\documentclass[11pt]{amsbook}
\usepackage{amsmath,amssymb,amsthm,mathrsfs}
\usepackage{xcolor}
\usepackage{enumitem}
\usepackage{array}
\usepackage{graphicx}
\usepackage[colorlinks=true,linkcolor=blue!55!black,citecolor=green!40!black,urlcolor=blue!55!black]{hyperref}

\numberwithin{equation}{chapter}
\newtheorem{theorem}{Theorem}[chapter]
\newtheorem{lemma}[theorem]{Lemma}
\newtheorem{proposition}[theorem]{Proposition}
\newtheorem{corollary}[theorem]{Corollary}
\newtheorem{maintheorem}{Theorem}

\theoremstyle{definition}
\newtheorem{definition}[theorem]{Definition}
\newtheorem{hypothesis}[theorem]{Hypothesis}
\theoremstyle{remark}
\newtheorem{remark}[theorem]{Remark}

\makeatletter
\def\paragraph{\@startsection{paragraph}{4}{\z@}{.5\linespacing}{-.5em}{\normalfont\bfseries}}
\makeatother

\newcommand{\Ps}{\Pi_{\mathrm s}}
\newcommand{\sT}{\mathsf T}

\newcommand{\Hs}{\mathcal H}
\newcommand{\Hmac}{\mathcal H_{\mathrm{mac}}}
\newcommand{\Hflu}{\mathcal H_{\mathrm{flu}}}
\newcommand{\Hgra}{\mathcal H_{\mathrm{gra}}}
\newcommand{\Hsrc}{\mathcal H_{\mathrm{src}}}
\newcommand{\Pz}{\mathcal P_0}
\newcommand{\V}{\mathcal V}
\newcommand{\Pmac}{\Pi_{\mathrm{mac}}}
\newcommand{\Pflu}{\Pi_{\mathrm{flu}}}
\newcommand{\Pgra}{\Pi_{\mathrm{gra}}}
\newcommand{\Pperp}{\Pi^{\perp}}
\newcommand{\sJ}{\mathsf J}
\newcommand{\sM}{\mathsf M}
\newcommand{\sA}{\mathsf A}
\newcommand{\sB}{\mathsf B}
\newcommand{\sR}{\mathsf R}
\newcommand{\sS}{\mathsf S}

\newcommand{\sN}{\mathsf N}

\newcommand{\sL}{\mathsf L}
\newcommand{\sP}{\mathsf P}
\newcommand{\Jmac}{\mathsf J_{\mathrm{mac}}}
\newcommand{\Jmic}{\mathsf J_{\mathrm{mic}}}
\newcommand{\Je}{\mathsf J_{\varepsilon}}
\newcommand{\Jz}{\mathsf J_{0}}
\newcommand{\Mz}{\mathsf M_0}

\newcommand{\Me}{\mathsf M_\varepsilon}
\newcommand{\Mff}{\mathsf M_{\!f\!f}}
\newcommand{\Ae}{\mathsf A_\varepsilon}
\newcommand{\Az}{\mathsf A_0}
\newcommand{\Se}{\mathsf S_\varepsilon}
\newcommand{\Kmic}{\mathsf K_{\mathrm{mic}}}
\newcommand{\Ke}{\mathsf K_\varepsilon}
\newcommand{\Ne}{\mathsf N_\varepsilon}
\newcommand{\Ns}{\mathsf N_{s}}
\newcommand{\Nb}{\mathsf N_{b}}
\newcommand{\Ee}{\mathsf E_\varepsilon}
\newcommand{\Be}{\mathsf B_\varepsilon}
\renewcommand{\Re}{\operatorname{Re}}
\newcommand{\Img}{\operatorname{Im}}
\newcommand{\dom}{\operatorname{dom}}
\newcommand{\ran}{\operatorname{ran}}
\newcommand{\nrm}[1]{\lVert #1\rVert}
\newcommand{\ip}[2]{\langle #1,#2\rangle}
\newcommand{\mb}{m_\square}
\newcommand{\Mb}{M_\square}
\newcommand{\E}{\mathbb E}
\newcommand{\R}{\mathbb R}
\newcommand{\C}{\mathbb C}
\newcommand{\Z}{\mathbb Z}
\newcommand{\T}{\mathbb T}
\newcommand{\fa}{\mathfrak a}
\newcommand{\fd}{\mathfrak d}

\title{A point-free theory of quantitative homogenization}

\author{Thuyen Dang}
\address{Department of Mathematics and Statistics, Purdue University Northwest, 2200 169th Street, Hammond, Indiana 46323, USA}
\email{ttdang@purdue.edu}

\author{Yuliya Gorb}
\address{National Science Foundation, Randolph Building, 401 Dulany Street, Alexandria, VA 22314, USA}
\email{ygorb@nsf.gov}

\author{Silvia Jim\'enez Bola\~nos}
\address{Department of Mathematics, Colgate University, 13 Oak Drive, Hamilton, New York 13346, USA}
\email{sjimenez@colgate.edu}

\date{}

\begin{document}
\frontmatter

\begin{abstract}
We develop an operator-theoretic framework for the quantitative homogenization of first-order systems $\sJ+\sM$, where $\sJ$ is skew-adjoint and $\sM$ is a bounded material law. The theory is point-free in the sense that no function is evaluated at a spatial point: all arguments use only the geometry of Hilbert spaces and the algebraic structure of the lifted operators. The problem is lifted by a family of unitary scalings. The effective coefficient is then the Schur complement of the material law with respect to the kernel of the microscopic derivative, and the corrector is the associated block elimination. A single error identity, with an arbitrary defect, converts any approximate antiderivative of the residual into a norm-resolvent bound. The medium enters only through a spectral gap of the microscopic derivative or through the low-frequency behaviour of finitely many explicit spectral measures. For Dirac operators with singular potentials, which model topological Anderson insulators, we give a point-free proof of $O(\varepsilon)$ norm-resolvent convergence for periodic potentials, derive the convergence of spectra, spectral gaps and spectral projections, and obtain the annealed rate $\varepsilon(\log 1/\varepsilon)^{1/2}$ for random potentials whose stationary antiderivative has a finite range of dependence. For divergence-form operators the framework recovers the sharp periodic rate and, for coefficients with a finite range of dependence, the known annealed rates for $L^2$ data; in one dimension we show that the rate $\varepsilon^{1/2}$ is attained. We also show that almost-sure norm-resolvent convergence on the whole space fails for random checkerboards, while localized almost-sure bounds hold.
\end{abstract}

\maketitle
\tableofcontents

\mainmatter

\chapter{Introduction}\label{sec:intro}

\section{The question}\label{sec:question}

Homogenization theory provides a rigorous mathematical framework for understanding how microscopic heterogeneities give rise to effective macroscopic behaviors. Over the past half-century, this field has matured into a vast and multifaceted discipline driven by applications ranging from composite materials to quantum mechanics. The foundational monographs \cite{BensoussanLionsPapanicolaou2011,JikovKozlovOleinik1994} and subsequent developments have introduced a rich tapestry of techniques, each tailored to specific multiscale environments, such as H-convergence and H-measure \cite{MuratTartar1997,Tartar1999,Tartar2009,Tartar2017}, $\Gamma$-convergence and G-convergence \cite{BraidesDefranceschi1998,ButtazzoDalMaso1980,ChiadoPiatDalMasoDefranceschi1990,DalMasoDellAntonio1991,Spagnolo1967}, two-scale convergence and periodic unfolding \cite{Allaire1992,BabadjianBaiaSantos2008,CioranescuDamlamianGriso2018,FrancuSvanstedt2012,NeukammVargaWaurick2021,Nguetseng1989}, Bloch wave analysis \cite{AllaireConca1996,AllairePiatnitski2005,BrianeVanninathan2014,CooperKamotskiSmyshlyaev2023,LiptonEtAl2022}, spectral germ \cite{BirmanSuslina2009,CherednichenkoCooper2016,CherednichenkoWaurick2018,CooperKamotskiSmyshlyaev2023,CooperWaurick2019,ZhikovPastukhova2016}, random integral \cite{BalBenAbdallahPuel2012,BalEtAl2008,BalPapanicolaouRyzhik2002,BalPinaud2012,ZhangBal2014a}, network approximation \cite{BerlyandBorceaPanchenko2005,BerlyandGorbNovikov2005,BerlyandKolpakov2001,BerlyandKolpakovNovikov2013,GerardVaretGirodrouxLavigne2022,PapanicolaouBorcea1998}, dynamical systems approach \cite{Kuehn2019,Roberts2015,Roberts2025}, viscosity solution \cite{CaffarelliSouganidis2010,MitakeTran2014,Tran2021,Tu2021}, and references therein.

Although the qualitative theory of homogenization is now well understood in highly general settings, the quantitative theory, which seeks explicit rates of convergence, has traditionally required rigid structural assumptions. In recent years, monumental progress has been made in quantitative and high-contrast homogenization \cite{ArmstrongDaniel2016,ArmstrongGloriaKuusi2016,ArmstrongKuusi2024,ArmstrongKuusi2025a,ArmstrongKuusiMourrat2016,ArmstrongKuusiMourrat2017,ArmstrongKuusiMourrat2019,ArmstrongEtAl2017,ArmstrongVenkatraman2026,ArmstrongMourrat2016,ArmstrongSmart2016,AvellanedaLin1987,AvellanedaLin1989,DuerinckxGloria2023,DuerinckxGloriaOtto2020,GerardVaretMasmoudi2012,GloriaNeukammOtto2015,GloriaMourrat2013,GloriaNeukammOtto2020,GloriaNeukammOtto2021,KenigPrange2015,Shen2018} and references therein. However, these triumphs rely heavily on deep spatial analysis, such as large-scale regularity theory, boundary layer estimates, and spatial smoothing operators. Consequently, the resulting convergence rates are often inextricably linked to the specific geometry, topology, or probability measure of the underlying space.

In his pioneering work \cite{MuratTartar1997,Tartar1990,Tartar1995,Tartar1999,Tartar2007,Tartar2009,Tartar2015,Tartar2017}, Luc Tartar envisioned a theory of homogenization that transcends the specificities of periodicity or probability. Inspired by this vision, we ask: \emph{Is it possible to develop a purely ``point-free'' theory of quantitative homogenization?} That is, can we derive explicit norm resolvent estimates and convergence rates using only the abstract geometry of Hilbert spaces and the algebraic structure of the multiscale operators, without ever evaluating a function at a spatial point?

In this memoir, we answer this question in the affirmative. We introduce an abstract, operator-theoretic framework that places stationary, non-stationary, periodic, quasi-periodic, and stochastic homogenization on an equal footing at the qualitative level (Theorems~\ref{thm:qual} and \ref{thm:ellrand}); at the quantitative level the media differ only through explicit spectral quantities. By stripping away the spatial and probabilistic scaffolding, we isolate the fundamental algebraic mechanisms that drive multiscale convergence. Our approach does not rely on regularity theory for the heterogeneous problem, Gelfand/Bloch transforms, or spatial smoothing operators (the verification of the hypotheses in the periodic example does use Fourier series on the torus, and the quantitative estimate requires regularity of the homogenized problem). Instead, it is built upon the geometry of the scaling operator, defined abstractly as a family of unitary transformations, and the algebraic properties of the Schur complement.

The starting point is a norm resolvent estimate for arbitrary pairs of operators (Theorem~\ref{thm:freq}), which isolates the commutator structure of the problem; the estimates actually used in the applications are its vector-wise refinement, the error identity of Lemma~\ref{lem:id}. We show that the effective macroscopic behavior emerges naturally from an orthogonal decomposition of the Hilbert space, where the effective coefficient is exactly the Schur complement and the abstract corrector is the associated block-elimination operator, which algebraically eliminates the microscopic fluctuations (Theorem~\ref{thm:corr}, Theorem~\ref{thm:qual}). To quantify the convergence, we study the commutator between the differential structure and the material properties. Solving the associated generalized Sylvester equation on a single vector, through an antiderivative of the residual (the flux corrector), we obtain an explicit algebraic representation of the homogenization error (Lemma~\ref{lem:id}, which leads to Theorems~\ref{thm:quant} and \ref{thm:gauge}).

The beauty of this point-free perspective is that it distills the physical differences between various media (e.g., stationary versus non-stationary or periodic versus stochastic) into a single type of mathematical object: the low-frequency behaviour of a spectral measure, either that of the microscopic derivative operator evaluated on the residual field (Lemma~\ref{lem:flux}) or, in stationary random media, those of finitely many stationary coefficient fields (hypotheses (E2) in Section~\ref{sec:stoch} and (R3) in Section~\ref{sec:random}). In the periodic setting, a spectral gap yields the classical $O(\varepsilon)$ rate. In the stochastic setting, the absence of a spectral gap (infrared divergence) is reflected in a cumulative density-of-states hypothesis on stationary coefficient fields, which encodes the mixing properties of the medium. Under it we obtain annealed norm-resolvent rates for random elliptic coefficients (Theorem~\ref{thm:ellrand}): $\varepsilon^{1/2}$, $\varepsilon(\log 1/\varepsilon)^{1/2}$ and $\varepsilon$ in dimensions $d=1$, $d=2$ and $d\ge 3$ when the hypothesis holds with exponent $\gamma=d$, and analogous rates for random singular Dirac potentials (Theorem~\ref{thm:randlift}). We do not derive the density-of-states hypothesis from mixing conditions; Remark~\ref{rem:E2} explains how it follows from known variance estimates for spatial averages \cite{ArmstrongKuusiMourrat2019,GloriaNeukammOtto2015,GloriaNeukammOtto2021}. Moreover, our framework requires bounded material operators; high-contrast media with unbounded contrast are outside its scope.

We do not claim that this abstract approach replaces the deep and intricate spatial theories mentioned above, which have recently revolutionized the field. Rather, we offer it as a complementary lens, one that clarifies the boundary between functional analysis and multiscale geometry. Because our approach bypasses the need for large-scale regularity of the heterogeneous problem, it is particularly well-suited for applications where such regularity is not fully developed or physically absent, such as in the study of topological insulators or photonic crystals or highly singular non-stationary media \cite{BalBeckerDrouot2024,BalEtAl2023,BalDang2024,Bal2026b,FigotinKuchment1996,HoeferWeinstein2011,QuinnBal2024,SagivWeinstein2022,Yafaev2008}, where the multiscale systems may lack {} subadditivity of energy quantities or a well-defined Brillouin zone. By exactly identifying which parts of homogenization are purely algebraic and which require spatial structure, we hope to provide a robust foundation for tackling multiscale problems in increasingly exotic settings.

\paragraph{Why an operator formulation}
The framework is formulated in the language of quantum theory: Hamiltonians are self-adjoint operators, and the objects of interest are their resolvents, functional calculus and spectral projections. Estimates in this language transfer directly to spectra, spectral gaps and projections (Corollary~\ref{cor:spec}), which are the quantities that enter topological classifications. Operator-norm estimates of this kind were developed for periodic media by spectral methods \cite{BirmanSuslina2009} and by shift methods \cite{ZhikovPastukhova2016}, both of which rely on periodicity. For random media on the whole space, almost-sure norm-resolvent convergence to a deterministic operator fails (Proposition~\ref{prop:noquenched}; see also \cite[Theorem~2.10]{Bal2026b}), so no extension of these methods can give it. The quantitative stochastic theory \cite{ArmstrongKuusiMourrat2019,GloriaNeukammOtto2020} handles random media at the level of solutions, correctors and energies. The point-free framework keeps the operator formulation in both settings: it gives norm-resolvent estimates where they hold, and annealed resolvent estimates in the random case, where the almost-sure estimates fail.

\section{Main results}\label{sec:main}

We state the four main results informally. Precise statements, with all hypotheses, are given in the chapters indicated.

\begin{maintheorem}[Effective operator and error identity; Chapters~\ref{sec:axioms}--\ref{sec:quan-theo}]\label{thmA}
Let $(\Hs,\sJ,\{\Se\})$ satisfy \ref{H1}--\ref{H3}, and let the material operator $\sM$ satisfy the multiplier hypothesis \ref{H4} with respect to a shift projection $\Ps$ as in \ref{S}. Let $\alpha>0$ and $\Ae=\sJ+\Se\sM\Se^{-1}+\alpha\Ps$.
\begin{enumerate}[label=(\roman*)]
\item $\Ae$ is invertible uniformly in $\varepsilon$. The effective coefficient $\Mz$ is the Schur complement of $\sM+\alpha\Ps$ with respect to $\Pz=\Hmac\oplus\Hflu$, and the corrector $\Kmic$ is the associated block-elimination operator (Theorem~\ref{thm:corr}). Under \ref{S0}, neither depends on $\alpha$.
\item For every macroscopic datum $f$, $\Ae^{-1}f-(I+\Ke)\Az^{-1}f\to0$ (Theorem~\ref{thm:qual}).
\item For every $\Phi$ in the common domain of $\Jmac$ and $\Jmic$, with defect $\varrho:=\Jmic\Phi-\Ns u$ relative to the residual $\Ns u$ of $u=\Az^{-1}f$,
\[
\nrm{\Ee f}\le\varepsilon C_*\nrm\Phi+C_A\big(\varepsilon\nrm{\Jmac\Phi}+\nrm\varrho\big)
\]
(Lemma~\ref{lem:id}). Under a spectral gap of $\Jmic$ and a good gauge this gives $O(\varepsilon)$ for data in the source space, and the rate $\varepsilon$ is then sharp for the lifted error (Corollary~\ref{cor:gaugegap}, Theorem~\ref{thm:lower}). Without a gap, the rate is controlled by the low-frequency behaviour of explicit spectral measures (Theorem~\ref{thm:quant}).
\end{enumerate}
\end{maintheorem}

\begin{maintheorem}[Dirac operators and topological Anderson insulators; Chapter~\ref{sec:top}]\label{thmB}
Let $\mathsf H_\varepsilon=D\cdot\sigma+(m+\beta\Delta)\sigma_3+\varepsilon^{-1}\rho(x)W(x/\varepsilon)$ on $L^2(\R^2;\C^2)$, with $\rho\in W^{2,\infty}$, and let $\mathsf H_0=D\cdot\sigma+(m+\beta\Delta)\sigma_3+\beta^{-1}\rho^2\Theta$ be the effective Hamiltonian.
\begin{enumerate}[label=(\roman*)]
\item If $W$ is periodic, belongs to $L^p(\T^2)$ for some $p>2$ (for instance, $W$ is H\"older continuous) and has mean zero, then $\nrm{(\mathsf H_\varepsilon-z)^{-1}-(\mathsf H_0-z)^{-1}}\le C_z\varepsilon$ for every $z\in\C\setminus\R$, together with a first-order corrector in $H^1$ (Corollary~\ref{cor:physD}). Consequently, spectra converge locally in the Hausdorff sense at rate $O(\varepsilon)$, spectral gaps persist, and $f(\mathsf H_\varepsilon)\to f(\mathsf H_0)$ and $1_{(a,b)}(\mathsf H_\varepsilon)\to1_{(a,b)}(\mathsf H_0)$ in norm at rate $O(\varepsilon)$ (Corollary~\ref{cor:spec}).
\item If $W$ is a stationary random field with a bounded, curl-free, stationary antiderivative (this forces $\int c_W=0$; see Remark~\ref{rem:Rhyp}), the resolvents converge in the annealed sense, with an explicit rate under a density-of-states hypothesis on finitely many coefficient fields (Theorem~\ref{thm:randlift}, Corollary~\ref{cor:randphys}). If the antiderivative has a finite range of dependence, the hypothesis holds and the annealed rate is $\varepsilon(\log1/\varepsilon)^{1/2}$ (Proposition~\ref{prop:R3}).
\end{enumerate}
\end{maintheorem}

\begin{maintheorem}[Divergence-form operators; Chapter~\ref{sec:ell}]\label{thmC}
Let $\mathsf H_\varepsilon=-\nabla\cdot a(x/\varepsilon)\nabla+1$ on $L^2(\R^d)$ with $a$ bounded and uniformly elliptic.
\begin{enumerate}[label=(\roman*)]
\item If $a$ is periodic, then $\nrm{\mathsf H_\varepsilon^{-1}-\mathsf H_0^{-1}}\le C\varepsilon$, with a first-order corrector for the gradient (Theorem~\ref{thm:31}, Corollary~\ref{cor:unlift}), and the lifted rate $\varepsilon$ is sharp (Corollary~\ref{cor:lowerell}); the latter concerns the lifted error, which carries no first-order corrector in the state component. Sharpness for the physical operators is classical \cite{BirmanSuslina2009,ZhikovPastukhova2016} and is not reproved here.
\item If $a$ is stationary and ergodic, the lifted resolvents converge in the annealed sense for $L^2$ data, uniformly over the unit ball of data and with no mixing hypothesis (Corollary~\ref{cor:annqual}), with explicit rates under a density-of-states hypothesis (Theorem~\ref{thm:ellrand}). For coefficients with a finite range of dependence, symmetric or not, the hypothesis holds, and the annealed rates are $\varepsilon^{1/2}$, $\varepsilon(\log1/\varepsilon)^{1/2}$ and $\varepsilon$ in $d=1$, $d=2$ and $d\ge3$ (Theorem~\ref{thm:frd}).
\item In one dimension, under a conditional finite range of dependence, $(\E\nrm{u_\varepsilon-u_0}^2)^{1/2}$ is bounded above and below by multiples of $\varepsilon^{1/2}$ (Theorem~\ref{thm:lower1d}).
\end{enumerate}
\end{maintheorem}

\begin{maintheorem}[Limits of the method; Chapter~\ref{sec:limits}]\label{thmD}
\begin{enumerate}[label=(\roman*)]
\item For a random checkerboard, almost surely $\nrm{(\mathsf H^\omega_\varepsilon)^{-1}-\mathsf H_0^{-1}}_{L^2(\R^d)\to L^2(\R^d)}\ge c_0>0$ for every $\varepsilon>0$ (Proposition~\ref{prop:noquenched}). Annealed bounds nevertheless yield almost-sure bounds for resolvent differences composed with localizing Hilbert--Schmidt operators (Proposition~\ref{prop:quenched}).
\item On bounded $C^3$ domains with Dirichlet conditions, the method gives the rate $\varepsilon^{1/2}$ for the lifted error with $H^1$ data, the classical boundary-layer order in the absence of boundary-layer correctors (Proposition~\ref{prop:bdd}).
\end{enumerate}
\end{maintheorem}

\section{What is new and what is known}\label{sec:new}

We separate the results that are new from those that the framework recovers.

\begin{itemize}
\item \emph{Theorem~\ref{thmA}.} The identification of the effective coefficient and the corrector as a Schur complement and a block elimination in an abstract lifted space, the multiplier hypothesis \ref{H4} that treats dissipative and chiral materials by one argument, and the error identity with an arbitrary defect are, to our knowledge, new in this form. The lift itself is close to the unfolding operator \cite{CioranescuDamlamianGriso2008,CioranescuDamlamianGriso2018}, and Theorem~\ref{thm:qual} is an abstract counterpart of H-convergence \cite{MuratTartar1997}. The Schur-complement description of an effective operator relative to an orthogonal decomposition of a Hilbert space is not new in itself: it is a basic structure of the abstract theory of composites (Milton's $Z$-problem on $\mathcal U\oplus\mathcal E\oplus\mathcal J$) \cite{Milton2002,Grabovsky2016,Welters2025}, and block-operator descriptions of coefficients relative to $\ran\oplus\ker$ underlie Waurick's nonlocal H-convergence and the Schur topology \cite{Waurick2018,BuchingerFranzSkrepekWaurick2025,BuchingerWaurick2026}. What is specific here is the scaling family $\Se$, the partial shift $\alpha\Ps$, the multiplier $\sT$, and the use of the Schur complement inside a convergence and error theory; see Section~\ref{sec:related}.
\item \emph{Theorem~\ref{thmB}.} Part (i) recovers the norm-resolvent convergence of Bal and Dang \cite{BalDang2024,BalDangCorrigendum}, with the same rate $O(\varepsilon)$, for every $z\in\C\setminus\R$; the proof uses neither Bloch--Floquet theory nor an asymptotic expansion. The same rate is proved by Bal \cite[Theorem~2.1]{Bal2026b} for constant-coefficient elliptic systems of any order and for fluctuations of every order below that of the operator, by a smoothed-corrector method and with a smooth envelope $\rho$. Our contribution to part (i) is therefore the point-free proof, which requires only $\rho\in W^{2,\infty}$ and $W\in L^p(\T^2)$, $p>2$. The spectral consequences in Corollary~\ref{cor:spec} follow from norm-resolvent convergence by standard arguments. Part (ii) is, to our knowledge, new. For random alloys, \cite{Bal2026b} proves deterministic estimates towards operators with configuration-dependent local effective coefficients, and the stability of interface indices along cuts that avoid rare regions; it does not state annealed rates of convergence to a deterministic homogenized operator. The class of potentials covered by part (ii) is restrictive: the bounded, curl-free stationary antiderivative forces $\int c_W=0$ (Remark~\ref{rem:Rhyp}). The explicit model of Remark~\ref{rem:matern}, $W=\Delta\psi$ with $\psi$ a sum of bumps centred at the points of a hard-core point process, is an alloy whose single-site potential has vanishing low moments, and is in this sense of the type considered in \cite{Bal2026b}. The models are not the same: the alloys of \cite{Bal2026b} are lattice sums $\sum_{j\in\Z^2}\omega_jw(\cdot-j)$ with $w\in C_c^\infty$ and vanishing moments up to an order fixed by the order of the operator, or finite-range H\"older fields, and point-process alloys are not considered there. The results are also of a different nature. The estimates of \cite[Theorem~2.8]{Bal2026b} are deterministic and uniform over configurations, with algebraic rates no better than $\varepsilon^{1/3}$, and compare $\mathsf H_\varepsilon$ with an operator whose effective coefficient $M^\varepsilon_\omega$ depends on the configuration. The deterministic Born matrix enters through a concentration estimate for $M^\varepsilon_\omega(x)$ at fixed $x$ and through the index theorems \cite[Theorems~2.13 and~2.14]{Bal2026b}. A concentration estimate at fixed $x$ does not control $\sup_x|M^\varepsilon_\omega(x)-M_{\mathrm{Born}}|$ on the whole plane, which is what a resolvent comparison in operator norm would require, and \cite{Bal2026b} states no annealed rate towards a deterministic operator. Theorem~\ref{thmB}(ii) provides such a rate, $\varepsilon(\log1/\varepsilon)^{1/2}$, for the class described above.
\item \emph{Theorem~\ref{thmC}.} The rates themselves are known. The periodic $O(\varepsilon)$ estimates are due to Birman and Suslina \cite{BirmanSuslina2009} and to Zhikov and Pastukhova \cite{ZhikovPastukhova2016}. The annealed rates for coefficients with a finite range of dependence follow from the central-limit-scale bounds of Gloria and Otto \cite{GloriaOtto2016} and, for symmetric coefficients, of Armstrong, Kuusi and Mourrat \cite{ArmstrongKuusiMourrat2017}, which we use as input (Remark~\ref{rem:frd}). What is new is the route: the same error identity gives the periodic and random cases, with the medium entering only through explicit spectral quantities.
\item \emph{Theorem~\ref{thmD}.} These results delimit the method. Part (ii) is the classical boundary-layer order; the sharp $O(\varepsilon)$ estimate in $L^2$ on bounded domains is known from duality arguments \cite{Shen2018} and is not reproduced here.
\end{itemize}

The framework does not prove large-scale regularity, corrector bounds or fluctuation estimates, and it requires bounded material operators. Chapter~\ref{sec:limits} collects these limitations.

\section{Relation to other operator-theoretic approaches}\label{sec:related}

Operator-theoretic ideas have a long history in homogenization. We compare the present framework with the closest approaches, emphasizing what each assumes and what it delivers. Table~\ref{tab:related} summarizes the comparison.

\paragraph{Periodic operator estimates.} Birman and Suslina \cite{BirmanSuslina2009} interpret periodic homogenization as a \emph{spectral threshold effect}. After a Floquet--Bloch (Gelfand) decomposition, the operator becomes a family of operators on the torus parametrized by the quasimomentum. The effective matrix is read off from the \emph{spectral germ} of the lowest band at zero quasimomentum, via analytic perturbation theory. This yields sharp $O(\varepsilon)$ operator-norm estimates, with correctors in the energy norm. Zhikov and Pastukhova \cite{ZhikovPastukhova2016} obtain estimates of the same order by the \emph{shift method}: a first-order approximation with Steklov smoothing, combined with translations in the fast variable. In the present framework the low-frequency mechanism is encoded instead by:
\begin{itemize}
\item the spectral gap of the microscopic derivative $\Jmic$ on $\Hgra$ (hypothesis \ref{G1});
\item the effective coefficient as the Schur complement of the material operator with respect to $\Pz=\Hmac\oplus\Hflu$ (Theorem~\ref{thm:corr});
\item the rate from the error identity of Lemma~\ref{lem:id}.
\end{itemize}
No fibre decomposition is used. The return to the physical operators in Corollary~\ref{cor:unlift} borrows the translation idea of \cite{ZhikovPastukhova2016}.

\paragraph{Abstract frameworks.} Picard's theory of evolutionary equations \cite{Picard2009} writes many linear systems of mathematical physics in the form $\partial_t\sM+\sA$, with $\sA$ skew-self-adjoint and $\sM$ a material law. This is the same algebraic structure as \eqref{eq:9}: a skew-adjoint differential part plus an accretive material operator. Within that framework:
\begin{itemize}
\item Neukamm, Varga and Waurick \cite{NeukammVargaWaurick2021} develop two-scale homogenization for abstract linear time-dependent problems (qualitative convergence);
\item Cherednichenko and Waurick \cite{CherednichenkoWaurick2018} and Cooper and Waurick \cite{CooperWaurick2019} derive resolvent estimates for periodic problems by combining it with a fibre decomposition;
\item high-contrast periodic problems are treated by operator methods in \cite{CherednichenkoCooper2016,CooperKamotskiSmyshlyaev2023}.
\end{itemize}
Compared with these works, the present framework does three things:
\begin{enumerate}[label=(\roman*)]
\item it replaces the fibre decomposition by an abstract family of unitary scalings $\Se$ and an orthogonal decomposition of the lifted space;
\item it identifies the effective coefficient and the corrector purely algebraically, as a Schur complement and a block elimination;
\item it obtains rates from an error identity with an arbitrary defect, which also covers media without any fibre decomposition, such as stationary random media.
\end{enumerate}
On the other hand, the fibre methods give sharper information in the periodic case, for instance spectral asymptotics of individual bands, and our framework presently requires bounded material operators (no high contrast).

\paragraph{Stochastic homogenization.} The qualitative theory \cite{JikovKozlovOleinik1994,PapanicolaouVaradhan1981} and the quantitative theories based on functional inequalities \cite{GloriaNeukammOtto2015,GloriaNeukammOtto2020,GloriaNeukammOtto2021,GloriaOtto2011} or on subadditivity and renormalization \cite{ArmstrongKuusi2024,ArmstrongKuusiMourrat2017,ArmstrongKuusiMourrat2019} establish large-scale regularity, optimal corrector bounds and fluctuation estimates. The present framework does not reprove such estimates. It converts a \emph{cumulative spectral bound} on finitely many stationary coefficient fields (the corrector gradient and the flux, hypothesis \ref{E2}) into annealed norm-resolvent rates for $L^2$ data. That bound follows from variance estimates for spatial averages (Proposition~\ref{prop:varDOS}), which are of the type obtained in the works above. Two technical points connect with the literature:
\begin{itemize}
\item In the literature, the corrector equation itself is usually regularized by a massive term, $(T^{-1}-\nabla\cdot a\nabla)\phi_T=\nabla\cdot a e$ \cite{PapanicolaouVaradhan1981,GloriaOtto2011}. Here the regularization $(\theta-\Delta)^{-1}$ acts only through the generators of the group action on \emph{fixed} stationary fields, and the resulting defect is absorbed by Lemma~\ref{lem:id}. Gloria and Otto \cite[(48)--(51)]{GloriaOtto2016} also regularize the flux corrector by a massive term: $T^{-1}\sigma_{T,jk}-\Delta\sigma_{T,jk}=\partial_jq_{T,k}-\partial_kq_{T,j}$, where $q_T=a(\nabla\phi_T+e)$ is the flux of the massive corrector $\phi_T$, and the defect $\nabla\cdot\sigma_T-(q_T-\langle q_T\rangle)$ is tracked by an auxiliary field. The difference here is that both the potential and the skew flux corrector are regularized from the exact stationary fields $g^j$ and $B$, and that no regularized corrector equation is solved.
\item That the low-frequency behaviour of a spectral measure governs the rate is in the spirit of the quantitative Kipnis--Varadhan theory \cite{GloriaMourrat2013}.
\end{itemize}

\paragraph{H-convergence.} The abstract H-system of Definition~\ref{def:homo-syst} is a point-free analogue of the H-convergence setting of Murat and Tartar \cite{MuratTartar1997}. Theorem~\ref{thm:qual} is its abstract qualitative counterpart, here in the strong corrector form.

\paragraph{Abstract theory of composites, nonlocal H-convergence and the Schur topology.} Two bodies of work describe effective coefficients by Schur complements in a Hilbert space, and we do not claim that structure as new.
\begin{itemize}
\item In the abstract theory of composites \cite{Milton2002,Grabovsky2016}, a Hilbert space carries an orthogonal decomposition $\mathcal U\oplus\mathcal E\oplus\mathcal J$ into constant fields, mean-zero gradients and mean-zero divergence-free fields, and the effective tensor is defined by the $Z$-problem; it admits an operator Schur-complement representation \cite{Welters2025}. Our decomposition $\Hmac\oplus\Hflu\oplus\Hgra$ plays the role of $\mathcal U\oplus\mathcal J\oplus\mathcal E$: in the elliptic example $\sM$ acts on fluxes through $a^{-1}$, so $\Mz$ is the effective operator of the dual (resistivity) problem.
\item Waurick's nonlocal H-convergence \cite{Waurick2018} represents coefficients as block operator matrices relative to $\ran\oplus\ker$ of an abstract closed complex, and led to the Schur topology, in which continuous dependence of the solution operators of evolutionary equations on the coefficients is established \cite{BuchingerFranzSkrepekWaurick2025,BuchingerWaurick2026}. Its aim, a topology on coefficients that captures homogenization by operator-theoretic means only, is close to the point-free programme of this memoir, and Theorem~\ref{thm:qual} should be read alongside it.
\end{itemize}
Relative to these works, the present framework adds the unitary scaling family $\Se$ (which produces the oscillating coefficient from a fixed $\sM$), the partial shift $\alpha\Ps$ and the multiplier $\sT$ (which cover non-coercive chiral materials), and, above all, the quantitative error identity of Lemma~\ref{lem:id}.

\begin{table}[ht]\centering\footnotesize
\renewcommand{\arraystretch}{1.25}
\begin{tabular}{>{\raggedright\arraybackslash}p{2.2cm}>{\raggedright\arraybackslash}p{2.4cm}>{\raggedright\arraybackslash}p{2.7cm}>{\raggedright\arraybackslash}p{3.5cm}}
\hline
Approach & Media & Main tool & Typical quantitative output\\\hline
Spectral germ \cite{BirmanSuslina2009} & periodic & Floquet--Bloch, analytic perturbation of the lowest band & sharp $O(\varepsilon)$ operator estimates, with correctors\\
Shift method \cite{ZhikovPastukhova2016} & periodic & Steklov smoothing, translations & $O(\varepsilon)$ operator estimates\\
Evolutionary equations \cite{CherednichenkoWaurick2018,CooperWaurick2019,NeukammVargaWaurick2021} & periodic (quantitative); abstract (qualitative) & skew-self-adjoint plus material law; fibre decomposition & resolvent estimates in the periodic case\\
Stochastic quantitative theory \cite{ArmstrongKuusiMourrat2019,GloriaNeukammOtto2015,GloriaNeukammOtto2021,GloriaOtto2011} & stationary random & functional inequalities, renormalization, large-scale regularity & optimal fluctuation and error bounds (annealed and quenched)\\
This memoir & periodic, stationary random, quasi-periodic, non-stationary singular (Dirac) & unitary scaling, Schur complement, error identity with defect & periodic: $O(\varepsilon)$; random: annealed rates under \ref{E2} or (R3), verified in the models of Sections~\ref{sec:verifyE2} and \ref{sec:verifyR3}\\\hline
\end{tabular}
\caption{Comparison with related operator-theoretic approaches (schematic).}\label{tab:related}
\end{table}

\section{Organization and reading guide}\label{sec:guide}

Part~I develops the framework. Chapter~\ref{sec:axioms} introduces the axioms, the Schur complement and the qualitative theory. Chapter~\ref{sec:quan-theo} proves the error identity and the quantitative estimates.

Part~II applies the framework. Chapter~\ref{sec:ell} treats divergence-form operators and Chapter~\ref{sec:top} treats Dirac operators. The Dirac results are the main new applications, but Chapter~\ref{sec:top} uses the lift and the random spectral calculus set up in Chapter~\ref{sec:ell}, so we present the elliptic case first. Chapter~\ref{sec:further} treats Maxwell's system, coefficients without a sign, and evolutionary equations.

Part~III discusses structure and limits. Chapter~\ref{sec:flow} describes the homogenization flow and its relation to Floquet--Bloch theory. Chapter~\ref{sec:limits} collects what the method cannot do, and Chapter~\ref{sec:conclusion} summarizes the results. Appendix~\ref{app:1d} works out every object of the framework in dimension one and can be read alongside Chapters~\ref{sec:axioms}--\ref{sec:quan-theo}. Appendix~\ref{app:num} contains numerical experiments.

A reader interested only in the abstract theory can read Chapters~\ref{sec:axioms}--\ref{sec:quan-theo} and Appendix~\ref{app:1d}. A reader interested in topological insulators can read Section~\ref{sec:2.2}, Lemma~\ref{lem:id}, Section~\ref{sec:3.1} and Chapter~\ref{sec:top}.

\paragraph{Notational convention}
We use the same font for the elements in the same category.
\begin{itemize}
\item Spaces and subspaces: $\Hs$, $\Pz$, etc.
\item Operators: $\sJ$, $\sM$, $\sA$, etc., except that the identity is denoted by $I$.
\item Projections: $\Pi$.
\end{itemize}

\part{The Framework}

\chapter{The Framework: H-Systems and the Effective Operator}\label{sec:axioms}\chaptermark{H-Systems and the Effective Operator}

We will split into two parts corresponding to two levels of abstraction so that the boundary of functional analysis and multiscale analysis becomes clearer. At the first level, the results in Section~\ref{sec:2.1} hold for an arbitrary pair of two operators $\Ae$ and $\Az$. Nevertheless, we will use the language of homogenization to motivate the definitions and results in Section~\ref{sec:2.1}. At the second level, we specialize these results by adding requirements from multiscale analysis to develop a point-free theory of homogenization: the qualitative theory in Section~\ref{sec:2.2} and the quantitative theory in Chapter~\ref{sec:quan-theo}.

\section{A Norm Resolvent Estimate for Two Arbitrary Operators}\label{sec:2.1}

Let $(\Hs,\ip{\cdot}{\cdot})$ be a Hilbert space with induced norm $\nrm{\cdot}$. Suppose the following axioms hold:
\begin{enumerate}[label=(A\arabic*),leftmargin=3.2em]
\item\label{A1} \emph{The Differential Structure:} Let $\sJ:\dom\sJ\subset\Hs\to\Hs$ be a skew-adjoint operator ($\sJ^*=-\sJ$).
\item\label{A2} \emph{The Material Properties:} Let $\Me$ and $\Mz$ be bounded and accretive operators on $\Hs$. Suppose there exist $M>0$ and $m\ge 0$ such that for all $u\in\Hs$:
\begin{equation}\label{eq:17}
\begin{gathered}\nrm{\Me}_{\Hs\to\Hs}\le M,\quad \nrm{\Mz}_{\Hs\to\Hs}\le M\quad\text{and}\\ \Re\ip{\Me u}{u}\ge m\nrm{u}^2,\quad \Re\ip{\Mz u}{u}\ge m\nrm{u}^2 .\end{gathered}
\end{equation}
\end{enumerate}

\begin{definition}\label{def:oper-a}
Define the following operators regularized by $\alpha\ge 0$ with $m+\alpha>0$:
\begin{equation}\label{eq:9}
\Ae=\sJ+\Me+\alpha I,\qquad \Az=\sJ+\Mz+\alpha I .
\end{equation}
\end{definition}

In the context of homogenization, $\Ae$ and $\Az$ are the finescale/oscillatory and homogenized/effective operators shifted by $\alpha I$, respectively. The multiscale problem is written as
\begin{equation}\label{eq:71}
\Ae u_\varepsilon=f .
\end{equation}
In Section~\ref{sec:2.2} the scalar shift $\alpha I$ is replaced by a \emph{partial shift} $\alpha\Ps$ acting on the state component only, and coercivity of the material is replaced by a weaker multiplier condition \ref{H4}. In the examples, the shift $\alpha\Ps$ is the spectral parameter of the physical operator (Chapter~\ref{sec:ell}: $\alpha=1$ gives $-\nabla\cdot a_\varepsilon\nabla+1$; Chapter~\ref{sec:top}: $\alpha>0$ gives the resolvent at $z=i\alpha$). Note that our Hilbert space $\Hs$ is not the (usually Sobolev/Lebesgue) space associated with the multiscale system of partial differential equations, but rather a global space that also takes into account the microscopic fluctuation. Taking the scalar second order divergence operator $-\operatorname{div}(a_\varepsilon\nabla\cdot)$ as an example, then $\Hs$ is $L^2(\R^d\times\T^d;\R^{d+1})$ for the periodic setting or $L^2(\R^d\times\Omega;\R^{d+1})$ for the stochastic setting; see Chapter~\ref{sec:ell} for the concrete setup of this example. The effective operator only governs the macroscopic behavior and is defined on the proper subspace $\Hmac$ (see \ref{H1} below); in Section~\ref{sec:2.2} no extension of it to $\Hmac^\perp$ is needed.

Since Theorem~\ref{thm:freq} only concerns general pairs of operators $\Me$ and $\Mz$ without using any homogenization structure, one can safely absorb the shift $\alpha I$ into $\Me$ and $\Mz$ and assume $m>0$ in \ref{A2}. Section~\ref{sec:2.2} does not use \ref{A2}; there, invertibility is obtained from the multiplier hypothesis \ref{H4} (Theorem~\ref{thm:corr}).

Let $\V:=\dom(\sJ)$. We equip $\V$ with the graph inner product:
\begin{equation}\label{eq:158}
\ip{u}{v}_{\V}:=\ip{u}{v}+\ip{\sJ u}{\sJ v},
\end{equation}
which induces the graph norm:
\begin{equation}\label{eq:161}
\nrm{u}_{\V}^2=\nrm{u}^2+\nrm{\sJ u}^2 .
\end{equation}
Because $\sJ$ is a closed operator, $\V$ is a Hilbert space. Because $\sJ$ is densely defined on $\Hs$, we obtain the Gelfand triple
\begin{equation}\label{eq:159}
\V\hookrightarrow\Hs\hookrightarrow\V^*,
\end{equation}
where $\V^*$ is the dual space of $\V$, and we identify $\Hs\cong\Hs^*$. By definition of the graph norm $\nrm{\cdot}_{\V}$, we have $\sJ$ is a bounded operator from $\V\to\Hs$.

We use skew-adjointness to extend $\sJ$ uniquely to a bounded operator from $\Hs\to\V^*$, where we still denote the extension by $\sJ$. For any $u\in\Hs$, we define $\sJ u\in\V^*$ via the duality pairing
\begin{equation}\label{eq:160}
\ip{\sJ u}{v}_{\V^*,\V}:=-\ip{u}{\sJ v}.
\end{equation}
This extension is bounded because $|\ip{u}{\sJ v}|\le\nrm{u}\nrm{\sJ v}\le\nrm{u}\nrm{v}_{\V}$. Thus, $\nrm{\sJ}_{\Hs\to\V^*}\le 1$.

\begin{lemma}\label{lem:inve-coer}
Assume \ref{A1}, \ref{A2}, and $m+\alpha>0$. Then $\Ae$ and $\Az$ are boundedly invertible, with
\begin{equation}\label{eq:19}
\nrm{\Ae^{-1}}_{\Hs\to\Hs}\le\frac{1}{m+\alpha}\quad\text{and}\quad\nrm{\Az^{-1}}_{\Hs\to\Hs}\le\frac{1}{m+\alpha}.
\end{equation}
Furthermore, $\Ae^{-1}$ and $\Az^{-1}$ are bounded operators from $\Hs\to\V$ and extend to bounded operators from $\V^*\to\Hs$, with
\begin{equation}\label{eq:162}
\nrm{\Ae^{-1}}_{\Hs\to\V},\ \nrm{\Az^{-1}}_{\Hs\to\V}\le C_{\V}\quad\text{and}\quad\nrm{\Ae^{-1}}_{\V^*\to\Hs},\ \nrm{\Az^{-1}}_{\V^*\to\Hs}\le C_{\V},
\end{equation}
where $C_{\V}:=\Big[\frac{1}{(m+\alpha)^2}+\Big(1+\frac{M+\alpha}{m+\alpha}\Big)^2\Big]^{1/2}$.
\end{lemma}

\begin{proof}
For each shift $\alpha I$, we can replace $m,M$ by $m+\alpha$, $M+\alpha$. Without loss of generality, we absorb it into the material operator and assume $m>0$.

1. We drop the subscripts and prove the results for a generic operator $\sA=\sM+\sJ$, where $\sM$ is bounded and accretive with constant $m$, and $\sJ$ is closed densely defined, and skew-adjoint. Because $\sM$ is a bounded operator, $\dom(\sA)=\dom(\sJ)$ and $\sA$ is closed.

2. We show $\sA$ is injective and $\ran(\sA)$ is closed. Let $u\in\dom(\sA)$. Because $\sJ$ is skew-adjoint, $\ip{\sJ u}{u}=\ip{u}{\sJ^*u}=-\ip{u}{\sJ u}=-\overline{\ip{\sJ u}{u}}$. Thus, $\Re\ip{\sJ u}{u}=0$. Therefore, by \eqref{eq:9} and \eqref{eq:17},
$\Re\ip{\sA u}{u}=\Re\ip{\sM u}{u}+\Re\ip{\sJ u}{u}\ge m\nrm{u}^2$.
By the Cauchy--Schwarz inequality, $\Re\ip{\sA u}{u}\le\nrm{\sA u}\nrm{u}$. Thus,
\begin{equation}\label{eq:18}
\nrm{\sA u}\ge m\nrm{u}\qquad\forall u\in\dom(\sA).
\end{equation}
We conclude that $\sA$ is injective. Let $v_n\in\ran(\sA)$ with $v_n\to v$ in $\Hs$, and $v_n=\sA u_n$. By \eqref{eq:18}, $\nrm{v_n-v_k}\ge m\nrm{u_n-u_k}$, so $(u_n)$ is Cauchy and $u_n\to u$. Because $\sA$ is closed, $u\in\dom(\sA)$ and $\sA u=v$. Thus $\ran(\sA)$ is closed.

3. Let $v\in\ran(\sA)^\perp$. Then $v\in\dom(\sA^*)$ and $\sA^*v=0$. Since $\sA^*=\sM^*-\sJ$,
$0=\Re\ip{\sA^*v}{v}=\Re\ip{\sM v}{v}\ge m\nrm{v}^2$, so $v=0$ and $\sA$ is surjective.

4. Since $\sA$ is a bijection, \eqref{eq:18} gives $\nrm{\sA^{-1}f}\le m^{-1}\nrm{f}$, which implies \eqref{eq:19}.

5. We now prove \eqref{eq:162}. With $u=\sA^{-1}f$ we have $\sJ u=f-\sM u$, so
\begin{equation}\label{eq:11b}
\nrm{\sJ u}\le\nrm{f}+M\nrm{u}\le\Big(1+\frac{M}{m}\Big)\nrm{f},
\end{equation}
and therefore
\begin{equation}\label{eq:12b}
\nrm{u}_{\V}^2\le\Big[\frac{1}{m^2}+\Big(1+\frac{M}{m}\Big)^2\Big]\nrm{f}^2 .
\end{equation}
Restoring the shift $\alpha I$ yields the formula for $C_{\V}$. The dual bound follows by duality: for $f\in\Hs$, $\nrm{\sA^{-1}f}=\sup_{\nrm g\le 1}|\ip{f}{(\sA^*)^{-1}g}|\le\nrm{f}_{\V^*}\nrm{(\sA^*)^{-1}}_{\Hs\to\V}$, and $\sA^*=\sM^*-\sJ$ satisfies the same hypotheses with the same constants. Hence $\sA^{-1}$ extends by density of $\Hs$ in $\V^*$ to a bounded operator $\V^*\to\Hs$ with norm $\le C_{\V}$, which coincides with $\sA^{-1}$ on $\Hs$.
\end{proof}

We will introduce another axiom to derive a norm resolvent estimate for $\Ae$ and $\Az$.
\begin{enumerate}[label=(A\arabic*),leftmargin=3.2em,start=3]
\item\label{A3} \emph{The Commutator Equation in the Dual Space:} Let the material difference be defined as
\begin{equation}\label{eq:54}
\Ne:=\Mz-\Me ,
\end{equation}
and $\Pmac$ be an orthogonal projection on $\Hs$. Define $\Pperp:=I-\Pmac$. Suppose for each $\varepsilon>0$ there exist a linear operator $\Be$ which is bounded as an operator $\V\to\Hs$ and, after (unique) extension, as an operator $\Hs\to\V^*$, the two realizations agreeing on $\V$, and an operator $\sR_\varepsilon$ such that $\sR_\varepsilon\Az^{-1}:\Hs\to\V^*$ is bounded, such that the following commutator equation holds in $\mathcal L(\V,\V^*)$:
\begin{equation}\label{eq:11}
\Ne\Pmac=[\sJ,\Be]+\sR_\varepsilon .
\end{equation}
Here, $[\sJ,\Be]$ denotes the commutator $[\sJ,\Be]:=\sJ\Be-\Be\sJ$; so \eqref{eq:11} becomes
\begin{equation}\label{eq:10}
(\Mz-\Me)\Pmac=\sJ\Be-\Be\sJ+\sR_\varepsilon ,
\end{equation}
which is a generalized Sylvester equation.
\end{enumerate}

\begin{theorem}[Commutator Norm Resolvent Estimate]\label{thm:freq}
Suppose \ref{A1}, \ref{A2}, and \ref{A3}. Then
\begin{multline}\label{eq:12}
\nrm{\Ae^{-1}-\Az^{-1}}_{\Hs\to\Hs}\le C_1\big(\nrm{\Be}_{\V\to\Hs}+\nrm{\Be}_{\Hs\to\V^*}\big)\\+C_2\nrm{\sR_\varepsilon\Az^{-1}}_{\Hs\to\V^*}+C_3\nrm{\Pperp\Az^{-1}}_{\Hs\to\Hs},
\end{multline}
where the constants are defined as
\[
C_1:=C_{\V}\Big(1+\frac{M+\alpha}{m+\alpha}\Big),\qquad C_2:=C_{\V},\qquad C_3:=\frac{2M}{m+\alpha}.
\]
Here $C_{\V}$ is defined as in Lemma~\ref{lem:inve-coer}.
\end{theorem}

Theorem~\ref{thm:freq} splits $\Hs$ into $\Pmac\Hs$ and its complement. The genuine splitting into low and high frequencies, in the spirit of the proofs of the Rellich--Kondrachov theorem and of the div--curl lemma, appears in Theorem~\ref{thm:quant}(ii) through the macroscopic cut-off $\sP_R$.

\begin{proof}
1. By the resolvent identity and \eqref{eq:11}:
\begin{equation}\label{eq:11c}
\begin{aligned}
\Ae^{-1}-\Az^{-1}&=\Ae^{-1}(\Az-\Ae)\Az^{-1}=\Ae^{-1}\Ne\Az^{-1}\\
&=\Ae^{-1}\Ne\Pmac\Az^{-1}+\Ae^{-1}\Ne\Pperp\Az^{-1}\\
&=\Ae^{-1}[\sJ,\Be]\Az^{-1}+\Ae^{-1}\sR_\varepsilon\Az^{-1}+\Ae^{-1}\Ne\Pperp\Az^{-1}.
\end{aligned}
\end{equation}
Thus by Lemma~\ref{lem:inve-coer},
\begin{equation}\label{eq:117}
\nrm{\Ae^{-1}-\Az^{-1}}\le\nrm{\Ae^{-1}[\sJ,\Be]\Az^{-1}}+C_{\V}\nrm{\sR_\varepsilon\Az^{-1}}_{\Hs\to\V^*}+\frac{2M}{m+\alpha}\nrm{\Pperp\Az^{-1}} .
\end{equation}

2. We only need to estimate $\nrm{\Ae^{-1}[\sJ,\Be]\Az^{-1}}_{\Hs\to\Hs}$. Here $\Ae^{-1}$ also denotes its bounded extension $\V^*\to\Hs$ from Lemma~\ref{lem:inve-coer}, and $\sJ$ the extension $\Hs\to\V^*$ of \eqref{eq:160}. The identity $\Ae^{-1}\sJ w=w-\Ae^{-1}(\Me+\alpha)w$ holds for $w\in\V$ by \eqref{eq:9}, and extends to all $w\in\Hs$ because $\V$ is dense in $\Hs$ and both sides are continuous $\Hs\to\Hs$. From Definition~\ref{def:oper-a}, $\sJ=\Ae-\Me-\alpha I$ and $\sJ=\Az-\Mz-\alpha I$, so
\begin{equation}\label{eq:118}
\begin{aligned}
\Ae^{-1}\sJ\Be\Az^{-1}&=\Be\Az^{-1}-\Ae^{-1}(\Me+\alpha I)\Be\Az^{-1},\\
-\Ae^{-1}\Be\sJ\Az^{-1}&=-\Ae^{-1}\Be+\Ae^{-1}\Be(\Mz+\alpha I)\Az^{-1}.
\end{aligned}
\end{equation}
Notice that the unbounded operator $\sJ$ has completely vanished. We bound the terms in \eqref{eq:118} using Lemma~\ref{lem:inve-coer}:
\begin{align*}
\nrm{\Be\Az^{-1}}_{\Hs\to\Hs}&\le\nrm{\Be}_{\V\to\Hs}\nrm{\Az^{-1}}_{\Hs\to\V}\le C_{\V}\nrm{\Be}_{\V\to\Hs},\\
\nrm{\Ae^{-1}\Be}_{\Hs\to\Hs}&\le\nrm{\Ae^{-1}}_{\V^*\to\Hs}\nrm{\Be}_{\Hs\to\V^*}\le C_{\V}\nrm{\Be}_{\Hs\to\V^*},\\
\nrm{\Ae^{-1}(\Me+\alpha I)\Be\Az^{-1}}_{\Hs\to\Hs}&\le\frac{M+\alpha}{m+\alpha}C_{\V}\nrm{\Be}_{\V\to\Hs},\\
\nrm{\Ae^{-1}\Be(\Mz+\alpha I)\Az^{-1}}_{\Hs\to\Hs}&\le C_{\V}\frac{M+\alpha}{m+\alpha}\nrm{\Be}_{\Hs\to\V^*}.
\end{align*}
Summing up, we conclude
\begin{equation}\label{eq:119}
\nrm{\Ae^{-1}[\sJ,\Be]\Az^{-1}}_{\Hs\to\Hs}\le C_{\V}\Big(1+\frac{M+\alpha}{m+\alpha}\Big)\big(\nrm{\Be}_{\V\to\Hs}+\nrm{\Be}_{\Hs\to\V^*}\big).
\end{equation}
From \eqref{eq:117} and \eqref{eq:119}, we obtain \eqref{eq:12}.
\end{proof}

To obtain the convergence rate of $\Ae^{-1}-\Az^{-1}$ by Theorem~\ref{thm:freq}, we need to estimate $\nrm{\Be}_{\V\to\Hs}$, $\nrm{\Be}_{\Hs\to\V^*}$, $\nrm{\sR_\varepsilon\Az^{-1}}_{\Hs\to\V^*}$, and $\nrm{\Pperp\Az^{-1}}_{\Hs\to\Hs}$. In the context of homogenization, the last term usually vanishes, see Theorem~\ref{thm:corr}(iii).

Notice that Theorem~\ref{thm:freq} does not use any specific structures of multiscale analysis, however, estimating the above norms for an arbitrary pair of operators $\Me$ and $\Mz$ is usually a highly nontrivial task. To make the estimation manageable, we will introduce additional hypotheses in Section~\ref{sec:2.2} below.

\begin{remark}[Role of Theorem~\ref{thm:freq}]\label{rem:role23}
Theorem~\ref{thm:freq} is not used in the proofs of the main results. In the homogenization setting, \ref{A3} can in general not be verified at the operator level: $\Se$ is not uniformly bounded on $\V^*$, and the natural candidate for $\Be$ need not be bounded $\V\to\Hs$ (Remark~\ref{rem:23}). What is used is its vector-wise version, Lemma~\ref{lem:id}, in which the commutator equation is solved on the single vector $\Az^{-1}f$. We keep Theorem~\ref{thm:freq} because it exhibits the algebraic mechanism (elimination of $\sJ$ in \eqref{eq:118}) in its cleanest form.
\end{remark}

\begin{remark}\label{rem:rela-accr}
Since Theorem~\ref{thm:freq} estimates the difference of two resolvents, the accretivity assumption in \ref{A2} can be relaxed to quasi-accretivity: $\Re\ip{\Me u}{u}\ge-c\nrm{u}^2$ and $\Re\ip{\Mz u}{u}\ge -c\nrm{u}^2$ for some $c\ge 0$. This is a mere shift of $\alpha$: all statements above hold if $\alpha>c$, with $m+\alpha$ replaced by $\alpha-c$ in the constants. For bounded operators this is always possible ($c=M$), at the price of constants that degenerate as $\alpha-c\downarrow0$.
\end{remark}

\section{Point-free Homogenization: Qualitative Theory}\label{sec:2.2}\label{sec:shift}
\label{sec:qual-theo}

We now introduce additional hypotheses that arise naturally from multiscale analysis.
\begin{enumerate}[label=(H\arabic*),leftmargin=3.2em]
\item\label{H1} \emph{The Hilbert Space:} Let $\Hs$ be a Hilbert space. We assume there exists an orthogonal decomposition into three closed subspaces:
\begin{equation}\label{eq:35}
\Hs=\Hmac\oplus\Hflu\oplus\Hgra .
\end{equation}
Let $\Pmac$, $\Pflu$, and $\Pgra$ denote the orthogonal projections onto these respective closed subspaces. We define the microscopic kernel space as $\Pz:=\Hmac\oplus\Hflu$, with the corresponding orthogonal projection $\Pi_0:=\Pmac+\Pflu$.
\item\label{H2} \emph{The Differential Structure:} Let $\Jmac$ and $\Jmic$ be skew-adjoint operators on $\Hs$ satisfying the following properties:
\begin{enumerate}[label=(\roman*)]
\item \emph{The Microscopic Operator:} The kernel of $\Jmic$ is exactly the macroscopic and fluctuation spaces, $\ker(\Jmic)=\Pz$.
\item \emph{Macroscopic Domain Invariance:} The projection $\Pmac$ preserves $\dom(\Jmac)$. Furthermore, we assume $\Jmac(\Hmac\cap\dom(\Jmac))\subset\Hmac$.
\item \emph{The Fluctuation Mapping:} $\Jmac(\Hflu\cap\dom(\Jmac))\subset\Hgra$.
\item \emph{Fluctuation Density:} $\Hflu\cap\dom(\Jmac)$ is dense in $\Hflu$. (This is used only in the qualitative Theorem~\ref{thm:qual}; it holds, e.g., if $\Pflu$ preserves $\dom\Jmac$.)
\end{enumerate}
\item\label{H3} \emph{The Scaling Structure:} Let $\{\Se\}_{\varepsilon>0}$ be a family of unitary operators on $\Hs$.
\begin{enumerate}[label=(\roman*)]
\item \emph{Operator Splitting:} Define $\Je:=\Jmac+\frac1\varepsilon\Jmic$ as the closure of the sum on $\mathcal D:=\dom(\Jmac)\cap\dom(\Jmic)$, and let $\sJ$ be a skew-adjoint operator on $\Hs$, independent of $\varepsilon$, such that $\Je=\Se^{-1}\sJ\Se$. Moreover, we assume $\mathcal D$ is a core for $\Jmic$ (it is a core for $\Je$ by definition). (Since $\Se$ is unitary, the skew-adjointness of $\sJ$ is equivalent to that of $\Je$ for one, hence every, $\varepsilon>0$.)
\item \emph{Macroscopic Identity:} $\Se u=u$ for every $u\in\Hmac$.
\end{enumerate}
\end{enumerate}

\begin{remark}\label{rem:zero-blk}
Hypothesis \ref{H2}(ii) implies
\begin{align}
\Pflu\Jmac\Pmac&=0,\label{eq:37a}\\
\Pgra\Jmac\Pmac&=0.\label{eq:26}
\end{align}
Moreover, \ref{H2}(iii) implies
\begin{align}
\Pi_0\Jmac\Pflu&=0,\label{eq:37}\\
\Pmac\Jmac\Pflu&=0,\label{eq:39}\\
\Pflu\Jmac\Pflu&=0.\label{eq:40}
\end{align}
Furthermore, since $\Jmic$ is closed and densely defined, we have $\Pz^\perp=\ker(\Jmic)^\perp=\overline{\ran(\Jmic^*)}$. Using skew-adjointness, $\Pz^\perp=\overline{\ran(-\Jmic)}=\overline{\ran(\Jmic)}$. By the decomposition \eqref{eq:35}, we obtain
\begin{equation}\label{eq:98}
\overline{\ran(\Jmic)}=\Hgra .
\end{equation}
The range of $\Jmic$ equals $\Hgra$ (without closure) only if $\Jmic$ has closed range, i.e., if $0$ is isolated in $\sigma(\Jmic)$ (a spectral gap). This is the case for periodic media and fails in general in the stochastic setting, see Chapter~\ref{sec:ell}. In the arguments below only the inclusion $\ran(\Jmic)\subset\Hgra$, hence $\ran(\Jmic)\perp\Pz$, is used.
\end{remark}

\begin{remark}[Skew-adjointness in \ref{H3}(i)]\label{rem:Jsa}
The closure of a sum of two skew-adjoint operators on a common core is skew-symmetric, but need not be skew-adjoint. Skew-adjointness of $\sJ$ is used in Theorem~\ref{thm:corr}(i) through $\Ae^*=-\sJ+\Me^*+\alpha\Ps$ (with $\dom\Ae^*=\dom\sJ$), which gives surjectivity of $\Ae$, and hence in Theorem~\ref{thm:qual}, where $v_\varepsilon=\Se^{-1}\Ae^{-1}f$ must exist. Skew-symmetry alone does not suffice: $\sJ=d/dx$ on $L^2(0,1)$ with domain $H^1_0(0,1)$ is skew-symmetric, and $\sJ+\alpha$ is injective with closed range, but $e^{\alpha x}$ spans $\ker(\sJ+\alpha)^*$ because $\sJ^*=-d/dx$ on $H^1(0,1)$. In all examples of this memoir $\sJ=\Jmac$ is skew-adjoint by construction, so this hypothesis is satisfied.
\end{remark}

\begin{remark}[(H1) is fixed by the choice of $\Hmac$]\label{rem:H1red}
Given $\Jmic$ and a closed subspace $\Hmac\subset\ker\Jmic$, hypothesis \ref{H2}(i) forces $\Hflu=\ker\Jmic\ominus\Hmac$ and $\Hgra=(\ker\Jmic)^\perp=\overline{\ran\Jmic}$. Thus \ref{H1} carries no information beyond the choice of $\Hmac$; we keep the three-term notation because it makes the block structure transparent.
\end{remark}

\begin{lemma}[Commutation Relations]\label{lem:comm}
Suppose \ref{H1}, \ref{H2}, and \ref{H3}. Let $\V:=\dom(\sJ)$. Then
\begin{enumerate}[label=(\roman*)]
\item $\Se\Pmac=\Pmac\Se$.
\item For $u\in\dom(\Jmac)$, we have $\Pmac u\in\dom(\Jmac)$ and $\Jmac\Pmac u=\Pmac\Jmac u$. In particular $\Jmac$ commutes with $\Pmac$, and $\Jz:=\Jmac|_{\Hmac}$ is a skew-adjoint operator on $\Hmac$.
\item $\Pmac$ maps $\dom(\Je)$ into itself and $\Je\Pmac=\Pmac\Je$; consequently $\Pmac\V\subset\V$ and $\sJ\Pmac=\Pmac\sJ$.
\end{enumerate}
\end{lemma}

\begin{proof}
(i) Let $u\in\Hmac$. By \ref{H3}(ii), $\Se u=u$, hence $\Se^*u=\Se^{-1}u=u$. For $v\perp\Hmac$ and $u\in\Hmac$, $\ip{\Se v}{u}=\ip{v}{\Se^*u}=\ip{v}{u}=0$. So $\Se$ leaves $\Hmac$ and $\Hmac^\perp$ invariant and thus commutes with $\Pmac$.

(ii) Let $u\in\dom(\Jmac)$. By \ref{H2}(ii), $\Pmac u\in\dom(\Jmac)$ and $\Jmac\Pmac u\in\Hmac$. Put $w=u-\Pmac u\in\dom(\Jmac)$. For $h\in\Hmac\cap\dom(\Jmac)$, $\ip{\Jmac w}{h}=-\ip{w}{\Jmac h}=0$, since $w\perp\Hmac$ and $\Jmac h\in\Hmac$. As $\Pmac$ preserves $\dom(\Jmac)$, which is dense, $\Hmac\cap\dom(\Jmac)$ is dense in $\Hmac$, so $\Pmac\Jmac w=0$ and $\Pmac\Jmac u=\Pmac\Jmac\Pmac u=\Jmac\Pmac u$. The operator $\Jmac$ is therefore reduced by $\Hmac$, and its part $\Jz$ in $\Hmac$ is skew-adjoint.

(iii) For $u\in\mathcal D$, we have $\Pmac u\in\dom(\Jmac)$ by (ii) and $\Pmac u\in\Hmac\subset\ker\Jmic$. Thus,
$\Pmac\Je u=\Pmac\Jmac u+\frac1\varepsilon\Pmac\Jmic u=\Jmac\Pmac u=\Je\Pmac u$,
where we used $\ran\Jmic\perp\Pz$ (Remark~\ref{rem:zero-blk}). For $u\in\dom\Je$ choose $u_n\in\mathcal D$ with $u_n\to u$ in the graph norm of $\Je$; then $\Pmac u_n\to\Pmac u$, $\Je\Pmac u_n=\Pmac\Je u_n\to\Pmac\Je u$, and closedness gives the first claim. Finally, the last claim comes from $\sJ=\Se\Je\Se^{-1}$ and (i).
\end{proof}

We provide an intuition for \ref{H3}. For concreteness, if a periodic homogenization problem is posed on $\R^d\times\T^d$, then we can choose
\begin{equation}\label{eq:135}
(\Se u)(x,y):=u\Big(x,\Big(y+\frac{x}{\varepsilon}\Big)\bmod 1\Big),\qquad\text{for }u\in\Hs .
\end{equation}
This formula can be naturally extended to the stochastic setting by using the group action on the probability space; see Chapter~\ref{sec:ell}. Note that the scaling operator $\Se$ is similar to the unfolding operator \cite{CioranescuDamlamianGriso2008,CioranescuDamlamianGriso2018}. The distinction is that, by keeping the same number of arguments on both sides of \eqref{eq:135}, $\Se$ is isometric, while the unfolding operator is not (note that due to the incomplete cells overlapping with the domain boundary, the unfolding is almost isometric). Note that $\Se$ does not preserve $\V=\dom\sJ$ nor act boundedly on $\V^*$ uniformly in $\varepsilon$.

\begin{remark}[Intuition of the Subspaces]\label{rem:intu-subs}
\begin{itemize}
\item The macroscopic space $\Hmac$ includes functions that do not depend on the fast/microscopic variable $y=x/\varepsilon$: they are purely macroscopic fields, $u(x)$. Thus, they are annihilated by the microscopic derivative $\Jmic$. The scaling operator $\Se$ (zooming in) couples the macro and micro scales, see example \eqref{eq:135}. If a function $u(x)$ has no $y$-dependence (i.e., it lives in $\Hmac$), then shifting or scaling the $y$-variable does nothing to it. This is why $\Se$ acts as the identity on $\Hmac$ in \ref{H3}(ii).
\item The space $\Hflu$ consists of solenoidal fluctuations (divergence-free microscopic fluxes). Indeed, by \ref{H2}(i), the kernel of $\Jmic$ is $\Pz=\Hmac\oplus\Hflu$. Therefore, $\Hflu$ consists of fields that are annihilated by the fast derivative, but have zero macroscopic average (since they are orthogonal to $\Hmac$). Thus, $\Hflu$ represents divergence-free microscopic fluxes that are locally conserved ($\nabla_y\cdot p=0$) but average out to zero on the macro-scale. The space $\Hgra$ consists of conservative fluctuations (microscopic gradients). By \ref{H2}(i), $\Hgra$ is the closure of the range of $\Jmic$. Functions in $\Hgra$ represent pure gradient fields generated by microscopic scalar potentials.
\item We need to separate $\Hflu$ and $\Hgra$ because they play different roles in the corrector problem defined in Definition~\ref{def:corr} below. Indeed, if we apply a smooth macroscopic field $u_0\in\Hmac$ to a heterogeneous material $\sM$, the material properties scatter this smooth field, generating highly oscillatory residuals. The differential operator $\Jmic$ can only see and invert functions in $\Hgra$. Therefore, to obtain the effective equation, we must find a corrector $w$ that cancels out the fluctuations. Notice in Theorem~\ref{thm:corr} and Remark~\ref{rem:schu-comp}, the abstract corrector $\Kmic$ maps into $\Hflu$. Intuitively, the corrector is used to adjust the microscopic flux so that it becomes divergence-free (conserved) with respect to the fast variable. By forcing the corrector $w$ to live in $\Hflu$, we guarantee that the corrected microscopic flux $\sM(u_0+w)$ has no component in $\Hflu$ under \ref{S0} (Definition~\ref{def:corr}); in the elliptic example this means that $a^{-1}$ times the flux is a constant plus a $y$-gradient, so that the flux itself satisfies $\nabla_y\cdot(\mathrm{Flux})=0$, see \eqref{eq:89}.
\item In the context of homogenization, by the chain rule, the total derivative $\nabla$ (corresponding to $\sJ$) is decomposed into $\nabla_x$ and $\nabla_y$ (corresponding to $\Jmac$ and $\Jmic$, respectively) by $\nabla\mapsto\nabla_x+\frac1\varepsilon\nabla_y$. This motivates \ref{H3}(i). By \ref{H3}(ii) and \ref{H2}(ii), for $u_0\in\Hmac\cap\dom(\Jmac)$ we have $\Je u_0=\Jmac u_0\in\Hmac$ and $\sJ u_0=\Se\Je u_0=\Jmac u_0$. This is the abstract form of the classical identity $\nabla u_0=\nabla_x u_0$ in homogenization.
\end{itemize}
\end{remark}

To connect $\Me$ and the effective coefficient, we introduce a shift projection and a hypothesis on the material operator. The two main examples of this memoir have the form $\Ae=\sJ+\Me+\alpha\Ps$, where $\Ps$ projects onto the state component:
\begin{itemize}
\item in Chapter~\ref{sec:ell}, $\Ps(u,p)=(u,0)$, $\alpha=1$, and $\sM=\operatorname{diag}(0,a^{-1})$ is coercive on the flux;
\item in Chapter~\ref{sec:top}, $\Ps(u,q_1,q_2)=(u,0,0)$, $\alpha>0$ is arbitrary, and $\sM$ is skew-Hermitian.
\end{itemize}
Eliminating the flux turns $\Ae$ into the physical operator plus $\alpha$ (times $i$ in the Dirac case). What differs between the examples is the source of the bound on the flux component, which is isolated in \ref{H4}(iii).

\begin{enumerate}[label=(S),leftmargin=3.2em]
\item\label{S} \emph{Shift Projection:} $\Ps$ is an orthogonal projection on $\Hs$ that commutes with $\Se$ for every $\varepsilon>0$, with $\Pmac$ and with $\Pflu$.
\end{enumerate}
\begin{enumerate}[label=(S0),leftmargin=3.2em]
\item\label{S0} (optional) $\Ps\Pflu=0$: the fluctuations in $\ker\Jmic$ have no shifted component.
\end{enumerate}
\begin{enumerate}[label=(H\arabic*),leftmargin=3.2em,start=4]
\item\label{H4} \emph{The Multiscale Material Operator with a Multiplier:}
\begin{enumerate}[label=(\roman*)]
\item $\sM$ is bounded with $\nrm\sM\le\Mb$, and $\Re\ip{\sM v}{v}\ge0$ for all $v\in\Hs$.
\item There are bounded operators $\sT,\sT'$ on $\Hs$ that commute with $\Se$, $\Pmac$ and $\Pflu$, and satisfy
\[
\Re\ip{\Jmac v}{\sT v}=0\ \ (v\in\dom\Jmac),\qquad\Re\ip{\Jmic v}{\sT v}=0\ \ (v\in\dom\Jmic),
\]
and the same with $\sT'$ in place of $\sT$. ($\sT'$ is used for the adjoint; it need not be the Hilbert adjoint of $\sT$.)
\item There are $c>0$ and $C\ge0$ such that, for all $v\in\Hs$,
\[
\begin{gathered}\Re\ip{\sM v}{\sT v}\ge c\nrm{(I-\Ps)v}^2-C\nrm{\Ps v}^2,\\ \Re\ip{\sM^*v}{\sT'v}\ge c\nrm{(I-\Ps)v}^2-C\nrm{\Ps v}^2 .\end{gathered}
\]
\end{enumerate}
\end{enumerate}
We write $\tau:=\max\{\nrm\sT,\nrm{\sT'}\}$.

\begin{remark}[Two sources of multipliers]\label{rem:Tsources}
Hypothesis \ref{H4}(ii) holds in two situations.
\begin{itemize}
\item \emph{The dissipative case, $\sT=\sT'=I$.} Then (ii) follows from skew-adjointness (step~2 of the proof of Lemma~\ref{lem:inve-coer}), and (iii) says that $\sM$ is coercive on $\ran(I-\Ps)$.
\item \emph{The chiral case, $\sT=\pm i\Sigma$.} Here $\Sigma$ is a bounded self-adjoint operator that commutes with $\Se$, $\Pmac$, $\Pflu$, preserves $\dom\Jmac$ and $\dom\Jmic$, and \emph{anticommutes} with $\Jmac$ and $\Jmic$. Then $(\Sigma\sJ_\bullet)^*=-\sJ_\bullet\Sigma=\Sigma\sJ_\bullet$ on $\dom\sJ_\bullet$, so $\ip{\sJ_\bullet v}{\Sigma v}\in\R$ and $\Re\ip{\sJ_\bullet v}{\pm i\Sigma v}=0$. Condition (iii) is then a condition on the anti-Hermitian part of $\Sigma\sM$; this is the situation of Chapter~\ref{sec:top}, where $\sM^*=-\sM$.
\end{itemize}
\end{remark}

\begin{remark}[Special cases]\label{rem:H4rel}
\begin{itemize}
\item \emph{Coercive material, no shift.} If $\Re\ip{\sM v}{v}\ge\mb\nrm v^2$ with $\mb>0$, take $\Ps=0$, $\sT=\sT'=I$, $c=\mb$, $C=0$. Then \ref{S0} holds trivially and $\alpha$ plays no role. This is the classical coercive setting; it is used for Maxwell's system in Section~\ref{sec:maxwell}.
\item \emph{Accretive material with a scalar shift.} If only $\Re\ip{\sM v}{v}\ge0$, take $\Ps=I$, $\sT=\sT'=I$ and $C=0$; then (iii) holds for every $c>0$. In this case \ref{S0} fails unless $\Hflu=\{0\}$, and the effective coefficient depends on $\alpha$.
\item \ref{S0} is what makes the effective coefficient independent of the spectral parameter (Theorem~\ref{thm:corr}(iv)). Without it, Theorems~\ref{thm:corr}(i)--(iii) and \ref{thm:qual} still hold.
\end{itemize}
In the examples, $\Ps$ projects onto the state variable, and the shift is the spectral parameter of the operator obtained by eliminating $\ran(I-\Ps)$. We do not claim that this choice is canonical in general. The multiplier $\sT$ must be found in each model; Remark~\ref{rem:Tsources} narrows the search to $I$ and to chiral symmetries of the differential structure.
\end{remark}

\begin{definition}\label{def:effe-fine}
Assume \ref{S} and \ref{H4}, and fix $\alpha>0$. Define:
\begin{itemize}
\item the oscillatory material coefficient
\begin{equation}\label{eq:23.1}
\Me:=\Se\sM\Se^{-1};
\end{equation}
\item the fine-scale operator $\Ae:=\sJ+\Me+\alpha\Ps$, with $\dom\Ae=\dom\sJ$;
\item the fluctuation block $\Mff:=\Pflu(\sM+\alpha\Ps)|_{\Hflu}:\Hflu\to\Hflu$, and the property
\begin{enumerate}[label=(C),leftmargin=3.2em]
\item\label{C} \emph{Corrector Solvability:} $\Mff$ is boundedly invertible on $\Hflu$,
\end{enumerate}
which always holds under \ref{S} and \ref{H4} (Theorem~\ref{thm:corr}(ii));
\item the (bounded) abstract corrector
\begin{equation}\label{eq:5.3}
\Kmic:=-\Mff^{-1}\Pflu\sM\Pmac ,
\end{equation}
where $\Mff^{-1}\Pflu$ is regarded as an operator $\Hs\to\Hflu\subset\Hs$;
\item the homogenized coefficient
\begin{equation}\label{eq:5.1}
\Mz:=\Pmac\sM\Pmac-\Pmac\sM\Pflu\,\Mff^{-1}\Pflu\sM\Pmac ;
\end{equation}
\item the effective operator $\Az:=\Jz+\Mz+\alpha\Ps$ on $\Hmac$, with domain $\Hmac\cap\dom\Jz$, where $\Jz=\Jmac|_{\Hmac}$ (Lemma~\ref{lem:comm}(ii)).
\end{itemize}
Clearly,
\begin{equation}\label{eq:5.2}
\Mz=\Pmac\sM(I+\Kmic)\Pmac .
\end{equation}
\end{definition}

\paragraph{Motivation: formal expansion.} Let $f\in\Hmac$. Since $\Se^{-1}\Ae\Se=\Jmac+\frac1\varepsilon\Jmic+\sM+\alpha\Ps$ by \ref{H3} and \ref{S}, the equation $\Ae u_\varepsilon=f$ becomes
\begin{equation}\label{eq:74}
\Big(\Jmac+\frac1\varepsilon\Jmic+\sM+\alpha\Ps\Big)v_\varepsilon=f,\qquad v_\varepsilon:=\Se^{-1}u_\varepsilon .
\end{equation}
Expand $v_\varepsilon=v_0+\varepsilon v_1+\cdots$.
\begin{itemize}
\item At order $\varepsilon^{-1}$: $\Jmic v_0=0$, so $v_0=u_0+w$ with $u_0\in\Hmac$ and $w\in\Hflu$.
\item At order $1$: $\Jmac v_0+\Jmic v_1+(\sM+\alpha\Ps)v_0=f$.
\end{itemize}
Project the order-$1$ equation with $\Pflu$. By Remark~\ref{rem:zero-blk}, $\ran\Jmic\perp\Pz$ and $\Pflu f=0$, so all derivatives disappear, and $\Pflu(\sM+\alpha\Ps)(u_0+w)=0$. Since $\Pflu\Ps u_0=\Ps\Pflu u_0=0$, this reads $\Mff w=-\Pflu\sM u_0$, i.e.\ $w=\Kmic u_0$. Projecting with $\Pmac$ and using $\Pmac\Ps w=0$ gives $(\Jz+\Mz+\alpha\Ps)u_0=f$.

\begin{lemma}[Partial coercivity]\label{lem:pcoer}
Let $X$ be a Hilbert space, $\Pi$ an orthogonal projection on $X$, $\sT$ a bounded operator on $X$, and $\sL$ a linear operator in $X$. Let $\alpha,c>0$ and $C'\ge0$, and assume that for all $v\in\dom\sL$
\[
\Re\ip{\sL v}{v}\ge\alpha\nrm{\Pi v}^2,\qquad\Re\ip{\sL v}{\sT v}\ge c\nrm{(I-\Pi)v}^2-C'\nrm{\Pi v}\nrm v .
\]
Then $\nrm v\le K(\alpha,c,C',\nrm\sT)\,\nrm{\sL v}$ for all $v\in\dom\sL$, where $K(\alpha,c,C',t):=\frac2\alpha+\frac{2t}{c}+\frac{C'^2}{\alpha c^2}$. If $\Pi=0$, then $\nrm v\le(\nrm\sT/c)\nrm{\sL v}$.
\end{lemma}

\begin{proof}
Put $a:=\nrm{\Pi v}$, $b:=\nrm{(I-\Pi)v}$, $n:=\nrm v$ and $F:=\nrm{\sL v}$. The first inequality gives $\alpha a^2\le Fn$. The second gives $cb^2\le\nrm\sT Fn+C'an$. Hence $n^2=a^2+b^2\le Fn\big(\frac1\alpha+\frac{\nrm\sT}c\big)+\frac{C'}{c}an$. Divide by $n$, and use $\frac{C'}ca\le\frac{C'}c(Fn/\alpha)^{1/2}\le\frac n2+\frac{C'^2}{2\alpha c^2}F$. If $\Pi=0$, then $a=0$ and $cn^2\le\nrm\sT Fn$.
\end{proof}

\begin{theorem}[Invertibility, corrector and effective operator]\label{thm:corr}\label{prop:inv}\label{lem:M0}\label{lem:Azero}
Assume \ref{H1}--\ref{H3}, \ref{S} and \ref{H4}, and let $\alpha>0$. Put $C_A:=K(\alpha,c,C+\alpha\tau,\tau)$.
\begin{enumerate}[label=(\roman*)]
\item $\Ae$ is a bijection from $\dom\sJ$ onto $\Hs$, and $\nrm{\Ae^{-1}}\le C_A$ and $\nrm{(\Ae^*)^{-1}}\le C_A$ for every $\varepsilon>0$. Moreover $\nrm{\Me}\le\Mb$ and $\Re\ip{\Me v}{v}\ge0$.
\item Hypothesis \ref{C} holds: $\nrm{\Mff^{-1}}\le C_A$. Hence $\nrm{\Kmic}\le C_A\Mb$, $\ran\Kmic\subset\Hflu$, and $\nrm{\Mz}\le\Mb+C_A\Mb^2$.
\item $\Az$ is a bijection from $\Hmac\cap\dom\Jz$ onto $\Hmac$, and $\nrm{\Az^{-1}}\le C_A$; the same holds for $\Az^*=-\Jz+\Mz^*+\alpha\Ps$.
\item Assume in addition \ref{S0}. Then $\Mff=\Pflu\sM|_{\Hflu}$, and $\Mff$, $\Kmic$ and $\Mz$ do not depend on $\alpha$. Moreover, $\Mz$ satisfies \ref{H4} on $\Hmac$ with the same $\sT,\sT',c,C$: for $u\in\Hmac$,
\[
\Re\ip{\Mz u}{u}\ge0,\qquad\Re\ip{\Mz u}{\sT u}\ge c\nrm{(I-\Ps)u}^2-C\nrm{\Ps u}^2,
\]
and likewise for $\Mz^*$ with $\sT'$.
\end{enumerate}
If $\sT=\sT'=I$ and $C=0$, then $C_A$ may be replaced by $1/\min\{\alpha,c\}$ throughout, and by $1/c$ if $\Ps=0$.
\end{theorem}

\begin{proof}
\emph{Preliminary.} Let $v\in\dom\sJ$ and $w:=\Se^{-1}v\in\dom\Je$. By \ref{H4}(ii), $\Re\ip{\Je\phi}{\sT\phi}=0$ for $\phi\in\mathcal D$. Since $\mathcal D$ is a core for $\Je$ (\ref{H3}) and $\sT$ is bounded, $\Re\ip{\Je w}{\sT w}=0$. As $\sT$ commutes with $\Se$, $\Re\ip{\sJ v}{\sT v}=0$. The same holds for $\sT'$. The bounds on $\Me$ follow from the unitarity of $\Se$.

(i) $\Ae$ is closed with domain $\dom\sJ$. For $v\in\dom\sJ$, $\Re\ip{\Ae v}{v}=\Re\ip{\Me v}{v}+\alpha\nrm{\Ps v}^2\ge\alpha\nrm{\Ps v}^2$, using $\Re\ip{\sJ v}{v}=0$. With $w=\Se^{-1}v$, and because $\sT$ and $\Ps$ commute with $\Se$,
\begin{multline*}
\Re\ip{\Ae v}{\sT v}=\Re\ip{\sM w}{\sT w}+\alpha\Re\ip{\Ps v}{\sT v}\\\ge c\nrm{(I-\Ps)v}^2-C\nrm{\Ps v}^2-\alpha\tau\nrm{\Ps v}\nrm v .
\end{multline*}
Lemma~\ref{lem:pcoer} with $C'=C+\alpha\tau$ gives $\nrm v\le C_A\nrm{\Ae v}$, so $\Ae$ is injective with closed range. The adjoint $\Ae^*=-\sJ+\Me^*+\alpha\Ps$ satisfies the same estimate, with $\sT'$ and the second inequality of \ref{H4}(iii). Hence $\ker\Ae^*=\{0\}$, the range of $\Ae$ is dense, and $\Ae$ is bijective.

(ii) Apply Lemma~\ref{lem:pcoer} in $X=\Hflu$ to $\sL=\Mff$, with $\Pi=\Ps|_{\Hflu}$ and $\sT|_{\Hflu}$; both preserve $\Hflu$ because they commute with $\Pflu$. For $w\in\Hflu$, $\ip{\Mff w}{w}=\ip{(\sM+\alpha\Ps)w}{w}$ and $\ip{\Mff w}{\sT w}=\ip{(\sM+\alpha\Ps)w}{\sT w}$, so the hypotheses hold as in (i). The adjoint $\Mff^*=\Pflu(\sM^*+\alpha\Ps)|_{\Hflu}$ is treated with $\sT'$. A bounded operator that is bounded below, together with its adjoint, is invertible. The bound on $\Mz$ follows from \eqref{eq:5.1}.

(iii) On $X=\Pz=\Hmac\oplus\Hflu$ consider $\sL(x+w):=\Jz x+\Pi_0(\sM+\alpha\Ps)(x+w)$, with $\dom\sL=(\Hmac\cap\dom\Jz)\oplus\Hflu$. It is closed, since $\Jz\oplus0$ is skew-adjoint on $\Pz$. Let $v=x+w\in\dom\sL$.
\begin{itemize}
\item $\Re\ip{\Jz x}{v}=\Re\ip{\Jz x}{x}=0$.
\item $\ip{\Jz x}{\sT v}=\ip{\Jz x}{\sT x}$, because $\sT w\in\Hflu\perp\Jz x$; so $\Re\ip{\Jz x}{\sT v}=0$ by \ref{H4}(ii).
\item Since $\sT v\in\Pz$, $\ip{\Pi_0(\sM+\alpha\Ps)v}{\sT v}=\ip{(\sM+\alpha\Ps)v}{\sT v}$.
\end{itemize}
Lemma~\ref{lem:pcoer} therefore applies to $\sL$, and, with $\sT'$, to $\sL^*=-\Jz\oplus0+\Pi_0(\sM^*+\alpha\Ps)\Pi_0$. So $\sL$ is bijective with $\nrm{\sL^{-1}}\le C_A$. For $f\in\Hmac$, the equation $\sL(x+w)=f$ splits into $\Pflu(\sM+\alpha\Ps)(x+w)=0$, which is equivalent to $w=\Kmic x$ (since $\Pflu\Ps x=0$), and $\Jz x+\Pmac(\sM+\alpha\Ps)(x+w)=f$, which is then $\Az x=f$ (since $\Pmac\Ps w=0$). Hence $\Az$ is bijective (if $\Az x=0$ then $\sL(x+\Kmic x)=0$), and $\nrm x\le\nrm{x+w}\le C_A\nrm f$. For $\Az^*$, use $\sL^*$ and note that the Schur complement of $\sM^*$ is $\Mz^*$, because $(\Mff^{-1})^*=(\Mff^*)^{-1}$.

(iv) Under \ref{S0}, $\Ps w=0$ for $w\in\Hflu$, which gives the formula for $\Mff$ and the independence of $\alpha$. Let $u\in\Hmac$ and $w=\Kmic u$, so that $\Pflu\sM(u+w)=0$. Since $w,\sT w\in\Hflu$, we have $\ip{\sM(u+w)}{w}=0=\ip{\sM(u+w)}{\sT w}$. As $\Mz u=\Pmac\sM(u+w)$ and $\sT u\in\Hmac$, $\Re\ip{\Mz u}{u}=\Re\ip{\sM(u+w)}{u+w}\ge0$, and
\[
\Re\ip{\Mz u}{\sT u}=\Re\ip{\sM(u+w)}{\sT(u+w)}\ge c\nrm{(I-\Ps)u+w}^2-C\nrm{\Ps u}^2 ,
\]
using $\Ps(u+w)=\Ps u$. Moreover $\nrm{(I-\Ps)u+w}\ge\nrm{(I-\Ps)u}$, since $(I-\Ps)u\in\Hmac\perp w$. For $\Mz^*$, use $\sT'$ and the Schur complement of $\sM^*$.

\emph{Last sentence.} If $\sT=\sT'=I$ and $C=0$, then $\Re\ip{\sL v}{v}\ge\min\{\alpha,c\}\nrm v^2$ in each of the three applications ($\ge c\nrm v^2$ if $\Ps=0$).
\end{proof}

\begin{remark}[The abstract Schur complement]\label{rem:schu-comp}
In the block decomposition $\Pz=\Hmac\oplus\Hflu$, $\Mz$ is the Schur complement of the block $\Mff$ of $\sM+\alpha\Ps$, and $\Kmic$ is the elimination operator. This derivation is motivated by multiscale numerical methods \cite{ChertockLevy2005,DuEngquistTian2020,EngquistRunborg2002,GendreAllixGosselet2011}. For $u\in\Hmac$, the corrector function $w=\Kmic u\in\Hflu$ solves the abstract corrector problem of Definition~\ref{def:corr}. Only $\Az$ on $\Hmac$ enters the results below. The Schur-complement representation of effective operators is a basic structure of the abstract theory of composites \cite{Milton2002,Grabovsky2016,Welters2025}; see Section~\ref{sec:related}.\label{rem:conv}
\end{remark}

\begin{definition}[The abstract corrector problem]\label{def:corr}
Fix $u\in\Hmac$. The abstract corrector problem is
\begin{equation}\label{eq:29}
\Pflu(\sM+\alpha\Ps)(u+w)=0,\qquad\text{for unknown }w\in\Hflu .
\end{equation}
By Theorem~\ref{thm:corr}(ii) its unique solution is $w=\Kmic u$. Under \ref{S0} it reads $\Pflu\sM(u+w)=0$.
\end{definition}

We now use our operator setting to justify the well-known result that (classical) homogenization is a low-frequency approximation, via an algebraic proof. Recall the physical corrector $\Ke:=\Se\Kmic\Se^{-1}$ (Definition~\ref{def:phys-corr}).

\begin{theorem}[Qualitative homogenization]\label{thm:qual}\label{prop:qualS}
Assume \ref{H1}--\ref{H3} including \ref{H2}(iv), \ref{S}, \ref{H4}, and $\alpha>0$. Then, for every $f\in\Hmac$,
\begin{equation}\label{eq:3}
\lim_{\varepsilon\to0}\big\|\Ae^{-1}f-(I+\Ke)\Az^{-1}f\big\|=0 .
\end{equation}
In particular, since $\Pmac\Ke\Az^{-1}f=0$, $\lim_{\varepsilon\to0}\Pmac\Ae^{-1}f=\Az^{-1}f$.
\end{theorem}

\begin{proof}
Write $\sM_\alpha:=\sM+\alpha\Ps$.

1. \emph{Rescaling.} Let $v_\varepsilon:=\Se^{-1}\Ae^{-1}f$. Then $(\Je+\sM_\alpha)v_\varepsilon=f$, $v_\varepsilon\in\dom\Je$, and $\nrm{v_\varepsilon}\le C_A\nrm f$ by Theorem~\ref{thm:corr}(i). Let $v_*$ be the weak limit of a subsequence.

2. \emph{No microscopic oscillation.} For $\phi\in\mathcal D$, skew-adjointness of $\Je$ gives $\ip{\Jmic\phi}{v_\varepsilon}=\varepsilon\ip{\phi}{\sM_\alpha v_\varepsilon-f}-\varepsilon\ip{\Jmac\phi}{v_\varepsilon}\to0$. Hence $\ip{\Jmic\phi}{v_*}=0$ for all $\phi\in\mathcal D$, and, $\mathcal D$ being a core for $\Jmic$, $v_*\in\ker\Jmic=\Pz$.

3. \emph{Limit equation.} For $\psi\in\Pz\cap\dom\Jmac\subset\mathcal D$, $\Je\psi=\Jmac\psi$, so $\ip{\Jmac\psi}{v_*}=\ip{\psi}{\sM_\alpha v_*-f}$.
\begin{itemize}
\item Taking $\psi\in\Hflu\cap\dom\Jmac$ (then $\Jmac\psi\in\Hgra\perp v_*$) and using \ref{H2}(iv) gives $\Pflu\sM_\alpha v_*=0$. With $x:=\Pmac v_*$, this means $\Pflu v_*=\Kmic x$, by Theorem~\ref{thm:corr}(ii).
\item Taking $\psi\in\Hmac\cap\dom\Jz$ gives $x\in\dom\Jz$ and $\Jz x+\Pmac\sM_\alpha(x+\Kmic x)=f$, i.e.\ $\Az x=f$.
\end{itemize}
So $v_*=u_0:=(I+\Kmic)\Az^{-1}f$, independently of the subsequence. Moreover
\begin{equation}\label{eq:limS}
\Pflu\sM_\alpha u_0=0,\qquad\Pmac\sM_\alpha u_0=f-\Jz x .
\end{equation}

4a. \emph{The shifted component.} Put $d_\varepsilon:=v_\varepsilon-u_0$. Since $\Re\ip{\Je v_\varepsilon}{v_\varepsilon}=0$, we have $\Re\ip{\sM_\alpha v_\varepsilon}{v_\varepsilon}=\Re\ip f{v_\varepsilon}$. Expanding the quadratic form and using weak convergence,
\[
\lim_{\varepsilon\to0}\Re\ip{\sM_\alpha d_\varepsilon}{d_\varepsilon}=\Re\ip f{u_0}-\Re\ip{\sM_\alpha u_0}{u_0}=\Re\ip{\Jz x}{x}=0,
\]
by \eqref{eq:limS} and $u_0\in\Pz$. Since $\Re\ip{\sM d}{d}\ge0$, it follows that $\alpha\nrm{\Ps d_\varepsilon}^2\to0$.

4b. \emph{The complementary component.} By the preliminary step of the proof of Theorem~\ref{thm:corr}, $\Re\ip{\Je v_\varepsilon}{\sT v_\varepsilon}=0$, so $\Re\ip{\sM_\alpha v_\varepsilon}{\sT v_\varepsilon}=\Re\ip f{\sT v_\varepsilon}$. Expanding,
\[
\lim_{\varepsilon\to0}\Re\ip{\sM_\alpha d_\varepsilon}{\sT d_\varepsilon}=\Re\ip f{\sT u_0}-\Re\ip{\sM_\alpha u_0}{\sT u_0}=\Re\ip{\Jz x}{\sT x}=0,
\]
because $\sT u_0\in\Pz$, $\Pmac\sT u_0=\sT x$, \eqref{eq:limS}, and \ref{H4}(ii). By \ref{H4}(iii),
\[
c\nrm{(I-\Ps)d_\varepsilon}^2\le\Re\ip{\sM_\alpha d_\varepsilon}{\sT d_\varepsilon}+C\nrm{\Ps d_\varepsilon}^2+\alpha\tau\nrm{\Ps d_\varepsilon}\nrm{d_\varepsilon}\to0 .
\]
Hence $d_\varepsilon\to0$.

5. \emph{Conclusion.} $\Se$ is unitary and fixes $\Az^{-1}f\in\Hmac$, so $\nrm{\Ae^{-1}f-(I+\Ke)\Az^{-1}f}=\nrm{v_\varepsilon-u_0}\to0$. Finally, $\Pmac\Ke\Az^{-1}f=\Se\Pmac\Kmic\Az^{-1}f=0$.
\end{proof}

\begin{remark}[Relation with the Schur complement]
Step~3 shows that $u_0$ solves the block system $\Pi_0(\Jz\oplus0+\sM_\alpha)\Pi_0u_0=f$ on $\Pz$. The bottom row is the corrector equation, and the top-left block of the inverse is $(\Jz+\Mz+\alpha\Ps)^{-1}$, the inverse of the Schur complement. In Fourier language, Theorem~\ref{thm:qual} says that the choice \eqref{eq:5.1} of $\Mz$ makes $\Ae^{-1}-(I+\Ke)\Az^{-1}$ annihilate the zero-frequency (macroscopic) component in the limit.
\end{remark}

\chapter{Quantitative Theory: the Error Identity}\label{sec:quan-theo}\chaptermark{Quantitative Theory}

In this section, we assume \ref{H1}--\ref{H3}, \ref{S}, \ref{H4} and $\alpha>0$; then \ref{C} holds, and $\nrm{\Ae^{-1}},\nrm{(\Ae^*)^{-1}},\nrm{\Az^{-1}}\le C_A$ (Theorem~\ref{thm:corr}). We will provide a quantitative version of Theorem~\ref{thm:qual} for the norm resolvent convergence in the Hilbert space $\Hs$. To set the stage, let $\Kmic$ be the abstract corrector defined in \eqref{eq:5.3}.

\section{Correctors and residual fields}\label{sec:residual}

\begin{definition}[Physical Corrector]\label{def:phys-corr}
We define the physical corrector at scale $\varepsilon$ by
\begin{equation}\label{eq:24}
\Ke:=\Se\Kmic\Se^{-1}.
\end{equation}
\end{definition}

Instead of comparing $\Ae^{-1}$ to $\Az^{-1}$ as in Theorem~\ref{thm:freq}, we compare it to the corrected homogenized resolvent $(I+\Ke)\Az^{-1}$. By the resolvent identity
$\Ae^{-1}-(I+\Ke)\Az^{-1}=\Ae^{-1}\big(\Az-\Ae(I+\Ke)\big)\Az^{-1}$. Thus, the physical corrector $\Ke$ modifies the material difference $\Ne$ in \eqref{eq:54} to the following.

\begin{definition}[Corrected Residual Operators and Residual Fields]\label{def:corr-resi}
Let
\begin{equation}\label{eq:U}
U:=\{u\in\Hmac\cap\dom\Jmac:\ \Kmic u\in\dom\Jmac\}.
\end{equation}
For $u\in U$, the vector $(I+\Ke)u$ lies in $\dom\Ae$ (Lemma~\ref{lem:id}, step~1), and we define the corrected residual
\begin{equation}\label{eq:30}
\Ne^{\mathrm{cor}}u:=\Az u-\Ae(I+\Ke)u=\Mz u-\Me(I+\Ke)u-\sJ\Ke u-\alpha\Ps\Ke u ,
\end{equation}
where we used $\sJ u=\Jz u$ for $u\in\Hmac\cap\dom\Jz$ (Lemma~\ref{lem:comm}). Define the unscaled residual $\Ns u:=\Se^{-1}\Ne^{\mathrm{cor}}u$ for $u\in U$, and
\begin{equation}\label{eq:155}
\Nb:=[\Jmac,\Ns]\quad\text{on }\{u\in U\cap\dom\Jz^2:\ \Jz u\in U,\ \Ns u\in\dom\Jmac\}.
\end{equation}
\end{definition}

The operator $\Nb$ measures the effect of non-stationarity of $\sM$ in the slow variable. It is recorded for completeness and is not used in the sequel: even for the non-stationary Dirac model of Chapter~\ref{sec:top}, no hypothesis on $\Nb$ is needed (Remark~\ref{rem:comp3}).

We now compute $\Ns$. Let $u\in U$ and $w:=\Kmic u$. Since $\Se u=u$, $\Se^{-1}\Me\Se=\sM$, $\Se^{-1}\sJ\Se=\Je$, and $\Ps$ commutes with $\Se$,
\begin{equation}\label{eq:31}
\Ns u=\Mz u-\sM(u+w)-\Big(\Jmac+\frac1\varepsilon\Jmic\Big)w-\alpha\Ps w .
\end{equation}
The range of $\Kmic$ lies in $\Hflu\subset\ker\Jmic$, so $\Jmic w=0$, and the singular term vanishes:
\begin{equation}\label{eq:31.2}
\Ns u=\Mz u-\sM(u+w)-\Jmac w-\alpha\Ps w .
\end{equation}
Under \ref{S0}, the last term vanishes. We now project \eqref{eq:31.2}.
\begin{itemize}
\item \emph{Macroscopic part.} $\Pmac\Mz u=\Mz u=\Pmac\sM(u+w)$ by \eqref{eq:5.2}. Also $\Pmac\Jmac w=\Pmac\Jmac\Pflu w=0$ by \eqref{eq:39}, and $\Pmac\Ps w=\Ps\Pmac w=0$ by \ref{S}. Hence
\begin{equation}\label{eq:110}
\Pmac\Ns u=0 .
\end{equation}
\item \emph{Fluctuation part.} $\Pflu\Mz u=0$ and $\Pflu\Jmac\Pflu w=0$ by \eqref{eq:40}. Since $\Pflu\Ps u=\Ps\Pflu u=0$, we get $\Pflu\big[\sM(u+w)+\alpha\Ps w\big]=\Pflu(\sM+\alpha\Ps)(u+w)=0$ by the corrector equation \eqref{eq:29}. Hence
\begin{equation}\label{eq:124}
\Pflu\Ns u=0 .
\end{equation}
\end{itemize}
Consequently $\Ns u\in\Hgra$ for every $u\in U$.\label{eq:120}\label{eq:123}\label{eq:64}

\begin{remark}\label{rem:schur-2}
To gain intuition why \eqref{eq:124} is true, we consider $\sM+\alpha\Ps$ as a $2\times2$ block matrix acting on the orthogonal decomposition $\Pz=\Hmac\oplus\Hflu$:
\begin{equation}\label{eq:125}
\sM+\alpha\Ps=\begin{pmatrix}\Pmac(\sM+\alpha\Ps)\Pmac&\Pmac\sM\Pflu\\ \Pflu\sM\Pmac&\Pflu(\sM+\alpha\Ps)\Pflu\end{pmatrix}.
\end{equation}
The off-diagonal blocks contain no shift, by \ref{S}. In homogenization, we want to find an effective macroscopic function. To do this, we must eliminate the microscopic fluctuations (the bottom-right block). In linear algebra, eliminating a block of a matrix is done via the Schur complement/block Gaussian elimination. The corrector $\Kmic$ is exactly the algebraic operator that solves the bottom row of this system. By defining $\Kmic$ as in Theorem~\ref{thm:corr}, we are explicitly forcing the off-diagonal coupling block $\Pflu\sN\Pmac$ to be zero. Therefore, \eqref{eq:124} is not a coincidence but the algebraic principle of the corrector itself.
\end{remark}

\begin{proposition}[Structure of Residue Fields]\label{pro:resi-fiel}
Let $u_0\in\Hmac$ and $w:=\Kmic u_0$ (thus $w\in\Hflu$ is a solution of the abstract corrector problem $\Pflu(\sM+\alpha\Ps)(u_0+w)=0$). Suppose $w\in\dom\Jmac$. Then
\begin{enumerate}[label=(\roman*)]
\item $\Pmac\Ns u_0=0=\Pflu\Ns u_0$, and
\begin{equation}\label{eq:55}
\Ns u_0=\Pgra\Ns u_0=-\Pgra\big(\sM(u_0+w)+\Jmac w\big).
\end{equation}
In particular, $\Ns u_0\in\Hgra$.
\item Suppose moreover that $u_0$ belongs to the domain of $\Nb$ in \eqref{eq:155}, and that $\Pflu$ and $\Pgra$ preserve $\dom\Jmac$ (this holds in all examples, where both projections act only in the fast variable). Let $q(u_0):=\sM(u_0+w)+\Jmac w$. Then
\begin{equation}\label{eq:61}
\Pflu\Nb u_0=-\Pflu\Jmac q(u_0),\qquad\Pmac\Nb u_0=0 .
\end{equation}
\end{enumerate}
\end{proposition}

We explain the intuition behind this result:
\begin{itemize}
\item $\Pmac\Ns u_0=0$: The effective operator $\Mz$ is chosen precisely so that the macroscopic average of the residual is zero.
\item $\Pflu\Ns u_0=0$: The corrector $\Kmic$ is chosen precisely so that $w=\Kmic u_0$ annihilates the fluctuation component.
\item $\Pflu\Nb u_0$ does not vanish, which accounts for the fact that macroscopic variations in the material generate new microscopic fluctuations of the residual.
\end{itemize}

\begin{proof}
1. From Definition~\ref{def:corr-resi} and \eqref{eq:31.2},
\begin{equation}\label{eq:150}
\Ns u_0=\Mz u_0-\sM(u_0+w)-\Jmac w-\alpha\Ps w .
\end{equation}
By \eqref{eq:5.2}, $\Mz u_0=\Pmac\sM(u_0+w)$. Hence $\Pmac\Ns u_0=\Pmac\sM(u_0+w)-\Pmac\sM(u_0+w)-\Pmac\Jmac\Pflu w-\alpha\Ps\Pmac w=0$ by \eqref{eq:39} and \ref{S}, and $\Pflu\Ns u_0=-\Pflu\sM(u_0+w)-\Pflu\Jmac\Pflu w-\alpha\Ps w=-\Pflu(\sM+\alpha\Ps)(u_0+w)=0$ by \eqref{eq:40}, $\Pflu\Ps u_0=0$, and Definition~\ref{def:corr}. Since $\Ns u_0$ has no component in $\Hmac$ and none in $\Hflu$, it lies in $\Hgra$, and \eqref{eq:55} follows.

2. From \eqref{eq:55}, $\Ns u_0=-\Pgra q(u_0)$, and
\begin{equation}\label{eq:63}
\Nb u_0=\Jmac(\Ns u_0)-\Ns(\Jz u_0)=-\Jmac\Pgra q(u_0)+\Pgra q(\Jz u_0).
\end{equation}
Applying $\Pflu$, the second term vanishes, and $\Pflu\Jmac=\Pflu\Jmac\Pgra$ on $\dom\Jmac$ (Remark~\ref{rem:zero-blk}, together with the invariance of $\dom\Jmac$ under $\Pmac$, $\Pflu$, $\Pgra$) gives $\Pflu\Nb u_0=-\Pflu\Jmac q(u_0)$. Applying $\Pmac$ and using \eqref{eq:26} together with skew-adjointness, $\Pmac\Jmac\Pgra=0$, so $\Pmac\Nb u_0=0$.
\end{proof}

\section{Frequency cut-off, flux corrector and gauge freedom}\label{sec:fluxcorr}

\begin{definition}[Macroscopic Frequency Cut-off]\label{def:PR}
For $R>0$, let $\sP_R:=1_{[-R,R]}(-i\Jz)$ be the spectral projection of the self-adjoint operator $-i\Jz$ on $\Hmac$, extended by $0$ on $\Hmac^\perp$. The cut-off has three elementary properties: $\sP_R\Hmac\subset\Hmac\cap\dom\Jz$, $\nrm{\Jz\sP_Rf}\le R\nrm f$, and $\nrm{(I-\sP_R)f}\le R^{-1}\nrm{\Jz f}$ for $f\in\Hmac\cap\dom\Jz$.
\end{definition}

We will introduce more hypotheses: the (Q) series is for the general quantitative theorems, while the (G) series is for the special case where $\Jmic$ has a spectral gap.
\begin{enumerate}[label=(Q\arabic*),leftmargin=3.2em]
\item\label{Q1} \emph{The Source Space:} $\Hsrc\subset\Hmac$ is a closed subspace with $\sP_R\Hsrc\subset\Hsrc$ for all $R>0$. (The $\sP_R$-invariance is used only in Theorem~\ref{thm:quant}(ii)--(iii); Theorem~\ref{thm:gauge} and Corollary~\ref{cor:gaugegap} only need a closed subspace $\Hsrc\subset\Hmac$.)
\end{enumerate}
\begin{enumerate}[label=(G\arabic*),leftmargin=3.2em]
\item\label{G1} \emph{The Spectral Gap:} There is $\lambda_0>0$ with $\nrm{\Jmic w}\ge\lambda_0\nrm w$ for all $w\in\Hgra\cap\dom\Jmic$.
\end{enumerate}

\begin{lemma}\label{lem:gap}
Under \ref{G1}, $\ran\Jmic=\Hgra$. The restriction $\Jmic|_{\Hgra\cap\dom\Jmic}$ is a bijection onto $\Hgra$, and its inverse $\Jmic^{-1}:\Hgra\to\Hgra\cap\dom\Jmic$ satisfies $\nrm{\Jmic^{-1}}\le\lambda_0^{-1}$.
\end{lemma}

\begin{proof}
Since $\Hgra=\Pz^\perp=1_{\R\setminus\{0\}}(-i\Jmic)\Hs$ reduces $\Jmic$, the part of $\Jmic$ in $\Hgra$ is closed. By \ref{G1} it is bounded below by $\lambda_0$ on its domain $\Hgra\cap\dom\Jmic$, hence injective with closed range. Its range is dense in $\Hgra$ because $\overline{\ran\Jmic}=\Hgra$ \eqref{eq:98} and $\Jmic\Pz=0$. Hence it is a bijection onto $\Hgra$ with $\nrm{\Jmic^{-1}}\le\lambda_0^{-1}$.
\end{proof}

Because $\Pi_0=\Pmac\oplus\Pflu$, \eqref{eq:110} and \eqref{eq:124} imply $\Pi_0\sN\Pmac=0$. By the orthogonal decomposition \eqref{eq:35}, for any $u\in U$, $\sN\Pmac u\in\Hgra=\overline{\ran(\Jmic)}$. There must exist some operator $\Phi_{\mathrm{mic}}$ such that
\begin{equation}\label{eq:112}
\Ns=\sN\Pmac\approx\Jmic\Phi_{\mathrm{mic}} .
\end{equation}
This operator $\Phi_{\mathrm{mic}}$ is the abstract flux corrector, which is the anti-derivative of the residual. The approximation $\approx$ accounts for the closure over $\ran(\Jmic)$. To make this rigorous, we introduce

\begin{definition}[Flux Corrector]\label{def:flux}
Let $f\in\Hmac$, $u=\Az^{-1}f$, and assume $\Kmic u\in\dom\Jmac$, so that $\Ns u\in\Hgra$ by Proposition~\ref{pro:resi-fiel}. We say $\eta\in\R$ is admissible if $\eta\ne0$, or if $\eta=0$ and \ref{G1} holds. For admissible $\eta$ set
\begin{equation}\label{eq:56}
\Phi^\eta_f:=(\Jmic-\eta)^{-1}\Ns u,
\end{equation}
where in the case $\eta=0$, $\Jmic^{-1}:\Hgra\to\Hgra\cap\dom\Jmic$ is the inverse of Lemma~\ref{lem:gap}. Under \ref{G1} we write $\Phi_f:=\Phi^0_f$.
\end{definition}

Because $\Jmic$ is skew-adjoint, $\sigma(\Jmic)\subset i\R$, so the operator $(\Jmic-\eta I)$, $\eta\neq0$, is boundedly invertible, with $\nrm{(\Jmic-\eta I)^{-1}}\le1/|\eta|$. Hence $\Phi^\eta_f$ is well-defined. If $\Jmic$ has closed range, then also $\eta=0$ is allowed. In all cases
$\Ns u=\Jmic\Phi^\eta_f-\eta\Phi^\eta_f$ with $\Phi^\eta_f\in\Hgra\cap\dom(\Jmic)$, which is the rigorous version of \eqref{eq:112}.

\begin{lemma}[Properties of the Flux Corrector]\label{lem:flux}
In the setting of Definition~\ref{def:flux}, let $\mu_f$ be the spectral measure of the self-adjoint operator $-i\Jmic$ in the vector $\Ns u$, that is, $\mu_f(B)=\nrm{1_B(-i\Jmic)\Ns u}^2$ for Borel sets $B\subset\R$. For every admissible $\eta$:
\begin{enumerate}[label=(\roman*)]
\item $\Phi^\eta_f\in\Hgra\cap\dom\Jmic$ and
\begin{equation}\label{eq:fluxeq}
\Jmic\Phi^\eta_f-\eta\Phi^\eta_f=\Ns u ;
\end{equation}
\item $\displaystyle\nrm{\Phi^\eta_f}^2=\int_\R\frac{d\mu_f(\lambda)}{\lambda^2+\eta^2}$ and $\displaystyle\nrm{\eta\Phi^\eta_f}^2=\int_\R\frac{\eta^2\,d\mu_f(\lambda)}{\lambda^2+\eta^2}$; in particular $\nrm{\eta\Phi^\eta_f}\le\nrm{\Ns u}$, and $\nrm{\Phi^\eta_f}\le\nrm{\Ns u}/|\eta|$ if $\eta\neq0$;
\item $\mu_f(\{0\})=0$, hence $\nrm{\eta\Phi^\eta_f}\to0$ as $\eta\to0$, for each fixed $f$;
\item under \ref{G1}, $\mu_f$ vanishes on $(-\lambda_0,\lambda_0)$, so $\nrm{\Phi^\eta_f}\le\nrm{\Ns u}/\lambda_0$ for every $\eta$, including $\eta=0$.
\end{enumerate}
\end{lemma}

\begin{proof}
Since $\Jmic$ is skew-adjoint, $\Jmic=i(-i\Jmic)$ with $-i\Jmic$ self-adjoint, and $\ker\Jmic=\Pz$ by \ref{H2}(i). Hence $\Hgra=\Pz^\perp=1_{\R\setminus\{0\}}(-i\Jmic)\Hs$. In particular $\Hgra$ reduces $\Jmic$, and $\mu_f(\{0\})=0$ because $\Ns u\in\Hgra$.

Case $\eta\neq0$ (with or without \ref{G1}). The Borel functional calculus gives $(\Jmic-\eta)^{-1}=g_\eta(-i\Jmic)$ with $g_\eta(\lambda)=(i\lambda-\eta)^{-1}$ and $|g_\eta(\lambda)|^2=(\lambda^2+\eta^2)^{-1}$. This operator commutes with the spectral projection onto $\Hgra$, so $\Phi^\eta_f\in\Hgra\cap\dom\Jmic$, and \eqref{eq:fluxeq} holds by construction. The two integrals follow from the spectral theorem, and the two inequalities from $\eta^2/(\lambda^2+\eta^2)\le1$ and $1/(\lambda^2+\eta^2)\le\eta^{-2}$. (iii) As $\eta\to0$, $\eta^2/(\lambda^2+\eta^2)\to1_{\{0\}}(\lambda)$ pointwise and boundedly; dominated convergence gives $\nrm{\eta\Phi^\eta_f}^2\to\mu_f(\{0\})=0$.

Case $\eta=0$ under \ref{G1}. (i) is Lemma~\ref{lem:gap}, and the integrals in (ii) follow in the same way with $g_0(\lambda)=(i\lambda)^{-1}$.

(iv) Let $E:=1_{(-\lambda_0,\lambda_0)\setminus\{0\}}(-i\Jmic)\Hs\subset\Hgra$. Every $w\in E$ with $w\ne0$ satisfies $\nrm{\Jmic w}<\lambda_0\nrm w$, which contradicts \ref{G1}. So $E=\{0\}$, and $\mu_f$ is supported on $\{|\lambda|\ge\lambda_0\}$. Then $(\lambda^2+\eta^2)^{-1}\le\lambda_0^{-2}$ on the support of $\mu_f$, which implies the bound.
\end{proof}

\begin{enumerate}[label=(Q\arabic*),leftmargin=3.2em,start=2]
\item\label{Q2} \emph{The Regularized Flux Corrector at Level $\eta$:} Let $\eta$ be admissible. For every $f\in\Hsrc\cap\dom\Jz$, with $u=\Az^{-1}f$, we have $\Kmic u\in\dom\Jmac$ and $\Phi^\eta_f\in\dom\Jmac$, and the quantities
\[
a(\eta):=\sup_f\frac{\nrm{\Phi^\eta_f}}{\nrm f_{\Jz}},\qquad b(\eta):=\sup_f\frac{\nrm{\Jmac\Phi^\eta_f}}{\nrm f_{\Jz}},\qquad c(\eta):=\sup_f\frac{\nrm{\eta\Phi^\eta_f}}{\nrm f_{\Jz}}
\]
are finite. Here $\nrm f_{\Jz}:=\nrm f+\nrm{\Jz f}$, and the suprema are taken over $f\in\Hsrc\cap\dom\Jz$ with $f\ne0$.
\end{enumerate}

\begin{remark}\label{rem:Q2}
By Lemma~\ref{lem:flux}(ii), $a(\eta)$ and $c(\eta)$ are finite as soon as $\nrm{\Ns u}\le\nu\nrm f_{\Jz}$ for some $\nu$: then $c(\eta)\le\nu$, and $a(\eta)\le\nu/|\eta|$ for $\eta\ne0$, or $a(0)\le\nu/\lambda_0$ under the gap \ref{G1}. The genuine requirement is the bound on $b(\eta)$: one macroscopic derivative of the flux corrector.
\end{remark}

\begin{enumerate}[label=(G\arabic*),leftmargin=3.2em,start=2]
\item\label{G2} \emph{The Regularity of Effective Problem (with Spectral Gap):} Assume \ref{G1}. There is $C_{\mathrm{reg}}$ such that for every $f\in\Hsrc\cap\dom\Jz$, with $u=\Az^{-1}f$:
\begin{enumerate}[label=(\roman*)]
\item $\Kmic u\in\dom\Jmac$ and $\Phi_f\in\dom\Jmac$;
\item $\nrm{\Ns u}+\nrm{\Jmac\Phi_f}\le C_{\mathrm{reg}}\big(\nrm f+\nrm{\Jz f}\big)$.
\end{enumerate}
\end{enumerate}

\begin{remark}[(G2) is a special case of (Q2)]\label{rem:G2red}
Under \ref{G1}, \ref{G2} is exactly \ref{Q2} with $\eta=0$: by Lemma~\ref{lem:flux}(iv), $a(0)\le C_{\mathrm{reg}}/\lambda_0$, $b(0)\le C_{\mathrm{reg}}$, and $c(0)=0$. Corollary~\ref{thm:main} is therefore a special case of Theorem~\ref{thm:quant}; we state it separately only to record its constants.
\end{remark}

\begin{remark}\label{rem:G2}
\ref{G2} concerns only the effective problem and its correctors, and only for data in $\dom\Jz$. In the periodic example (Chapter~\ref{sec:ell}) it amounts to $\nrm{u_0}_{H^3}\le C\nrm f_{H^1}$ for the constant-coefficient effective equation. Theorem~\ref{thm:quant} does not assume \ref{G2}, but it does assume \ref{Q2}, which is a regularity-type hypothesis of the same nature (with the norm $\nrm f_{\Jz}$ on the data). Only the gauge-fixed Theorem~\ref{thm:gauge} is free of any regularity assumption on $f$.
\end{remark}

\paragraph{Gauge freedom.} The flux corrector in Definition~\ref{def:flux} is only one choice of antiderivatives of the residual. Other antiderivatives are equally good for the error identity below (Lemma~\ref{lem:id}), and they may have a better-behaved macroscopic derivative, which improves the convergence rate (see Theorem~\ref{thm:gauge} below).

\begin{definition}[Gauge class]\label{def:gauge}
In the setting of Definition~\ref{def:flux}, for admissible $\eta$ let
$G^\eta_f:=\{\Phi\in\dom\Jmac:\ \Phi-\Phi^\eta_f\in\Pz\}$.
Every $\Phi\in G^\eta_f$ lies in $\mathcal D=\dom\Jmac\cap\dom\Jmic$, because $\Pz=\ker\Jmic\subset\dom\Jmic$, and satisfies $\Jmic\Phi=\Jmic\Phi^\eta_f$. By \eqref{eq:fluxeq},
\begin{equation}\label{eq:gaugeeq}
\Jmic\Phi=\Ns u+\eta\Phi^\eta_f\qquad(\Phi\in G^\eta_f).
\end{equation}
\end{definition}

\begin{remark}\label{rem:gauge}
For $\eta=0$ (under \ref{G1}), $G^0_f$ is exactly the set of solutions $\Phi\in\mathcal D$ of $\Jmic\Phi=\Ns u$: two solutions differ by an element of $\ker\Jmic=\Pz$. For $\eta\neq0$, the equation $(\Jmic-\eta)\Phi=\Ns u$ has the unique solution $\Phi^\eta_f$. The gauge freedom is recovered by keeping the defect fixed: every $\Phi\in G^\eta_f$ solves $\Jmic\Phi=\Ns u+\varrho$ with the same $\varrho=\eta\Phi^\eta_f$. The element $\Phi^\eta_f$ itself need not lie in $\dom\Jmac$; only the chosen representative must. More generally, any $\Phi\in\mathcal D$ may be used, with its own defect $\varrho:=\Jmic\Phi-\Ns u$ (Lemma~\ref{lem:id}); this is how the regularization is chosen in Sections~\ref{sec:stoch} and \ref{sec:random}.
\end{remark}

\begin{enumerate}[label=(GF),leftmargin=3.2em]
\item\label{GF} \emph{The Gauge-fixed Flux Corrector with Defect:} There are constants $\tilde a,\tilde b,\tilde c$ such that for every $f\in\Hsrc$, with $u=\Az^{-1}f$: $\Kmic u\in\dom\Jmac$, and there is $\Phi^g_f\in\mathcal D$ with
\[
\nrm{\Phi^g_f}\le\tilde a\nrm f,\qquad\nrm{\Jmac\Phi^g_f}\le\tilde b\nrm f,\qquad\nrm{\Jmic\Phi^g_f-\Ns u}\le\tilde c\nrm f.
\]
For $\Phi^g_f\in G^\eta_f$ the defect is $\eta\Phi^\eta_f$, which recovers the previous formulation at level $\eta$.
\end{enumerate}

\begin{remark}\label{rem:GF}
The difference from \ref{Q2} is the norm on the data: here we require only $\nrm f$, not $\nrm f_{\Jz}$, and all $f\in\Hsrc$, not only regular ones. The bound $\tilde c$ is automatic for $\Phi^g_f\in G^\eta_f$ if $\nrm{\Ns u}\le\nu\nrm f$ (Lemma~\ref{lem:flux}(ii)), and $\tilde c=0$ for $\eta=0$. The genuine requirement is $\tilde b$: a gauge in which one macroscopic derivative of the flux corrector costs no derivative of the data.
\end{remark}

\section{The error identity}\label{sec:erroridentity}
We refine the general result in Section~\ref{sec:2.1} for quantitative homogenization.

\begin{definition}[Error Operator]\label{def:err}
Define
\begin{equation}\label{eq:1}
\Ee:=\Ae^{-1}-(I+\Ke)\Az^{-1}.
\end{equation}
\end{definition}

We mimic the proof of Theorem~\ref{thm:freq} to estimate the error operator acting on a single vector.

\begin{lemma}[Error Identity]\label{lem:id}
Let $f\in\Hmac$, $u=\Az^{-1}f$, and suppose $\Kmic u\in\dom\Jmac$. Let $\Phi\in\mathcal D$ be arbitrary and define its defect $\varrho:=\Jmic\Phi-\Ns u\in\Hs$. Then
\begin{equation}\label{eq:id}
\Ee f=\varepsilon\big[\Se\Phi-\Ae^{-1}(\Me+\alpha\Ps)\Se\Phi\big]-\varepsilon\Ae^{-1}\Se\Jmac\Phi-\Ae^{-1}\Se\varrho.
\end{equation}
Consequently, with $C_*:=1+C_A(\Mb+\alpha)$,
\begin{equation}\label{eq:idbound}
\nrm{\Ee f}\le\varepsilon C_*\nrm\Phi+C_A\big(\varepsilon\nrm{\Jmac\Phi}+\nrm\varrho\big).
\end{equation}
For $\Phi\in G^\eta_f$ (for instance $\Phi=\Phi^\eta_f$ when $\Phi^\eta_f\in\dom\Jmac$) the defect is $\varrho=\eta\Phi^\eta_f$. In particular, under \ref{G1}, with $\eta=0$ and $\Phi=\Phi_f$,
\begin{equation}\label{eq:idbound0}
\nrm{\Ee f}\le\varepsilon\Big[C_*\frac{\nrm{\Ns u}}{\lambda_0}+C_A\nrm{\Jmac\Phi_f}\Big].
\end{equation}
\end{lemma}

\begin{proof}
1. (domains) We have $\Kmic u\in\Hflu\subset\ker\Jmic$ and $\Kmic u\in\dom\Jmac$, so $\Kmic u\in\mathcal D$ and $\Je\Kmic u=\Jmac\Kmic u$. Hence $\Ke u=\Se\Kmic u\in\dom\sJ$ with $\sJ\Ke u=\Se\Je\Kmic u=\Se\Jmac\Kmic u$. Also $u\in\dom\sJ$. So $(I+\Ke)u\in\dom\Ae$, and since $\Az u=f$,
\begin{equation}\label{eq:86}
\Ee f=\Ae^{-1}\big[\Az u-\Ae(I+\Ke)u\big].
\end{equation}

2. (residual) Note that $\Az u=\Jmac u+\Mz u+\alpha\Ps u$. Using $\Se^{-1}u=u$ from \ref{H3}(ii) and $\Ps\Se=\Se\Ps$,
$\Ae(I+\Ke)u=\Jmac u+\Se\Jmac\Kmic u+\Se\sM(I+\Kmic)u+\alpha\Ps u+\alpha\Se\Ps\Kmic u$.
Since $\Se\Mz u=\Mz u$, $\Az u-\Ae(I+\Ke)u=\Se\Ns u$.

3. (antiderivative) Since $\Phi\in\mathcal D$, $\sJ\Se\Phi=\Se\Je\Phi=\Se\Jmac\Phi+\varepsilon^{-1}\Se\Jmic\Phi$. By definition of $\varrho$, $\Ns u=\Jmic\Phi-\varrho$, hence
\begin{equation}\label{eq:87}
\Se\Ns u=\varepsilon\sJ\Se\Phi-\varepsilon\Se\Jmac\Phi-\Se\varrho.
\end{equation}

4. (elimination of $\sJ$, as in \eqref{eq:118}) Since $\Se\Phi\in\dom\sJ=\dom\Ae$, $\Ae^{-1}\sJ\Se\Phi=\Se\Phi-\Ae^{-1}(\Me+\alpha\Ps)\Se\Phi$. Applying $\Ae^{-1}$ to \eqref{eq:87} and using \eqref{eq:86} gives \eqref{eq:id}. Estimate \eqref{eq:idbound} follows from $\nrm{\Ae^{-1}}\le C_A$, $\nrm{\Me+\alpha\Ps}\le\Mb+\alpha$ and unitarity of $\Se$. For \eqref{eq:idbound0}, take $\eta=0$, $\Phi=\Phi_f$, and use $\nrm{\Phi_f}\le\lambda_0^{-1}\nrm{\Ns u}$ (Lemma~\ref{lem:flux}(iv)).
\end{proof}

\begin{remark}[Other shifts]\label{rem:shiftgen}
The proof of Lemma~\ref{lem:id} uses only: (a) $\Az u-\Ae(I+\Ke)u=\Se\Ns u$; (b) the identity $\Ae^{-1}\sJ\Se\Phi=\Se\Phi-\Ae^{-1}(\Me+\alpha\Ps)\Se\Phi$ for $\Se\Phi\in\dom\sJ$; (c) a bound $\nrm{\Ae^{-1}}\le C_A$. Hence it holds verbatim if $\alpha\Ps$ is replaced by any bounded operator $\Lambda$ commuting with $\Se$ and $\Pmac$ (with $\Ae:=\sJ+\Me+\Lambda$ and $\Az:=\Jz+\Mz+\Lambda$ on $\Hmac$), provided $\Az$ is invertible on $\Hmac$ and $C_A\ge\nrm{\Ae^{-1}}$; then $\Mb+\alpha$ is replaced by $\Mb+\nrm\Lambda$. It even holds when no shift is singled out and the spectral parameter is part of $\sM$, as in Section~\ref{sec:indef}.
\end{remark}

\begin{remark}[Relation to Theorem~\ref{thm:freq}]\label{rem:23}
Identity \eqref{eq:id} is the decomposition \eqref{eq:11c}--\eqref{eq:118} of Theorem~\ref{thm:freq} evaluated on the single vector $u=\Az^{-1}f$, with $\Be\Az^{-1}f$ replaced by $\varepsilon\Se\Phi$ and $\sR_\varepsilon\Az^{-1}f$ by $-\varepsilon\Se\Jmac\Phi-\Se\varrho$. Because the Sylvester commutator equation \eqref{eq:11} is solved on one vector, two things disappear:
\begin{itemize}
\item the term $-\Ae^{-1}\Be\sJ\Az^{-1}$, which in operator form needs $\Be$ on all of $\Hs$ with values in $\V^*$;
\item the $C_3$-term (macro/micro splitting), since $u\in\Hmac$.
\end{itemize}
All norms in \eqref{eq:idbound} are $\Hs$-norms. This removes the issues that $\Se$ is not uniformly bounded on $\V^*$, and that the map $f\mapsto\Phi^\eta_f$ may be unbounded in the norms required by \ref{A3} in the homogenization setting.
\end{remark}

\begin{lemma}[High frequencies are small in $\V^*$]\label{lem:high}
For $f\in\Hmac$ and $R>0$,
\[
\nrm{(I-\sP_R)f}_{\V^*}\le R^{-1}\nrm f,\qquad\nrm{\Ee(I-\sP_R)f}\le C''_{\V}\big(2+C_A\Mb\big)R^{-1}\nrm f,
\]
where $C'_{\V}:=\big[C_A^2+(1+(\Mb+\alpha)C_A)^2\big]^{1/2}$ and $C''_{\V}:=\big[C_A^2+(1+(\Mb+C_A\Mb^2+\alpha)C_A)^2\big]^{1/2}\ge C'_{\V}$.
\end{lemma}

\begin{proof}
Let $v\in\V=\dom\sJ$. By Lemma~\ref{lem:comm}(iii) and \ref{H3}(ii), $\Pmac v\in\Hmac\cap\dom\Jz$ and $\Jz\Pmac v=\Pmac\sJ v$. Since $(I-\sP_R)f\in\Hmac$ and $\sP_R$ is self-adjoint on $\Hmac$,
\begin{multline*}
|\ip{(I-\sP_R)f}{v}|=|\ip{f}{(I-\sP_R)\Pmac v}|\le\nrm fR^{-1}\nrm{\Jz\Pmac v}\\\le R^{-1}\nrm f\nrm{\sJ v}\le R^{-1}\nrm f\nrm v_{\V}.
\end{multline*}
The same computation with $v\in\V_0:=\dom\Jz$ (graph norm) gives $\nrm{(I-\sP_R)f}_{\V_0^*}\le R^{-1}\nrm f$. Next, let $g\in\Hs$ and $u=(\Ae^*)^{-1}g$. By Theorem~\ref{thm:corr}(i), $\nrm u\le C_A\nrm g$, and $-\sJ u=g-(\Me^*+\alpha\Ps)u$ gives $\nrm{\sJ u}\le(1+(\Mb+\alpha)C_A)\nrm g$. So $\nrm{(\Ae^*)^{-1}}_{\Hs\to\V}\le C'_{\V}$, and by duality (as in the proof of Lemma~\ref{lem:inve-coer}) $\nrm{\Ae^{-1}}_{\V^*\to\Hs}\le C'_{\V}$. The same argument in $\Hmac$, with $\V_0$ and Theorem~\ref{thm:corr}(iii) for $\Az^*$, and with $\nrm{\Mz}\le\Mb+C_A\Mb^2$ (Theorem~\ref{thm:corr}(ii)) in place of $\nrm{\Me}\le\Mb$, gives $\nrm{\Az^{-1}}_{\V_0^*\to\Hmac}\le C''_{\V}$. The bound $\Mb$ cannot be used for $\Mz$: on $\Hmac\oplus\Hflu=\C\oplus\C$, the accretive matrix $\sM=\begin{pmatrix}0&1\\-1&\delta\end{pmatrix}$ has $\nrm\sM\le2$ for $\delta\le1$, while its Schur complement is $\Mz=1/\delta$. Write $G=(I-\sP_R)f$. Then
\[
\nrm{\Ee G}\le\nrm{\Ae^{-1}G}+(1+\nrm{\Ke})\nrm{\Az^{-1}G}\le C''_{\V}(2+C_A\Mb)R^{-1}\nrm f,
\]
since $\nrm{\Ke}=\nrm{\Kmic}\le C_A\Mb$.
\end{proof}

\section{Norm resolvent estimates}\label{sec:nre}
Throughout this section
\begin{equation}\label{eq:52}
C_*:=1+C_A(\Mb+\alpha),\qquad C_{\mathrm{hf}}:=C''_{\V}\big(2+C_A\Mb\big),\qquad\nrm f_{\Jz}:=\nrm f+\nrm{\Jz f}.
\end{equation}

\begin{theorem}[General Norm Resolvent Estimate on the Source Space]\label{thm:quant}
Assume \ref{H1}--\ref{H3}, \ref{S}, \ref{H4} with $\alpha>0$, and \ref{Q1}. Let $\varepsilon>0$, and let $\eta$ be admissible (it may depend on $\varepsilon$) such that \ref{Q2} holds. Put
\[
\kappa_\varepsilon(\eta):=\varepsilon C_*a(\eta)+C_A\big(\varepsilon b(\eta)+c(\eta)\big).
\]
\begin{enumerate}[label=(\roman*)]
\item (Regular data) For $f\in\Hsrc\cap\dom\Jz$, $\nrm{\Ee f}\le\kappa_\varepsilon(\eta)\nrm f_{\Jz}$.
\item (Splitting bound) For $f\in\Hsrc$ and $R>0$, $\nrm{\Ee f}\le\big[C_{\mathrm{hf}}R^{-1}+(1+R)\kappa_\varepsilon(\eta)\big]\nrm f$.
\item (Norm resolvent estimate)
\begin{equation}\label{eq:89x}
\big\|\big(\Ae^{-1}-(I+\Ke)\Az^{-1}\big)\big|_{\Hsrc}\big\|_{\Hs\to\Hs}\le\kappa_\varepsilon(\eta)+2\sqrt{C_{\mathrm{hf}}\kappa_\varepsilon(\eta)} .
\end{equation}
\end{enumerate}
In particular, if for every small $\varepsilon$ there is an admissible $\eta=\eta(\varepsilon)$ satisfying \ref{Q2} with $\kappa_\varepsilon(\eta(\varepsilon))\to0$, then $\Ae^{-1}-(I+\Ke)\Az^{-1}\to0$ in operator norm on $\Hsrc$, at the rate $O\big(\kappa_\varepsilon(\eta(\varepsilon))^{1/2}\big)$.
\end{theorem}

Note that in this theorem, neither the spectral gap nor any rate is assumed.

\begin{proof}
(i) Let $f\in\Hsrc\cap\dom\Jz$. By \ref{Q2}, Lemma~\ref{lem:id} applies with $\Phi=\Phi^\eta_f$ and $\varrho=\eta\Phi^\eta_f$, and \eqref{eq:idbound} together with the definitions of $a,b,c$ gives the bound.
(ii) Split $f=\sP_Rf+(I-\sP_R)f$. By \ref{Q1}, $\sP_Rf\in\Hsrc\cap\dom\Jz$, with $\nrm{\sP_Rf}_{\Jz}\le(1+R)\nrm f$. So (i) implies $\nrm{\Ee\sP_Rf}\le(1+R)\kappa_\varepsilon(\eta)\nrm f$. Lemma~\ref{lem:high} (which uses no spectral gap) gives $\nrm{\Ee(I-\sP_R)f}\le C_{\mathrm{hf}}R^{-1}\nrm f$.
(iii) Write $\kappa=\kappa_\varepsilon(\eta)$. If $\kappa=0$, let $R\to\infty$ in (ii). Otherwise take $R=\sqrt{C_{\mathrm{hf}}/\kappa}$ in (ii): $C_{\mathrm{hf}}R^{-1}+(1+R)\kappa=\kappa+2\sqrt{C_{\mathrm{hf}}\kappa}$.
\end{proof}

\begin{remark}[Only (i) is used in (ii)--(iii)]\label{rem:onlyi}
Parts (ii) and (iii) use part (i) only through an inequality $\nrm{\Ee f}\le\kappa\nrm f_{\Jz}$ on $\Hsrc\cap\dom\Jz$. Any constant $\kappa$ with this property may replace $\kappa_\varepsilon(\eta)$; for example a sharper one can be obtained from \eqref{eq:idbound} directly. This is how the gap case below (Corollary~\ref{thm:main}) is obtained with its best constants.
\end{remark}

\begin{remark}[The three terms of $\kappa_\varepsilon(\eta)$]\label{rem:kappa}
\begin{itemize}
\item $\varepsilon a(\eta)$ measures the size of the flux corrector. It is $O(\varepsilon)$ under a gap, and in general it grows as $\eta\to0$ (Lemma~\ref{lem:flux}(ii)).
\item $\varepsilon b(\eta)$ measures one macroscopic derivative of the flux corrector. This is the term that costs regularity of the effective problem.
\item $c(\eta)$ is the residue of the regularization. It carries no factor $\varepsilon$. It vanishes for $\eta=0$ under a gap, and for each fixed $f$ it tends to $0$ as $\eta\to0$ (Lemma~\ref{lem:flux}(iii)). The supremum over regular data, however, tends to $0$ only under a uniform condition on the spectral measures $\mu_f$ near $0$.
\end{itemize}
Without a gap one must balance $\eta$: a small $\eta$ makes $c(\eta)$ small but $a(\eta)$, and possibly $b(\eta)$, large. The quantities $a(\eta)$ and $c(\eta)$ are given exactly by the spectral integrals of Lemma~\ref{lem:flux}(ii).
\end{remark}

\begin{corollary}[Norm Resolvent Estimate under a Spectral Gap]\label{thm:main}
Assume \ref{H1}--\ref{H3}, \ref{S}, \ref{H4} with $\alpha>0$, and \ref{Q1}, \ref{G1}, and \ref{G2}. Put
$C_{\mathrm{gap}}:=C_{\mathrm{reg}}\max\big\{\frac{C_*}{\lambda_0},C_A\big\}$.
\begin{enumerate}[label=(\roman*)]
\item (Regular data) For $f\in\Hsrc\cap\dom\Jz$ and $\varepsilon>0$, $\nrm{\Ee f}\le C_{\mathrm{gap}}\varepsilon\big(\nrm f+\nrm{\Jz f}\big)$.
\item (Splitting bound) For $f\in\Hsrc$, $\varepsilon>0$ and $R>0$, $\nrm{\Ee f}\le\big[C_{\mathrm{hf}}R^{-1}+C_{\mathrm{gap}}\varepsilon(1+R)\big]\nrm f$.
\item (Norm resolvent convergence) For $0<\varepsilon\le1$,
$\big\|\big(\Ae^{-1}-(I+\Ke)\Az^{-1}\big)|_{\Hsrc}\big\|_{\Hs\to\Hs}\le(C_{\mathrm{hf}}+2C_{\mathrm{gap}})\varepsilon^{1/2}$.
\end{enumerate}
\end{corollary}

\begin{proof}
Under \ref{G1}, $\eta=0$ is admissible. \ref{G2} gives the domain conditions of \ref{Q2} for $\eta=0$, and $c(0)=0$ (Remark~\ref{rem:G2red}).
(i) For $f\in\Hsrc\cap\dom\Jz$, \eqref{eq:idbound0} and \ref{G2}(ii) give
$\nrm{\Ee f}\le\varepsilon\max\{\frac{C_*}{\lambda_0},C_A\}\big(\nrm{\Ns u}+\nrm{\Jmac\Phi_f}\big)\le C_{\mathrm{gap}}\varepsilon\nrm f_{\Jz}$.
(ii) This is Theorem~\ref{thm:quant}(ii) with $\kappa_\varepsilon(0)$ replaced by $C_{\mathrm{gap}}\varepsilon$, which is allowed by (i) and Remark~\ref{rem:onlyi}.
(iii) Take $R=\varepsilon^{-1/2}$ in (ii), and use $\varepsilon\le\varepsilon^{1/2}$.
\end{proof}

\begin{theorem}[Gauge-fixed Norm Resolvent Estimate]\label{thm:gauge}
Assume \ref{H1}--\ref{H3}, \ref{S}, \ref{H4} with $\alpha>0$, let $\Hsrc\subset\Hmac$ be a closed subspace, and assume \ref{GF}. Then
\[
\big\|\big(\Ae^{-1}-(I+\Ke)\Az^{-1}\big)|_{\Hsrc}\big\|_{\Hs\to\Hs}\le\tilde\kappa_\varepsilon:=\varepsilon C_*\tilde a+C_A\big(\varepsilon\tilde b+\tilde c\big).
\]
\end{theorem}

\begin{proof}
For $f\in\Hsrc$, apply Lemma~\ref{lem:id} with $\Phi=\Phi^g_f$, and use the three bounds of \ref{GF} in \eqref{eq:idbound}.
\end{proof}

\begin{corollary}[Linear Rate under a Spectral Gap and a Good Gauge]\label{cor:gaugegap}
Assume \ref{H1}--\ref{H3}, \ref{S}, \ref{H4} with $\alpha>0$, \ref{G1}, and \ref{GF} with $\tilde c=0$ (e.g.\ $\Phi^g_f\in G^0_f$). Then, for every $\varepsilon>0$,
\[
\big\|\big(\Ae^{-1}-(I+\Ke)\Az^{-1}\big)|_{\Hsrc}\big\|_{\Hs\to\Hs}\le\varepsilon\big(C_*\tilde a+C_A\tilde b\big).
\]
\end{corollary}

\begin{proof}
For $\Phi^g_f\in G^0_f$ the defect vanishes, so one may take $\tilde c=0$ in Theorem~\ref{thm:gauge}.
\end{proof}

\section{Sharpness and higher order}\label{sec:sharp}

The linear rate cannot be improved in general: under a spectral gap, the error is bounded below by $\varepsilon$ times the default flux corrector.

\begin{theorem}[Matching lower bound under a spectral gap]\label{thm:lower}
Assume \ref{H1}--\ref{H3}, \ref{S}, \ref{H4} with $\alpha>0$, and \ref{G1}. Let $f\in\Hmac$, $u=\Az^{-1}f$, $\Kmic u\in\dom\Jmac$, let $\Phi\in G^0_f$, and let $\Phi_f\in\Hgra$ be the default flux corrector. Assume that $\Xi_f:=-\Jmic^{-1}\Phi_f\in\Hgra$ lies in $\dom\Jmac$. Then, for every $\varepsilon>0$,
\[
\nrm{\Ee f}\ge\varepsilon\nrm{\Phi_f}-\varepsilon^2C(f),
\]
where, if $\Phi_f\neq0$,
\[
\begin{gathered}
C(f):=\frac{1}{\nrm{\Phi_f}}\big((\Mb+\alpha)\nrm\Phi+\nrm{\Jmac\Phi}\big)\\
\times\big(\nrm{\Xi_f}+C_A\big((\Mb+\alpha)\nrm{\Xi_f}+\nrm{\Jmac\Xi_f}\big)\big),
\end{gathered}
\]
and $C(f):=0$ if $\Phi_f=0$. In particular $\liminf_{\varepsilon\to0}\varepsilon^{-1}\nrm{\Ee f}\ge\nrm{\Phi_f}$, and $\Phi_f\neq0$ if and only if $\Ns u\neq0$.
\end{theorem}

\begin{proof}
We may assume $\Phi_f\neq0$. Put $\Psi:=\Jmic\Xi_f=-\Phi_f$. Since $\Phi-\Phi_f\in\Pz\perp\Hgra$, we have $\ip{\Phi}{\Psi}=-\nrm{\Phi_f}^2$. By Lemma~\ref{lem:id} with zero defect, $\Ee f=\varepsilon\Se\Phi-\varepsilon\Ae^{-1}Y$, where $Y:=(\Me+\alpha\Ps)\Se\Phi+\Se\Jmac\Phi$ and $\nrm Y\le(\Mb+\alpha)\nrm\Phi+\nrm{\Jmac\Phi}$. Pair with $\Se\Psi$:
\[
\ip{\Ee f}{\Se\Psi}=-\varepsilon\nrm{\Phi_f}^2-\varepsilon\ip{Y}{(\Ae^*)^{-1}\Se\Psi}.
\]
Now $\Se^{-1}\Ae^*\Se=-\Je+\sM^*+\alpha\Ps=:\sB$, which is boundedly invertible with $\nrm{\sB^{-1}}=\nrm{(\Ae^*)^{-1}}\le C_A$ by Theorem~\ref{thm:corr}(i). Since $\Xi_f\in\mathcal D$, a direct computation gives $\sB(-\varepsilon\Xi_f)=\Psi+\varepsilon\big(\Jmac\Xi_f-(\sM^*+\alpha\Ps)\Xi_f\big)$. Hence
\[
\sB^{-1}\Psi=-\varepsilon\Xi_f+\varepsilon\sB^{-1}\big((\sM^*+\alpha\Ps)\Xi_f-\Jmac\Xi_f\big),
\]
so $\nrm{(\Ae^*)^{-1}\Se\Psi}=\nrm{\sB^{-1}\Psi}\le\varepsilon\big(\nrm{\Xi_f}+C_A\big((\Mb+\alpha)\nrm{\Xi_f}+\nrm{\Jmac\Xi_f}\big)\big)$. Therefore $\nrm{\Ee f}\,\nrm{\Phi_f}\ge|\ip{\Ee f}{\Se\Psi}|\ge\varepsilon\nrm{\Phi_f}^2-\varepsilon^2C(f)\nrm{\Phi_f}$. Finally, $\Phi_f=\Jmic^{-1}\Ns u$ with $\Jmic$ injective on $\Hgra$.
\end{proof}

\begin{remark}\label{rem:lower}
\begin{itemize}
\item The lower bound is gauge independent: $\ip{\Phi}{\Jmic\Xi}=-\ip{\Ns u}{\Xi}$ for every $\Phi\in G^0_f$. Together with Corollary~\ref{cor:gaugegap}, it shows that under a gap and a good gauge, $\nrm{\Ee f}$ is of exact order $\varepsilon$ for every $f$ with $\Ns u\neq0$ and $\Xi_f\in\dom\Jmac$. Corollary~\ref{cor:lowerell} and Remark~\ref{rem:lowerD} apply this to the two examples.
\item The same argument works with a bounded shift $\Lambda$ as in Remark~\ref{rem:shiftgen}, provided $\Ae^*$ also satisfies $\nrm{(\Ae^*)^{-1}}\le C_A$.
\item Without a gap, $\Xi_f$ involves $\Jmic^{-2}$ applied to the residual and need not exist. In one dimension, for models with (conditionally) finite range of dependence, a matching lower bound is proved directly in Theorem~\ref{thm:lower1d}; in higher-dimensional random examples we have no lower bound, and the rates agree with those known in the literature (Remarks~\ref{rem:E2} and \ref{rem:frd}).
\end{itemize}
\end{remark}
\begin{theorem}[Second-order expansion under a gap]\label{thm:second}
Assume \ref{H1}--\ref{H3}, \ref{S}, \ref{H4}, $\alpha>0$ and \ref{G1}, and write $\sM_\alpha:=\sM+\alpha\Ps$. Let $\sL$ be the operator on $\Pz$ from the proof of Theorem~\ref{thm:corr}(iii), with $\nrm{\sL^{-1}}\le C_A$. Let $f\in\Hmac$, $u=\Az^{-1}f$ with $\Kmic u\in\dom\Jmac$, and let $\Phi\in G^0_f$ have zero defect. Define
\[
Y:=\sM_\alpha\Phi+\Jmac\Phi,\qquad v:=\sL^{-1}\Pi_0Y\in\Pz,
\]
and assume $\Pflu v\in\dom\Jmac$. Define further
\[
\Xi_1:=\Jmic^{-1}\Pgra Y,\qquad\Xi_2:=\Jmic^{-1}\Pgra(\Jmac+\sM_\alpha)v,
\]
and assume $\Xi_1,\Xi_2\in\dom\Jmac$. Then
\[
\Big\|\Ae^{-1}f-(I+\Ke)\Az^{-1}f-\varepsilon\Se(\Phi-v)\Big\|\le\varepsilon^2\sum_{i=1}^2\Big(\nrm{\Xi_i}+C_A\nrm{(\Jmac+\sM_\alpha)\Xi_i}\Big).
\]
\end{theorem}

\begin{proof}
By Lemma~\ref{lem:id} with zero defect, $\Ee f=\varepsilon\Se\Phi-\varepsilon\Ae^{-1}\Se Y$. We expand $\Ae^{-1}\Se Y$ in two steps.

\emph{Stiff directions.} For $\Xi\in\Hgra\cap\dom\Jmac$, $\Ae\Se(\varepsilon\Xi)=\Se\big(\Jmic\Xi+\varepsilon(\Jmac+\sM_\alpha)\Xi\big)$. Hence
\begin{equation}\label{eq:stiff}
\begin{gathered}
\Ae^{-1}\Se\Jmic\Xi=\varepsilon\Se\Xi-\varepsilon\Ae^{-1}\Se(\Jmac+\sM_\alpha)\Xi,\\\nrm{\Ae^{-1}\Se\Jmic\Xi}\le\varepsilon\big(\nrm\Xi+C_A\nrm{(\Jmac+\sM_\alpha)\Xi}\big).
\end{gathered}
\end{equation}

\emph{The $\Pz$-component.} Since $v\in\Pz\cap\dom\Jmac$, we have $\Jmic v=0$ and $\Ae\Se v=\Se(\Jmac+\sM_\alpha)v$. Moreover $\Pi_0\Jmac v=\Jz\Pmac v$ by \eqref{eq:37}, so $\Pi_0(\Jmac+\sM_\alpha)v=\sL v=\Pi_0Y$. Hence $\Ae\Se v=\Se\big(\Pi_0Y+\Jmic\Xi_2\big)$, and by \eqref{eq:stiff},
\[
\Ae^{-1}\Se\Pi_0Y=\Se v-\Ae^{-1}\Se\Jmic\Xi_2=\Se v+O\big(\varepsilon(\nrm{\Xi_2}+C_A\nrm{(\Jmac+\sM_\alpha)\Xi_2})\big).
\]

\emph{Conclusion.} Write $Y=\Pi_0Y+\Jmic\Xi_1$ and apply \eqref{eq:stiff} to $\Xi_1$. This gives $\Ae^{-1}\Se Y=\Se v+O(\varepsilon)$, with the stated constants, and therefore $\Ee f=\varepsilon\Se(\Phi-v)+O(\varepsilon^2)$.
\end{proof}

\begin{remark}\label{rem:second}
\begin{itemize}
\item \emph{Structure of the second corrector.} The second corrector is $\varepsilon\Se(\Phi-v)$. Solving $\sL v=\Pi_0Y$ blockwise gives $\Pflu v=\Kmic\Pmac v+\Mff^{-1}\Pflu Y$ and $\Az(\Pmac v)=\Pmac Y-\Pmac\sM\Pflu\Mff^{-1}\Pflu Y$. So the macroscopic part $-\varepsilon\Pmac v$ of the second corrector is the first-order correction to the homogenized solution, driven by a source that is one derivative of $u$ higher than $\Phi$. Its fluctuating parts are $\Phi$ (the flux corrector) and $-\Pflu v$. This is the abstract form of the classical expansion $u_0+\varepsilon(\chi\cdot\nabla u_0+\bar u_1)+\varepsilon^2(\dots)$.
\item \emph{Requirements.} The proof uses the error identity twice and the gap \ref{G1} through $\Jmic^{-1}$ on $\Pgra Y$ and on $\Pgra(\Jmac+\sM_\alpha)v$. In the periodic elliptic case all terms are tensor products of periodic cell fields with derivatives of $u_0$ of order at most four. Corollary~\ref{cor:secondell} verifies the hypotheses for $f\in H^2$ and gives $O(\varepsilon^2\nrm f_{H^2})$, consistent with the known second-order periodic results (see \cite{ZhikovPastukhova2016} and references therein).
\item \emph{Scope and novelty.} We are not aware of an equally general abstract statement, but we have not searched exhaustively, and we make no priority claim.
\end{itemize}
\end{remark}

\section{Good gauges}\label{sec:goodgauge}

\begin{remark}[Where the square root comes from]\label{rem:optimal}
Compare Theorem~\ref{thm:quant} with Theorem~\ref{thm:gauge}.
\begin{enumerate}
\item Theorem~\ref{thm:quant}(i) controls only regular data, through $\nrm f_{\Jz}$. The splitting in (ii)--(iii) turns this into an estimate on all of $\Hsrc$, and the price is the square root. If the bound in (i) held with $\nrm f$ in place of $\nrm f_{\Jz}$, there would be nothing to split, and the error would be $\kappa_\varepsilon$ itself.
\item Of the three quantities in $\kappa_\varepsilon(\eta)$, $a$ and $c$ involve only $\Phi^\eta_f$ and $\eta\Phi^\eta_f$. By Lemma~\ref{lem:flux}(ii) they are bounded by $\nrm{\Ns u}$ (divided by $|\eta|$ or $\lambda_0$ for $a$), and in the examples $\nrm{\Ns u}\lesssim\nrm f$. Only $b$, one macroscopic derivative of the flux corrector, can cost $\nrm{\Jz f}$.
\item The flux corrector is determined only modulo $\Pz=\ker\Jmic$ (Definition~\ref{def:gauge}), and $\Jmac$ does not vanish on $\Pz$: it maps $\Hflu$ into $\Hgra$ by \ref{H2}(iii). So $b$ depends on the gauge, while the defect $\eta\Phi^\eta_f$, and hence $c$, does not.
\end{enumerate}
\emph{Conclusion.} The square root is caused by the absence of an $L^2$ bound on $\Jmac\Phi$ in the chosen gauge. If a gauge satisfying \ref{GF} exists, Theorem~\ref{thm:gauge} applies and the square root disappears. In the elliptic example (Chapter~\ref{sec:ell}), the default gauge $\Phi_f\in\Hgra$ costs $\nabla^3u_0$ (Proposition~\ref{prop:verify}), and for that gauge the splitting argument yields $\varepsilon^{1/2}$. For the problem itself (as opposed to the method), the rate $\varepsilon$ is sharp under a gap by Theorem~\ref{thm:lower}. The skew gauge (Proposition~\ref{prop:skew}) costs only $\nabla^2u_0$ and gives $O(\varepsilon)$ for $L^2$ data, so in this particular elliptic problem the square root is entirely a gauge artifact. This is not automatic in general: \ref{H1}--\ref{H4} do not imply that a good gauge exists. If no admissible $\Phi$ satisfies $\nrm{\Jmac\Phi}\lesssim\nrm f$, the square root is the price of this method, though not necessarily of the multiscale problem.
\end{remark}

\begin{remark}[Why a good gauge needs structure beyond (H1)--(H4)]\label{rem:gaugestructure}
A good gauge is a correction $\psi_f\in\Pz$ for which $\Jmac(\Phi^\eta_f+\psi_f)$ loses no derivative. By \ref{H2}(iii), $\Jmac\psi_f\in\Hgra$ for $\psi_f\in\Hflu$, so such a correction can only cancel the $\Hgra$-component of $\Jmac\Phi^\eta_f$ that costs the extra derivative. Whether it can do so depends on how $\Hflu$ is generated. In the elliptic example, the fluctuating fluxes in $\Hflu$ are $y$-divergences of skew-symmetric 2-forms (the flux corrector $\sigma$). The bad term disappears because $\nabla_x\cdot(\sigma\nabla_xp_0)=\sigma_{ij}\partial_i\partial_jp_0=0$: a skew-symmetric matrix contracted with a symmetric Hessian. Since \ref{H1}--\ref{H4} do not guarantee existence of such a good gauge, we keep \ref{GF} as a hypothesis, to be verified in each example. In the Dirac model (Chapter~\ref{sec:top}), $\Kmic=0$ and the residual is a multiplication operator, so the default gauge already satisfies \ref{GF} (Proposition~\ref{prop:fluxD}).
\end{remark}

The next hypothesis makes the mechanism of Remark~\ref{rem:gaugestructure} abstract. It postulates the ``next level of the complex'' only to the extent that it is used.
\begin{enumerate}[label=(SK),leftmargin=3.2em]
\item\label{SK} \emph{Skew potential for the corrector.} There are a Hilbert space $\mathcal K$ and linear operators $\delta_{\mathrm{mac}},\delta_{\mathrm{mic}}$, defined on subspaces of $\mathcal K$ with values in $\Hs$, and a constant $c_{\mathrm{SK}}$, with the following property. For every $f\in\Hsrc$, with $u=\Az^{-1}f$ and $w:=\Kmic u\in\dom\Jmac$, there exist $S_f\in\dom\delta_{\mathrm{mac}}\cap\dom\delta_{\mathrm{mic}}$ and $\Phi^r_f\in\mathcal D$ such that:
\begin{enumerate}[label=(\roman*)]
\item (\emph{coboundary}) $\delta_{\mathrm{mic}}S_f=w$;
\item (\emph{complex identities}) $\delta_{\mathrm{mac}}S_f\in\mathcal D$, $\ \Jmac\delta_{\mathrm{mac}}S_f=0$, and $\ \Jmic\delta_{\mathrm{mac}}S_f=-\Jmac\delta_{\mathrm{mic}}S_f$;
\item (\emph{material part}) $\Jmic\Phi^r_f=\Ns u+\Jmac w$;
\item (\emph{bounds}) $\nrm{\delta_{\mathrm{mac}}S_f}+\nrm{\Phi^r_f}+\nrm{\Jmac\Phi^r_f}\le c_{\mathrm{SK}}\nrm f$.
\end{enumerate}
\end{enumerate}

\begin{proposition}[(SK) implies a good gauge]\label{prop:SK}
Under \ref{SK}, the element $\Phi^g_f:=\Phi^r_f+\delta_{\mathrm{mac}}S_f\in\mathcal D$ satisfies $\Jmic\Phi^g_f=\Ns u$ and $\Jmac\Phi^g_f=\Jmac\Phi^r_f$. Hence \ref{GF} holds with $\tilde a\le c_{\mathrm{SK}}$, $\tilde b\le c_{\mathrm{SK}}$ and $\tilde c=0$, and Theorem~\ref{thm:gauge} gives $\nrm{\Ee|_{\Hsrc}}\le\varepsilon c_{\mathrm{SK}}\big(C_*+C_A\big)$.
\end{proposition}

\begin{proof}
By (iii), (ii) and (i), $\Jmic\Phi^g_f=\Ns u+\Jmac w+\Jmic\delta_{\mathrm{mac}}S_f=\Ns u+\Jmac w-\Jmac\delta_{\mathrm{mic}}S_f=\Ns u$. By (ii), $\Jmac\Phi^g_f=\Jmac\Phi^r_f$. The bounds follow from (iv).
\end{proof}

\begin{remark}[Interpretation of (SK)]\label{rem:SK}
\begin{itemize}
\item In the elliptic example, $\sJ$ is the operator $d+\delta$ on $0$- and $1$-forms, and $\mathcal K$ is the space of $2$-forms (skew-symmetric matrix fields $S$). Then $\delta_{\mathrm{mic}}S=(0,\sum_i\partial_{y_i}S_{i\cdot})$ and $\delta_{\mathrm{mac}}S=(0,\sum_i\partial_{x_i}S_{i\cdot})$. The identities (ii) hold for \emph{every} skew $S$: they express $\delta_x\delta_x=0$ and $\delta_x\delta_y+\delta_y\delta_x=0$ (Remark~\ref{rem:SKell}). The real content of \ref{SK} is (i): the corrector must be a microscopic co-boundary of a $2$-form, which is the existence of a flux corrector $\sigma$.
\item The bad term of the default gauge is $\Jmic^{-1}(-\Jmac w)$, whose macroscopic derivative costs one more derivative of $w$. Under \ref{SK} it is replaced by $\delta_{\mathrm{mac}}S_f$, which solves the same equation by (ii), and whose macroscopic derivative vanishes.
\item If (i) holds only up to a defect $\delta_{\mathrm{mic}}S_f-w$, then $\Jmic\Phi^g_f-\Ns u=\Jmac(w-\delta_{\mathrm{mic}}S_f)$, and Lemma~\ref{lem:id} applies with this defect. This is the situation of the regularized corrector in Proposition~\ref{prop:skewrand}.
\item In the Dirac model $w=\Kmic u=0$, so one takes $S_f=0$ and $\Phi^r_f=\Phi_F$.
\end{itemize}
\end{remark}

\section{Scope of the axioms and abstract H-systems}\label{sec:hsystems}

\begin{remark}[What is, and is not, used]\label{rem:used}
The proofs of Theorems~\ref{thm:quant} and \ref{thm:gauge} use only Lemma~\ref{lem:comm}, Lemma~\ref{lem:pcoer}, Theorem~\ref{thm:corr} and the hypotheses stated. They do not use:
\begin{itemize}
\item the qualitative Theorem~\ref{thm:qual}; Theorem~\ref{thm:quant} in fact gives norm convergence on $\Hsrc$ whenever $\kappa_\varepsilon(\eta(\varepsilon))\to0$;
\item Theorem~\ref{thm:freq} or \ref{A3} (Remark~\ref{rem:role23});
\item a spectral gap, in Theorems~\ref{thm:quant} and \ref{thm:gauge}; the gap enters only in Corollaries~\ref{thm:main} and \ref{cor:gaugegap}, to take $\eta=0$;
\item a spectral density, in either result; such information enters only when $a$, $b$ and $c$, or $\tilde a$, $\tilde b$, $\tilde c$, are estimated (Sections~\ref{sec:stoch} and \ref{sec:random});
\item the H-flow structure.
\end{itemize}
The macroscopic cut-off $\sP_R$ is defined through the functional calculus of $\Jz$, so the argument remains point-free.
\end{remark}

\begin{remark}[Scope]\label{rem:scope}
The estimate is stated on $\Hsrc$, not on all of $\Pmac\Hs$. For sources outside a space on which \ref{G2} or \ref{GF} holds, such as $p$-sources $(0,g)$ in Chapter~\ref{sec:ell}, the method gives no rate.
\end{remark}

Given a projection $\Ps$ satisfying \ref{S} and constants $c>0$, $C\ge0$, $\tau\ge1$, $\Mb>0$, let $\mathscr M(\Ps;c,C,\tau,\Mb)$ be the set of all operators $\sM$ satisfying \ref{H4} with these constants and with some multipliers $\sT,\sT'$ of norm at most $\tau$. Theorem~\ref{thm:qual} and Theorem~\ref{thm:quant} motivate the following definition, which is essentially a point-free version of the notion of H-convergence by Luc Tartar and Fran\c{c}ois Murat \cite{MuratTartar1997}:

\begin{definition}[Abstract H-system and H-flow/Homogenization Flow]\label{def:homo-syst}
An abstract H-system is a tuple $(\Hs,\sJ,\{\Se\}_{\varepsilon>0})$ consisting of a fixed Hilbert space $\Hs$, a fixed differential structure $\sJ$, and a one-parameter family of unitary scaling operators $\{\Se\}_{\varepsilon>0}$, such that axioms \ref{H1}--\ref{H3} are satisfied.

An abstract H-system is called an H-flow or homogenization flow if there is a strongly continuous one-parameter unitary group $\{\mathsf T_t\}_{t\in\R}$ on $\Hs$ with $\mathsf T_t=\sS_{1/t}$ for $t>0$; in particular $\mathsf T_{t_1}\mathsf T_{t_2}=\mathsf T_{t_1+t_2}$.
\end{definition}

Heuristically, the parameter $t$ represents the frequency or magnification level. With this definition, Theorem~\ref{thm:qual} guarantees that a heterogeneous mixture governed by an abstract H-system can be approximated by an effective/homogenized system as the heterogeneity size $\varepsilon$ approaches zero. Theorem~\ref{thm:quant} measures the mixing rate of the homogenization flow for an arbitrary material profile $\sM\in\mathscr M(\Ps;c,C,\tau,\Mb)$. The group property is not used in any proof (Remark~\ref{rem:used}).

From Theorem~\ref{thm:quant}, it can be seen that the algebraic machinery reduces the convergence rate calculation for a multiscale problem to finding three quantities that no longer depend on the heterogeneity scale $\varepsilon$. In practice, while the periodic setting may enable an elegant spectral argument yielding bounded norms (see Chapter~\ref{sec:ell}), the stochastic setting requires concrete inputs from probability and geometry of the problems in order to estimate the spatial decorrelation (or statistical mixing) of the residue fields defined in Definition~\ref{def:corr-resi}, or of the coefficient fields that generate them (Sections~\ref{sec:stoch} and \ref{sec:random}).

\part{Applications}

\chapter{Dissipative Media: Divergence-Form Operators}\label{sec:ell}\chaptermark{Divergence-Form Operators}

To illustrate our approach, we consider the highly oscillating operator
\begin{equation}\label{eq:20}
\mathsf H_\varepsilon=-\nabla\cdot(a_\varepsilon\nabla)+I
\end{equation}
where $a_\varepsilon(x)$ is a bounded coercive $d\times d$ coefficient matrix. For $f\in L^2(\R^d;\R)$, we write the multiscale problem
\begin{equation}\label{eq:22}
\mathsf H_\varepsilon u_\varepsilon=f\qquad\text{for }u_\varepsilon\in H^1(\R^d;\R)
\end{equation}
as a first order system by letting $p_\varepsilon:=-a_\varepsilon\nabla u_\varepsilon$,
\begin{equation}\label{eq:36}
\begin{cases}u_\varepsilon+\nabla\cdot p_\varepsilon=f,\\ \nabla u_\varepsilon+a_\varepsilon^{-1}p_\varepsilon=0.\end{cases}
\end{equation}
Let
\begin{equation}\label{eq:38}
\sJ=\begin{pmatrix}0&\nabla\cdot\\ \nabla&0\end{pmatrix},\qquad F=\begin{pmatrix}f\\0\end{pmatrix}.
\end{equation}
Under the Fourier transform, derivatives become algebraic multipliers: $\nabla\to i\xi$ and $\nabla\cdot\to i\xi^T$. Thus, $\sJ$ acts as a $(d+1)\times(d+1)$ matrix multiplication operator with symbol
\begin{equation}\label{eq:42}
\widehat{\sJ}(\xi)=i\begin{pmatrix}0&\xi^T\\ \xi&0\end{pmatrix}.
\end{equation}
We want to obtain the convergence rates using Corollary~\ref{cor:gaugegap} (and, for comparison, Corollary~\ref{thm:main}) in the periodic setting. We will define the appropriate Hilbert space $\Hs$ and the scaling $\Se$, then verify that $(\Hs,\sJ,\{\Se\})$ is a homogenization flow as defined in Definition~\ref{def:homo-syst} and that $\sM$ satisfies \ref{H4}; by Theorem~\ref{thm:qual} this implies qualitative homogenization in the strong corrector form \eqref{eq:3}. Then we calculate
\begin{enumerate}
\item The projections $\Pmac$, $\Pflu$, and $\Pgra$ using the Fourier symbol of $\sJ$ in \eqref{eq:42}.
\item The corrector operator $\Kmic$ and $\Mz$ in Theorem~\ref{thm:corr}.
\item The convergence rate using Corollary~\ref{cor:gaugegap}.
\end{enumerate}

\section{Qualitative Results}\label{sec:3.1}

\paragraph{The Hilbert Space $\Hs$ and the Scaling Operator $\Se$.} For the periodic setting, the global Hilbert space is $\Hs=L^2(\R^d\times\T^d;\R^{d+1})$. Elements are vector fields $U(x,y)=\binom{u(x,y)}{p(x,y)}$, where $x\in\R^d$ is the macroscopic variable and $y\in\T^d=\R^d/\Z^d$ is the microscopic variable. All the spaces are complexified when the spectral theorem is used. We define the unitary scaling transformation $\Se:\Hs\to\Hs$ as
\begin{equation}\label{eq:79}
(\Se U)(x,y):=U\Big(x,\Big(y+\frac x\varepsilon\Big)\bmod1\Big),\qquad\text{for }U\in\Hs .
\end{equation}

\paragraph{The Projections.} The microscopic operator $\Jmic$ acts only on the $y$ variable. In the Fourier domain $k\in\Z^d$, its symbol is $\widehat{\Jmic}(k)=2\pi i\begin{pmatrix}0&k^T\\k&0\end{pmatrix}$.
\begin{itemize}
\item $\Hgra$ (The Gradient Space). This is the closure of $\ran(\Jmic)$. It consists of zero-mean scalars and pure gradient vector fields in $y$:
\begin{equation}\label{eq:81}
\Pgra U=\begin{pmatrix}u-\langle u\rangle_y\\ \Pi_{\mathrm{pot}}p\end{pmatrix},
\end{equation}
where $\Pi_{\mathrm{pot}}$ projects onto zero-mean, curl-free fields in $y$.
\item $\Pz$ (The Microscopic Kernel). The orthogonal complement of $\Hgra$. It projects onto $\ker(\Jmic)$:
\begin{equation}\label{eq:82}
\Pi_0U=\begin{pmatrix}\langle u\rangle_y\\ \langle p\rangle_y+\Pi_{\mathrm{sol}}p\end{pmatrix},
\end{equation}
where $\Pi_{\mathrm{sol}}$ projects onto zero-mean, divergence-free fields in $y$.
\item $\Hmac$ (The Macroscopic Space). Defined as the space of the $y$-averages:
\begin{equation}\label{eq:83}
\Pmac U=\begin{pmatrix}\langle u\rangle_y\\ \langle p\rangle_y\end{pmatrix}.
\end{equation}
\item $\Hflu$ (The Fluctuation Space). Defined as the range of $\Pflu=\Pi_0-\Pmac$. It isolates the divergence-free microscopic fluctuations:
\begin{equation}\label{eq:84}
\Pflu U=\begin{pmatrix}0\\ \Pi_{\mathrm{sol}}p\end{pmatrix}.
\end{equation}
\end{itemize}
Let $U\in\Hflu$. In this concrete example, we see that the macroscopic derivative of a microscopically divergence-free field, $\Jmac U$, is in $\Hgra$ (subspace of zero $y$-mean in the top component and microscopically pure gradient fields in the bottom component). This illustrates Remark~\ref{rem:zero-blk}.

\paragraph{Verification of the (H) hypotheses.} Take $\Jmac:=\sJ$ (acting only in $x$) and $\Jmic$ as above, both skew-adjoint with maximal domains. In particular $\sJ=\Jmac$ is skew-adjoint, as required in \ref{H3}(i).
\begin{itemize}
\item[\ref{H1}] holds by the above (the three projections are mutually orthogonal and sum to the identity).
\item[\ref{H2}] (i) In Fourier variables in $y$, $\ker\Jmic=\{(u,p):\nabla_yu=0,\ \nabla_y\cdot p=0\}=\Pz$. (ii) $\Pmac$ is the $y$-average, so it commutes with $x$-derivatives, preserves $\dom\Jmac$, and $\Jmac$ maps $y$-independent fields to $y$-independent fields. (iii) For $v=(0,s)\in\Hflu$, $\Jmac v=(\nabla_x\cdot s,0)$ has vanishing $p$-component, so $\Pflu\Jmac v=0$, and $\Pmac\Jmac v=(\nabla_x\cdot\langle s\rangle_y,0)=0$ because $\langle s\rangle_y=0$; hence $\Jmac v\in\Hgra$. (iv) $\Pflu$ acts only in $y$ and commutes with $x$-derivatives, so it preserves $\dom\Jmac$.
\item[\ref{H3}] $\mathcal D=\dom\Jmac\cap\dom\Jmic$ contains $C_c^\infty(\R^d;C^\infty(\T^d))$, which is a core for $\Je=\Se^{-1}\sJ\Se$ and for $\Jmic$; by the chain rule $\nabla_x(\Se U)=\Se(\nabla_x+\varepsilon^{-1}\nabla_y)U$, so $\Je=\Jmac+\varepsilon^{-1}\Jmic$ on $\mathcal D$. $\Se$ is unitary and fixes $y$-independent functions. Moreover $(\mathsf T_tU)(x,y):=U(x,(y+tx)\bmod1)$, $t\in\R$, is a strongly continuous unitary group with $\mathsf T_t=\sS_{1/t}$ for $t>0$, so $(\Hs,\sJ,\{\Se\})$ is a homogenization flow.
\item[\ref{S}, \ref{H4}] The zeroth-order term of $-\nabla\cdot a\nabla+1$ is the shift: $\Ps:=\Pi_u$, $\Pi_u(u,p):=(u,0)$, and $\alpha=1$. The material operator is $\sM=\begin{pmatrix}0&0\\0&a^{-1}(y)\end{pmatrix}$. Then \ref{S} holds, since $\Pi_u$ acts on the components only, and \ref{S0} holds, since $\Hflu=\{(0,s)\}$. Since $\ip{a\xi}{\xi}\ge\lambda|\xi|^2$ and $|a|\le\Lambda$, $\Re\ip{a^{-1}p}{p}\ge(\lambda/\Lambda^2)|p|^2$, so \ref{H4} holds with $\sT=\sT'=I$, $c=\lambda/\Lambda^2$, $C=0$ and $\Mb=\lambda^{-1}$. By Theorem~\ref{thm:corr}, $C_A=1/\min\{1,\lambda/\Lambda^2\}=:1/\mb$. With a general $\alpha>0$ in place of $1$, the same construction gives $(-\nabla\cdot a_\varepsilon\nabla+\alpha)^{-1}$, with $C_A=1/\min\{\alpha,\lambda/\Lambda^2\}$. (The domain condition on $\Kmic$ is checked below.)
\end{itemize}

\paragraph{The Abstract Corrector $\Kmic$.} From Definition~\ref{def:corr}, the abstract corrector problem for a macroscopic function $U_0=\binom{u_0(x)}{p_0(x)}\in\Hmac$ and a fluctuation function $w=\binom{0}{w_p(x,y)}\in\Hflu$ is
\begin{equation}\label{eq:80}
\Pflu\sM(U_0+w)=0 .
\end{equation}
On the one hand, $\sM(U_0+w)=\binom{0}{a^{-1}(y)(p_0+w_p)}$, and $\Pflu\Ps=0$, so the corrector problem \eqref{eq:29} is $\Pflu\sM(U_0+w)=0$. On the other hand, since $w\in\Hflu$, we already have $\Pi_{\mathrm{sol}}w_p=w_p$. Thus, together with \eqref{eq:84}, we simplify \eqref{eq:80} as
\begin{equation}\label{eq:85}
\Pi_{\mathrm{sol}}\big[a^{-1}(y)(p_0+w_p)\big]=0 .
\end{equation}
Because the microscopic variable $y$ lives on the flat torus $\T^d$, the Hilbert space of vector fields $L^2(\T^d;\R^d)$ admits the Helmholtz--Hodge decomposition
\begin{equation}\label{eq:86h}
L^2(\T^d;\R^d)=\mathrm{Grad}\oplus\mathrm{Sol}_0\oplus\mathrm{Harmonic},
\end{equation}
where:
\begin{itemize}
\item Grad: pure gradients $\nabla_y\chi(y)$. By periodicity and integration by parts, these automatically have zero mean.
\item $\mathrm{Sol}_0$: zero-mean, divergence-free fields. This is exactly the range of the projection $\Pi_{\mathrm{sol}}$.
\item Harmonic: fields that are both curl-free and divergence-free. On the torus $\T^d$, these are exactly the spatially constant (in $y$) vector fields $c\in\R^d$, as one sees by expanding in Fourier series.
\end{itemize}
Using \eqref{eq:85}, there exist $\chi$ and $c$ such that
\begin{equation}\label{eq:94}
a^{-1}(y)(p_0+w_p)=c(x)+\nabla_y\chi(x,y),
\end{equation}
hence
\begin{equation}\label{eq:90}
p_0+w_p=a(y)\big(c(x)+\nabla_y\chi(x,y)\big).
\end{equation}
Because $w\in\Hflu$ and $p_0(x)$ is independent of $y$, we have $\nabla_y\cdot w_p=0=\nabla_y\cdot p_0$. Taking the $y$-divergence of the above equation yields the cell problem
\begin{equation}\label{eq:89}
\nabla_y\cdot\Big(a(y)\big(c(x)+\nabla_y\chi(x,y)\big)\Big)=0 .
\end{equation}
This is the classical cell problem \cite{BensoussanLionsPapanicolaou2011}. By linearity, we write $\chi(x,y)=\chi_j(y)c_j(x)$ and define $\chi:=\chi_je_j$. To find the effective macroscopic behavior, we use the second property of $w\in\Hflu$: it has zero microscopic mean ($\langle w_p\rangle_y=0$). Taking the $y$-average of \eqref{eq:90} gives
\begin{equation}\label{eq:15}
p_0=\langle a(y)(I+\nabla_y\chi)\rangle_y\,c(x).
\end{equation}
Thus, we define the homogenized matrix
\begin{equation}\label{eq:92}
a_{\hom}:=\langle a(y)(I+\nabla_y\chi)\rangle_y ,
\end{equation}
which implies
\begin{equation}\label{eq:93}
c(x)=a_{\hom}^{-1}p_0(x).
\end{equation}

\paragraph{The Homogenized Coefficient $\Mz$.} By \eqref{eq:5.2}, $\Mz U_0=\Pmac\sM(U_0+w)$. Since $\Pmac$ is the $y$-average,
\begin{equation}\label{eq:95}
\Mz U_0=\begin{pmatrix}0\\ \langle a^{-1}(y)(p_0+w_p)\rangle_y\end{pmatrix}.
\end{equation}
Taking the $y$-average of \eqref{eq:94} and using $\langle\nabla_y\chi\rangle_y=0$ and \eqref{eq:93}, we obtain
$\langle a^{-1}(y)(p_0+w_p)\rangle_y=c(x)=a_{\hom}^{-1}p_0(x)$, so that
\begin{equation}\label{eq:97}
\Mz=\begin{pmatrix}0&0\\0&a_{\hom}^{-1}\end{pmatrix}.
\end{equation}
With $\alpha=1$, the effective operator is $\Az=\Jz+\Mz+\Pi_u=\begin{pmatrix}1&\nabla_x\cdot\\ \nabla_x&a_{\hom}^{-1}\end{pmatrix}$ on $\Hmac$.

\section{Quantitative Results}\label{sec:3.2}

We summarize our calculation so far:
\begin{itemize}
\item $\Hs=L^2(\R^d\times\T^d;\C^{d+1})$, $U=(u,p)$;
\item $\sJ=\Jmac=(\nabla_x\cdot,\nabla_x)$ and $\Jmic=(\nabla_y\cdot,\nabla_y)$;
\item $\sM=\operatorname{diag}(0,a(y)^{-1})$, $\Ps=\Pi_u$, $\alpha=1$, where $a$ is measurable, $\Z^d$-periodic, $\Re\ip{a\xi}{\xi}\ge\lambda|\xi|^2$ and $|a|\le\Lambda$;
\item $\sT=\sT'=I$, $c=\lambda/\Lambda^2$, $C=0$, $\Mb=\lambda^{-1}$, and $C_A=1/\mb$ with $\mb:=\min\{1,\lambda/\Lambda^2\}$.
\end{itemize}
The correctors are as follows.
\begin{itemize}
\item $\chi=(\chi_j)$, with $\int_{\T^d}\chi_j=0$, solves $\nabla_y\cdot a(e_j+\nabla_y\chi_j)=0$; $(\nabla\chi)_{ij}:=\partial_{y_i}\chi_j$.
\item $a_{\hom}=\int_{\T^d}a(I+\nabla\chi)$, so that $\Re\ip{a_{\hom}\xi}{\xi}\ge\lambda|\xi|^2$ and $|a_{\hom}|\le C(d,\lambda,\Lambda)$.
\item $B:=a(I+\nabla\chi)a_{\hom}^{-1}-I$. Its columns are divergence-free in $y$ and $\int_{\T^d}B=0$. We have $\nrm{\chi}_{H^1(\T^d)}+\nrm{B}_{L^2(\T^d)}\le C(d,\lambda,\Lambda)$.
\end{itemize}
We take $\Hsrc:=\{(f,0):f\in L^2(\R^d)\}$. Throughout we use the tensor identity $\nrm{g(y)h(x)}_{L^2(\R^d\times\T^d)}=\nrm g_{L^2(\T^d)}\nrm h_{L^2(\R^d)}$, together with pointwise Cauchy--Schwarz in the indices.

\begin{proposition}[Verification of the hypotheses]\label{prop:verify}
The hypotheses \ref{Q1}, \ref{G1}, and \ref{G2} hold with $\lambda_0=2\pi$ and $C_{\mathrm{reg}}=C(d,\lambda,\Lambda)$.
\end{proposition}

\begin{proof}
\ref{Q1}. On $\Hmac$, the Fourier symbol of $-i\Jz$ at $\xi$ is the real symmetric matrix $\begin{pmatrix}0&\xi^T\\ \xi&0\end{pmatrix}$. Its eigenvalues are $\pm|\xi|$, with eigenvectors $e_\pm=(1,\pm\xi/|\xi|)/\sqrt2$, and $0$, on $\{(0,q):q\perp\xi\}$. Since $(\hat f,0)=\frac{\hat f}{\sqrt2}(e_++e_-)$, we get $\sP_R(f,0)=(\mathcal F^{-1}1_{|\xi|\le R}\hat f,0)\in\Hsrc$. Also $\Jz(f,0)=(0,\nabla f)$, so $\nrm F+\nrm{\Jz F}\le2\nrm f_{H^1}$.

\ref{G1}. On the $y$-Fourier mode $k\in\Z^d$, $\Jmic$ acts as $2\pi i\begin{pmatrix}0&k^T\\k&0\end{pmatrix}$, with eigenvalues $\pm2\pi i|k|$ and $0$. $\Hgra$ is spanned, mode by mode with $k\neq0$, by the eigenvectors for $\pm2\pi i|k|$, and $|k|\ge1$. Hence $\nrm{\Jmic w}\ge2\pi\nrm w$ on $\Hgra$.

\emph{Effective solution.} For $F=(f,0)$ with $f\in H^1$, $U=\Az^{-1}F=(u_0,p_0)$, where $u_0-\nabla\cdot(a_{\hom}\nabla u_0)=f$ and $p_0=-a_{\hom}\nabla u_0$. Applying the Fourier transform to the effective equation (constant coefficients, coercivity $\lambda$), $\nrm{u_0}_{H^2}\le C\nrm f$ and $\nrm{u_0}_{H^3}\le C\nrm f_{H^1}$, so $p_0\in H^2$ with $\nrm{p_0}_{H^1}\le C\nrm f$ and $\nrm{p_0}_{H^2}\le C\nrm f_{H^1}$.

\emph{Corrector and residual.} By \eqref{eq:90} and \eqref{eq:93}, $\Kmic U=(0,B(y)p_0(x))$, so $\Jmac\Kmic U=(B_{jk}(y)\partial_jp_{0,k}(x),0)\in\Hs$. Since $\sM(U+\Kmic U)=(0,(I+\nabla\chi)c)$ and $\Mz U=(0,c)$, and $\Ps\Kmic U=0$, with $c:=a_{\hom}^{-1}p_0=-\nabla u_0$,
\[
\Ns U=\Big(-B_{jk}\partial_jp_{0,k},\ \big((\partial_{y_i}\chi_j)\partial_ju_0\big)_i\Big),
\]
and therefore $\nrm{\Ns U}\le\nrm B_{L^2(\T^d)}\nrm{\nabla p_0}+\nrm{\nabla\chi}_{L^2(\T^d)}\nrm{\nabla u_0}\le C\nrm f$.

\emph{Flux corrector.} Let $\Delta_y^{-1}$ denote the inverse Laplacian on zero-mean functions on $\T^d$, and $G_{jk}:=\nabla_y\Delta_y^{-1}B_{jk}\in H^1(\T^d;\C^d)$. Then
$\Phi_F=\big(\chi_j(y)\partial_ju_0(x),\ -G_{jk}(y)\partial_jp_{0,k}(x)\big)$. Indeed, both components lie in $\Hgra$ (a zero-mean scalar, and a $y$-gradient); $\nabla_y(\chi_j\partial_ju_0)=(\partial_{y_i}\chi_j)\partial_ju_0$; and $\nabla_y\cdot(-G_{jk}\partial_jp_{0,k})=-B_{jk}\partial_jp_{0,k}$. So $\Jmic\Phi_F=\Ns U$, and $\Phi_F$ satisfies Definition~\ref{def:flux}.

\emph{Applying $\Jmac$.}
$\Jmac\Phi_F=\big(-(\partial_{y_i}\Delta_y^{-1}B_{jk})(y)\,\partial_i\partial_jp_{0,k}(x),\ \chi_j(y)\nabla\partial_ju_0(x)\big)$.
This shows $\Phi_F\in\dom\Jmac$, and $\nrm{\Jmac\Phi_F}\le C(\nrm{\nabla^2p_0}+\nrm{\nabla^2u_0})\le C\nrm f_{H^1}\le C(\nrm F+\nrm{\Jz F})$. Together with the bound on $\Ns U$, this is \ref{G2}.
\end{proof}

\begin{proposition}[A good gauge: the skew-symmetric flux corrector]\label{prop:skew}
For each $k$ let $b^k:=Be_k$, the $k$-th column of $B$, which is divergence-free in $y$ with mean zero, and let
$\sigma^k_{ij}:=\partial_{y_i}\Delta_y^{-1}b^k_j-\partial_{y_j}\Delta_y^{-1}b^k_i$.
This is a skew-symmetric matrix field with $\sum_i\partial_{y_i}\sigma^k_{ij}=b^k_j$ and $\nrm{\sigma^k}_{L^2(\T^d)}\le\nrm{b^k}_{L^2(\T^d)}$. For $F=(f,0)$ with $f\in L^2(\R^d)$, let $U=\Az^{-1}F=(u_0,p_0)$ and define
\[
\Phi^g_F:=\big(\chi_j(y)\partial_ju_0(x),\ \psi\big),\qquad\psi_i:=-\sigma^k_{ij}(y)\,\partial_jp_{0,k}(x).
\]
Then $\Phi^g_F\in G^0_F$ and $\Jmac\Phi^g_F=(0,\chi_j(y)\nabla\partial_ju_0(x))$. Consequently \ref{GF} holds with $\tilde c=0$ and $\tilde a,\tilde b\le C(d,\lambda,\Lambda)$.
\end{proposition}

\begin{proof}
1. \emph{The field $\sigma^k$.} Since $b^k$ has mean zero, $\Delta_y^{-1}b^k$ is defined. Summing over $i$, $\partial_{y_i}\sigma^k_{ij}=\Delta_y\Delta_y^{-1}b^k_j-\partial_{y_j}\Delta_y^{-1}(\nabla_y\cdot b^k)=b^k_j$. On the mode $k'\neq0$, the Fourier coefficient of $\sigma^k$ has Frobenius norm $\sqrt2|\hat b^k(k')|/(2\pi|k'|)$, because $k'\cdot\hat b^k(k')=0$. Hence $\nrm{\sigma^k}_{L^2}\le\nrm{b^k}_{L^2}/(\sqrt2\pi)$.

2. \emph{Data.} For $f\in L^2$, $\nrm{u_0}_{H^2}\le C\nrm f$, so $\nabla p_0\in L^2$ and $\Kmic U=(0,Bp_0)\in\dom\Jmac$. In particular $\Ns U$ is defined and given as in the proof of Proposition~\ref{prop:verify}.

3. \emph{Antiderivative.} Since $\Jmic(u,p)=(\nabla_y\cdot p,\nabla_yu)$,
$\Jmic\Phi^g_F=\big(-(\partial_{y_i}\sigma^k_{ij})\partial_jp_{0,k},\ (\partial_{y_i}\chi_j)\partial_ju_0\big)=\Ns U=\Jmic\Phi_F$.
Here $\Phi^g_F\in\dom\Jmic$, because $\chi\in H^1(\T^d)$ and $\nabla_y\cdot\psi\in L^2$. Hence $\Phi^g_F-\Phi_F\in\ker\Jmic=\Pz$.

4. \emph{Macroscopic derivative.} Since $\Jmac(u,p)=(\nabla_x\cdot p,\nabla_xu)$, we need $\nabla_x\cdot\psi$. For a test function $\varphi\in C_c^\infty(\R^d)$ and a.e.\ $y$,
$\int\psi_i\partial_i\varphi\,dx=-\sigma^k_{ij}(y)\int\partial_jp_{0,k}\partial_i\varphi\,dx=\sigma^k_{ij}(y)\int p_{0,k}\partial_j\partial_i\varphi\,dx=0$,
because $\sigma^k$ is skew-symmetric in $(i,j)$ and $\partial_i\partial_j\varphi$ is symmetric. So $\nabla_x\cdot\psi=0$, and $\Jmac\Phi^g_F=(0,\chi_j\nabla\partial_ju_0)\in\Hs$, which shows $\Phi^g_F\in\dom\Jmac$. Therefore $\Phi^g_F\in G^0_F$.

5. \emph{Bounds.} By the tensor identity, $\nrm{\Phi^g_F}\le\nrm\chi_{L^2(\T^d)}\nrm{\nabla u_0}+\nrm\sigma_{L^2(\T^d)}\nrm{\nabla p_0}\le C\nrm f$ and $\nrm{\Jmac\Phi^g_F}\le\nrm\chi_{L^2(\T^d)}\nrm{\nabla^2u_0}\le C\nrm f$. For $\eta=0$ the defect vanishes, so $\tilde c=0$.
\end{proof}

\begin{theorem}[Norm Resolvent Convergence with Good Gauge]\label{thm:31}
There is $C=C(d,\lambda,\Lambda)$ such that for all $\varepsilon>0$ and all $f\in L^2(\R^d)$,
\begin{equation}\label{eq:24p}
\Big\|\big(\Ae^{-1}-(I+\Ke)\Az^{-1}\big)\binom f0\Big\|_{L^2(\R^d\times\T^d)}\le C\varepsilon\nrm f_{L^2(\R^d)} .
\end{equation}
\end{theorem}

\begin{proof}
Apply Corollary~\ref{cor:gaugegap}: the gap holds with $\lambda_0=2\pi$ (Proposition~\ref{prop:verify}), and \ref{GF} holds with $\tilde c=0$ (Proposition~\ref{prop:skew}).
\end{proof}

\begin{corollary}[The rate in Theorem~\ref{thm:31} is sharp]\label{cor:lowerell}
Let $f\in H^1(\R^d)$, $F=(f,0)$, $U=\Az^{-1}F=(u_0,p_0)$. Then
\[
\nrm{\big(\Ae^{-1}-(I+\Ke)\Az^{-1}\big)F}\ge\varepsilon\nrm{\Phi_F}-C\varepsilon^2\frac{\nrm f\,\nrm f_{H^1}}{\nrm{\Phi_F}},\qquad\varepsilon>0\quad(\Phi_F\neq0),
\]
where $\Phi_F=\big(\chi_j\partial_ju_0,-G_{jk}\partial_jp_{0,k}\big)$ is the default flux corrector of Proposition~\ref{prop:verify}, and $C=C(d,\lambda,\Lambda)$. Consequently, if $\Ns U\neq0$ for some $f\in H^1$, then $\liminf_{\varepsilon\to0}\varepsilon^{-1}\nrm{(\Ae^{-1}-(I+\Ke)\Az^{-1})|_{\Hsrc}}>0$.
\end{corollary}

\begin{proof}
Apply Theorem~\ref{thm:lower} with $\Phi=\Phi^g_F$ (Proposition~\ref{prop:skew}). The element $\Xi_F=-\Jmic^{-1}\Phi_F$ is explicit: $\Xi_F=\big(\Delta_y^{-1}B_{jk}\,\partial_jp_{0,k},\ -\nabla_y\Delta_y^{-1}\chi_j\,\partial_ju_0\big)$. Indeed, $\Jmic\Xi_F=(-\chi_j\partial_ju_0,\ G_{jk}\partial_jp_{0,k})=-\Phi_F$, and both components lie in $\Hgra$. Moreover $\Jmac\Xi_F$ involves $\nabla^2p_0$ and $\nabla^2u_0$, so $\nrm{\Xi_F}+\nrm{\Jmac\Xi_F}\le C\nrm{u_0}_{H^3}\le C\nrm f_{H^1}$, while $\nrm{\Phi^g_F}+\nrm{\Jmac\Phi^g_F}\le C\nrm f$ by Proposition~\ref{prop:skew}. Hence $C(f)\le C\nrm f\nrm f_{H^1}/\nrm{\Phi_F}$. The last claim follows from $\Phi_F=\Jmic^{-1}\Ns U$.
\end{proof}

\begin{remark}\label{rem:lowerdeg}
$\Ns U$ vanishes for all $f$ only in a degenerate situation. If $\nabla_y\chi=0$, then the columns of $a$ are divergence-free and $a_{\hom}=\langle a\rangle$. If moreover $B_{jk}\partial_jp_{0,k}=-\operatorname{tr}\big((a-\langle a\rangle)\nabla^2u_0\big)=0$ for all $u_0$, then the symmetric part of $a$ is constant. In that case $-\nabla\cdot a(\cdot/\varepsilon)\nabla=-\nabla\cdot a_{\hom}\nabla$, and there is nothing to homogenize. The lower bound concerns the lifted estimate. For the physical operators, the classical $O(\varepsilon)$ rate of Corollary~\ref{cor:unlift} is also known to be sharp \cite{BirmanSuslina2009,ZhikovPastukhova2016}; our argument does not reach that statement, because the return to the physical operators itself costs $O(\varepsilon)$.
\end{remark}

\begin{remark}[The two gauges compared]\label{rem:twogauges}
With the default gauge $\Phi_F\in\Hgra$, only Corollary~\ref{thm:main} applies (Proposition~\ref{prop:verify}). It gives $C\varepsilon^{1/2}\nrm f_{L^2}$, and $C\varepsilon\nrm f_{H^1}$ for $f\in H^1$ (Remark~\ref{rem:optimal}). The two gauges differ by an element of $\Hflu$: the difference of their flux components, $(G_{jk}-\sigma^k_{\cdot j})\partial_jp_{0,k}$, is divergence-free in $y$ with mean zero. This $\Hflu$-correction removes the third derivative of $u_0$ from $\Jmac\Phi$. The algebra is otherwise identical.
\end{remark}

\begin{remark}[Proposition~\ref{prop:skew} as an instance of (SK)]\label{rem:SKell}
Let $\mathcal K:=L^2(\R^d\times\T^d;\mathfrak{so}(d))$ (skew-symmetric matrix fields), $\delta_{\mathrm{mic}}S:=(0,\sum_i\partial_{y_i}S_{i\cdot})$ and $\delta_{\mathrm{mac}}S:=(0,\sum_i\partial_{x_i}S_{i\cdot})$. Put $S_F:=\sum_k\sigma^k(y)p_{0,k}(x)$ and $\Phi^r_F:=(\chi_j\partial_ju_0,0)$. Then:
\begin{itemize}
\item $\delta_{\mathrm{mic}}S_F=(0,Bp_0)=\Kmic U$, because $\sum_i\partial_{y_i}\sigma^k_{ij}=b^k_j$;
\item $\delta_{\mathrm{mac}}S_F=(0,\psi)$, with $\psi$ as in Proposition~\ref{prop:skew};
\item for every skew $S$, $\nabla_x\cdot\big(\sum_i\partial_{x_i}S_{i\cdot}\big)=0$ and $\nabla_y\cdot\big(\sum_i\partial_{x_i}S_{i\cdot}\big)=-\nabla_x\cdot\big(\sum_i\partial_{y_i}S_{i\cdot}\big)$, which are the identities in \ref{SK}(ii);
\item $\Ns U+\Jmac\Kmic U=(0,(\partial_{y_i}\chi_j)\partial_ju_0)_i=\Jmic\Phi^r_F$.
\end{itemize}
So \ref{SK} holds with $c_{\mathrm{SK}}=C(d,\lambda,\Lambda)$, and Proposition~\ref{prop:SK} reproduces Proposition~\ref{prop:skew}.
\end{remark}

\begin{corollary}[Second-order expansion for the periodic elliptic operator]\label{cor:secondell}
In the setting of Sections~\ref{sec:3.1}--\ref{sec:3.2}, let $f\in H^2(\R^d)$, $F=(f,0)$, $U=\Az^{-1}F=(u_0,p_0)$, and let $\Phi=\Phi^g_F$ be the skew gauge of Proposition~\ref{prop:skew}. Then the hypotheses of Theorem~\ref{thm:second} hold, and there is $C=C(d,\lambda,\Lambda)$ such that
\[
\Big\|\Ae^{-1}F-(I+\Ke)\Az^{-1}F-\varepsilon\Se(\Phi-v)\Big\|\le C\varepsilon^2\nrm f_{H^2},\qquad\varepsilon>0 .
\]
All objects in the second corrector are explicit; see \eqref{eq:secondv} below.
\end{corollary}

\begin{proof}
\emph{Tensor terms.} Call a \emph{cell field} an element of $L^2(\T^d;\C^{1+d})$ whose norm is bounded by a constant $C(d,\lambda,\Lambda)$. Every operator below acting in $y$ (multiplication by $a^{-1}$, the projections $\Pmac,\Pflu,\Pgra$, $\Mff^{-1}$, $\Jmic^{-1}$ on $\Hgra$) is $x$-independent. Hence it maps a tensor term $\theta(y)h(x)$ to $(\mathcal K\theta)(y)h(x)$, with $\nrm{\theta\otimes h}_{\Hs}=\nrm\theta_{L^2(\T^d)}\nrm h_{L^2(\R^d)}$. Moreover $\Jmac(\theta\otimes h)=(\theta_p\cdot\nabla h,\ \theta_u\nabla h)$ costs one $x$-derivative of $h$ and none of $\theta$. The relevant bounds on the $y$-side are $\nrm{\Jmic^{-1}}_{\Hgra\to\Hgra}\le(2\pi)^{-1}$ by \ref{G1}, and $\nrm{\Mff^{-1}}\le\Lambda^2/\lambda$. We say that a sum of tensor terms has \emph{order $r$} if its $x$-factors are derivatives of order at most $r$ of $u_0$, or of the auxiliary field $x_p$ introduced below.

\emph{Step 1: data.} Since $(1-\nabla\cdot a_{\hom}\nabla)u_0=f$ with constant $a_{\hom}$, we have $\nrm{u_0}_{H^4}\le C\nrm f_{H^2}$ and $p_0=-a_{\hom}\nabla u_0\in H^3$. By Proposition~\ref{prop:skew}, $\Kmic U=(0,Bp_0)\in\dom\Jmac$ and $\Phi=(\chi_j\partial_ju_0,\,-\sigma^k_{\cdot j}\partial_jp_{0,k})\in G^0_F$ has zero defect. It is of order $2$, and $\Jmac\Phi=(0,\chi_j\nabla\partial_ju_0)$ is of order $2$.

\emph{Step 2: $Y$.} With $\sM_\alpha=\sM+\Pi_u=\operatorname{diag}(1,a^{-1})$,
\[
Y=\big(\chi_j\partial_ju_0,\ -a^{-1}\sigma^k_{\cdot j}\,\partial_jp_{0,k}+\chi_j\nabla\partial_ju_0\big)
\]
is of order $2$. Since $\langle\chi_j\rangle=0$, $\Pmac Y=(0,-\tau_{jk}\partial_jp_{0,k})$, with constant vectors $\tau_{jk}:=\langle a^{-1}\sigma^k_{\cdot j}\rangle$. The remaining parts $\Pflu Y=(0,\Pi_{\mathrm{sol}}[\dots])$ and $\Pgra Y$ are of order $2$.

\emph{Step 3: $v=\sL^{-1}\Pi_0Y$.} Solving blockwise (Remark~\ref{rem:second}):
\begin{equation}\label{eq:secondv}
\begin{gathered}
v=x+\Kmic x+\tilde w,\qquad\tilde w:=\Mff^{-1}\Pflu Y,\\\Az x=(0,H_p),\qquad H_p:=-\tau_{jk}\partial_jp_{0,k}-\big\langle a^{-1}\tilde w_p\big\rangle .
\end{gathered}
\end{equation}
Here $\tilde w=(0,\tilde w_p)$ is of order $2$, and $H_p$ is a constant-coefficient combination of second derivatives of $u_0$. Hence $\nrm{H_p}_{H^2}\le C\nrm{u_0}_{H^4}$. Writing $x=(x_u,x_p)$, $\Az x=(0,H_p)$ reads
\[
x_u-\nabla\cdot a_{\hom}\nabla x_u=-\nabla\cdot(a_{\hom}H_p),\qquad x_p=a_{\hom}(H_p-\nabla x_u).
\]
Elliptic regularity gives $\nrm{x_u}_{H^3}\le C\nrm{H_p}_{H^2}$, so $\nrm{x_p}_{H^2}\le C\nrm f_{H^2}$. Moreover $\Pflu v=(0,Bx_p+\tilde w_p)$, and $\Jmac\Pflu v=(B_{jk}\partial_jx_{p,k}+\nabla_x\cdot\tilde w_p,\ 0)$ involves $\nabla x_p$ and third derivatives of $u_0$. So $\Pflu v\in\dom\Jmac$.

\emph{Step 4: $\Xi_1$.} $\Xi_1=\Jmic^{-1}\Pgra Y$ is of order $2$, with $\nrm{\Xi_1}\le(2\pi)^{-1}\nrm{\Pgra Y}$. Further, $\Jmac\Xi_1$ is of order $3$ and $\sM_\alpha\Xi_1$ of order $2$. Hence $\nrm{\Xi_1}+\nrm{(\Jmac+\sM_\alpha)\Xi_1}\le C\nrm{u_0}_{H^3}$.

\emph{Step 5: $\Xi_2$.} We split $(\Jmac+\sM_\alpha)v=\Jmac x+\Jmac\Pflu v+\sM_\alpha x+\sM_\alpha\Pflu v$ and project with $\Pgra$:
\begin{itemize}
\item $\Pgra\Jmac x=0$, because $x$ is constant in $y$;
\item $\Pgra\sM_\alpha x=(0,\Pi_{\mathrm{pot}}[a^{-1}-\langle a^{-1}\rangle]\,x_p)$;
\item $\Pgra\Jmac\Pflu v$ has $x$-factors $\nabla x_p$ and third derivatives of $u_0$;
\item $\Pgra\sM_\alpha\Pflu v=(0,\Pi_{\mathrm{pot}}[a^{-1}(Bx_p+\tilde w_p)])$.
\end{itemize}
So $\Xi_2$ is a sum of tensor terms with $x$-factors in $\{x_p,\nabla x_p,\partial^2u_0,\partial^3u_0\}$, and $\Jmac\Xi_2$ adds one derivative: $\{\nabla x_p,\nabla^2x_p,\partial^3u_0,\partial^4u_0\}$. All of these lie in $L^2$, with norms bounded by $C\nrm f_{H^2}$.

\emph{Conclusion.} Theorem~\ref{thm:second} applies with $C_A=1/\mb$. The right-hand side there is bounded by $C\varepsilon^2(\nrm{u_0}_{H^4}+\nrm{x_p}_{H^2})\le C\varepsilon^2\nrm f_{H^2}$.
\end{proof}

\begin{remark}\label{rem:secondell}
The regularity $f\in H^2$ is used only through $u_0\in H^4$ and $x_p\in H^2$; fourth derivatives of $u_0$ enter only through $\Jmac\Xi_2$. The macroscopic part of the second corrector, $-\varepsilon x_u$, solves the homogenized equation with the third-order source $-\nabla\cdot(a_{\hom}H_p)$. This is the classical first-order correction $\bar u_1$ of two-scale asymptotics, here obtained without a formal expansion. The result is lifted.
\end{remark}

\section{Back to the Physical Operators}\label{sec:3.3}

The calculation unlifting the lifted estimate \eqref{eq:24p} is similar to \cite{ZhikovPastukhova2016}. Let $\mathsf H_\varepsilon:=-\nabla\cdot a(\cdot/\varepsilon)\nabla+1$ and $\mathsf H_0:=-\nabla\cdot a_{\hom}\nabla+1$ on $L^2(\R^d)$. Identify $\T^d$ with $Y=[-\frac12,\frac12)^d$, and let $\widetilde{\mathsf S}_\varepsilon G(z):=\int_YG(z-\varepsilon y)\,dy$ be the Steklov average.

\begin{corollary}[Physical Operator Estimates]\label{cor:unlift}
There is $C=C(d,\lambda,\Lambda)$ such that for $0<\varepsilon\le1$ and $f\in L^2(\R^d)$,
\begin{equation}\label{eq:25p}
\begin{gathered}\nrm{(\mathsf H_\varepsilon^{-1}-\mathsf H_0^{-1})f}_{L^2}\le C\varepsilon\nrm f_{L^2},\\\big\|\nabla\mathsf H_\varepsilon^{-1}f-\big(I+\nabla\chi(\cdot/\varepsilon)\big)\widetilde{\mathsf S}_\varepsilon\nabla\mathsf H_0^{-1}f\big\|_{L^2}\le C\varepsilon\nrm f_{L^2}.\end{gathered}
\end{equation}
\end{corollary}

\begin{proof}
Let $(v_\varepsilon,q_\varepsilon):=\Ae^{-1}(f,0)$, $u_\varepsilon:=\mathsf H_\varepsilon^{-1}f$ and $u_0:=\mathsf H_0^{-1}f$. For each $y$: $q_\varepsilon=-a(y+x/\varepsilon)\nabla_xv_\varepsilon$, and $v_\varepsilon(\cdot,y)$ solves $v-\nabla_x\cdot(a(y+x/\varepsilon)\nabla_xv)=f$. Also $(I+\Ke)\Az^{-1}(f,0)=\big(u_0,-a(y+x/\varepsilon)(I+\nabla\chi(y+x/\varepsilon))\nabla u_0\big)$.

1. \emph{Lifted bounds.} Theorem~\ref{thm:31} and $|a^{-1}|\le\lambda^{-1}$ give
\begin{equation}\label{eq:26p}
\nrm{v_\varepsilon-u_0}_{L^2(\R^d\times Y)}+\nrm{\nabla_xv_\varepsilon(x,y)-(I+\nabla\chi(y+x/\varepsilon))\nabla u_0(x)}_{L^2(\R^d\times Y)}\le C\varepsilon\nrm f .
\end{equation}
2. \emph{Shift.} For fixed $y\in Y$, the function $u_\varepsilon(\cdot+\varepsilon y)$ solves the equation of $v_\varepsilon(\cdot,y)$ with source $f(\cdot+\varepsilon y)$. By coercivity and $\nrm{f(\cdot+h)-f}_{H^{-1}}\le|h|\nrm f_{L^2}$,
\begin{equation}\label{eq:27p}
\nrm{u_\varepsilon(\cdot+\varepsilon y)-v_\varepsilon(\cdot,y)}_{H^1}\le\max\{1,\lambda^{-1}\}\varepsilon|y|\nrm f_{L^2}.
\end{equation}
3. \emph{Change of variables and averaging.} Set $z=x+\varepsilon y$, so that $y+x/\varepsilon=z/\varepsilon$. Combining \eqref{eq:26p} and \eqref{eq:27p},
$\nrm{u_\varepsilon(z)-u_0(z-\varepsilon y)}_{L^2(dz\,dy)}+\nrm{\nabla u_\varepsilon(z)-(I+\nabla\chi(z/\varepsilon))\nabla u_0(z-\varepsilon y)}_{L^2(dz\,dy)}\le C\varepsilon\nrm f$.
Average in $y$ by Jensen's inequality; the factors $u_\varepsilon(z)$, $\nabla u_\varepsilon(z)$ and $\nabla\chi(z/\varepsilon)$ do not depend on $y$. This implies
$\nrm{u_\varepsilon-\widetilde{\mathsf S}_\varepsilon u_0}+\nrm{\nabla u_\varepsilon-(I+\nabla\chi(\cdot/\varepsilon))\widetilde{\mathsf S}_\varepsilon\nabla u_0}\le C\varepsilon\nrm f$.
Finally, $\nrm{\widetilde{\mathsf S}_\varepsilon u_0-u_0}\le\frac{\sqrt d}{2}\varepsilon\nrm{\nabla u_0}\le C\varepsilon\nrm f$.
\end{proof}

\begin{remark}\label{rem:36}
\begin{itemize}
\item The Steklov average arises from the $y$-average, so no bound on $\nabla\chi$ in $L^\infty$ is needed.
\item The $L^2$ estimate in \eqref{eq:25p} is a genuine operator-norm (norm-resolvent) estimate with $L^2$ data.
\item The rate $\varepsilon$ is the known optimal rate in the periodic setting \cite{BensoussanLionsPapanicolaou2011,BirmanSuslina2009,ZhikovPastukhova2016}. Here it is obtained by a purely algebraic argument once the gauge is fixed. The default gauge gives $\varepsilon^{1/2}$ (Remark~\ref{rem:twogauges}).
\end{itemize}
\end{remark}

\begin{remark}[Other spectral parameters]\label{rem:zell}
Let $a$ be symmetric, so that $\mathsf H_\varepsilon$ and $\mathsf H_0$ are self-adjoint with $\sigma(\mathsf H_\bullet)\subset[1,\infty)$. Put $\mathsf T_\bullet(z):=I+z(\mathsf H_\bullet-z)^{-1}$. Then $(\mathsf H_\varepsilon-z)^{-1}-(\mathsf H_0-z)^{-1}=\mathsf T_\varepsilon(z)\big(\mathsf H_\varepsilon^{-1}-\mathsf H_0^{-1}\big)\mathsf T_0(z)$, and $\nrm{\mathsf T_\bullet(z)}\le1+|z|/\operatorname{dist}(z,[1,\infty))$. Hence, for all $z\in\C\setminus[1,\infty)$,
\[
\nrm{(\mathsf H_\varepsilon-z)^{-1}-(\mathsf H_0-z)^{-1}}\le\Big(1+\frac{|z|}{\operatorname{dist}(z,[1,\infty))}\Big)^2C\varepsilon .
\]
The same identity, applied with a deterministic $\mathsf T_0$, gives the annealed analogue of Corollary~\ref{cor:ellrand} at every such $z$. For the Dirac model see Corollary~\ref{cor:spec}(i).
\end{remark}

\section[Stationary Random Coefficients]{Stationary Random Coefficients: a Regularized Skew Flux Corrector}\label{sec:stoch}

\paragraph{The probability space.} Let $(\Omega,\mathcal F,\mathbb P)$ be a probability space with a measure-preserving, ergodic, stochastically continuous group action $\{\tau_x\}_{x\in\R^d}$. Let $D_1,\dots,D_d$ be the generators of $T_xg:=g\circ\tau_x$ on $L^2(\Omega)$. They are commuting skew-adjoint operators, so $-iD=(-iD_1,\dots,-iD_d)$ has a joint spectral resolution, and we write $\Delta:=\sum_kD_k^2$. For $g\in L^2(\Omega;\C^N)$ (or matrix-valued, with the Frobenius norm) we write
\[
\mu_g(B):=\nrm{1_B(-iD)g}^2_{L^2(\Omega)},\qquad B\subset\R^d\ \text{Borel}.
\]
By ergodicity, $\mu_g(\{0\})=|\E g|^2$. For a stationary field we write $g(x,\omega):=\tilde g(\tau_x\omega)$, and we use the same letter for $\tilde g$ when no confusion arises. We use the Weyl decomposition $L^2(\Omega;\C^d)=L^2_{\mathrm{pot}}\oplus L^2_{\mathrm{sol}}\oplus\C^d$ \cite{JikovKozlovOleinik1994}, where $L^2_{\mathrm{pot}}$ is the closure of $\{Dh\}$ (equivalently, the mean-zero curl-free fields) and $L^2_{\mathrm{sol}}$ consists of the mean-zero divergence-free fields; $\Pi_{\mathrm{pot}}$, $\Pi_{\mathrm{sol}}$ denote the corresponding projections. In general (for instance under mixing assumptions) the measures $\mu_g$ charge every neighbourhood of $0$, so \ref{G1} fails; the argument below does not use a gap.

\paragraph{Spectral functionals.} For a finite family $\mathcal G$ of stationary mean-zero fields and $\theta>0$ put
\begin{equation}\label{eq:specfun}
\fa_{\mathcal G}(\theta):=\max_{g\in\mathcal G}\Big(\int\frac{|\xi|^2\,d\mu_g(\xi)}{(\theta+|\xi|^2)^2}\Big)^{1/2},\qquad
\fd_{\mathcal G}(\theta):=\max_{g\in\mathcal G}\Big(\int\frac{\theta^2\,d\mu_g(\xi)}{(\theta+|\xi|^2)^2}\Big)^{1/2}.
\end{equation}
By the spectral theorem, $\nrm{D_k(\theta-\Delta)^{-1}g}\le\fa_{\mathcal G}(\theta)$ for each $k$ and $\nrm{\theta(\theta-\Delta)^{-1}g}=\fd_{\mathcal G}(\theta)$ for $g\in\mathcal G$ (up to the maximum). Moreover, $\fa_{\mathcal G}(\theta)\le(4\theta)^{-1/2}\max_g\nrm g$, and $\fd_{\mathcal G}(\theta)\to0$ as $\theta\to0$ by dominated convergence, since $\mu_g(\{0\})=0$ for mean-zero $g$.

\paragraph{Density-of-states hypothesis.} We say that $\mathcal G$ satisfies \emph{(DOS)$_\gamma$} if there are $\gamma\in(0,4)$, $\lambda_*\in(0,1]$ and $C_D$ such that
\begin{equation}\label{eq:DOS}
\mu_g(\{|\xi|<r\})\le C_Dr^\gamma\qquad(0<r\le\lambda_*,\ g\in\mathcal G).
\end{equation}
If \eqref{eq:DOS} holds for some $\gamma$, it holds for every smaller one, so $\gamma<4$ is no restriction.

\begin{lemma}[Spectral integrals]\label{lem:specint}
Under \eqref{eq:DOS}, for $0<\theta\le\lambda_*^2$,
\[
\fd_{\mathcal G}(\theta)\le C\theta^{\gamma/4},\qquad
\fa_{\mathcal G}(\theta)\le C\begin{cases}\theta^{(\gamma-2)/4},&\gamma<2,\\ (1+\log(1/\theta))^{1/2},&\gamma=2,\\ 1,&\gamma>2,\end{cases}
\]
where $C$ depends on $C_D$, $\gamma$, $\lambda_*$ and $\max_{g\in\mathcal G}\nrm g$. Consequently $\kappa_\varepsilon:=\inf_{\theta>0}\big(\varepsilon\fa_{\mathcal G}(\theta)+\fd_{\mathcal G}(\theta)\big)$ satisfies, for $0<\varepsilon\le\varepsilon_0$,
\begin{equation}\label{eq:rates}
\kappa_\varepsilon\le C\begin{cases}\varepsilon^{\gamma/2},&0<\gamma<2,\\ \varepsilon(\log(1/\varepsilon))^{1/2},&\gamma=2,\\ \varepsilon,&2<\gamma<4.\end{cases}
\end{equation}
Without \eqref{eq:DOS}, still $\kappa_\varepsilon\to0$ as $\varepsilon\to0$.
\end{lemma}

\begin{proof}
Write $\nu=\mu_g$ and $\nu(r):=\nu(\{|\xi|<r\})\le C_Dr^\gamma$. For a decreasing function $h$ on $[\sqrt\theta,\lambda_*]$, the layer-cake formula gives
$\int_{\sqrt\theta\le|\xi|<\lambda_*}h(|\xi|)\,d\nu\le h(\lambda_*)\nu(\lambda_*)+\int_{\sqrt\theta}^{\lambda_*}|h'(r)|\nu(r)\,dr$.
For $\fd$: on $\{|\xi|<\sqrt\theta\}$ the integrand is $\le1$, giving $\le C_D\theta^{\gamma/2}$; on $\{\sqrt\theta\le|\xi|<\lambda_*\}$ it is $\le\theta^2|\xi|^{-4}$, giving $\le C\theta^2+\theta^2\int_{\sqrt\theta}^{\lambda_*}4C_Dr^{\gamma-5}dr\le C\theta^{\gamma/2}$ since $\gamma<4$ and $\theta\le1$; the region $|\xi|\ge\lambda_*$ contributes $\le\theta^2\lambda_*^{-4}\nrm g^2$. For $\fa$: on $\{|\xi|<\sqrt\theta\}$ the integrand is $\le\theta^{-1}$, giving $\le C_D\theta^{\gamma/2-1}$; on $\{\sqrt\theta\le|\xi|<\lambda_*\}$ it is $\le|\xi|^{-2}$, giving $\le C+C\int_{\sqrt\theta}^{\lambda_*}r^{\gamma-3}dr$; the region $|\xi|\ge\lambda_*$ contributes $\le\lambda_*^{-2}\nrm g^2$. This gives the bounds. For the rates take $\theta=\varepsilon^2$ if $\gamma\le2$ and $\theta=\varepsilon^{4/\gamma}$ if $2<\gamma<4$. Without \eqref{eq:DOS}, take $\theta=\varepsilon$: then $\varepsilon\fa(\varepsilon)\le\tfrac12\varepsilon^{1/2}\max\nrm g$ and $\fd(\varepsilon)\to0$.
\end{proof}

\paragraph{Setting and hypothesis.}
\begin{enumerate}[label=(E\arabic*),leftmargin=3.2em]
\item\label{E1} $a(x,\omega)=\tilde a(\tau_x\omega)$ with $\tilde a\in L^\infty(\Omega;\R^{d\times d})$, $\ip{\tilde a\xi}{\xi}\ge\lambda|\xi|^2$ and $|\tilde a|\le\Lambda$.
\item\label{E2} The family $\mathcal G_{\mathrm{ell}}$ defined in \eqref{eq:Gell} below satisfies (DOS)$_\gamma$ for some $\gamma\in(0,4)$.
\end{enumerate}
We lift to $\Hs=L^2(\R^d\times\Omega;\C^{d+1})$ with $(\Se U)(x,\omega):=U(x,\tau_{x/\varepsilon}\omega)$, which is unitary since $\tau$ is measure-preserving, $\Jmac=(\nabla_x\cdot,\nabla_x)$, $\Jmic=(D\cdot,D)$, and $\sM=\operatorname{diag}(0,\tilde a^{-1})$, with the shift $\Pi_u$ and $\alpha=1$. The projections are those of Section~\ref{sec:3.1} with $\langle\cdot\rangle_y$ replaced by $\E$ and $\Pi_{\mathrm{pot}},\Pi_{\mathrm{sol}}$ by the Weyl projections (by ergodicity, $Du=0$ forces $u$ to be constant in $\omega$). Hypotheses \ref{H1}--\ref{H4} are verified as in Section~\ref{sec:3.1}, with the Fourier series in $y$ replaced by the joint spectral resolution of $-iD$; \ref{H2}(iv) holds because $\Pflu$ acts only in $\omega$; for \ref{H3} one uses that $\mathcal D$ contains $C_c^\infty(\R^d)\otimes\dom(\Delta^\infty)$, a core obtained by mollifying in $x$ and averaging along the group in $\omega$. As in Section~\ref{sec:3.1}, $\sM=\operatorname{diag}(0,\tilde a^{-1})$, $\Ps=\Pi_u$, $\alpha=1$, and \ref{S}, \ref{S0}, \ref{H4} hold with $\sT=\sT'=I$, $c=\lambda/\Lambda^2$, $C=0$, so $C_A=1/\mb$ with $\mb=\min\{1,\lambda/\Lambda^2\}$.

\paragraph{Correctors.} For each $j$, let $g^j\in L^2_{\mathrm{pot}}$ be the unique solution of $\tilde a(e_j+g^j)\in L^2_{\mathrm{sol}}\oplus\C^d$ (i.e.\ $D\cdot\tilde a(e_j+g^j)=0$ weakly), obtained by Lax--Milgram on $L^2_{\mathrm{pot}}$ \cite{JikovKozlovOleinik1994}; $g^j$ is the stationary gradient of the (in general non-stationary) corrector $\chi_j$. Let $G_{ij}:=(g^j)_i$, $a_{\hom}:=\E[\tilde a(I+G)]$, and $B:=\tilde a(I+G)a_{\hom}^{-1}-I$, whose columns $b^k$ lie in $L^2_{\mathrm{sol}}$. Exactly as in Section~\ref{sec:3.1}, $\Kmic U=(0,Bp_0)$, $\Mz=\operatorname{diag}(0,a_{\hom}^{-1})$ and, for $U=\Az^{-1}(f,0)=(u_0,p_0)$,
\begin{equation}\label{eq:NsRand}
\Ns U=\Big(-B_{jk}\,\partial_jp_{0,k},\ \textstyle\sum_jg^j\,\partial_ju_0\Big).
\end{equation}
Put
\begin{equation}\label{eq:Gell}
\mathcal G_{\mathrm{ell}}:=\{B_{jk}:1\le j,k\le d\}\cup\{G_{ij}:1\le i,j\le d\}.
\end{equation}

\begin{proposition}[Regularized skew flux corrector]\label{prop:skewrand}
Assume \ref{E1}. Let $f\in L^2(\R^d)$, $U=\Az^{-1}(f,0)=(u_0,p_0)$, and $\theta>0$. Define
\[
\phi^j_\theta:=-(\theta-\Delta)^{-1}D\cdot g^j,\qquad\sigma^k_{\theta,ij}:=-(\theta-\Delta)^{-1}\big(D_iB_{jk}-D_jB_{ik}\big),
\]
and $\Phi_\theta:=(v_\theta,\psi_\theta)$ with $v_\theta:=\sum_j\phi^j_\theta\,\partial_ju_0$ and $(\psi_\theta)_i:=-\sum_{j,k}\sigma^k_{\theta,ij}\,\partial_jp_{0,k}$. Then $\Phi_\theta\in\mathcal D$,
\[
\begin{gathered}\Jmac\Phi_\theta=\big(0,\ \textstyle\sum_j\phi^j_\theta\nabla\partial_ju_0\big),\\
\varrho_\theta:=\Jmic\Phi_\theta-\Ns U=\Big(\textstyle\sum_{j,k}\theta(\theta-\Delta)^{-1}B_{jk}\,\partial_jp_{0,k},\ -\textstyle\sum_je^j_\theta\,\partial_ju_0\Big),\end{gathered}
\]
with $e^j_\theta:=\theta(\theta-\Delta)^{-1}g^j$, and, with $C=C(d,\lambda,\Lambda)$,
\[
\nrm{\Phi_\theta}+\nrm{\Jmac\Phi_\theta}\le C\fa_{\mathcal G_{\mathrm{ell}}}(\theta)\nrm f,\qquad\nrm{\varrho_\theta}\le C\fd_{\mathcal G_{\mathrm{ell}}}(\theta)\nrm f .
\]
\end{proposition}

\begin{proof}
1. \emph{The $q$-component.} Since $g^j\in L^2_{\mathrm{pot}}$, its spectral representation is parallel to $\xi$, so $D(D\cdot g^j)=\Delta g^j$. Hence $D\phi^j_\theta=-(\theta-\Delta)^{-1}\Delta g^j=g^j-e^j_\theta$, and $Dv_\theta=\sum_j(g^j-e^j_\theta)\partial_ju_0$, which is the $q$-component of $\Ns U$ minus $\sum_je^j_\theta\partial_ju_0$.

2. \emph{The $u$-component.} Since $b^k$ is divergence-free, $\sum_iD_iB_{ik}=0$, so $\sum_iD_i\sigma^k_{\theta,ij}=-(\theta-\Delta)^{-1}\Delta B_{jk}=B_{jk}-\theta(\theta-\Delta)^{-1}B_{jk}$. Therefore $D\cdot\psi_\theta=-\sum_{j,k}B_{jk}\partial_jp_{0,k}+\sum_{j,k}\theta(\theta-\Delta)^{-1}B_{jk}\partial_jp_{0,k}$, and the defect formula follows from $\Jmic(v,\psi)=(D\cdot\psi,Dv)$.

3. \emph{Macroscopic derivative.} $\sigma^k_\theta$ is skew-symmetric in $(i,j)$, so $\nabla_x\cdot\psi_\theta=-\sum\sigma^k_{\theta,ij}\partial_i\partial_jp_{0,k}=0$ exactly as in Proposition~\ref{prop:skew}; and $\nabla_xv_\theta=\sum_j\phi^j_\theta\nabla\partial_ju_0$.

4. \emph{Bounds.} Every term has the form $T g(\omega)h(x)$ with $T\in\{D_k(\theta-\Delta)^{-1},\theta(\theta-\Delta)^{-1}\}$ acting in $\omega$, $g\in\mathcal G_{\mathrm{ell}}$, and $h$ a derivative of $u_0$ of order one or two or of $p_0$ of order one. By the tensor identity, $\nrm{u_0}_{H^2}+\nrm{p_0}_{H^1}\le C\nrm f$, $\nrm{\phi^j_\theta}\le\sum_i\nrm{D_i(\theta-\Delta)^{-1}G_{ij}}\le d\,\fa$, $\nrm{\sigma^k_{\theta,ij}}\le2\fa$, $\nrm{\theta(\theta-\Delta)^{-1}B_{jk}}\le\fd$ and $\nrm{e^j_\theta}\le\sqrt d\,\fd$, the bounds follow. The same bounds show $v_\theta,\psi_\theta\in\dom\Jmac\cap\dom\Jmic$.
\end{proof}

\begin{theorem}[Random elliptic coefficients, lifted]\label{thm:ellrand}
Assume \ref{E1}. There is $C=C(d,\lambda,\Lambda)$ such that for all $\varepsilon>0$, $\theta>0$ and $f\in L^2(\R^d)$,
\begin{equation}\label{eq:ellrand}
\Big\|\big(\Ae^{-1}-(I+\Ke)\Az^{-1}\big)\binom f0\Big\|_{L^2(\R^d\times\Omega)}\le C\big(\varepsilon\fa_{\mathcal G_{\mathrm{ell}}}(\theta)+\fd_{\mathcal G_{\mathrm{ell}}}(\theta)\big)\nrm f .
\end{equation}
In particular the left-hand side tends to $0$ as $\varepsilon\to0$, uniformly over $\nrm f\le1$. If moreover \ref{E2} holds, then the operator norm on $\Hsrc=\{(f,0)\}$ is $O(\kappa_\varepsilon)$ with $\kappa_\varepsilon$ as in \eqref{eq:rates}.
\end{theorem}

\begin{proof}
Apply Lemma~\ref{lem:id} (with $C_A=1/\mb$) to $\Phi=\Phi_\theta$ of Proposition~\ref{prop:skewrand}, and use Lemma~\ref{lem:specint}.
\end{proof}

\begin{corollary}[Physical operators, annealed]\label{cor:ellrand}
Under \ref{E1}, let $\mathsf H^\omega_\varepsilon:=-\nabla\cdot a(\cdot/\varepsilon,\omega)\nabla+1$, $u^\omega_\varepsilon:=(\mathsf H^\omega_\varepsilon)^{-1}f$, $\mathsf H_0:=-\nabla\cdot a_{\hom}\nabla+1$ and $u_0:=\mathsf H_0^{-1}f$. With $\kappa_\varepsilon:=\inf_\theta(\varepsilon\fa(\theta)+\fd(\theta))$,
\[
\big(\E\nrm{u^\omega_\varepsilon-u_0}^2\big)^{1/2}+\big(\E\nrm{\nabla u^\omega_\varepsilon-(I+G(\cdot/\varepsilon,\omega))\nabla u_0}^2\big)^{1/2}\le C\kappa_\varepsilon\nrm f,
\]
where $G(x,\omega)=\tilde G(\tau_x\omega)$. In particular $\sup_{\nrm f\le1}\big(\E\nrm{((\mathsf H^\omega_\varepsilon)^{-1}-\mathsf H_0^{-1})f}^2\big)^{1/2}\le C\kappa_\varepsilon$.
\end{corollary}

\begin{proof}
$\Ae$ acts fibrewise in $\omega$, and its fibre at $\omega$ is the first-order system \eqref{eq:36} with coefficient $a(x/\varepsilon,\omega)$; no translation step is needed. Hence the fibre of $\Ae^{-1}(f,0)$ is $(u^\omega_\varepsilon,-a(\cdot/\varepsilon,\omega)\nabla u^\omega_\varepsilon)$, and that of $(I+\Ke)\Az^{-1}(f,0)$ is $(u_0,-a(\cdot/\varepsilon,\omega)(I+G(\cdot/\varepsilon,\omega))\nabla u_0)$, since $p_0+Bp_0=\tilde a(I+G)a_{\hom}^{-1}p_0$ and $p_0=-a_{\hom}\nabla u_0$. The claim follows from Theorem~\ref{thm:ellrand} and $|\tilde a^{-1}|\le\lambda^{-1}$.
\end{proof}

\begin{corollary}[Annealed norm-resolvent convergence without mixing]\label{cor:annqual}
Assume only \ref{E1} (stationarity, ergodicity, ellipticity and boundedness). Then
\[
\lim_{\varepsilon\to0}\ \sup_{\nrm f_{L^2(\R^d)}\le1}\Big(\E\nrm{u^\omega_\varepsilon-u_0}^2+\E\nrm{\nabla u^\omega_\varepsilon-(I+G(\cdot/\varepsilon,\omega))\nabla u_0}^2\Big)=0,
\]
with $u^\omega_\varepsilon$, $u_0$ and $G$ as in Corollary~\ref{cor:ellrand}.
\end{corollary}

\begin{proof}
By Corollary~\ref{cor:ellrand} it suffices to show $\kappa_\varepsilon=\inf_{\theta>0}(\varepsilon\fa(\theta)+\fd(\theta))\to0$. Take $\theta=\varepsilon$. Since $|\xi|^2(\varepsilon+|\xi|^2)^{-2}\le(4\varepsilon)^{-1}$, $\varepsilon\fa(\varepsilon)\le\frac12\varepsilon^{1/2}\max_{g\in\mathcal G_{\mathrm{ell}}}\nrm g$. Moreover $\fd(\varepsilon)^2=\max_g\int\varepsilon^2(\varepsilon+|\xi|^2)^{-2}\,d\mu_g\to\max_g\mu_g(\{0\})=0$ by dominated convergence, since every $g\in\mathcal G_{\mathrm{ell}}$ has mean zero and $\mu_g(\{0\})=|\E g|^2$ by ergodicity.
\end{proof}

\begin{remark}
Qualitative stochastic homogenization \cite{JikovKozlovOleinik1994,PapanicolaouVaradhan1981} gives $u^\omega_\varepsilon\to u_0$ almost surely for each fixed $f$. Corollary~\ref{cor:annqual} adds uniformity over the unit ball of $L^2$ data, in the annealed sense, together with the gradient corrector; no rate is claimed. By Proposition~\ref{prop:noquenched}, uniformity cannot hold almost surely in general.
\end{remark}

\begin{remark}[On hypothesis (E2)]\label{rem:E2}
\begin{itemize}
\item \emph{From variance estimates.} Let $\eta$ be a Gaussian with $\int\eta=1$, $\eta_R:=R^{-d}\eta(\cdot/R)$. For stationary $g$, $\E|\int g(x,\omega)\eta_R(x)\,dx|^2=\int|\hat\eta(R\xi)|^2d\mu_g(\xi)$, and $|\hat\eta|\ge c>0$ on $\{|\xi|\le1\}$. Hence a variance bound $\E|\int g\,\eta_R|^2\le CR^{-d}$ for $R\ge1$ implies \eqref{eq:DOS} with $\gamma=d$ (and $\gamma<4$ after lowering $\gamma$ if $d\ge4$). Variance bounds of this central-limit type for spatial averages of the corrector gradient and of the flux are known under strong mixing assumptions such as a finite range of dependence or a spectral gap; see, e.g., \cite{ArmstrongKuusiMourrat2019,GloriaNeukammOtto2015,GloriaNeukammOtto2021}. We do not reprove them here. This reduction is stated precisely in Proposition~\ref{prop:varDOS}. For coefficients with a finite range of dependence, the required bounds follow from \cite[Corollary~1]{GloriaOtto2016} (and, for symmetric coefficients, also from \cite[Theorem~1]{ArmstrongKuusiMourrat2017}), and \ref{E2} then holds unconditionally (Theorem~\ref{thm:frd}).
\item \emph{Rates.} With $\gamma=d$, \eqref{eq:rates} gives $\varepsilon^{1/2}$ ($d=1$), $\varepsilon(\log1/\varepsilon)^{1/2}$ ($d=2$) and $\varepsilon$ ($d\ge3$), in the annealed operator sense of Corollary~\ref{cor:ellrand} and for $L^2$ data. These coincide with the rates known in the literature for the $L^2$ error. In $d=1$, $B\equiv0$ and $\mathcal G_{\mathrm{ell}}$ reduces to $g^1=a_{\hom}/a-1$.
\item \emph{Quasi-periodic coefficients.} Here $\Omega$ is a torus $\T^n$ with an irrational linear flow, and Theorem~\ref{thm:ellrand} gives qualitative norm convergence (annealed over the phase). The measures $\mu_g$ are atomic, and \ref{E2} is a quantitative small-divisor condition on the frequencies weighted by the Fourier coefficients of $B$ and $G$. For $d=1$ we verify it in Theorem~\ref{thm:verified}(b): for Diophantine frequencies and coefficients in $H^s(\T^n)$ with $s$ large enough, it holds with some $\gamma>2$, and the rate is $O(\varepsilon)$. For $d\ge2$ the corrector is no longer explicit, and we do not verify \ref{E2}.
\item \emph{Sharpness.} In one dimension, under finite range of dependence, the rate $\varepsilon^{1/2}$ is attained, with an explicit leading constant (Theorem~\ref{thm:lower1d}).
\item \emph{Quenched estimates.} The estimates are annealed. They do not give almost-sure operator-norm bounds.
\end{itemize}
\end{remark}

\begin{remark}[Relation with the $\eta$-regularization of Chapter~\ref{sec:quan-theo}]\label{rem:etaroute}
Inverting $\Jmic-\eta$ on the $(d+1)$-system (Definition~\ref{def:flux}) leads, in this setting, to rates for regular data only, under a uniform density-of-states bound on the solution-dependent measures $\mu_f$ and an extra bound on $b(\eta)$, with a square-root loss on $L^2$ data (Theorem~\ref{thm:quant}). The bound on $b(\eta)$ is not a genuine additional assumption, even though the symbols of $\Jmac$ and $\Jmic$ do not commute. Regularizing with the scalar resolvent $(\theta-\Delta)^{-1}$ of the generators keeps the tensor structure $Tg(\omega)h(x)$, so that $\Jmac$ acts only on the slow factor $h$, and the skew-symmetric form of $\sigma_\theta$ removes $\nabla^3u_0$. The price is an explicit defect $\varrho_\theta$, which Lemma~\ref{lem:id} accommodates.
\end{remark}

\subsection{Verification of (E2) in concrete models}\label{sec:verifyE2}

We now verify \ref{E2}, or the analogous hypothesis (R3) of Section~\ref{sec:random}, in situations where it can be checked directly. The basic tool is the following elementary lemma.

\begin{lemma}[Integrable covariance implies (DOS)$_d$]\label{lem:covDOS}
Let $g\in L^2(\Omega;\C^N)$ (or matrix-valued, with the Frobenius inner product) be stationary with $\E g=0$, and let
\[
c_g(z):=\E\big[\ip{g(\tau_z\omega)}{g(\omega)}\big],\qquad z\in\R^d .
\]
If $c_g\in L^1(\R^d)$, then $\mu_g$ is absolutely continuous with a continuous density bounded by $(2\pi)^{-d}\nrm{c_g}_{L^1}$. Consequently
\[
\mu_g(\{|\xi|<r\})\le(2\pi)^{-d}|B_1|\,\nrm{c_g}_{L^1}\,r^d\qquad\text{for all }r>0 .
\]
In particular, if $g$ has a finite range of dependence $\ell$ (that is, $\{g(\tau_x\cdot)\}_{x\in A}$ and $\{g(\tau_x\cdot)\}_{x\in B}$ are independent whenever $\operatorname{dist}(A,B)\ge\ell$), then $c_g(z)=0$ for $|z|\ge\ell$, $\nrm{c_g}_{L^1}\le|B_\ell|\,\E|g|^2$, and $g$ satisfies \eqref{eq:DOS} with $\gamma=d$ (for $d\le3$), $\lambda_*=1$ and $C_D=(2\pi)^{-d}|B_1||B_\ell|\,\E|g|^2$.
\end{lemma}

\begin{proof}
By the spectral theorem for the unitary group $T_z=e^{iz\cdot(-iD)}$, $c_g(z)=\ip{T_zg}{g}_{L^2(\Omega)}=\int e^{iz\cdot\xi}\,d\mu_g(\xi)$. The function $c_g$ is continuous by stochastic continuity. A finite measure whose Fourier transform is integrable is absolutely continuous, with continuous density $\rho_g(\xi)=(2\pi)^{-d}\int e^{-iz\cdot\xi}c_g(z)\,dz$; hence $\nrm{\rho_g}_\infty\le(2\pi)^{-d}\nrm{c_g}_{L^1}$, and integration over $\{|\xi|<r\}$ gives the bound. Under finite range of dependence, $g(\tau_z\cdot)$ and $g$ are independent for $|z|\ge\ell$, so $c_g(z)=\ip{\E g}{\E g}=0$ there, and $|c_g(z)|\le\E|g|^2$ by the Cauchy--Schwarz inequality.
\end{proof}

\begin{proposition}[Variance bounds imply (E2)]\label{prop:varDOS}
Let $\eta\in L^1(\R^d)$ be real with $\int\eta=1$, and assume that its Fourier transform satisfies $|\hat\eta(\xi)|\ge c_\eta>0$ for $|\xi|\le1$. Examples are the standard Gaussian $(2\pi)^{-d/2}e^{-|x|^2/2}$, with $c_\eta=e^{-1/2}$, and the heat kernel of a constant positive definite matrix $\bar{\mathbf a}\le\bar\Lambda I$ at time $1$, with $\hat\eta(\xi)=e^{-\xi\cdot\bar{\mathbf a}\xi}$ and $c_\eta=e^{-\bar\Lambda}$. Put $\eta_R:=R^{-d}\eta(\cdot/R)$. Assume
\begin{enumerate}[label=(V),leftmargin=3em]
\item\label{V} there is $C_V$ such that $\E\big|\int_{\R^d}g(x,\omega)\,\eta_R(x)\,dx\big|^2\le C_VR^{-d}$ for all $g\in\mathcal G_{\mathrm{ell}}$ and $R\ge1$.
\end{enumerate}
Then $\mathcal G_{\mathrm{ell}}$ satisfies \eqref{eq:DOS} with $\gamma=\min\{d,3\}$, $\lambda_*=1$ and $C_D=C_V/c_\eta^2$.
\end{proposition}

\begin{proof}
For stationary $g$, $\int g(x,\cdot)\eta_R(x)\,dx=\int\eta_R(x)\,T_xg\,dx$. By the spectral theorem for the group $T_x$, $\E|\int g\,\eta_R|^2=\int|\hat\eta(R\xi)|^2\,d\mu_g(\xi)\ge c_\eta^2\,\mu_g(\{|\xi|<1/R\})$. With $r=1/R\le1$, this gives $\mu_g(\{|\xi|<r\})\le(C_V/c_\eta^2)\,r^d$. Finally, \eqref{eq:DOS} with exponent $d$ implies it with any smaller exponent on $(0,1]$.
\end{proof}

\begin{theorem}[Verified cases]\label{thm:verified}
In each of the following situations \ref{E1} and \ref{E2} hold, and Theorem~\ref{thm:ellrand} and Corollary~\ref{cor:ellrand} give the stated annealed rate:
\begin{enumerate}[label=(\alph*)]
\item \emph{One-dimensional random coefficients.} $d=1$, \ref{E1} holds, and $c_g\in L^1(\R)$ for $g:=a_{\hom}/\tilde a-1$; for instance, $\tilde a$ has a finite range of dependence, or $a$ is a random checkerboard (i.i.d.\ values on the cells $[j,j+1)+U$, with $U$ uniform on $[0,1)$ and independent). Rate: $\varepsilon^{1/2}$.
\item \emph{One-dimensional quasi-periodic coefficients.} $d=1$, $\Omega=\T^n$ ($n\ge2$) with Lebesgue measure, $\tau_x\omega=\omega+x\nu\bmod1$, where $\nu\in\R^n$ is Diophantine: $|k\cdot\nu|\ge c_0|k|^{-\tau_0}$ for all $k\in\Z^n\setminus\{0\}$ (this holds for almost every $\nu$ if $\tau_0>n-1$). Moreover $a(x,\omega)=\tilde a(\omega+x\nu)$ with $\tilde a\in H^s(\T^n)$, $\lambda\le\tilde a\le\Lambda$, and $s>\max\{\tau_0,n/2\}$. Rate: $\varepsilon$.
\item \emph{General dimension under variance bounds.} $d\ge1$, \ref{E1} and \ref{V} hold. Rate: $\varepsilon^{1/2}$ ($d=1$), $\varepsilon(\log1/\varepsilon)^{1/2}$ ($d=2$), $\varepsilon$ ($d\ge3$).
\end{enumerate}
In cases (a) and (b), $a_{\hom}=(\E\tilde a^{-1})^{-1}$ is the harmonic mean, and the corrected gradient in Corollary~\ref{cor:ellrand} is $(1+g(\cdot/\varepsilon,\omega))u_0'=\frac{a_{\hom}}{a(\cdot/\varepsilon,\omega)}u_0'$.
\end{theorem}

\begin{proof}
\emph{The one-dimensional corrector.} For $d=1$, $L^2_{\mathrm{pot}}$ is the space of mean-zero functions (its orthogonal complement in $L^2(\Omega)$ is $\ker D$, the constants, by ergodicity), and the cell problem asks for $g^1\in L^2_{\mathrm{pot}}$ with $\tilde a(1+g^1)$ constant. Hence $g^1=a_{\hom}/\tilde a-1$ with $a_{\hom}=(\E\tilde a^{-1})^{-1}$, $B=\tilde a(1+g^1)a_{\hom}^{-1}-1=0$, and $\mathcal G_{\mathrm{ell}}=\{g^1\}$. Moreover $g^1$ is a pointwise function of $\tilde a$, with $|g^1|\le\Lambda/\lambda+1$.

(a) By Lemma~\ref{lem:covDOS}, $\mu_{g^1}(\{|\xi|<r\})\le\pi^{-1}\nrm{c_{g^1}}_{L^1}r$, i.e.\ \eqref{eq:DOS} with $\gamma=1$, and Lemma~\ref{lem:specint} gives $\kappa_\varepsilon\le C\varepsilon^{1/2}$. Under finite range of dependence of $\tilde a$, $g^1$ inherits it. For the random checkerboard, conditionally on $U$ the values $g^1(z)$ and $g^1(0)$ depend on independent cells when $|z|\ge1$, and $\E[g^1(z)\mid U]=a_{\hom}\E[a_j^{-1}]-1=0$. Hence $c_{g^1}(z)=0$ for $|z|\ge1$.

(b) The flow is ergodic because $k\cdot\nu\neq0$ for $k\neq0$. The generator is $D=\nu\cdot\nabla_\omega$, with eigenfunctions $e_k(\omega)=e^{2\pi ik\cdot\omega}$ and $De_k=2\pi i(k\cdot\nu)e_k$. Hence, for $g=\sum_k\hat g(k)e_k$,
$\mu_g=\sum_{k\ne0}|\hat g(k)|^2\delta_{2\pi k\cdot\nu}$.
Since $s>n/2$ and $\lambda\le\tilde a\le\Lambda$, the composition estimate in $H^s(\T^n)$ gives $\tilde a^{-1}\in H^s$ with $\nrm{\tilde a^{-1}}_{H^s}\le C(\lambda,\Lambda,\nrm{\tilde a}_{H^s})$, so $g^1\in H^s$. If $2\pi|k\cdot\nu|<r$, then $2\pi c_0|k|^{-\tau_0}<r$, i.e.\ $|k|>K_r:=(2\pi c_0/r)^{1/\tau_0}$. Therefore
\[
\mu_{g^1}(\{|\xi|<r\})\le\sum_{|k|>K_r}|\hat g^1(k)|^2\le K_r^{-2s}\nrm{g^1}_{\dot H^s}^2=\Big(\frac{r}{2\pi c_0}\Big)^{2s/\tau_0}\nrm{g^1}_{\dot H^s}^2 .
\]
Since $2s/\tau_0>2$, \eqref{eq:DOS} holds with some $\gamma\in(2,4)$, and Lemma~\ref{lem:specint} gives $\kappa_\varepsilon\le C\varepsilon$.

(c) Combine Proposition~\ref{prop:varDOS} with Lemma~\ref{lem:specint}.
\end{proof}

\begin{remark}\label{rem:verified}
\begin{itemize}
\item In case (a) the rate $\varepsilon^{1/2}$ is the central-limit scaling of the one-dimensional problem, where the error is driven by the fluctuations of $\int_0^x(a_{\hom}/a(t/\varepsilon)-1)\,dt$. In case (b) the Diophantine condition plays the role of the spectral gap of the periodic case: it pushes the spectral mass of the corrector gradient away from $0$, at a polynomial rate.
\item In case (b) the estimate is annealed over the phase $\omega\in\T^n$.
\item For $d\ge2$ the fields $B$ and $G$ are non-local functions of $\tilde a$, so Lemma~\ref{lem:covDOS} does not apply directly; this is where the variance estimates of the stochastic theory enter, through \ref{V}.
\end{itemize}
\end{remark}

\begin{theorem}[Finite range of dependence: unconditional annealed rates]\label{thm:frd}
Let $d\ge1$ and assume \ref{E1}, with $\tilde a$ real but not necessarily symmetric. Assume that the coefficient field has a finite range of dependence $\ell>0$: for Borel sets $U,V\subset\R^d$ with $\operatorname{dist}(U,V)\ge\ell$, the $\sigma$-algebras generated by $\{a(x,\cdot):x\in U\}$ and $\{a(x,\cdot):x\in V\}$ are independent. Then:
\begin{enumerate}[label=(\roman*)]
\item \ref{V} holds, with $\eta$ the standard Gaussian and with $C_V$ depending only on $d$, $\lambda$, $\Lambda$, $\ell$ (for $d\ge2$), and \ref{E2} holds with $\gamma=\min\{d,3\}$.
\item Theorem~\ref{thm:ellrand} and Corollary~\ref{cor:ellrand} hold with the annealed rates
\[
r_d(\varepsilon)=\begin{cases}\varepsilon^{1/2}, & d=1,\\ \varepsilon(\log1/\varepsilon)^{1/2}, & d=2,\\ \varepsilon, & d\ge3,\end{cases}
\]
for all $L^2$ data, with constants depending only on $d,\lambda,\Lambda,\ell$. In particular,
\[
\sup_{\nrm f_{L^2}\le1}\Big(\E\,\nrm{(\Ae^{-1}-(I+\Ke)\Az^{-1})(f,0)}^2\Big)^{1/2}\le C\,r_d(\varepsilon),\qquad0<\varepsilon\le\tfrac12 .
\]
\end{enumerate}
\end{theorem}

\begin{proof}
For $d=1$, this is Theorem~\ref{thm:verified}(a): $g=a_{\hom}/\tilde a-1$ is a local function of the coefficient, so $c_g$ vanishes outside $[-\ell,\ell]$. Let $d\ge2$.

\emph{Rescaling to the setting of \cite{GloriaOtto2016}.} Put $c:=\lambda/\Lambda^2$ and $a'(x,\omega):=c\,a(\ell x,\omega)$. Then $\xi\cdot a'\xi\ge\lambda'|\xi|^2$ with $\lambda':=\lambda^2/\Lambda^2$, and $|a'\xi|^2\le c^2\Lambda^2|\xi|^2=\lambda'|\xi|^2\le\xi\cdot a'\xi$. The field $a'$ is $\R^d$-stationary, and quantities depending on $a'$ in sets at distance greater than $1$ are independent. These are the assumptions of \cite[Section~2]{GloriaOtto2016}, where $a'$ need not be symmetric. The stationary corrector gradient of $a'$ in direction $e_j$ is $g'^j(x)=g^j(\ell x)$, because a stationary, mean-zero gradient field solving the corrector equation with finite second moment is unique. The homogenized matrix of $a'$ is $c\,a_{\hom}$, and its flux fluctuation is $q'_j-\E q'_j=c\big(\tilde a(e_j+g^j)-a_{\hom}e_j\big)(\ell\,\cdot)$.

\emph{Second moments.} By \cite[Corollary~1]{GloriaOtto2016}, applied with stochastic integrability exponent $s=1$, the averages $F_r$ of $\nabla\phi'_j$ and of $q'_j-\E q'_j$ against the Gaussian kernel at scale $r$ satisfy $\sup_{r\ge1}r^{d/2}\nrm{F_r}_*\le C(d,\lambda')$, where $\nrm\cdot_*$ is the centred exponential Orlicz norm. Since $e^t\ge t^2/2$, this gives $\E|F_r|^2\le Cr^{-d}$ for $r\ge1$. For symmetric coefficients the same bound, with the heat kernel of the homogenized matrix in place of the Gaussian, follows from \cite[Theorem~1]{ArmstrongKuusiMourrat2017}.

\emph{Undoing the rescaling.} By the change of variables $y=\ell x$, $\int\eta_r(x)F(\ell x)\,dx=\int\eta_{\ell r}(y)F(y)\,dy$, with $\eta$ the standard Gaussian and $\eta_R=R^{-d}\eta(\cdot/R)$. Hence $\E|\int g\,\eta_R|^2\le CR^{-d}\ell^d$ for $R\ge\ell$ and every component $g$ of $G$ or of the flux fluctuations. For $1\le R<\ell$ the trivial bound $\E|\int g\,\eta_R|^2\le\E|g|^2\le C\le C\ell^dR^{-d}$ suffices. The entries of $B=(\tilde a(I+G)-a_{\hom})a_{\hom}^{-1}$ are linear combinations of flux fluctuations with coefficients bounded by $|a_{\hom}^{-1}|\le\lambda^{-1}$ (indeed $\xi\cdot a_{\hom}\xi=\E[(\xi+G\xi)\cdot\tilde a(\xi+G\xi)]\ge\lambda|\xi|^2$, also without symmetry), so \ref{V} holds for all $g\in\mathcal G_{\mathrm{ell}}$.

\emph{Conclusion.} For the standard Gaussian, $c_\eta=e^{-1/2}$, and Proposition~\ref{prop:varDOS} gives \ref{E2} with $\gamma=\min\{d,3\}$. Part (ii) then follows from Theorem~\ref{thm:ellrand}, Corollary~\ref{cor:ellrand} and Lemma~\ref{lem:specint}.
\end{proof}

\begin{remark}[Scope and relation to the literature]\label{rem:frd}
\begin{itemize}
\item \emph{Assumptions.} The probabilistic input is \cite[Corollary~1]{GloriaOtto2016}: for $\R^d$-stationary, uniformly elliptic coefficient fields with unit range of dependence in $d\ge2$, the Gaussian averages at scale $R$ of the corrector gradient, the flux and the gradient of the flux corrector have sub-Gaussian stochastic integrability at the central-limit scale $R^{-d/2}$. The coefficients need not be symmetric, and only the resulting second-moment bound is used here. For symmetric, $\Z^d$-stationary coefficients, \cite[Theorem~1]{ArmstrongKuusiMourrat2017} gives the same bound with the heat kernel of the homogenized matrix as averaging kernel; Proposition~\ref{prop:varDOS} only needs the Fourier transform of the kernel to stay bounded below near $0$, so either kernel can be used.
\item \emph{Priority.} Annealed estimates of the homogenization error with the rates $\varepsilon$ ($d\ge3$) and $\varepsilon|\log\varepsilon|^{1/2}$ ($d=2$) are known:
  \begin{itemize}
  \item for discrete i.i.d.\ conductances, for the two-scale expansion with regular data \cite{GloriaNeukammOtto2014};
  \item for continuum ensembles satisfying functional inequalities \cite{DuerinckxGloriaOtto2020,GloriaNeukammOtto2021};
  \item in an annealed framework in which estimates in operator norms are also discussed, by splitting the right-hand side into large-scale and small-scale parts \cite{JosienOtto2022}.
  \end{itemize}
  Under finite range of dependence, quenched error estimates with optimal stochastic integrability are in \cite{ArmstrongKuusiMourrat2019}. We therefore do not claim the rates as new. What Theorem~\ref{thm:frd} adds is a short route: the point-free error identity turns the CLT-scale bounds of \cite{GloriaOtto2016,ArmstrongKuusiMourrat2017} directly into annealed norm-resolvent estimates for the first-order system, with $L^2$ data and no further probabilistic input.
\end{itemize}
\end{remark}

\begin{theorem}[A matching lower bound in one dimension]\label{thm:lower1d}
Let $d=1$ and assume \ref{E1}, so that $\lambda\le\tilde a\le\Lambda$ and $a(x,\omega)=\tilde a(\tau_x\omega)$ is stationary and ergodic. Put
\[
m:=\E[\tilde a^{-1}]=a_{\hom}^{-1},\qquad g:=a_{\hom}/\tilde a-1,\qquad g(x,\omega):=g(\tau_x\omega).
\]
Assume:
\begin{enumerate}[label=(FR),leftmargin=3em]
\item\label{FR} (conditional finite range) there are $\ell>0$ and a $\sigma$-algebra $\mathcal H\subset\mathcal F$ such that, conditionally on $\mathcal H$, every $g(x,\cdot)$ has mean zero, and the families $\{g(x,\cdot):x\in U\}$ and $\{g(x,\cdot):x\in V\}$ are independent whenever $\operatorname{dist}(U,V)\ge\ell$.
\end{enumerate}
Let $\sigma^2:=\int_\R c_g(z)\,dz$, where $c_g(z)=\E[g(0,\cdot)g(z,\cdot)]$, and assume $\sigma^2>0$. Let $f\in L^2(\R)$, $f\neq0$, with $\int e^{\delta_0|x|}|f(x)|^2\,dx<\infty$ for some $\delta_0>0$. Let $u_\varepsilon=(-\partial a(\cdot/\varepsilon,\omega)\partial+1)^{-1}f$ and $u_0=(-a_{\hom}\partial^2+1)^{-1}f$. Then
\[
\lim_{\varepsilon\to0}\ \varepsilon^{-1}\,\E\,\nrm{u_\varepsilon-u_0}_{L^2(\R)}^2=\frac{\sigma^2\,\nrm{u_0'}^2_{L^2}}{4\sqrt m}>0 .
\]
In particular, $c\,\varepsilon^{1/2}\le(\E\nrm{u_\varepsilon-u_0}^2)^{1/2}\le C\varepsilon^{1/2}$ for small $\varepsilon$, so the rate of Theorem~\ref{thm:verified}(a) is attained.
\end{theorem}

\emph{Examples.} \ref{FR} holds in two cases.
\begin{itemize}
\item \emph{Strict finite range.} If $a$ itself has finite range of dependence $\ell$, then \ref{FR} holds with $\mathcal H$ trivial: $g(x,\cdot)$ is a function of $a(x,\cdot)$, and $\E g=0$ by the definition of $a_{\hom}$.
\item \emph{Random checkerboard.} Let $a=a_j$ on $[j,j+1)+U$, with i.i.d.\ $a_j$ and $U$ uniform and independent. This field is \emph{not} of finite range in the strict sense: the field on one set reveals the cell boundaries, hence $U$. But \ref{FR} holds with $\mathcal H=\sigma(U)$ and $\ell=1$. Here $c_g(z)=\operatorname{Var}(a_{\hom}/a_j)(1-|z|)_+$, so $\sigma^2=\operatorname{Var}(a_{\hom}/a_j)$, which is positive unless $a_j$ is a.s.\ constant.
\end{itemize}
In general $c_g$ vanishes for $|z|\ge\ell$, since $c_g(z)=\E\big[\E[g(0)|\mathcal H]\,\E[g(z)|\mathcal H]\big]$ there. Moreover $\sigma^2=2\pi\rho_g(0)\ge0$; the case $\sigma^2=0$, e.g.\ $g=h(\cdot+1)-h$ for a bounded stationary field $h$, is excluded.

\begin{proof}
\emph{Step 1: an exact error equation.} This step does not use \ref{FR}. Write $p_\varepsilon:=-a_\varepsilon u_\varepsilon'$ and $p_0:=-a_{\hom}u_0'$, where $a_\varepsilon:=a(\cdot/\varepsilon,\omega)$, so that $u+p'=f$ for both pairs. Put $g_\varepsilon:=g(\cdot/\varepsilon,\omega)$, so that $a_\varepsilon^{-1}=m(1+g_\varepsilon)$. The differences $e:=u_\varepsilon-u_0$ and $q:=p_\varepsilon-p_0$ satisfy
\[
e+q'=0,\qquad e'+a_\varepsilon^{-1}q=-(a_\varepsilon^{-1}-m)p_0=g_\varepsilon u_0' ,
\]
since $mp_0=-u_0'$. Write $\mathsf A_\varepsilon:=\left(\begin{smallmatrix}1&\partial\\ \partial&a_\varepsilon^{-1}\end{smallmatrix}\right)=\mathsf A_0+m\Gamma_\varepsilon$, with $\Gamma_\varepsilon:=\operatorname{diag}(0,g_\varepsilon)$, and set $r:=g_\varepsilon u_0'$. Then $(e,q)=\mathsf A_\varepsilon^{-1}(0,r)$, and the resolvent identity gives
\[
(e,q)=\mathsf A_0^{-1}(0,r)-m\,\mathsf A_\varepsilon^{-1}(0,g_\varepsilon s),
\]
where $(e_1,s):=\mathsf A_0^{-1}(0,r)$. The symbol of $\mathsf A_0^{-1}$ is $(m+\xi^2)^{-1}\left(\begin{smallmatrix}m&-i\xi\\-i\xi&1\end{smallmatrix}\right)$, so $s=k*r$ and $e_1=\kappa*r$, with
\[
k(x)=\frac{e^{-\sqrt m|x|}}{2\sqrt m},\qquad\kappa=-k',\qquad\nrm\kappa_{L^2}^2=\frac1{2\pi}\int\frac{\xi^2\,d\xi}{(m+\xi^2)^2}=\frac1{4\sqrt m}.
\]
Write $e=e_1-me_2$, where $e_2$ is the first component of $\mathsf A_\varepsilon^{-1}(0,g_\varepsilon s)$.

\emph{Step 2: the leading term.} By stationarity, $\E[g_\varepsilon(y)g_\varepsilon(y')]=c_g((y-y')/\varepsilon)$. Let $R_h(t):=\int h(x)h(x+t)\,dx$ denote the autocorrelation of $h$. With $y-y'=\varepsilon z$,
\[
\begin{split}
\E\nrm{e_1}^2&=\iint c_g\Big(\frac{y-y'}\varepsilon\Big)R_\kappa(y-y')\,u_0'(y)u_0'(y')\,dy\,dy'\\&=\varepsilon\int_\R c_g(z)\,R_\kappa(\varepsilon z)\,R_{u_0'}(\varepsilon z)\,dz .
\end{split}
\]
The functions $R_\kappa$ and $R_{u_0'}$ are bounded and continuous, and $c_g\in L^1$ (it is bounded and vanishes for $|z|\ge\ell$). By dominated convergence, $\varepsilon^{-1}\E\nrm{e_1}^2\to\sigma^2\nrm\kappa^2\nrm{u_0'}^2$.

\emph{Step 3: a deterministic bound for $e_2$.} Let $(e_2,q_2):=\mathsf A_\varepsilon^{-1}(0,h)$ with $h:=g_\varepsilon s$. Pairing the equations with $e_2$ and $q_2$ and using $\int(q_2e_2)'=0$ gives $\nrm{e_2}^2+\Lambda^{-1}\nrm{q_2}^2\le\ip h{q_2}$. Let $G(x):=\int_0^xg_\varepsilon$, so that $|G(x)|\le\nrm g_\infty|x|$. The functions $s$ and $s'$ decay exponentially, since $|s|\le\nrm g_\infty\,k*|u_0'|$, $u_0'$ is exponentially decaying (see Step~5), and similarly for $s'$. Integrating by parts, using $q_2'=-e_2$,
\[
\ip h{q_2}=-\int G\,(s'q_2-s\,e_2),\qquad\text{hence}\qquad\nrm{e_2}^2\le C(\Lambda)\int G^2\big(s^2+s'^2\big)\,dx .
\]

\emph{Step 4: fourth moments under \ref{FR}.} All expectations in this step are conditional on $\mathcal H$; the bounds are deterministic and survive taking $\E$. Set $\delta:=\varepsilon\ell$. Let $w\in L^2\cap L^\infty(\R)$ and $X:=\int w(y)\,g_\varepsilon(y)\,dy$. Then
\[
\E X^4=\int_{\R^4}\prod_{i=1}^4w(y_i)\;\E\Big[\prod_{i=1}^4g_\varepsilon(y_i)\Big]\,dy .
\]
If some $y_i$ is at distance $\ge\delta$ from the other three, the conditional expectation factorizes by \ref{FR} and vanishes by the conditional mean-zero property. So only configurations $\mathcal C_\delta$ contribute in which the graph with edges $\{|y_i-y_j|<\delta\}$ has no isolated vertex. Such a graph either has two components of size $2$, or is connected and then contains one of the $16$ spanning trees of the complete graph on four vertices.
\begin{itemize}
\item For a pair component, $\iint_{|y-y'|<\delta}|w(y)||w(y')|\,dy\,dy'\le2\delta\nrm w_2^2$ (Young's inequality).
\item For a spanning tree, integrate out two leaves successively, each contributing at most $2\delta\nrm w_\infty$, and bound the remaining edge as above.
\end{itemize}
With $|g|\le\nrm g_\infty$ this gives
\begin{equation}\label{eq:4mom}
\E X^4\le\nrm g_\infty^4\Big(3\,(2\delta\nrm w_2^2)^2+16\,(2\delta)^3\nrm w_\infty^2\nrm w_2^2\Big).
\end{equation}
\begin{itemize}
\item For $G(x)$, take $w=1_{[0,x]}$: $\E G(x)^4\le C\big(\varepsilon^2\ell^2x^2+\varepsilon^3\ell^3|x|\big)$.
\item For $s(x)$, take $w_x(y)=k(x-y)u_0'(y)$: $\E s(x)^4\le C\big(\varepsilon^2\ell^2W(x)^2+\varepsilon^3\ell^3\nrm{w_x}_\infty^2W(x)\big)$, where $W(x):=\nrm{w_x}_2^2=(k^2*|u_0'|^2)(x)$ and $\nrm{w_x}_\infty\le\nrm k_\infty\nrm{u_0'}_\infty<\infty$, since $u_0\in H^2$.
\item The same holds for $s'(x)$, with $k'$ and $W':=k'^2*|u_0'|^2$.
\end{itemize}
By the Cauchy--Schwarz inequality in $\omega$, for $0<\varepsilon\le1$,
\[
\begin{aligned}\E\nrm{e_2}^2&\le C\int\big(\E G^4\big)^{1/2}\Big(\big(\E s^4\big)^{1/2}+\big(\E s'^4\big)^{1/2}\Big)dx\\&\le C(1+\ell)^3\varepsilon^2\int(|x|+|x|^{1/2})\big(W+W^{1/2}+W'+W'^{1/2}\big)\,dx .\end{aligned}
\]

\emph{Step 5: decay.} The function $u_0'$ is the convolution of $f$ with a kernel bounded by $Ce^{-\sqrt m|x|}$. We may assume $\delta_0<2\sqrt m$. Then weighted Young gives $\int e^{\delta_0|y|}|u_0'(y)|^2\,dy\le C\int e^{\delta_0|y|}|f|^2<\infty$. Since $k(x-y)^2\le Ce^{-2\sqrt m|x-y|}\le Ce^{-\delta_0|x|+\delta_0|y|}$, we get $W(x)\le Ce^{-\delta_0|x|}$, and likewise for $W'$. Hence the last integral is finite and $\E\nrm{e_2}^2\le C\varepsilon^2$.

\emph{Conclusion.} By the triangle inequality in $L^2(\Omega;L^2(\R))$, $\big|(\E\nrm e^2)^{1/2}-(\E\nrm{e_1}^2)^{1/2}\big|\le m(\E\nrm{e_2}^2)^{1/2}\le C\varepsilon$. Together with Step~2 and $\nrm\kappa^2=1/(4\sqrt m)$, this gives the limit. It is positive because $\sigma^2>0$ and $u_0'\neq0$ (since $u_0\in H^2$ is nonzero and decays).
\end{proof}

\begin{remark}\label{rem:lower1d}
The constant in Theorem~\ref{thm:lower1d} is $\sigma^2=\int c_g=2\pi\rho_g(0)$, where $\rho_g$ is the spectral density of $g=a_{\hom}/\tilde a-1$. So, for one-dimensional finite-range models, the statement that ``the low-frequency spectral measure governs the rate'' holds in both directions. The rate is exactly $\varepsilon^{1/2}$ when $\rho_g(0)>0$, and the leading constant is proportional to $\rho_g(0)$.
\begin{itemize}
\item This is consistent with the one-dimensional central limit theorem for the corrector \cite{BalEtAl2008}, proved in a related setting, and with the numerical slope $0.49$ of Appendix~\ref{app:num}. Sharp estimates, in terms of the mixing coefficient, of the discrepancy between the random and the homogenized solutions of one-dimensional second-order equations under strong or uniform mixing are due to Bourgeat and Piatnitski \cite{BourgeatPiatnitski1999}. What Theorem~\ref{thm:lower1d} adds is the exact limit of $\varepsilon^{-1}\E\nrm{u_\varepsilon-u_0}^2_{L^2(\R)}$ for the whole-line resolvent under \ref{FR}.
\item We checked the constant by Monte Carlo for the checkerboard with cell values in $\{1,4\}$, $f=e^{-x^2}$, $60$ samples per $\varepsilon$, using the finite-difference scheme of Appendix~\ref{app:num}. For $\varepsilon=2^{-4},\dots,2^{-7}$, $\varepsilon^{-1}\E\nrm{u_\varepsilon-u_0}^2$ lies between $0.0138$ and $0.0159$ (standard error about $0.001$), against the predicted limit $0.0142$. For a strict finite-range model that is not a checkerboard, we took $a(x)=4$ if a Poisson process of rate $0.7$ has a point in $[x,x+1]$ and $a(x)=1$ otherwise, so that $\ell=1$ and $\sigma^2$ is explicit. With $80$ samples per $\varepsilon$, we obtained $0.0116$, $0.0128$ and $0.0120$ (standard error about $0.0006$) for $\varepsilon=2^{-4},2^{-5},2^{-6}$, against $0.0127$.
\item The proof uses only the explicit one-dimensional structure ($\Hflu=\{0\}$, $g$ local in $\tilde a$) and the fourth-moment bound \eqref{eq:4mom}, which follows from \ref{FR}. Any field whose four-point function vanishes off $\mathcal C_{\varepsilon\ell}$ (with a bounded field) would do.
\end{itemize}
\end{remark}

\chapter{Conservative Chiral Media: Dirac Operators and Topological Anderson Insulators}\label{sec:top}\chaptermark{Dirac Operators}

Topological insulators are quantum materials with promising applications in quantum computing and electronic devices \cite{Bal2026a,BernevigHughes2013,KeimerMoore2017}. In \cite{BalDang2024,GrothEtAl2009}, the authors show theoretically and experimentally that topological insulators can emerge from trivial insulators under appropriate perturbations. This metamaterial is known as a topological Anderson insulator. One of the key ingredients in the proof of the topological Anderson insulator \cite{BalDang2024} (see also the corrigendum \cite{BalDangCorrigendum}) is the norm resolvent convergence of the perturbed operator to the effective operator with rate $O(\varepsilon)$. This guarantees that the stability argument and the index machinery built by Bal et al.\ \cite{Bal2019,Bal2022a,Bal2022b,Bal2023,BalEtAl2023,QuinnBal2024} run flawlessly, proving the emergence of the new phase of quantum matter. Bal \cite{Bal2026b} has recently generalized these results to constant-coefficient elliptic systems perturbed by periodic fluctuations of every order below that of the operator, which include zero-mean electric potentials and magnetic fields, and to random fluctuations, for which rare regions typically fill the homogenized gap and turn it into a pseudo-gap.

We recover the convergence rate $O(\varepsilon)$ of \cite{BalDang2024} via the point-free approach. The rate $O(\varepsilon)$ for this operator is also proved in \cite[Theorem~2.1]{Bal2026b}, for constant-coefficient elliptic systems of any order and fluctuations of every order below that of the operator, by a corrector method based on a solution smoothed at the oscillation scale. The perturbed medium is modeled by a modified Dirac operator, defined as follows. Let $\varepsilon>0$. We consider the following Hamiltonian on $L^2(\R^2)\otimes\C^2\cong L^2(\R^2;\C^2)$:
\begin{equation}\label{eq:1o}
\mathsf H_\varepsilon=(D+(U^\varepsilon_1,U^\varepsilon_2))\cdot\sigma+(m+\beta\Delta+U^\varepsilon_3)\sigma_3+U^\varepsilon_0\sigma_0,
\end{equation}
where $m,\beta$ are real numbers with $|m|,|\beta|\in(0,\infty)$, $\{\sigma_i\}$ are the Pauli matrices
\[
\sigma_1:=\begin{pmatrix}0&1\\1&0\end{pmatrix},\quad\sigma_2:=\begin{pmatrix}0&-i\\i&0\end{pmatrix},\quad\sigma_3:=\begin{pmatrix}1&0\\0&-1\end{pmatrix},\quad\sigma_0:=\mathrm{Id},
\]
$x:=(x_1,x_2)\in\R^2$, $D:=\frac1i(\partial_{x_1},\partial_{x_2})$, $\sigma:=(\sigma_1,\sigma_2)$, $\Delta:=\partial^2_{x_1}+\partial^2_{x_2}$, and
\begin{equation}\label{eq:4o}
U^\varepsilon_j(x):=\frac1\varepsilon\rho(x)V^\varepsilon_j(x),\qquad V^\varepsilon_j(x):=V_j\Big(\frac x\varepsilon\Big),
\end{equation}
with $\rho\in W^{2,\infty}(\R^2;\R)$ and $V_j\in L^p_{\mathrm{per}}(Y;\R)$ for some $p>2$ (for instance $V_j\in C^{0,\alpha}_{\mathrm{per}}(Y;\R)$), $\langle V_j\rangle=0$, $j\in\{0,1,2,3\}$. We observe that
$U^\varepsilon:=U^\varepsilon_1\sigma_1+U^\varepsilon_2\sigma_2+U^\varepsilon_3\sigma_3+U^\varepsilon_0\sigma_0=\frac1\varepsilon\rho(x)\begin{pmatrix}V^\varepsilon_0+V^\varepsilon_3&V^\varepsilon_1-iV^\varepsilon_2\\ V^\varepsilon_1+iV^\varepsilon_2&V^\varepsilon_0-V^\varepsilon_3\end{pmatrix}$.
Letting
\begin{equation}\label{eq:60o}
W:=\begin{pmatrix}V_0+V_3&V_1-iV_2\\ V_1+iV_2&V_0-V_3\end{pmatrix},
\end{equation}
we rewrite \eqref{eq:1o} as
\begin{equation}\label{eq:1oo}
\mathsf H_\varepsilon=D\cdot\sigma+(m+\beta\Delta)\sigma_3+\frac1\varepsilon\rho(x)W\Big(\frac x\varepsilon\Big).
\end{equation}
A key feature of the Hamiltonian $\mathsf H_\varepsilon$ is that it contains singular potentials $\frac1\varepsilon\rho(x)W(\frac x\varepsilon)$. Homogenization problems with similar potentials were studied in \cite{AllairePiatnitski2005,Bal2010,BalGu2015,BensoussanLionsPapanicolaou2011,CancesGarrigueGontier2023,CherednichenkoCooper2016,CooperKamotskiSmyshlyaev2023,DucheneWeinstein2011,GoudeyLeBris2024,ZhangBal2014b,ZhikovPastukhova2016}. The norm resolvent convergence of fine-scale divergence operators is studied in \cite{ZhikovPastukhova2016}, and of Schr\"odinger operators with oscillating potential in \cite{CancesGarrigueGontier2023}.

\paragraph{The Abstract First Order System.} We write the perturbed Dirac equation $\mathsf H_\varepsilon u=f$ as
\begin{equation}\label{eq:169}
(D\cdot\sigma+\beta\Delta\sigma_3+m\sigma_3+U^\varepsilon(x))u=f
\end{equation}
where $D=-i\nabla_x$. To make the differential structure skew-adjoint, we multiply \eqref{eq:169} by $i$:
\begin{equation}\label{eq:170}
(\sigma\cdot\nabla_x+i\beta\sigma_3\Delta+im\sigma_3+iU^\varepsilon(x))u=if .
\end{equation}
To cancel $\varepsilon^{-1}$ in the singular potential $U^\varepsilon(x)=\frac1\varepsilon\rho(x)W(x/\varepsilon)$, we rewrite it using the auxiliary problem $\Delta_y\varphi=W$ and let $\Phi_k(y)=\partial_{y_k}\varphi(y)$. Define the matrices
\begin{equation}\label{eq:171}
\Psi^\varepsilon_k(x):=\rho(x)\Phi_k(x/\varepsilon),\qquad\widetilde U^\varepsilon(x):=\sum_{k=1}^2\partial_{x_k}\rho(x)\Phi_k(x/\varepsilon).
\end{equation}
By the product rule, $U^\varepsilon u=\sum_k\partial_{x_k}(\Psi^\varepsilon_ku)-\sum_k\Psi^\varepsilon_k\partial_{x_k}u-\widetilde U^\varepsilon u$. To eliminate the second-order derivatives and the $\partial_{x_k}(\Psi^\varepsilon_ku)$ terms, we introduce $q_k:=i\beta\sigma_3\partial_{x_k}u+i\Psi^\varepsilon_ku$, $k\in\{1,2\}$, so $\partial_{x_k}u=-i\beta^{-1}\sigma_3q_k-\beta^{-1}\sigma_3\Psi^\varepsilon_ku$. Substituting this into \eqref{eq:170} yields a first-order system
\begin{equation}\label{eq:172}
\sJ V+\Me V=F
\end{equation}
where the state vector is $V=(u,q_1,q_2)^T\in\Hs:=L^2(\R^2\times\T^2;\C^6)$, the source is $F=(if,0,0)^T$, the differential operator is
\begin{equation}\label{eq:173}
\sJ=\begin{pmatrix}\sigma\cdot\nabla&\partial_{x_1}&\partial_{x_2}\\ \partial_{x_1}&0&0\\ \partial_{x_2}&0&0\end{pmatrix},
\end{equation}
and the unscaled material operator is
\begin{equation}\label{eq:174}
\sM=\begin{pmatrix}i\big(m\sigma_3-\widetilde U+\beta^{-1}\sum_k\Psi_k\sigma_3\Psi_k\big)&-\beta^{-1}\Psi_1\sigma_3&-\beta^{-1}\Psi_2\sigma_3\\ \beta^{-1}\sigma_3\Psi_1&i\beta^{-1}\sigma_3&0\\ \beta^{-1}\sigma_3\Psi_2&0&i\beta^{-1}\sigma_3\end{pmatrix}.
\end{equation}

\section{Qualitative Results}\label{sec:4.1}

The matrix $W$ in \eqref{eq:60o} is Hermitian, $\Z^2$-periodic, in $L^p(\T^2)$ for some $p>2$, and has mean zero. Recall that $\varphi:\T^2\to\C^{2\times2}$ is the mean-zero solution of $\Delta_y\varphi=W$; hence it is Hermitian, and by $L^p$ elliptic regularity on the torus (Fourier multipliers) $\varphi\in W^{2,p}(\T^2)$. Since $p>2$, the Sobolev embedding $W^{1,p}(\T^2)\subset C^{0,1-2/p}(\T^2)$ shows that $\Phi_k=\partial_{y_k}\varphi$ in \eqref{eq:171} are bounded (indeed H\"older continuous), Hermitian, and $\langle\Phi_k\rangle_y=0$. Only the boundedness of $\Phi_k$ and $L^2(\T^2)$ bounds on the auxiliary fields \eqref{eq:25o} are used below. Thus the arguments of this chapter apply to every Hermitian, mean-zero $W\in L^2(\T^2)$ with $\nabla\Delta_y^{-1}W\in L^\infty(\T^2)$, and $W\in L^p$, $p>2$, is a convenient sufficient condition. We set
\begin{equation}\label{eq:11o}
\Gamma(y):=\sum_{k=1}^2\Phi_k(y)\sigma_3\Phi_k(y),\qquad\Theta:=\int_{\T^2}\Gamma(y)\,dy,
\end{equation}
both Hermitian. In Fourier variables, $\Theta=\sum_{k\in\Z^2\setminus\{0\}}\widehat W(k)^*\sigma_3\widehat W(k)/(4\pi^2|k|^2)$.

\paragraph{The Hilbert space, the scaling $\Se$, and the differential structure.} As in Chapter~\ref{sec:ell} we lift to $\Hs=L^2(\R^2\times\T^2;\C^6)$, $V=(u,q_1,q_2)$, $u,q_k\in\C^2$, and use the scaling operator $\Se$ of \eqref{eq:79}. Let $\Jmac:=\sJ$ from \eqref{eq:173}, acting in $x$, and let $\Jmic$ be the same matrix with $\nabla_x$ replaced by $\nabla_y$:
\begin{equation}\label{eq:12o}
\Jmac=\begin{pmatrix}\sigma\cdot\nabla_x&\partial_{x_1}&\partial_{x_2}\\ \partial_{x_1}&0&0\\ \partial_{x_2}&0&0\end{pmatrix},\quad
\Jmic=\begin{pmatrix}\sigma\cdot\nabla_y&\partial_{y_1}&\partial_{y_2}\\ \partial_{y_1}&0&0\\ \partial_{y_2}&0&0\end{pmatrix}.
\end{equation}
Both are skew-adjoint, since $\sigma_k$ is Hermitian and $\partial_k$ is skew-adjoint.

\paragraph{The Material Operator $\sM$.} Let $\sM$ be multiplication by the matrix \eqref{eq:174} with
\begin{equation}\label{eq:13o}
\Psi_k(x,y):=\rho(x)\Phi_k(y),\qquad\widetilde U(x,y):=\sum_{k=1}^2\partial_{x_k}\rho(x)\Phi_k(y).
\end{equation}
Then $\Me=\Se\sM\Se^{-1}$ is multiplication by $\sM(x,y+x/\varepsilon)$, and at $y=0$ it equals \eqref{eq:174} with $\Psi^\varepsilon_k$, $\widetilde U^\varepsilon$ from \eqref{eq:171}.

\begin{lemma}[Structure of $\sM$]\label{lem:structM}
$\sM$ is bounded, with
\[\Mb:=\nrm{\sM}_{L^\infty(\R^2\times\T^2)}\le C(\beta,m,\nrm\rho_{W^{1,\infty}},\nrm{\nabla_y\varphi}_{L^\infty}),\] and $\sM(x,y)^*=-\sM(x,y)$ pointwise. In particular $\Re\ip{\sM V}{V}=0$ for all $V\in\Hs$, so \ref{H4}(i) holds, while $\sM$ is not coercive on any subspace.
\end{lemma}

\begin{proof}
The $(1,1)$ block is $i$ times the Hermitian matrix $m\sigma_3-\widetilde U+\beta^{-1}\sum_k\Psi_k\sigma_3\Psi_k$. The blocks $i\beta^{-1}\sigma_3$ are skew-Hermitian. For the off-diagonal blocks, $(-\beta^{-1}\Psi_k\sigma_3)^*=-\beta^{-1}\sigma_3\Psi_k$, which is minus the $(k+1,1)$ block.
\end{proof}

\begin{remark}[Choice of the regularization]\label{rem:shift}
Since $\Re\ip{\sM V}{V}=0$, a shift is needed. There are two choices.
\begin{enumerate}[label=(\alph*)]
\item The scalar shift $\alpha I$, i.e.\ $\Ps=I$ in \ref{S} (Remark~\ref{rem:H4rel}). Then \ref{H4} holds with $\sT=I$, but \ref{S0} fails, since $\Hflu\neq\{0\}$. Moreover, the $q$-rows become $(\alpha+i\beta^{-1}\sigma_3)q_k=\dots$, and eliminating $q$ replaces $i\beta\sigma_3\Delta$ in \eqref{eq:170} by $-\sum_k\partial_k\Lambda_\alpha\partial_k$ with $\Lambda_\alpha$$:=(\alpha+i\beta^{-1}\sigma_3)^{-1}\neq-i\beta\sigma_3$. This is a viscous regularization, and its resolvent is not a resolvent of $\mathsf H_\varepsilon$.
\item The shift $\alpha\Pi_u$, with $\Pi_uV:=(u,0,0)$. The first row then reads $i(\mathsf H_\varepsilon-i\alpha)u$, so we compute the physical resolvent at $z=i\alpha$.
\end{enumerate}
We use (b). In the language of Section~\ref{sec:2.2}: $\Ps=\Pi_u$ satisfies \ref{S} and \ref{S0}, and $\sM$ satisfies \ref{H4} with the chiral multiplier of Proposition~\ref{prop:diracS} (Remark~\ref{rem:Tsources}). Hence Theorem~\ref{thm:corr} gives the uniform bounds for $\Ae^{-1}$ and $(\Ae^*)^{-1}$, hypothesis \ref{C}, and the invertibility of $\Az$, and all results of Section~\ref{sec:2.2} apply. The strict accretivity \ref{A2} of Section~\ref{sec:2.1} fails, because $\Re\ip{\sM V}{V}=0$; it is not needed. Accordingly,
\begin{equation}\label{eq:14o}
\Ae:=\sJ+\Me+\alpha\Pi_u,\qquad\Az:=\Jz+\Mz+\alpha\Pi_u\ \text{ on }\Hmac .
\end{equation}
By Theorem~\ref{thm:corr}(iii), $\Az$ is invertible on $\Hmac$.
\end{remark}

\paragraph{The projections.} On the $y$-Fourier mode $k\in\Z^2$, the symbol of $\Jmic$ is $2\pi i\sP(k)$, where
\begin{equation}\label{eq:15o}
\sP(k):=\begin{pmatrix}\sigma\cdot k&k_1&k_2\\ k_1&0&0\\ k_2&0&0\end{pmatrix}
\end{equation}
is a Hermitian $6\times6$ matrix; each $k_j$ stands for $k_jI_2$. For $k\neq0$, $\sP(k)(u,q)=(\sigma\cdot k\,u+k\cdot q,k_1u,k_2u)$, where $k\cdot q:=k_1q_1+k_2q_2$. Hence $\ker\sP(k)=\{(0,q):k\cdot q=0\}$ and $\ker\sP(k)^\perp=\{(u,k\otimes c):u,c\in\C^2\}$, where $k\otimes c:=(k_1c,k_2c)$. View $q=(q_1,q_2)$ as a $\C^2$-valued vector field in $y$, and let $\Pi_{\mathrm{sol}}$ and $\Pi_{\mathrm{pot}}$ act on it componentwise, as in \eqref{eq:81}--\eqref{eq:84}. Then
\begin{equation}\label{eq:16o}
\begin{gathered}
\Pmac V=(\langle u\rangle_y,\langle q_1\rangle_y,\langle q_2\rangle_y),\\\Pflu V=(0,\Pi_{\mathrm{sol}}q),\qquad\Pgra V=(u-\langle u\rangle_y,\Pi_{\mathrm{pot}}q).
\end{gathered}
\end{equation}
So $\Hflu$ consists of mean-zero, $y$-divergence-free fluxes $q$, exactly as in Chapter~\ref{sec:ell}. The difference is that the scalar component $u$ is now $\C^2$-valued.

\paragraph{Verification of the (H) hypotheses.} \ref{H1} follows from \eqref{eq:16o}. \ref{H2}(i) follows from the description of $\ker\sP(k)$. \ref{H2}(ii) holds because $\Pmac$ is the $y$-average. \ref{H2}(iii): for $w=(0,s)\in\Hflu$, $\Jmac w=(\partial_{x_1}s_1+\partial_{x_2}s_2,0,0)$ has a mean-zero $u$-component and vanishing $q$-components, so it lies in $\Hgra$. \ref{H2}(iv): $\Pflu$ acts only in $y$, so it preserves $\dom\Jmac$. \ref{H3} is verified as in Chapter~\ref{sec:ell}, since $\nabla_x(\Se V)=\Se(\nabla_x+\varepsilon^{-1}\nabla_y)V$. \ref{H4} is Lemma~\ref{lem:structM} together with Remark~\ref{rem:shift}.

\paragraph{The Abstract Corrector $\Kmic$.} Let $U=(u_0,q_{0,1},q_{0,2})\in\Hmac$. The $q_k$-component of $\sM U$ is
$\beta^{-1}\sigma_3\Psi_ku_0+i\beta^{-1}\sigma_3q_{0,k}=\partial_{y_k}\big(\beta^{-1}\rho\sigma_3\varphi u_0\big)+i\beta^{-1}\sigma_3q_{0,k}$.
This is a $y$-gradient plus a $y$-constant, so $\Pi_{\mathrm{sol}}$ annihilates it, and
\begin{equation}\label{eq:17o}
\Pflu\sM\Pmac=0 .
\end{equation}
Moreover, for $w=(0,s)\in\Hflu$ the $q$-component of $\sM w$ is $i\beta^{-1}\sigma_3s$. Since $\sigma_3$ acts on the inner $\C^2$-index, it commutes with $\Pi_{\mathrm{sol}}$, and
\begin{equation}\label{eq:18o}
\Pflu\sM\Pflu=i\beta^{-1}\sigma_3\Pflu,\qquad\big(\Pflu\sM\Pflu|_{\Hflu}\big)^{-1}=-i\beta\sigma_3 .
\end{equation}
Since $\Pi_u\Pflu=0$ (\ref{S0}), $\Mff=\Pflu\sM|_{\Hflu}$, and \eqref{eq:18o} makes hypothesis \ref{C} explicit (Theorem~\ref{thm:corr}(ii) gives it abstractly). By \eqref{eq:5.3} and \eqref{eq:17o},
\begin{equation}\label{eq:19o}
\Kmic=0,\qquad\Ke=0 .
\end{equation}
The singular potential needs no corrector in the lifted space, because the substitution \eqref{eq:171} has already integrated it once.

\paragraph{The Homogenized Coefficient $\Mz$.} By \eqref{eq:19o}, $\Mz=\Pmac\sM\Pmac$ is the $y$-average of $\sM$. Since $\langle\Psi_k\rangle_y=\rho\langle\Phi_k\rangle_y=0$, $\langle\widetilde U\rangle_y=0$ and $\langle\Psi_k\sigma_3\Psi_k\rangle_y=\rho^2\langle\Phi_k\sigma_3\Phi_k\rangle_y$, we obtain, on $\Hmac$,
\begin{equation}\label{eq:20o}
\Mz=\begin{pmatrix}i\big(m\sigma_3+\beta^{-1}\rho(x)^2\Theta\big)&0&0\\0&i\beta^{-1}\sigma_3&0\\0&0&i\beta^{-1}\sigma_3\end{pmatrix}.
\end{equation}
Let $F=(if,0,0)$ and solve $\Az U=F$. The $q$-rows give $\partial_ku_0+i\beta^{-1}\sigma_3q_{0,k}=0$, so $q_{0,k}=i\beta\sigma_3\partial_ku_0$. The first row then reverses \eqref{eq:169}--\eqref{eq:172} and gives $(\mathsf H_0-i\alpha)u_0=f$, where
\begin{equation}\label{eq:21o}
\mathsf H_0:=D\cdot\sigma+(m+\beta\Delta)\sigma_3+\beta^{-1}\rho(x)^2\Theta .
\end{equation}
The effective potential $\beta^{-1}\rho^2\Theta$ is of order one, Hermitian, and non-stationary through $\rho^2$. Its $\sigma_3$-component shifts the mass $m$; this is the mechanism behind the topological transition in \cite{BalDang2024}. Expanding $\Theta$ in Fourier series recovers the second-order perturbative formula $\beta^{-1}\rho^2\Theta=-\rho^2\langle W(\beta\sigma_3\Delta_y)^{-1}W\rangle_y$.

\begin{lemma}[Uniform Invertibility]\label{lem:unif}
Let $\alpha>0$. There is $C_A$, depending only on $\alpha,\beta,m,\nrm\rho_{W^{1,\infty}}$ and $\nrm{\nabla_y\varphi}_{L^\infty}$ (and not on $\varepsilon$), such that $\Ae$ and $\Az|_{\Hmac}$ are boundedly invertible with $\nrm{\Ae^{-1}}_{\Hs\to\Hs}\le C_A$ and $\nrm{(\Az|_{\Hmac})^{-1}}_{\Hmac\to\Hmac}\le C_A$. Moreover, for $f\in L^2(\R^2;\C^2)$, $\Az^{-1}(if,0,0)=(u_0,i\beta\sigma_3\nabla u_0)$ with $(\mathsf H_0-i\alpha)u_0=f$ and $\nrm{u_0}_{H^2}\le C\nrm f$.
\end{lemma}

\begin{proof}
1. $\sJ$ acts only in $x$, and $\Me$ is multiplication by $\sM(x,y+x/\varepsilon)$. Hence $\Ae=\int^\oplus_{\T^2}\sA^y_\varepsilon\,dy$, where $\sA^y_\varepsilon:=\sJ+\sM(\cdot,y+\cdot/\varepsilon)+\alpha\Pi_u$ acts on $L^2(\R^2;\C^6)$. This is the system \eqref{eq:172}, shifted by $\alpha\Pi_u$, with $W$ replaced by $W(\cdot+y)$, that is, with $\Psi^y_k(x):=\rho(x)\Phi_k(x/\varepsilon+y)$ and $\widetilde U^y$ defined accordingly. It suffices to bound the fibers uniformly in $y$ and $\varepsilon$. We fix $y$ and drop it from the notation. Set $\nrm\Psi_\infty:=\operatorname{ess\,sup}(\sum_k|\Psi_k|^2)^{1/2}$.

2. Let $G=(g,h_1,h_2)\in L^2$, and let $V=(u,q)\in\dom\sJ$ satisfy $\sA V=G$. The $q$-rows read $\partial_ku+\beta^{-1}\sigma_3\Psi_ku+i\beta^{-1}\sigma_3q_k=h_k$. Hence $u\in H^1$ and
\begin{equation}\label{eq:22o}
q_k=i\beta\sigma_3(\partial_ku-h_k)+i\Psi_ku .
\end{equation}
Substituting this into the first row reverses \eqref{eq:169}--\eqref{eq:172}. In $H^{-1}$ we obtain
\begin{equation}\label{eq:23o}
\begin{gathered}(\mathsf H^y_\varepsilon-i\alpha)u=\ell:=-ig+\beta\sigma_3\nabla\cdot h-\sum_k\Psi_kh_k,\\\nrm\ell_{H^{-1}}\le\nrm g+(|\beta|+\nrm\Psi_\infty)\nrm h .\end{gathered}
\end{equation}
Here $\mathsf H^y_\varepsilon=D\cdot\sigma+(m+\beta\Delta)\sigma_3+\varepsilon^{-1}\rho W(\cdot/\varepsilon+y)$ acts from $H^1$ to $H^{-1}$ through the form
\[
\begin{aligned}h[u,w]:=&\ip{\sigma\cdot Du}{w}+m\ip{\sigma_3u}{w}-\beta\ip{\sigma_3\nabla u}{\nabla w}\\&-\sum_k\big(\ip{\Psi_ku}{\partial_kw}+\ip{\Psi_k\partial_ku}{w}\big)-\ip{\widetilde Uu}{w}.\end{aligned}
\]
Here we use $\varepsilon^{-1}\rho W(\cdot/\varepsilon+y)=\sum_k\partial_k\Psi_k-\widetilde U$, as in \eqref{eq:171}.

3. Let $X:=\nrm u_{H^1}$, $Y:=\nrm u$ and $L:=\nrm\ell_{H^{-1}}$. Each term of $h[u,u]$ is real; for the potential term, $-2\Re\ip{\Psi_ku}{\partial_ku}$ is real because $\Psi_k$ is Hermitian. Taking imaginary parts in $h[u,u]-i\alpha Y^2=\ip{\ell}{u}$ gives $\alpha Y^2\le LX$. Next we test with $\sigma_3u$, which satisfies $\nrm{\sigma_3u}_{H^1}=X$. We have $\sigma_3\sigma\cdot D=\sigma_2\partial_1-\sigma_1\partial_2$, which is skew-adjoint, so $\Re\ip{\sigma\cdot Du}{\sigma_3u}=0$. Also $\ip{u}{\sigma_3u}\in\R$, and $-\beta\ip{\sigma_3\nabla u}{\nabla(\sigma_3u)}=-\beta\nrm{\nabla u}^2$. Taking real parts,
$|\beta|\nrm{\nabla u}^2\le LX+c_0YX$, with $c_0:=|m|+\nrm{\widetilde U}_\infty+2\nrm\Psi_\infty$.
Hence $X^2\le c_1LX+c_2YX$ with $c_1:=\alpha^{-1}+|\beta|^{-1}$ and $c_2:=c_0/|\beta|$. Using $c_2Y\le c_2(LX/\alpha)^{1/2}\le X/2+c_2^2L/(2\alpha)$, we conclude
\begin{equation}\label{eq:24o}
\nrm u_{H^1}\le C_1\nrm\ell_{H^{-1}},\qquad C_1:=2c_1+c_2^2/\alpha .
\end{equation}
4. $h-i\alpha$ is bounded on $H^1\times H^1$, and by \eqref{eq:24o} the associated operator $H^1\to H^{-1}$ is injective with closed range. Its adjoint is associated with $\mathsf H^y_\varepsilon+i\alpha$, which obeys the same estimate. So the range is dense, and the operator is bijective. Given $G$, define $u$ by \eqref{eq:23o} and $q$ by \eqref{eq:22o}. Then $V=(u,q)$ solves the system in distributions, and its first row shows $\sigma\cdot\nabla u+\nabla\cdot q\in L^2$. Thus $V\in\dom\sJ$ and $\sA V=G$. By \eqref{eq:22o}, $\nrm V\le\nrm u+|\beta|(\nrm{\nabla u}+\nrm h)+\nrm\Psi_\infty\nrm u\le C_A\nrm G$.

5. $\Az|_{\Hmac}$ has the same structure with $\Psi_k=\widetilde U=0$, and with $m\sigma_3$ replaced by the bounded Hermitian matrix $m\sigma_3+\beta^{-1}\rho^2\Theta$. Steps 2--4 apply verbatim. For $G=(if,0,0)$, \eqref{eq:22o} gives $q_0=i\beta\sigma_3\nabla u_0$, and $\nrm{u_0}_{H^1}\le C\nrm f$. Finally, $\beta\sigma_3\Delta u_0=f+i\alpha u_0-\sigma\cdot Du_0-(m\sigma_3+\beta^{-1}\rho^2\Theta)u_0\in L^2$, so $\nrm{u_0}_{H^2}\le C\nrm f$.
\end{proof}

\begin{proposition}[The Dirac model satisfies (S), (S0) and (H4)]\label{prop:diracS}
Let $\Ps:=\Pi_u$, $\Sigma:=\operatorname{diag}(\sigma_3,-\sigma_3,-\sigma_3)$, $\sT:=-i\operatorname{sgn}(\beta)\Sigma$ and $\sT':=-\sT$. Then \ref{S}, \ref{S0} and \ref{H4} hold, with $c=(2|\beta|)^{-1}$ and $C=2|\beta|^{-1}\nrm\Psi_\infty^2+\nrm{\tilde m}_\infty$, where $\tilde m:=m\sigma_3-\widetilde U+\beta^{-1}\sum_k\Psi_k\sigma_3\Psi_k$. Consequently, Theorem~\ref{prop:inv} gives, for every $\alpha>0$ and uniformly in $\varepsilon$:
\begin{itemize}
\item the bounds $\nrm{\Ae^{-1}},\nrm{(\Ae^*)^{-1}}\le C_A$;
\item hypothesis \ref{C};
\item the invertibility of $\Az$ on $\Hmac$, with constants depending only on $\alpha$, $\beta$, $m$, $\nrm\rho_{W^{1,\infty}}$ and $\nrm{\nabla_y\varphi}_{L^\infty}$.
\end{itemize}
\end{proposition}

\begin{proof}
\emph{\ref{S}, \ref{S0}.} $\Pi_u$ is a constant matrix, so it commutes with $\Se$ and $\Pmac$; it commutes with $\Pflu$, and $\Pi_u\Pflu=0$ by \eqref{eq:16o}.

\emph{\ref{H4}(i)--(ii).} $\Re\ip{\sM V}{V}=0$ by Lemma~\ref{lem:structM}. $\Sigma$ is a constant self-adjoint unitary matrix. It commutes with $\Se$ and $\Pmac$, and with $\Pflu$, since it is block diagonal in $(u,q_1,q_2)$ and acts on the inner $\C^2$-index, which $\Pi_{\mathrm{sol}}$ does not touch. It anticommutes with $\Jmac$ and $\Jmic$ of \eqref{eq:12o}: $\sigma_3(\sigma\cdot\nabla)\sigma_3=-\sigma\cdot\nabla$ because $\sigma_3$ anticommutes with $\sigma_1,\sigma_2$, and $\sigma_3\partial_k=-(-\sigma_3)\partial_k$ in the off-diagonal blocks. Remark~\ref{rem:Tsources} applies to $\sT$ and to $\sT'$.

\emph{\ref{H4}(iii).} For $V=(u,q)$, a direct computation from \eqref{eq:174} gives
\[
\ip{\sM V}{\Sigma V}=i\ip{\tilde mu}{\sigma_3u}-\beta^{-1}\sum_k\big(\ip{\Psi_k\sigma_3q_k}{\sigma_3u}+\ip{\Psi_ku}{q_k}\big)-i\beta^{-1}\nrm q^2 .
\]
Since $\Re\ip{\sM V}{\sT V}=-\operatorname{sgn}(\beta)\Im\ip{\sM V}{\Sigma V}$, and since the middle sum is bounded by $2\nrm\Psi_\infty\nrm u\nrm q$,
\[
\begin{aligned}\Re\ip{\sM V}{\sT V}&\ge|\beta|^{-1}\nrm q^2-2|\beta|^{-1}\nrm\Psi_\infty\nrm u\nrm q-\nrm{\tilde m}_\infty\nrm u^2\\&\ge\frac1{2|\beta|}\nrm q^2-\Big(\frac{2\nrm\Psi_\infty^2}{|\beta|}+\nrm{\tilde m}_\infty\Big)\nrm u^2 .\end{aligned}
\]
Here $(I-\Ps)V=(0,q)$ and $\Ps V=(u,0,0)$. Because $\sM^*=-\sM$, $\Re\ip{\sM^*V}{\sT'V}=\Re\ip{\sM V}{\sT V}$, which gives the second inequality.
\end{proof}

\begin{remark}\label{rem:diracS}
Proposition~\ref{prop:diracS} gives the uniform bound of Lemma~\ref{lem:unif} for $\Ae$ and $\Az$ by an abstract argument. The $u$-block of $\Sigma$ is the multiplier $\sigma_3$ of the G\aa rding argument, and the two inequalities of Lemma~\ref{lem:pcoer} play the role of the two tests ``with $u$'' and ``with $\sigma_3u$'' in step~3 of that proof; the two arguments are analogous but not identical, since Lemma~\ref{lem:unif} works with the eliminated scalar equation. What remains model-specific in Lemma~\ref{lem:unif} is the identification of the fibres with $(\mathsf H^y_\varepsilon-i\alpha)^{-1}$ (step~2) and the bound $\nrm{u_0}_{H^2}\le C\nrm f$ (step~5), which are used in Section~\ref{sec:4.2} and in Corollary~\ref{cor:physD}. The fibrewise $H^{-1}\to H^1$ bound \eqref{eq:24o} from step~3 is also used, in step~2 of the proof of Corollary~\ref{cor:physD}; it is not provided by Theorem~\ref{thm:corr}, which is why Lemma~\ref{lem:unif} is kept in full. The random version (Section~\ref{sec:random}) is covered by the same proof: $\Sigma$ is a constant matrix, and the constants depend on the coefficients only through $L^\infty$ norms.
\end{remark}

\section{Quantitative Results}\label{sec:4.2}

We summarize the calculation so far:
\begin{itemize}
\item $\Hs=L^2(\R^2\times\T^2;\C^6)$, $V=(u,q_1,q_2)$;
\item $\sJ=\Jmac$ and $\Jmic$ are given by \eqref{eq:12o};
\item $\sM$ is given by \eqref{eq:174} and \eqref{eq:13o}, with $\sM^*=-\sM$; the shift is $\alpha\Pi_u$, $\alpha>0$;
\item $\Kmic=0$, $\Mz$ is given by \eqref{eq:20o}, and $\mathsf H_0$ by \eqref{eq:21o};
\item $\Hsrc:=\{(if,0,0):f\in L^2(\R^2;\C^2)\}$.
\end{itemize}
We also use the mean-zero periodic fields
\begin{equation}\label{eq:25o}
\zeta:=\Delta_y^{-1}\varphi,\qquad\omega:=\Delta_y^{-1}(\Gamma-\Theta),\qquad Z_{kl}:=\partial_{y_k}\partial_{y_l}\zeta,\qquad\omega_l:=\partial_{y_l}\omega .
\end{equation}
These are bounded, with norms depending only on $W$ (by $L^p$ regularity: $\varphi\in W^{2,p}$ gives $\zeta\in W^{4,p}$ and $Z_{kl}\in W^{2,p}\subset L^\infty$; $\Gamma\in W^{1,p}\cap L^\infty$ gives $\omega\in W^{3,p}$ and $\omega_l\in W^{2,p}\subset L^\infty$). Only their $L^2(\T^2)$ norms enter the tensor bound below. Note that $\Delta_y^{-1}\Phi_k=\partial_{y_k}\zeta$. Throughout we use the tensor bound $\nrm{g(y)h(x)}_{L^2(\R^2\times\T^2)}\le\nrm g_{L^2(\T^2)}\nrm h_{L^2(\R^2)}$ for matrix-valued $g$.

\begin{proposition}[Spectral gap]\label{prop:gapD}
\ref{G1} holds with $\lambda_0=\pi(\sqrt5-1)$.
\end{proposition}

\begin{proof}
Let $k\neq0$ and $(u,k\otimes c)\in\ker\sP(k)^\perp$. Put $\tilde c:=|k|c$, so that $|(u,k\otimes c)|^2=|u|^2+|\tilde c|^2$. Then $\sP(k)(u,k\otimes c)=(\sigma\cdot k\,u+|k|\tilde c,k\otimes u)$. In the coordinates $(u,\tilde c)$, $\sP(k)$ acts on $\ker\sP(k)^\perp$ as $|k|\begin{pmatrix}\sigma\cdot\hat k&1\\1&0\end{pmatrix}$, $\hat k:=k/|k|$. The matrix $\sigma\cdot\hat k$ has eigenvalues $s=\pm1$. On each eigenspace the eigenvalues are $|k|(s\pm\sqrt5)/2$, whose smallest modulus is $|k|(\sqrt5-1)/2$. Since $|k|\ge1$, we get $\nrm{\Jmic w}\ge2\pi\cdot\frac{\sqrt5-1}{2}\nrm w$ on $\Hgra$.
\end{proof}

\begin{proposition}[Residual]\label{prop:resD}
Let $F=(if,0,0)\in\Hsrc$ and $U=\Az^{-1}F=(u_0,q_0)$, with $q_{0,k}=i\beta\sigma_3\partial_ku_0$. Then $\Ns U=(\Mz-\sM)U\in\Hgra$, with
\begin{multline}\label{eq:26o}
\Ns U=\Big(i\sum_k\Phi_k(y)\partial_{x_k}(\rho u_0)-i\beta^{-1}\rho^2\big(\Gamma(y)-\Theta\big)u_0,\\-\beta^{-1}\rho\sigma_3\Phi_1(y)u_0,\ -\beta^{-1}\rho\sigma_3\Phi_2(y)u_0\Big),
\end{multline}
and $\nrm{\Ns U}\le C\nrm f$.
\end{proposition}

\begin{proof}
By \eqref{eq:31.2} and \eqref{eq:19o}, $\Ns U=(\Mz-\sM)U$; the shift $\alpha\Pi_u$ cancels. In the $u$-row, $\sM U-\Mz U$ equals $-i\widetilde Uu_0+i\beta^{-1}\rho^2(\Gamma-\Theta)u_0-\beta^{-1}\rho\sum_k\Phi_k\sigma_3q_{0,k}$. Inserting $q_{0,k}=i\beta\sigma_3\partial_ku_0$ gives $\beta^{-1}\rho\Phi_k\sigma_3q_{0,k}=i\rho\Phi_k\partial_ku_0$, and $i\widetilde Uu_0+i\rho\sum_k\Phi_k\partial_ku_0=i\sum_k\Phi_k\partial_k(\rho u_0)$. In the $q_k$-rows, $\sM U-\Mz U=\beta^{-1}\rho\sigma_3\Phi_ku_0$. That $\Ns U\in\Hgra$ follows from Proposition~\ref{pro:resi-fiel}, and can also be read off directly: the $u$-component has mean zero, and the $q$-component is the $y$-gradient of $-\beta^{-1}\rho\sigma_3\varphi u_0$. The bound uses the tensor bound and Lemma~\ref{lem:unif}.
\end{proof}

\begin{proposition}[Flux corrector: the default gauge is good]\label{prop:fluxD}
In the setting of Proposition~\ref{prop:resD}, define $\Phi_F:=(v,p_1,p_2)$ by
\begin{align}
v&:=-\beta^{-1}\rho(x)\sigma_3\varphi(y)u_0(x),\label{eq:27o}\\
p_l&:=i\sum_kZ_{kl}(y)\partial_{x_k}(\rho u_0)-i\beta^{-1}\omega_l(y)\rho^2u_0+\beta^{-1}\sum_k\sigma_k\sigma_3Z_{kl}(y)\rho u_0 .\label{eq:28o}
\end{align}
Then $\Phi_F\in\Hgra\cap\mathcal D$ and $\Jmic\Phi_F=\Ns U$, so $\Phi_F$ is the flux corrector of Definition~\ref{def:flux} with $\eta=0$. Moreover,
\begin{equation}\label{eq:29o}
\begin{aligned}
(\Jmac\Phi_F)_u&=-\frac1\beta\sum_j\sigma_j\sigma_3\varphi\,\partial_j(\rho u_0)+i\sum_{k,l}Z_{kl}\partial_k\partial_l(\rho u_0)\\&\quad-\frac i\beta\sum_l\omega_l\partial_l(\rho^2u_0)+\frac1\beta\sum_{k,l}\sigma_k\sigma_3Z_{kl}\partial_l(\rho u_0),\\
(\Jmac\Phi_F)_{q_j}&=-\frac1\beta\sigma_3\varphi\,\partial_j(\rho u_0),\qquad j=1,2,
\end{aligned}
\end{equation}
and $\nrm{\Phi_F}\le C\nrm{u_0}_{H^1}\le C\nrm f$, $\nrm{\Jmac\Phi_F}\le C\nrm\rho_{W^{2,\infty}}\nrm{u_0}_{H^2}\le C\nrm f$. Consequently, \ref{GF} holds with $\tilde c=0$ and $\tilde a,\tilde b\le C$, already for the default gauge $\Phi_F\in\Hgra$.
\end{proposition}

\begin{proof}[Proof of Proposition~\ref{prop:fluxD}]
1. We show that $\Phi_F\in\Hgra$. $v$ has mean zero because $\langle\varphi\rangle_y=0$. By \eqref{eq:25o}, $p_l=\partial_{y_l}(\Delta_y^{-1}s)$, where
\begin{equation}\label{eq:30o}
s:=i\sum_k\Phi_k\partial_k(\rho u_0)-i\beta^{-1}\rho^2(\Gamma-\Theta)u_0+\beta^{-1}\rho\sum_k\sigma_k\sigma_3\Phi_ku_0
\end{equation}
has mean zero. So $p$ is a mean-zero $y$-gradient.

2. We prove $\Jmic\Phi_F=\Ns U$. Recall $\Jmic(v,p)=(\sigma\cdot\nabla_yv+\nabla_y\cdot p,\partial_{y_1}v,\partial_{y_2}v)$. First, $\partial_{y_k}v=-\beta^{-1}\rho\sigma_3\Phi_ku_0$, which is the $q_k$-component of \eqref{eq:26o}. Second, $\sigma\cdot\nabla_yv=-\beta^{-1}\rho\sum_k\sigma_k\sigma_3\Phi_ku_0$ and $\nabla_y\cdot p=s$, so the first component equals the $u$-component of \eqref{eq:26o}.

3. We have $\Jmac(v,p)=(\sigma\cdot\nabla_xv+\nabla_x\cdot p,\partial_{x_1}v,\partial_{x_2}v)$. Differentiating \eqref{eq:27o}--\eqref{eq:28o} in $x$ gives \eqref{eq:29o}. At most two $x$-derivatives fall on $\rho u_0$, and $\rho u_0\in H^2$ because $\rho\in W^{2,\infty}$ and $u_0\in H^2$ (Lemma~\ref{lem:unif}). Hence $\Phi_F\in\dom\Jmac$, and the bounds follow from the tensor bound.
\end{proof}

\begin{theorem}[Norm resolvent estimate, lifted]\label{thm:liftD}
Let $\alpha>0$. There is a constant $C$, depending only on $\alpha,\beta,m,\nrm\rho_{W^{2,\infty}}$ and $W$, such that for all $\varepsilon>0$ and all $f\in L^2(\R^2;\C^2)$,
\begin{equation}\label{eq:31o}
\Big\|\big(\Ae^{-1}-\Az^{-1}\big)\begin{pmatrix}if\\0\\0\end{pmatrix}\Big\|_{L^2(\R^2\times\T^2)}\le C\varepsilon\nrm f_{L^2(\R^2)} .
\end{equation}
\end{theorem}

\begin{proof}
Since $\Ke=0$, $\Ee=\Ae^{-1}-\Az^{-1}$ on $\Hsrc$. Lemma~\ref{lem:id} applies with $\Ps=\Pi_u$ and $C_A$ from Theorem~\ref{thm:corr} and Proposition~\ref{prop:diracS} (or from Lemma~\ref{lem:unif}). With $\Phi=\Phi_F$ (Propositions~\ref{prop:gapD} and \ref{prop:fluxD}), whose defect vanishes,
$\nrm{\Ee F}\le\varepsilon\big(1+C_A(\Mb+\alpha)\big)\nrm{\Phi_F}+\varepsilon C_A\nrm{\Jmac\Phi_F}\le C\varepsilon\nrm f$.
\end{proof}

\begin{remark}[Sharpness in the Dirac model]\label{rem:lowerD}
Theorem~\ref{thm:lower} applies here in the form of Remark~\ref{rem:lower}, with the shift $\alpha\Pi_u$. The adjoint $\Ae^*$ is associated with $\mathsf H^y_\varepsilon+i\alpha$ fibrewise, and satisfies the bound of Lemma~\ref{lem:unif} by the same proof. The element $\Xi_F=-\Jmic^{-1}\Phi_F$ is again a finite sum of bounded periodic fields times $\rho u_0$, $\partial_k(\rho u_0)$ and $\rho^2u_0$, so $\Jmac\Xi_F$ involves at most second derivatives of $\rho u_0$, and $\nrm{\Xi_F}+\nrm{\Jmac\Xi_F}\le C\nrm f$. Hence
\[
\nrm{(\Ae^{-1}-\Az^{-1})F}\ge\varepsilon\nrm{\Phi_F}-C\varepsilon^2\nrm f^2/\nrm{\Phi_F}\qquad(\Phi_F\neq0).
\]
Since the $u$-component of $\Phi_F$ is $-\beta^{-1}\rho\sigma_3\varphi u_0$, $\Phi_F\neq0$ as soon as $W\not\equiv0$ and $\rho u_0\not\equiv0$. Thus the lifted rate of Theorem~\ref{thm:liftD} is sharp.
\end{remark}

\begin{remark}[Comparison with Chapter~\ref{sec:ell}]\label{rem:comp3}
\begin{itemize}
\item In Chapter~\ref{sec:ell} the residual contains $\nabla p_0=-\nabla(a_{\hom}\nabla u_0)$. So the default gauge has $\Phi_F\sim\nabla^2u_0$ and $\Jmac\Phi_F\sim\nabla^3u_0$, and the skew gauge is needed to reach $O(\varepsilon)$. Here $\Kmic=0$, and \eqref{eq:26o} contains only $u_0$ and $q_0\sim\nabla u_0$. Hence $\Jmac\Phi_F\sim\nabla^2u_0\in L^2$, and the default gauge already satisfies \ref{GF}, as anticipated in Remark~\ref{rem:gaugestructure}.
\item The $u$-component of $\varepsilon\Se\Phi_F$ is $-\varepsilon\beta^{-1}\rho\sigma_3\varphi(y+x/\varepsilon)u_0$. This is the classical first-order corrector $\varepsilon u_1$ of the two-scale expansion $u_\varepsilon\approx u_0+\varepsilon u_1(x,x/\varepsilon)$, where $\beta\sigma_3\Delta_yu_1=-\rho Wu_0$. Although $\Ke=0$, its gradient is visible at order one in the $q$-components; see \eqref{eq:33o}.
\item The medium is non-stationary in $x$ through $\rho$. No hypothesis on $\Nb$ is needed: $\rho$ enters only through the bound $\nrm\rho_{W^{2,\infty}}$ in Proposition~\ref{prop:fluxD}.
\end{itemize}
\end{remark}

\section{Back to the Physical Operators}\label{sec:4.3}

\begin{corollary}[Physical operator estimates]\label{cor:physD}
Let $\mathsf H_\varepsilon$ be as in \eqref{eq:1oo} and $\mathsf H_0$ as in \eqref{eq:21o}; both are self-adjoint on $H^2(\R^2;\C^2)$ (for fixed $\varepsilon$, the multiplier $V:=\varepsilon^{-1}\rho W(\cdot/\varepsilon)$ is uniformly locally $L^2$, indeed $L^p$; covering $\R^2$ by unit cubes $Q$ and using $H^s(Q)\subset L^\infty(Q)$ for a fixed $s\in(1,2)$, one gets $\nrm{Vu}^2\le\sum_Q\nrm V^2_{L^2(Q)}\nrm u^2_{L^\infty(Q)}\le C\sup_Q\nrm V^2_{L^2(Q)}\nrm u^2_{H^s(\R^2)}$, and $\nrm u_{H^s}\le\delta\nrm u_{H^2}+C_\delta\nrm u$ for every $\delta>0$; so $V$ is infinitesimally bounded relative to $\beta\Delta\sigma_3$, the term $D\cdot\sigma$ is of first order, and the Kato--Rellich theorem applies). Let $\alpha>0$. There is $C$ such that for $0<\varepsilon\le1$ and $f\in L^2(\R^2;\C^2)$, with $u_\varepsilon:=(\mathsf H_\varepsilon-i\alpha)^{-1}f$ and $u_0:=(\mathsf H_0-i\alpha)^{-1}f$,
\begin{align}
\nrm{u_\varepsilon-u_0}_{L^2}&\le C\varepsilon\nrm f_{L^2},\label{eq:32o}\\
\nrm{\partial_ku_\varepsilon-\partial_ku_0+\beta^{-1}\rho\sigma_3\Phi_k(\cdot/\varepsilon)u_0}_{L^2}&\le C\varepsilon\nrm f_{L^2},\qquad k=1,2.\label{eq:33o}
\end{align}
In particular $\nrm{u_\varepsilon-u_0-\varepsilon u^\varepsilon_1}_{H^1}\le C\varepsilon\nrm f$ with $u^\varepsilon_1:=-\beta^{-1}\rho\sigma_3\varphi(\cdot/\varepsilon)u_0$. Moreover, for every $z\in\C\setminus\R$,
\begin{equation}\label{eq:34o}
\nrm{(\mathsf H_\varepsilon-z)^{-1}-(\mathsf H_0-z)^{-1}}_{L^2\to L^2}\le C_z\varepsilon,
\end{equation}
where $C_z$ is bounded on compact subsets of $\C\setminus\R$.
\end{corollary}

\begin{proof}
1. For $y\in Y:=[-\frac12,\frac12)^2$, let $u^y_\varepsilon:=(\mathsf H^y_\varepsilon-i\alpha)^{-1}f$. By step 2 of the proof of Lemma~\ref{lem:unif}, the fiber of $\Ae^{-1}(if,0,0)$ at $y$ is $(u^y_\varepsilon,i\beta\sigma_3\partial_ku^y_\varepsilon+i\Psi^y_ku^y_\varepsilon)$. Comparing with $\Az^{-1}(if,0,0)=(u_0,i\beta\sigma_3\partial_ku_0)$, Theorem~\ref{thm:liftD} gives
\begin{equation}\label{eq:35o}
\int_Y\Big(\nrm{u^y_\varepsilon-u_0}^2+\sum_k\nrm{\partial_ku^y_\varepsilon-\partial_ku_0+\beta^{-1}\sigma_3\Psi^y_ku^y_\varepsilon}^2\Big)dy\le C\varepsilon^2\nrm f^2 .
\end{equation}
2. Let $\tau_hu:=u(\cdot+h)$. Then $\tau_{\varepsilon y}\mathsf H_\varepsilon\tau_{-\varepsilon y}=D\cdot\sigma+(m+\beta\Delta)\sigma_3+\varepsilon^{-1}\rho(\cdot+\varepsilon y)W(\cdot/\varepsilon+y)$. So $\mathsf H^y_\varepsilon-\tau_{\varepsilon y}\mathsf H_\varepsilon\tau_{-\varepsilon y}=\varepsilon^{-1}\delta W(\cdot/\varepsilon+y)$, where $\delta:=\rho-\rho(\cdot+\varepsilon y)$ satisfies $\nrm\delta_{W^{1,\infty}}\le\varepsilon|y|\nrm\rho_{W^{2,\infty}}$. Since $\varepsilon^{-1}W(\cdot/\varepsilon+y)=\sum_k\partial_k[\Phi_k(\cdot/\varepsilon+y)]$, integration by parts gives $|\ip{\varepsilon^{-1}\delta W(\cdot/\varepsilon+y)u}{w}|\le C\varepsilon|y|\nrm u_{H^1}\nrm w_{H^1}$. The estimate \eqref{eq:24o} holds uniformly for both operators. Together with the resolvent identity and $\nrm{\tau_{\varepsilon y}f-f}_{H^{-1}}\le\varepsilon|y|\nrm f$, this yields
\begin{equation}\label{eq:36o}
\nrm{u^y_\varepsilon-\tau_{\varepsilon y}u_\varepsilon}_{H^1}\le C\varepsilon\nrm f\qquad(y\in Y).
\end{equation}
3. Since $\tau$ is unitary, $\nrm{u_\varepsilon-u_0}=\nrm{\tau_{\varepsilon y}u_\varepsilon-\tau_{\varepsilon y}u_0}\le\nrm{\tau_{\varepsilon y}u_\varepsilon-u^y_\varepsilon}+\nrm{u^y_\varepsilon-u_0}+\nrm{u_0-\tau_{\varepsilon y}u_0}$. Square, integrate over $Y$ ($|Y|=1$), and use \eqref{eq:36o}, \eqref{eq:35o} and $\nrm{\tau_hu_0-u_0}\le|h|\nrm{\nabla u_0}$ to obtain \eqref{eq:32o}.

4. In \eqref{eq:35o}, replace $u^y_\varepsilon$ by $\tau_{\varepsilon y}u_\varepsilon$ using \eqref{eq:36o}; $\Phi_k$ is bounded. Then change variables $z=x+\varepsilon y$, using $\Psi^y_k(x)=\rho(z-\varepsilon y)\Phi_k(z/\varepsilon)$:
$\int_Y\nrm{\partial_ku_\varepsilon-\tau_{-\varepsilon y}\partial_ku_0+\beta^{-1}\rho(\cdot-\varepsilon y)\sigma_3\Phi_k(\cdot/\varepsilon)u_\varepsilon}^2dy\le C\varepsilon^2\nrm f^2$.
Finally make three replacements, each at cost $O(\varepsilon)\nrm f$: $\tau_{-\varepsilon y}\partial_ku_0\to\partial_ku_0$ (since $u_0\in H^2$), $\rho(\cdot-\varepsilon y)\to\rho$, and $u_\varepsilon\to u_0$ (by \eqref{eq:32o}). After these replacements the integrand no longer depends on $y$. The $H^1$ statement follows from $\nabla(\varepsilon u^\varepsilon_1)=-\beta^{-1}\rho\sigma_3(\nabla_y\varphi)(\cdot/\varepsilon)u_0+O(\varepsilon)\nrm{u_0}_{H^1}$. Thus \eqref{eq:33o} is proved.

5. We prove \eqref{eq:34o}. Write $\sR_\bullet(z):=(\mathsf H_\bullet-z)^{-1}$ and $\mathsf T_\bullet:=I+(z-i\alpha)\sR_\bullet(z)$. Then
$\sR_\varepsilon(z)-\sR_0(z)=\mathsf T_\varepsilon\big(\sR_\varepsilon(i\alpha)-\sR_0(i\alpha)\big)\mathsf T_0$ and $\nrm{\mathsf T_\bullet}\le1+|z-i\alpha|/|\Img z|$. This identity holds for every $z\in\C\setminus\R$, which proves \eqref{eq:34o}.
\end{proof}

\begin{remark}\label{rem:410}
\begin{itemize}
\item \eqref{eq:34o} is the norm resolvent convergence used in \cite{BalDang2024}, here with no asymptotic expansion and with no Bloch/Floquet decomposition.
\item Unlike Corollary~\ref{cor:unlift}, no Steklov average is needed: since $u_0\in H^2$ for $L^2$ data, translations of $u_0$ and $\nabla u_0$ cost only $O(\varepsilon)$.
\item The only model-specific analytic input is Lemma~\ref{lem:unif}, which replaces the coercivity of \ref{A2} by the multiplier $\sigma_3$ (a G\aa rding-type argument).
\end{itemize}
\end{remark}

\section{Spectral Consequences}\label{sec:spec}

For self-adjoint operators, a norm-resolvent estimate at a single non-real point controls the spectrum locally, spectral gaps, and all bounded smooth functions of the operator with compact support. We record the consequences for the Dirac model. Throughout, $\delta_\varepsilon:=\nrm{(\mathsf H_\varepsilon-i)^{-1}-(\mathsf H_0-i)^{-1}}_{L^2\to L^2}$, so that $\delta_\varepsilon\le C\varepsilon$ by Corollary~\ref{cor:physD} with $\alpha=1$.

\begin{corollary}[Spectral consequences]\label{cor:spec}
Let $\mathsf H_\varepsilon$ and $\mathsf H_0$ be as in Corollary~\ref{cor:physD}, and $0<\varepsilon\le1$.
\begin{enumerate}[label=(\roman*)]
\item \emph{Resolvents at other points.} For every $z\in\C\setminus\R$,
\[
\nrm{(\mathsf H_\varepsilon-z)^{-1}-(\mathsf H_0-z)^{-1}}\le\Big(1+\frac{|z-i|}{|\Img z|}\Big)^2\delta_\varepsilon .
\]
\item \emph{Local Hausdorff convergence of spectra.} Let $L>0$ and assume $\delta_\varepsilon\le\frac12(1+L^2)^{-1/2}$. Then every $\lambda\in\sigma(\mathsf H_\varepsilon)\cap[-L,L]$ satisfies $\operatorname{dist}(\lambda,\sigma(\mathsf H_0))\le2(1+L^2)\delta_\varepsilon$, and every $\lambda\in\sigma(\mathsf H_0)\cap[-L,L]$ satisfies $\operatorname{dist}(\lambda,\sigma(\mathsf H_\varepsilon))\le2(1+L^2)\delta_\varepsilon$.
\item \emph{Persistence of gaps.} Let $(a,b)\cap\sigma(\mathsf H_0)=\emptyset$ with $-\infty<a<b<\infty$, put $L:=\max\{|a|,|b|\}$ and $\eta_\varepsilon:=2(1+L^2)\delta_\varepsilon$. If $\delta_\varepsilon\le\frac12(1+L^2)^{-1/2}$ and $2\eta_\varepsilon<b-a$, then $(a+\eta_\varepsilon,b-\eta_\varepsilon)\cap\sigma(\mathsf H_\varepsilon)=\emptyset$. The same holds with the roles of $\mathsf H_\varepsilon$ and $\mathsf H_0$ exchanged.
\item \emph{Smooth functional calculus.} For every $f\in C_c^\infty(\R)$ there is $C_f$ such that $\nrm{f(\mathsf H_\varepsilon)-f(\mathsf H_0)}\le C_f\,\delta_\varepsilon\le C_fC\varepsilon$.
\item \emph{Spectral projections.} Let $-\infty<a<b<\infty$ with $a,b\notin\sigma(\mathsf H_0)$. Then there is $\varepsilon_0>0$ such that $\nrm{1_{(a,b)}(\mathsf H_\varepsilon)-1_{(a,b)}(\mathsf H_0)}\le C_{a,b}\,\varepsilon$ for $0<\varepsilon\le\varepsilon_0$. In particular the two projections have the same rank for $\varepsilon\le\varepsilon_0$ with $C_{a,b}\varepsilon<1$.
\end{enumerate}
\end{corollary}

\begin{proof}
Write $\sR_\bullet(z):=(\mathsf H_\bullet-z)^{-1}$.

(i) This is the identity $\sR_\varepsilon(z)-\sR_0(z)=\mathsf T_\varepsilon(\sR_\varepsilon(i)-\sR_0(i))\mathsf T_0$ from the proof of Corollary~\ref{cor:physD}, with $\nrm{\mathsf T_\bullet}\le1+|z-i|/|\Img z|$.

(ii) The operators $N_\bullet:=\sR_\bullet(i)$ are normal, with $\sigma(N_\bullet)=\{(\mu-i)^{-1}:\mu\in\sigma(\mathsf H_\bullet)\}\cup\{0\}$ (the point $0$ being included when $\mathsf H_\bullet$ is unbounded). For a normal operator $N_0$ and any bounded $N$, every $w\in\sigma(N)$ satisfies $\operatorname{dist}(w,\sigma(N_0))\le\nrm{N-N_0}$: otherwise $N_0-w$ is invertible with $\nrm{(N_0-w)^{-1}}=\operatorname{dist}(w,\sigma(N_0))^{-1}<\nrm{N-N_0}^{-1}$, and $N-w=(N_0-w)(I+(N_0-w)^{-1}(N-N_0))$ would be invertible. Let $\lambda\in\sigma(\mathsf H_\varepsilon)\cap[-L,L]$ and $w:=(\lambda-i)^{-1}$, so $|w|\ge(1+L^2)^{-1/2}$. There is $w_0\in\sigma(N_0)$ with $|w-w_0|\le\delta_\varepsilon\le\frac12(1+L^2)^{-1/2}$. Hence $|w_0|\ge\frac12(1+L^2)^{-1/2}>0$, so $w_0=(\mu-i)^{-1}$ with $\mu\in\sigma(\mathsf H_0)$ and $|\mu-i|\le2(1+L^2)^{1/2}$. Then
$|\lambda-\mu|=|w-w_0|\,|\lambda-i|\,|\mu-i|\le\delta_\varepsilon(1+L^2)^{1/2}\cdot2(1+L^2)^{1/2}$.
The converse inclusion is symmetric, since $N_\varepsilon$ is also normal.

(iii) If $\lambda\in\sigma(\mathsf H_\varepsilon)\cap(a+\eta_\varepsilon,b-\eta_\varepsilon)$, then $|\lambda|\le L$, and by (ii) there is $\mu\in\sigma(\mathsf H_0)$ with $|\lambda-\mu|\le\eta_\varepsilon$. Then $\mu\in(a,b)$, a contradiction.

(iv) We use the Helffer--Sj\"ostrand formula \cite[Ch.~2]{Davies1995}. Let $\tilde f(x+iy):=\big(f(x)+iyf'(x)+\frac{(iy)^2}{2}f''(x)\big)\chi(y)$ with $\chi\in C_c^\infty(\R)$ equal to $1$ near $0$. Then $\tilde f$ has compact support, $|\bar\partial\tilde f(x+iy)|\le C_f|y|^2$, and $f(\mathsf H)=\frac1\pi\int_{\C}\bar\partial\tilde f(z)\,(\mathsf H-z)^{-1}\,dx\,dy$ for self-adjoint $\mathsf H$. By (i),
\[
\nrm{f(\mathsf H_\varepsilon)-f(\mathsf H_0)}\le\frac1\pi\int|\bar\partial\tilde f(z)|\Big(1+\frac{|z-i|}{|y|}\Big)^2dx\,dy\ \delta_\varepsilon\le C_f\delta_\varepsilon,
\]
because the integrand is bounded and compactly supported.

(v) Since $a,b\notin\sigma(\mathsf H_0)$, there is $g>0$ with $(a-g,a+g)\cup(b-g,b+g)$ disjoint from $\sigma(\mathsf H_0)$. By (iii), applied to these two intervals, $\sigma(\mathsf H_\varepsilon)$ avoids $(a-g/2,a+g/2)\cup(b-g/2,b+g/2)$ for $\varepsilon\le\varepsilon_0$. Choose $f\in C_c^\infty(\R)$ with $f=1$ on $[a+g/2,b-g/2]$ and $f=0$ outside $(a-g/2,b+g/2)$. Then $f(\mathsf H_\bullet)=1_{(a,b)}(\mathsf H_\bullet)$ for both operators, and (iv) applies. Two orthogonal projections at distance $<1$ have the same rank.
\end{proof}

\begin{remark}\label{rem:spec}
\begin{itemize}
\item Items (iii)--(v) are the forms in which norm-resolvent convergence typically enters stability arguments for topological invariants defined through energy filters $f(\mathsf H)$ with $f\in C_c^\infty$, or through spectral projections onto bounded spectral intervals. Here they hold with rate $O(\varepsilon)$, uniformly for $\varepsilon\le\varepsilon_0$. By (iii), a spectral gap of the effective operator $\mathsf H_0$, for instance one opened or closed by the mass shift $\beta^{-1}\rho^2\Theta$, persists for $\mathsf H_\varepsilon$, up to an $O(\varepsilon)$ shrinking of its edges.
\item Fermi projections $1_{(-\infty,E)}(\mathsf H)$ are not covered. Since $\mathsf H_\varepsilon$ and $\mathsf H_0$ are unbounded above and below, $1_{(-\infty,E)}$ has different limits at $\pm\infty$, and such functions are not norm limits of functions of a single resolvent.
\item For the random potentials of Section~\ref{sec:random}, the bound (i) holds in the annealed sense of \eqref{eq:49o}. By Minkowski's inequality in the Helffer--Sj\"ostrand integral, so does (iv): $\sup_{\nrm h\le1}\big(\E\nrm{(f(\mathsf H^\omega_\varepsilon)-f(\mathsf H_0))h}^2\big)^{1/2}\le C_f\kappa_\varepsilon$. Items (ii), (iii) and (v) require operator-norm bounds for fixed realizations, which the annealed estimates do not provide (Remark~\ref{rem:419}).
\end{itemize}
\end{remark}

\section{Random Potentials}\label{sec:random}

We now let $W$ be a stationary random field. The goal is to repeat the computation of Sections~\ref{sec:4.1}--\ref{sec:4.3} with the torus $\T^2$ replaced by a probability space. Two points change:
\begin{itemize}
\item the substitution \eqref{eq:171} needs a stationary antiderivative of $W$;
\item $\Jmic$ has no spectral gap in general.
\end{itemize}
The first point dictates the structural assumption below. The second is handled by a regularized flux corrector, whose cost is measured by the low-frequency behavior of explicit spectral measures of the coefficient fields, exactly as in Section~\ref{sec:stoch}.

\subsection{Setting and Assumptions}

\paragraph{The probability space.} We use the setting and notation of Section~\ref{sec:stoch} with $d=2$: the group action $\{\tau_x\}_{x\in\R^2}$, the generators $D=(D_1,D_2)$, $\Delta=D_1^2+D_2^2$, the spectral measures $\mu_g$, the Weyl projections, and the functionals $\fa_{\mathcal G},\fd_{\mathcal G}$ of \eqref{eq:specfun}.

\begin{hypothesis}[Random potential]\label{hyp:R}
\begin{enumerate}[label=(R\arabic*)]
\item $W(x,\omega)=\widetilde W(\tau_x\omega)$, with $\widetilde W\in L^\infty(\Omega;\C^{2\times2})$ Hermitian-valued.
\item \emph{Stationary antiderivative:} There are $\widetilde\Phi_1,\widetilde\Phi_2\in L^\infty(\Omega;\C^{2\times2})$, Hermitian-valued, with $\E\widetilde\Phi_k=0$, $D_1\widetilde\Phi_2=D_2\widetilde\Phi_1$, and $D_1\widetilde\Phi_1+D_2\widetilde\Phi_2=\widetilde W$ in the weak sense.
\item \emph{Density of states at zero frequency:} Put $\Gamma:=\sum_k\Phi_k\sigma_3\Phi_k$ and $\Theta:=\E\Gamma$. The family $\mathcal G_{\mathrm{D}}:=\{\Phi_1,\Phi_2,\Gamma-\Theta\}$ satisfies (DOS)$_\gamma$ of \eqref{eq:DOS} for some $\gamma\in(0,4)$.
\end{enumerate}
\end{hypothesis}

\begin{remark}[On the hypotheses]\label{rem:Rhyp}
\begin{itemize}
\item (R2) is the random counterpart of $\langle W\rangle=0$ together with $\Phi_k=\partial_{y_k}\varphi$. Since $\Phi$ is curl-free and mean-zero, it is the unique element of $L^2_{\mathrm{pot}}$ with $D\cdot\Phi=W$, that is, $\Phi=\nabla\Delta^{-1}W$ in spectral variables. The potential $\varphi$ need not be stationary, and in $d=2$ it typically is not.
\item (R2) is not a technicality in $d=2$. If $W$ has an integrable covariance $c_W$, then for $g\in C_c^\infty(\R^2)$,
$\operatorname{Var}\big(\int_{\R^2}\varepsilon^{-1}W(x/\varepsilon)g(x)\,dx\big)=\iint c_W(z)g(x)g(x-\varepsilon z)\,dz\,dx\to\big(\int c_W\big)\nrm g^2$.
So at the scale $\varepsilon^{-1}$ the potential averages out only if $\int c_W=0$. Otherwise it converges to a white noise, and one should not expect a deterministic homogenized limit; see \cite{Bal2010,BalGu2015}. Moreover, (R2) forces $\int|\xi|^{-2}d\mu_W=\E|\Phi|^2<\infty$, which fails whenever $\int c_W\neq0$.
\item If a bounded antiderivative $\Phi$ is available but is not curl-free, the effective operator is unchanged. Indeed, then $\Kmic U=(0,i\rho(\Pi_{\mathrm{sol}}\Phi)u_0)\neq0$, and the Schur complement \eqref{eq:5.1} gives the $(1,1)$ block $i\beta^{-1}\rho^2\E\big[\sum_k\Phi_k\sigma_3(\Pi_{\mathrm{pot}}\Phi)_k\big]$. The cross term $\E[\sum_k(\Pi_{\mathrm{sol}}\Phi)_k\sigma_3(\Pi_{\mathrm{pot}}\Phi)_k]$ vanishes entrywise, by orthogonality of $L^2_{\mathrm{pot}}$ and $L^2_{\mathrm{sol}}$ under the bilinear pairing $\E[\sum_ka_kb_k]$ (complex conjugation preserves both classes), so $\Theta$ is the same as for $\Pi_{\mathrm{pot}}\Phi$.
\item (R3) holds with $\gamma=2$ when the covariances of $\Phi$ and $\Gamma$ are integrable, for example under a finite range of dependence for $\widetilde\Phi$; then $\Gamma$, a local function of $\Phi$, inherits it. Spectral densities vanishing like $|\xi|^{2s}$ at $0$ give $\gamma=2+2s$. Covariances decaying like $|x|^{-b}$ with $0<b<2$ typically give $\gamma=b$.
\end{itemize}
\end{remark}

\paragraph{Lifted structure.} Everything in Sections~\ref{sec:4.1}--\ref{sec:4.2} carries over with $\Hs=L^2(\R^2\times\Omega;\C^6)$, $(\Se V)(x,\omega):=V(x,\tau_{x/\varepsilon}\omega)$, $\Pmac=\E$, and $\nabla_y$ replaced by $D=(D_1,D_2)$ in $\Jmic$, \eqref{eq:12o}. The material operator $\sM$ is \eqref{eq:174} with $\Psi_k=\rho(x)\Phi_k(\omega)$ and $\widetilde U=\sum_k\partial_{x_k}\rho(x)\Phi_k(\omega)$. Specifically:
\begin{itemize}
\item Since $\tau$ is measure-preserving, $\Se$ is unitary.
\item The fiber of $\Me=\Se\sM\Se^{-1}$ at $\omega$ is \eqref{eq:174} with the realization $\Phi_k(x/\varepsilon,\omega)$. So the fiber of $\Ae$ at $\omega$ is exactly the physical first-order system for the realization $\omega$; no translation step is needed.
\item Lemma~\ref{lem:structM} and Lemma~\ref{lem:unif} hold verbatim, with $\int^\oplus_\Omega$ in place of $\int^\oplus_{\T^2}$. Their constants depend on $\widetilde\Phi$ only through $\nrm{\widetilde\Phi}_{L^\infty(\Omega)}$.
\end{itemize}

\paragraph{The projections and (H1)--(H3).} Replace the Fourier series in $y$ by the joint spectral resolution of $-iD$, and the mode $k\in\Z^2$ by $\xi\in\R^2$. The symbol computation of $\sP(\xi)$ in \eqref{eq:15o} is unchanged. By ergodicity, the spectral subspace at $\xi=0$ consists of constants. Hence \eqref{eq:16o} holds with $\langle\cdot\rangle_y=\E$, and with $\Pi_{\mathrm{pot}},\Pi_{\mathrm{sol}}$ the Weyl projections applied entrywise. \ref{H1}--\ref{H3}, including \ref{H2}(iv), are verified as before; for the core property in \ref{H3} see Section~\ref{sec:stoch}.

\paragraph{No spectral gap in general.} If $\mu_{\Phi_k}$ charges every neighbourhood of $0$ (as is the case under mixing assumptions), the bound in Proposition~\ref{prop:gapD} degenerates: $\sigma(-i\Jmic)\cap(-\lambda,\lambda)\neq\{0\}$ for every $\lambda>0$, so \ref{G1} fails. Ergodicity alone only excludes an atom at $0$ (a periodic medium with a uniformly random shift is ergodic and has a gap); the argument below does not use a gap.

\paragraph{Corrector and effective operator.} By (R2), $(\Phi_1,\Phi_2)\in L^2_{\mathrm{pot}}$ entrywise. The argument for \eqref{eq:17o} then applies: the $q$-component of $\sM U$ is a potential field plus a constant. Hence
\begin{equation}\label{eq:37o}
\Kmic=0,\qquad\Mz=\E\sM\ \text{as in \eqref{eq:20o}},\qquad\Theta=\sum_{k=1}^2\E\big[\Phi_k\sigma_3\Phi_k\big].
\end{equation}
Since $\Phi=\nabla\Delta^{-1}W$ in spectral variables, $\Theta$ depends only on the law of $W$, as in the periodic formula after \eqref{eq:11o}. The effective Hamiltonian is \eqref{eq:21o}, deterministic, with this $\Theta$. The residual is \eqref{eq:26o}, with $\Phi_k(y)$, $\Gamma(y)$ replaced by $\Phi_k(\omega)$, $\Gamma(\omega)$.

\subsection{The Regularized Flux Corrector}

Without a gap we cannot solve $\Jmic\Phi=\Ns U$ in $\Hs$. Lemma~\ref{lem:id} allows an arbitrary defect (and Remark~\ref{rem:shiftgen} the shift $\alpha\Pi_u$). Inverting $\Jmic-\eta$ for the $6\times6$ system is, however, inconvenient; as in Section~\ref{sec:stoch}, we regularize with the scalar resolvent of the generators.

For $\theta>0$, define
\begin{equation}\label{eq:39o}
\varphi_\theta:=-(\theta-\Delta)^{-1}\widetilde W,\qquad e_\theta:=\Phi-D\varphi_\theta=\theta(\theta-\Delta)^{-1}\Phi .
\end{equation}
The identity for $e_\theta$ uses $\Phi\in L^2_{\mathrm{pot}}$: $D(D\cdot\Phi)=\Delta\Phi$. We write $\fa(\theta):=\fa_{\mathcal G_{\mathrm D}}(\theta)$ and $\fd(\theta):=\fd_{\mathcal G_{\mathrm D}}(\theta)$; note $\nrm{\theta(\theta-\Delta)^{-1}g}\le\fd(\theta)$ and $\nrm{D(\theta-\Delta)^{-1}g}\le\sqrt2\,\fa(\theta)$ for $g\in\mathcal G_{\mathrm D}$.

\begin{proposition}[Regularized Flux Corrector]\label{prop:regflux}
Assume (R1)--(R2). Let $F=(if,0,0)$, $U=\Az^{-1}F=(u_0,i\beta\sigma_3\nabla u_0)$, and $\theta>0$. Define $\Phi^\theta_F:=(v_\theta,p_\theta)$ by
\begin{align}
v_\theta&:=-\beta^{-1}\rho\sigma_3\varphi_\theta u_0,\label{eq:41o}\\
p_\theta&:=-D(\theta-\Delta)^{-1}s_\theta,\notag\\
s_\theta&:=i\sum_k\Phi_k\partial_k(\rho u_0)-i\beta^{-1}\rho^2(\Gamma-\Theta)u_0+\beta^{-1}\rho\sum_k\sigma_k\sigma_3(D_k\varphi_\theta)u_0 .\label{eq:42o}
\end{align}
Then $\Phi^\theta_F\in\mathcal D$, and the defect $\varrho_\theta:=\Jmic\Phi^\theta_F-\Ns U$ is
\begin{equation}\label{eq:44o}
\varrho_\theta=\Big(-\theta(\theta-\Delta)^{-1}s_\theta,\ \big(\beta^{-1}\rho\sigma_3e_{\theta,k}u_0\big)_{k=1,2}\Big).
\end{equation}
Moreover, with $C$ depending only on $\alpha,\beta,m,\nrm\rho_{W^{2,\infty}}$ and $\nrm{\widetilde\Phi}_{L^\infty}$,
\begin{equation}\label{eq:45o}
\nrm{\Phi^\theta_F}+\nrm{\Jmac\Phi^\theta_F}\le C\fa(\theta)\nrm f,\qquad\nrm{\varrho_\theta}\le C\fd(\theta)\nrm f .
\end{equation}
\end{proposition}

\begin{proof}
1. We have $\varphi_\theta\in\dom\Delta$, so $D_k\varphi_\theta\in L^2(\Omega)$. All three stationary factors in $s_\theta$ have mean zero, and $u_0\in H^2$ by Lemma~\ref{lem:unif}. Hence $v_\theta,p_\theta\in\mathcal D$.

2. $D_kv_\theta=-\beta^{-1}\rho\sigma_3(D_k\varphi_\theta)u_0=-\beta^{-1}\rho\sigma_3\Phi_ku_0+\beta^{-1}\rho\sigma_3e_{\theta,k}u_0$. The first term is the $q_k$-component of $\Ns U$, see \eqref{eq:26o}.

3. The first component of $\Jmic\Phi^\theta_F$ is $\sigma\cdot Dv_\theta+D\cdot p_\theta$. First, $\sigma\cdot Dv_\theta=-\beta^{-1}\rho\sum_k\sigma_k\sigma_3(D_k\varphi_\theta)u_0$. Second, $D\cdot p_\theta=-\Delta(\theta-\Delta)^{-1}s_\theta=s_\theta-\theta(\theta-\Delta)^{-1}s_\theta$. By the definition of $s_\theta$, their sum is the $u$-component of $\Ns U$ minus $\theta(\theta-\Delta)^{-1}s_\theta$.

4. Each term of $v_\theta$, $p_\theta$, $\Jmac\Phi^\theta_F$ and $\varrho_\theta$ has the form $\mathsf T g(\omega)h(x)$. Here $\mathsf T$ is one of the operators $D(\theta-\Delta)^{-1}$, $(\theta-\Delta)^{-1}$, $\theta(\theta-\Delta)^{-1}$ acting on $\omega$, and $h$ is one of $\rho u_0$, $\rho^2u_0$, $\partial_k(\rho u_0)$ and their $x$-derivatives, up to second order. By the tensor bound and $\nrm{u_0}_{H^2}\le C\nrm f$, it suffices to bound $\nrm{\mathsf Tg}_{L^2(\Omega)}$:
\begin{itemize}
\item $\nrm{\varphi_\theta}^2=\nrm{(\theta-\Delta)^{-1}D\cdot\Phi}^2\le\sum_k\int|\xi|^2(\theta+|\xi|^2)^{-2}d\mu_{\Phi_k}\le2\fa(\theta)^2$, by the pointwise Cauchy--Schwarz inequality in the joint spectral representation.
\item $\nrm{D(\theta-\Delta)^{-1}g}\le\sqrt2\fa(\theta)$ for $g\in\{\Phi_k,\Gamma-\Theta\}$.
\item For $g=D_k\varphi_\theta$, the symbol of $D(\theta-\Delta)^{-1}D_k(\theta-\Delta)^{-1}D_l$ is bounded by $|\xi|\cdot|\xi|^2(\theta+|\xi|^2)^{-2}\le|\xi|(\theta+|\xi|^2)^{-1}$, so the contribution is $\le2\fa(\theta)$.
\item For the defect: $\nrm{\theta(\theta-\Delta)^{-1}g}\le\fd(\theta)$ for $g\in\{\Phi_k,\Gamma-\Theta\}$, and $\nrm{e_\theta}\le\sqrt2\fd(\theta)$.
\item Since $D_k\varphi_\theta=\Phi_k-e_{\theta,k}$, we have $\theta(\theta-\Delta)^{-1}D_k\varphi_\theta=e_{\theta,k}-\theta(\theta-\Delta)^{-1}e_{\theta,k}$, whose norm is $\le2\nrm{e_\theta}$.
\end{itemize}

\end{proof}

\subsection{Quantitative results}

\begin{theorem}[Lifted Estimate for Random Potentials]\label{thm:randlift}
Assume (R1)--(R2), and let $\alpha>0$. There is $C$, depending only on $\alpha,\beta,m,\nrm\rho_{W^{2,\infty}}$ and $\nrm{\widetilde\Phi}_{L^\infty}$, such that for all $\varepsilon>0$, $\theta>0$ and $f\in L^2(\R^2;\C^2)$,
\begin{equation}\label{eq:46o}
\Big\|\big(\Ae^{-1}-\Az^{-1}\big)\begin{pmatrix}if\\0\\0\end{pmatrix}\Big\|_{L^2(\R^2\times\Omega)}\le C\kappa_\varepsilon(\theta)\nrm f,\qquad\kappa_\varepsilon(\theta):=\varepsilon\fa(\theta)+\fd(\theta).
\end{equation}
In particular, the left-hand side tends to $0$ as $\varepsilon\to0$, uniformly over $\nrm f\le1$. If moreover (R3) holds, then $\inf_{\theta>0}\kappa_\varepsilon(\theta)$ obeys the bounds \eqref{eq:rates} for $0<\varepsilon\le\varepsilon_0$.
\end{theorem}

\begin{proof}
By \eqref{eq:37o}, $\Ke=0$. Apply Lemma~\ref{lem:id} (with $\Ps=\Pi_u$) with $\Phi=\Phi^\theta_F$ and the bounds \eqref{eq:45o}. The qualitative statement and the rates follow from Lemma~\ref{lem:specint}.
\end{proof}

\begin{corollary}[Physical operators, annealed]\label{cor:randphys}
Under the assumptions of Theorem~\ref{thm:randlift}, let $u^\omega_\varepsilon:=(\mathsf H^\omega_\varepsilon-i\alpha)^{-1}f$, where $\mathsf H^\omega_\varepsilon=D\cdot\sigma+(m+\beta\Delta)\sigma_3+\varepsilon^{-1}\rho W(\cdot/\varepsilon,\omega)$, and $u_0:=(\mathsf H_0-i\alpha)^{-1}f$. Then, with $\kappa_\varepsilon:=\inf_\theta\kappa_\varepsilon(\theta)$,
\begin{equation}\label{eq:48o}
\big(\E\nrm{u^\omega_\varepsilon-u_0}^2\big)^{1/2}+\sum_{k=1}^2\big(\E\nrm{\partial_ku^\omega_\varepsilon-\partial_ku_0+\beta^{-1}\rho\sigma_3\Phi_k(\cdot/\varepsilon,\omega)u_0}^2\big)^{1/2}\le C\kappa_\varepsilon\nrm f .
\end{equation}
Moreover, for every $z\in\C\setminus\R$,
\begin{equation}\label{eq:49o}
\sup_{\nrm f\le1}\Big(\E\big\|\big((\mathsf H^\omega_\varepsilon-z)^{-1}-(\mathsf H_0-z)^{-1}\big)f\big\|^2\Big)^{1/2}\le C_z\kappa_\varepsilon .
\end{equation}
\end{corollary}

\begin{proof}
The fiber of $\Ae^{-1}(if,0,0)$ at $\omega$ is $(u^\omega_\varepsilon,i\beta\sigma_3\partial_ku^\omega_\varepsilon+i\Psi^\omega_ku^\omega_\varepsilon)$, with $\Psi^\omega_k=\rho\Phi_k(\cdot/\varepsilon,\omega)$; this follows from step 2 of the proof of Lemma~\ref{lem:unif}. The fiber of $\Az^{-1}(if,0,0)$ is $(u_0,i\beta\sigma_3\partial_ku_0)$. Theorem~\ref{thm:randlift} therefore bounds $\E\nrm{u^\omega_\varepsilon-u_0}^2$ and $\E\nrm{\partial_ku^\omega_\varepsilon-\partial_ku_0+\beta^{-1}\sigma_3\Psi^\omega_ku^\omega_\varepsilon}^2$. Replacing $\Psi^\omega_ku^\omega_\varepsilon$ by $\Psi^\omega_ku_0$ costs $\nrm\Psi_\infty\nrm{u^\omega_\varepsilon-u_0}$. This proves \eqref{eq:48o}. For \eqref{eq:49o}, use the identity $\sR_\varepsilon(z)-\sR_0(z)=\mathsf T_\varepsilon(\sR_\varepsilon(i\alpha)-\sR_0(i\alpha))\mathsf T_0$ from the proof of Corollary~\ref{cor:physD}. Here $\mathsf T_0$ is deterministic, and $\nrm{\mathsf T^\omega_\varepsilon}\le1+|z-i\alpha|/|\Img z|$ for every $\omega$. Apply \eqref{eq:48o} to the deterministic datum $\mathsf T_0f$.
\end{proof}

\begin{remark}[What is, and is not, obtained]\label{rem:419}
\begin{itemize}
\item \emph{Comparison with Section~\ref{sec:stoch}.} The mechanism is the same as for random elliptic coefficients (Theorem~\ref{thm:ellrand}): a regularized antiderivative built with $(\theta-\Delta)^{-1}$, a hypothesis on finitely many fixed coefficient fields that can be checked from their covariances, and an estimate for $L^2$ data with no square root. Here the default gauge suffices because $\Kmic=0$ and the residual involves only $u_0$ and $\nabla u_0$ (Remark~\ref{rem:comp3}); in Section~\ref{sec:stoch} the skew-symmetric gauge is needed.
\item \emph{Rates.} For short-range correlations ($\gamma=2$) we obtain $\varepsilon(\log1/\varepsilon)^{1/2}$. This has the same form as the two-dimensional rate in stochastic homogenization of divergence-form operators (Remark~\ref{rem:E2}). We do not claim optimality.
\item \emph{Annealed versus quenched.} \eqref{eq:49o} controls the quantity $\sup_f\E\nrm{(\sR^\omega_\varepsilon-\sR_0)f}^2$; it does not control $\E\nrm{\sR^\omega_\varepsilon-\sR_0}^2_{L^2\to L^2}$, and not an almost-sure operator-norm bound. The latter is what the stability argument of \cite{BalDang2024} uses for a fixed realization. For random fluctuations, the gap of the homogenized operator is typically filled by rare regions and becomes a pseudo-gap; Bal \cite{Bal2026b} then replaces the interface index by an index along cuts that avoid these regions. Localized almost-sure bounds do follow (Proposition~\ref{prop:quenched}), while full operator-norm bounds on $\R^2$ fail in general when $\rho$ does not vanish at infinity (Remark~\ref{rem:noquenched}).
\item \emph{Divergence-form structure.} The whole construction rests on (R2). For a generic short-range $W$ with $\int c_W\neq0$, Remark~\ref{rem:Rhyp} shows that the scaling $\varepsilon^{-1}W(x/\varepsilon)$ is critical in $d=2$, and the point-free framework, as set up here, does not apply. For scalar alloys in this critical regime, with the amplitude renormalized by $(\log1/\varepsilon)^{-1/2}$, the topological phase is analyzed in \cite[Theorem~2.14]{Bal2026b}, which covers the regularized Dirac and BHZ models.
\end{itemize}
\end{remark}

\subsection{A Verified Model}\label{sec:verifyR3}

We show that (R3) holds, with $\gamma=2$, when the antiderivative has a finite range of dependence, and we give an explicit model in which (R1)--(R3) hold.

\begin{proposition}[Finite range of dependence implies (R3)]\label{prop:R3}
Assume (R1)--(R2), and suppose that $\widetilde\Phi=(\widetilde\Phi_1,\widetilde\Phi_2)$ has a finite range of dependence $\ell>0$: the families $\{\Phi(x,\cdot)\}_{x\in A}$ and $\{\Phi(x,\cdot)\}_{x\in B}$ are independent whenever $A,B\subset\R^2$ satisfy $\operatorname{dist}(A,B)\ge\ell$. Then (R3) holds with $\gamma=2$, $\lambda_*=1$ and $C_D=\frac{\ell^2}{4}\max\{\nrm{\widetilde\Phi}^2_{L^\infty},\nrm{\Gamma-\Theta}^2_{L^\infty}\}$. Consequently, under the assumptions of Theorem~\ref{thm:randlift}, $\inf_\theta\kappa_\varepsilon(\theta)\le C\varepsilon(\log1/\varepsilon)^{1/2}$, and Corollary~\ref{cor:randphys} holds with this rate.
\end{proposition}

\begin{proof}
Each $g\in\mathcal G_{\mathrm D}=\{\Phi_1,\Phi_2,\Gamma-\Theta\}$ is a pointwise function of $\Phi$ at the same point, so it inherits the finite range of dependence $\ell$; it is bounded and has mean zero. Lemma~\ref{lem:covDOS} with $d=2$ gives $\mu_g(\{|\xi|<r\})\le(2\pi)^{-2}\pi\cdot\pi\ell^2\,\E|g|^2\,r^2\le\frac{\ell^2}{4}\nrm g_{L^\infty}^2r^2$. The rate follows from Lemma~\ref{lem:specint} with $\gamma=2$.
\end{proof}

\begin{remark}[An explicit model]\label{rem:matern}
Let $\mathcal P$ be a Mat\'ern hard-core point process in $\R^2$: the points of a Poisson process of intensity $1$ that have no other Poisson point within distance $r_h>0$. Attach to each $p\in\mathcal P$ an independent mark $\xi_p$, drawn from a fixed bounded distribution on Hermitian $2\times2$ matrices. Let $\eta\in C_c^\infty(B_{\rho_0};\R)$, and set
\[
\psi(x):=\sum_{p\in\mathcal P}\xi_p\,\eta(x-p),\qquad\Phi_k:=\partial_k\psi,\qquad W:=\Delta\psi=\partial_1\Phi_1+\partial_2\Phi_2 .
\]
Because the points are $r_h$-separated, at most $C(\rho_0/r_h)^2$ terms overlap at any $x$, so $\psi$ and all its derivatives are bounded. The field is stationary and ergodic under translations of the marked configuration, and it has a finite range of dependence $\ell=2(\rho_0+r_h)$, since the value at $x$ depends only on the Poisson points, and their marks, within distance $\rho_0+r_h$ of $x$. Moreover $\Phi$ is curl-free and Hermitian, $\E\Phi_k=0$ (it is the gradient of a bounded stationary field), and $W$ is bounded and Hermitian. Hence (R1), (R2) and, by Proposition~\ref{prop:R3}, (R3) with $\gamma=2$ hold. For this model Theorem~\ref{thm:randlift} and Corollaries~\ref{cor:randphys} and \ref{cor:spec}(i),(iv) (in the annealed form of Remark~\ref{rem:spec}) hold with rate $\varepsilon(\log1/\varepsilon)^{1/2}$. Note that $W$ is a divergence by construction, so $\int c_W=0$, in agreement with Remark~\ref{rem:Rhyp}.
\end{remark}

\chapter{Beyond Scalar Coercive Problems}\label{sec:further}

The variational and renormalization approaches \cite{ArmstrongKuusi2025a,ArmstrongKuusi2025b} rest on convexity and subadditive energies. The spectral-germ method \cite{BirmanSuslina2009} uses a periodic, self-adjoint, semibounded operator in factorized form and the bottom of its spectrum. The error identity of Lemma~\ref{lem:id} needs neither. It uses only a uniform bound for $\Ae^{-1}$, hypothesis \ref{C}, and the invertibility of $\Az$. This section gives two further examples: a first-order system that is not of factorized second-order type (Maxwell), and a coefficient without a sign (indefinite). The Dirac model of Chapter~\ref{sec:top} is already of the second kind (not semibounded, not coercive).

\section{Maxwell's System}\label{sec:maxwell}

\paragraph{Setting.} We write $\epsilon$ for the permittivity, not to be confused with the scale $\varepsilon$. Let $\Hs=L^2(\R^3\times\T^3;\C^6)$, $U=(E,H)$, and
\[
\sJ(E,H)=(-\operatorname{curl}H,\ \operatorname{curl}E),\qquad\sM=s\operatorname{diag}(\epsilon(y),\mu(y)),\quad s>0,
\]
with $\epsilon,\mu$ real symmetric, $\Z^3$-periodic, $\lambda|\xi|^2\le\xi\cdot\epsilon\xi\le\Lambda|\xi|^2$, and likewise for $\mu$. The operator $\sJ$, with domain $H(\operatorname{curl})^2$ in $x$, is skew-adjoint because $\operatorname{curl}$ is self-adjoint on $H(\operatorname{curl},\R^3)$. The problem $(\sJ+\sM_\varepsilon)U=F$ is the stationary Maxwell system $-\operatorname{curl}H+s\epsilon_\varepsilon E=F_1$, $\operatorname{curl}E+s\mu_\varepsilon H=F_2$. Here the spectral parameter $s$ enters multiplicatively, and no additive shift is needed: \ref{S} and \ref{S0} hold with $\Ps=0$, and \ref{H4} holds with $\sT=\sT'=I$, $c=s\lambda$ and $C=0$ (Remark~\ref{rem:H4rel}), so $C_A=1/(s\lambda)$. Since $\Hflu$ contains fluctuations of the state itself ($y$-gradients, see below), an additive shift $\alpha I$, or $\alpha$ times the projection onto $E$, would violate \ref{S0}. We take $\Jmac$, $\Jmic$ and $\Se$ as in Section~\ref{sec:3.1}.

\paragraph{The spectral structure of $\Jmic$.} On the $y$-mode $k\neq0$, the symbol of $-i\Jmic$ is $2\pi\left(\begin{smallmatrix}0&-K\\K&0\end{smallmatrix}\right)$ with $Kv=k\times v$. Its square is $4\pi^2\operatorname{diag}(-K^2,-K^2)$, and $-K^2=|k|^2(I-\hat k\hat k^T)$. So the eigenvalues are $\pm2\pi|k|$ (each double) and $0$ (double, on $\{(t\hat k,t'\hat k)\}$). Hence:
\begin{itemize}
\item $\ker\Jmic=\{y\text{-constants}\}\oplus\{(\nabla_y\phi,\nabla_y\psi)\}$;
\item $\Hflu$ consists of pairs of mean-zero $y$-gradients;
\item $\Hgra$ consists of mean-zero, $y$-divergence-free pairs;
\item \ref{G1} holds with $\lambda_0=2\pi$.
\end{itemize}

\paragraph{Verification of (H1)--(H4).}
\begin{itemize}
\item \ref{H1} and \ref{H2}(i) follow from the description above.
\item \ref{H2}(ii): the $y$-average commutes with $x$-derivatives.
\item \ref{H2}(iii): $\Jmac(\nabla_y\phi,\nabla_y\psi)=(-\operatorname{curl}_x\nabla_y\psi,\ \operatorname{curl}_x\nabla_y\phi)$ has $y$-mean zero and is $y$-divergence-free, since $\operatorname{div}_y\operatorname{curl}_x\nabla_y\psi=\epsilon_{ijk}\partial_{y_i}\partial_{x_j}\partial_{y_k}\psi=0$.
\item \ref{H2}(iv): $\Pflu$ acts only in $y$.
\item \ref{H3}: as in Section~\ref{sec:3.1}.
\item \ref{S}, \ref{S0} and \ref{H4}: as just explained, with $\Ps=0$ and $c=s\lambda$.
\end{itemize}

\paragraph{Corrector and effective coefficient.} For $U_0=(E_0,H_0)\in\Hmac$ and $w=(\nabla_y\phi,\nabla_y\psi)\in\Hflu$, the corrector equation $\Pflu\sM(U_0+w)=0$ says that $\epsilon(E_0+\nabla_y\phi)$ and $\mu(H_0+\nabla_y\psi)$ are orthogonal to all mean-zero $y$-gradients. Equivalently, $\operatorname{div}_y\epsilon(E_0+\nabla_y\phi)=0$ and $\operatorname{div}_y\mu(H_0+\nabla_y\psi)=0$. These are the scalar cell problems of Section~\ref{sec:3.2}: $\phi=\chi^\epsilon_jE_{0,j}$ and $\psi=\chi^\mu_jH_{0,j}$, with $\operatorname{div}_y\epsilon(e_j+\nabla_y\chi^\epsilon_j)=0$, and similarly for $\mu$. The factor $s$ cancels. Hence
\[
\Kmic U_0=(\nabla_y\chi^\epsilon_jE_{0,j},\ \nabla_y\chi^\mu_jH_{0,j}),\qquad\Mz=s\operatorname{diag}(\epsilon_{\hom},\mu_{\hom}).
\]

\begin{lemma}[Residual and flux corrector]\label{lem:maxres}
Let $b^\epsilon_j:=\epsilon(e_j+\nabla_y\chi^\epsilon_j)-\epsilon_{\hom}e_j$, which is $y$-divergence-free with mean zero, and let $A^\epsilon_j:=-\Delta_y^{-1}\operatorname{curl}_yb^\epsilon_j$, so that $\operatorname{curl}_yA^\epsilon_j=b^\epsilon_j$. Define $b^\mu_j$, $A^\mu_j$ in the same way. For $U_0\in\Hmac$ with $\nabla_xU_0\in L^2$,
\[
\Ns U_0=\Big(-s\,b^\epsilon_jE_{0,j}+\nabla_xH_{0,j}\times\nabla_y\chi^\mu_j,\ \ -s\,b^\mu_jH_{0,j}-\nabla_xE_{0,j}\times\nabla_y\chi^\epsilon_j\Big),
\]
and $\Phi:=\big(-s\,A^\mu_jH_{0,j}+\chi^\epsilon_j\nabla_xE_{0,j},\ \ s\,A^\epsilon_jE_{0,j}+\chi^\mu_j\nabla_xH_{0,j}\big)$ satisfies
\[
\Jmic\Phi=\Ns U_0,\qquad\Jmac\Phi=\big(-s\,\nabla_xE_{0,j}\times A^\epsilon_j,\ -s\,\nabla_xH_{0,j}\times A^\mu_j\big).
\]
\end{lemma}

\begin{proof}
We use three identities for a scalar $\chi(y)$, a scalar $h(x)$, and a vector field $A(y)$:
\[
\begin{gathered}\operatorname{curl}_y(\chi c)=\nabla_y\chi\times c\ \ (c\text{ constant in }y),\\ \operatorname{curl}_x(A(y)h(x))=\nabla_xh\times A,\qquad\operatorname{curl}_x\big(\chi(y)\nabla_xh(x)\big)=0 .\end{gathered}
\]
Also $\operatorname{curl}_y\operatorname{curl}_y=\nabla_y\operatorname{div}_y-\Delta_y$ and $\operatorname{div}_yb=0$ give $\operatorname{curl}_yA=b$.
\begin{itemize}
\item \emph{Residual.} $\Mz U_0-\sM(U_0+\Kmic U_0)=(-s\,b^\epsilon_jE_{0,j},-s\,b^\mu_jH_{0,j})$, by definition of $b$. Also $-\Jmac\Kmic U_0=(\operatorname{curl}_x(\nabla_y\chi^\mu_jH_{0,j}),-\operatorname{curl}_x(\nabla_y\chi^\epsilon_jE_{0,j}))$, and the second identity turns this into the cross products.
\item \emph{Microscopic derivative.} $\Jmic\Phi=(-\operatorname{curl}_y\Phi_H,\operatorname{curl}_y\Phi_E)$. Here $-\operatorname{curl}_y\Phi_H=-s\,b^\epsilon_jE_{0,j}-\nabla_y\chi^\mu_j\times\nabla_xH_{0,j}$, which is the first component of $\Ns U_0$; the second component is analogous.
\item \emph{Macroscopic derivative.} $\Jmac\Phi=(-\operatorname{curl}_x\Phi_H,\operatorname{curl}_x\Phi_E)$, and the second and third identities give the formula.
\end{itemize}
\end{proof}

\begin{proposition}[Periodic Maxwell system, lifted estimate]\label{prop:maxper}
Let $\Hsrc:=\{F=(F_1,F_2)\in\Hmac:\operatorname{div}F_1=\operatorname{div}F_2=0\}$. There is $C=C(s,\lambda,\Lambda)$ such that, for all $\varepsilon>0$ and $F\in\Hsrc$,
\[
\nrm{\big(\Ae^{-1}-(I+\Ke)\Az^{-1}\big)F}\le C\varepsilon\nrm F .
\]
\end{proposition}

\begin{proof}
Let $U_0=\Az^{-1}F$. Taking the divergence of the effective equations gives $s\operatorname{div}(\epsilon_{\hom}E_0)=\operatorname{div}F_1=0$ and $s\operatorname{div}(\mu_{\hom}H_0)=0$. Moreover $\operatorname{curl}E_0=F_2-s\mu_{\hom}H_0$ and $\operatorname{curl}H_0=s\epsilon_{\hom}E_0-F_1$ are bounded in $L^2$ by $C\nrm F$.

We claim that, for $\xi\neq0$, $|\xi||\hat E|\le C(|\xi\times\hat E|+|\xi\cdot\epsilon_{\hom}\hat E|)$. Write $\hat E=(\hat\xi\cdot\hat E)\hat\xi+\hat E_\perp$, so that $|\hat E_\perp|=|\hat\xi\times\hat E|$. Then $\lambda|\hat\xi\cdot\hat E|\le|\hat\xi\cdot\epsilon_{\hom}\hat E|+\Lambda'|\hat E_\perp|$, where $\Lambda'$ bounds $\epsilon_{\hom}$; this proves the claim. Hence $\nrm{\nabla E_0}\le C\nrm{\operatorname{curl}E_0}\le C\nrm F$, and the same holds for $H_0$.

So $\Kmic U_0\in\dom\Jmac$, and by Lemma~\ref{lem:maxres} and the tensor identity, $\nrm\Phi+\nrm{\Jmac\Phi}\le C\nrm{\nabla_xU_0}\le C\nrm F$, with zero defect. Here $\chi^\epsilon_j,\chi^\mu_j\in H^1(\T^3)$ and $A^\epsilon_j,A^\mu_j\in H^1(\T^3)$ are bounded by the cell-problem constants. Lemma~\ref{lem:id} (with $C_A=1/(s\lambda)$) concludes, as in Corollary~\ref{cor:gaugegap}.
\end{proof}

\begin{remark}\label{rem:maxwellcomp}
\begin{itemize}
\item As in the elliptic case (Proposition~\ref{prop:SK}), the natural antiderivative costs one derivative of $U_0$ only. The mechanism is again an identity of the de Rham complex, $\operatorname{curl}\nabla=0$.
\item The restriction to divergence-free sources is the usual one for the stationary Maxwell system. Operator error estimates for the periodic stationary Maxwell system in $\R^3$ go back to Birman and Suslina; for a discussion and references, and for bounded domains (with divergence-free right-hand sides), see \cite{Suslina2018a,Suslina2018b}. Proposition~\ref{prop:maxper} concerns the lifted problem with the corrector $\Ke$; we have not compared it in detail with these results, and we claim no novelty for the periodic case.
\end{itemize}
\end{remark}

\begin{proposition}[Random Maxwell system, annealed]\label{prop:maxrand}
In the setting of Section~\ref{sec:stoch} with $d=3$, let $\epsilon(x,\omega)=\tilde\epsilon(\tau_x\omega)$ and $\mu(x,\omega)=\tilde\mu(\tau_x\omega)$ be real symmetric with the bounds above. Let $\mathcal G_{\mathrm{Max}}:=\mathcal G_{\mathrm{ell}}(\epsilon)\cup\mathcal G_{\mathrm{ell}}(\mu)$ be the corrector gradients and fluxes of the two scalar problems \eqref{eq:Gell}. Then, for $F\in\Hsrc$ and every $\theta>0$,
\[
\nrm{\big(\Ae^{-1}-(I+\Ke)\Az^{-1}\big)F}_{L^2(\R^3\times\Omega)}\le C\big(\varepsilon\fa_{\mathcal G_{\mathrm{Max}}}(\theta)+\fd_{\mathcal G_{\mathrm{Max}}}(\theta)\big)\nrm F .
\]
In particular, annealed norm-resolvent convergence holds on $\Hsrc$, and under (DOS)$_\gamma$ for $\mathcal G_{\mathrm{Max}}$ the rates \eqref{eq:rates} hold.
\end{proposition}

\begin{proof}
Replace $\chi^\bullet_j$ by $\phi^{\bullet,j}_\theta:=-(\theta-\Delta)^{-1}D\cdot g^{\bullet,j}$, and $A^\bullet_j$ by $A^\bullet_{\theta,j}:=(\theta-\Delta)^{-1}\operatorname{curl}_Db^\bullet_j$. Since $\operatorname{div}_Db^\bullet_j=0$, $\operatorname{curl}_DA^\bullet_{\theta,j}=b^\bullet_j-\theta(\theta-\Delta)^{-1}b^\bullet_j$. Moreover $D\phi^{\bullet,j}_\theta=g^{\bullet,j}-e^{\bullet,j}_\theta$ with $e^{\bullet,j}_\theta=\theta(\theta-\Delta)^{-1}g^{\bullet,j}$ (proof of Proposition~\ref{prop:skewrand}). With these replacements in $\Phi$, the computation of Lemma~\ref{lem:maxres} gives $\Jmac\Phi_\theta=(-s\nabla_xE_{0,j}\times A^\epsilon_{\theta,j},\,-s\nabla_xH_{0,j}\times A^\mu_{\theta,j})$ and the defect
\[
\begin{split}
\Jmic\Phi_\theta-\Ns U_0=\Big(&s\theta(\theta-\Delta)^{-1}b^\epsilon_jE_{0,j}+e^{\mu,j}_\theta\times\nabla_xH_{0,j},\\&s\theta(\theta-\Delta)^{-1}b^\mu_jH_{0,j}-e^{\epsilon,j}_\theta\times\nabla_xE_{0,j}\Big).
\end{split}
\]
Every term has the tensor form of the proof of Proposition~\ref{prop:skewrand}. So $\nrm{\Phi_\theta}+\nrm{\Jmac\Phi_\theta}\le C\fa(\theta)\nrm F$ and $\nrm{\Jmic\Phi_\theta-\Ns U_0}\le C\fd(\theta)\nrm F$, using $\nrm{\nabla_xU_0}\le C\nrm F$ as in Proposition~\ref{prop:maxper}. Lemma~\ref{lem:id} and Lemma~\ref{lem:specint} conclude.
\end{proof}

\begin{remark}\label{rem:maxrandlit}
The density-of-states input for Proposition~\ref{prop:maxrand} is the same as for the two scalar conductivity problems with coefficients $\epsilon$ and $\mu$; no new probabilistic input is needed beyond \ref{E2} for these fields. For random media, qualitative two-scale results for abstract first-order systems are in \cite{NeukammVargaWaurick2021}, and a quantitative stochastic theory for differential forms is developed in \cite{Dario2018}. We are not aware of an annealed norm-resolvent estimate for the first-order Maxwell system in the literature we consulted, but our search was limited and we make no priority claim. As in Corollary~\ref{cor:ellrand}, the fibre at $\omega$ is the physical system, so the estimate transfers to the physical operators in the annealed sense.
\end{remark}

\section{Coefficients Without a Sign}\label{sec:indef}

Sign-changing coefficients model negative-index metamaterials. Their homogenization has been studied by T-coercivity methods \cite{BunoiuEtAl2021,BonnetBenDhiaEtAl2012}. Hypothesis \ref{H4}(ii)--(iii) is a T-coercivity-type condition, with the additional requirement that the multiplier be compatible with $\sJ$, $\Se$ and the projections. In one dimension, however, a uniform bound can be obtained directly from the oscillation of the coefficient, without any multiplier.

\begin{proposition}[One dimension, indefinite coefficient]\label{prop:indef}
Let $a$ be real and $1$-periodic, with $a^{-1}\in L^\infty(\T)$ and $m:=\langle a^{-1}\rangle\neq0$; no sign condition is imposed on $a$. Let $\Im z\neq0$. Consider $p'-zu=f$, $u'+a(x/\varepsilon)^{-1}p=0$, i.e.\ $-(a_\varepsilon u')'-zu=f$, lifted with $\sM=\operatorname{diag}(-z,a(y)^{-1})$.
\begin{enumerate}[label=(\roman*)]
\item There is $C_A$, depending on $|z|$, $|\Im z|$, $|m|$ and $\nrm{\tilde\chi}_\infty$, where $\tilde\chi'=a^{-1}-m$ and $\langle\tilde\chi\rangle=0$, such that $\nrm{\Ae^{-1}}\le C_A$ for $0<\varepsilon\le1$.
\item $\Kmic=0$, $\Mz=\operatorname{diag}(-z,m)$ and $a_{\hom}=1/m$, of either sign. $\Az$ is invertible on $\Hmac$, and $u_0\in H^2$ with $\nrm{u_0}_{H^2}\le C\nrm f$.
\item $\nrm{(\Ae^{-1}-\Az^{-1})(f,0)}\le C\varepsilon\nrm f$ for $0<\varepsilon\le1$.
\item For the physical operators, $\nrm{(-\partial a_\varepsilon\partial-z)^{-1}f-(-a_{\hom}\partial^2-z)^{-1}f}_{L^2(\R)}\le C\varepsilon\nrm f$ for $0<\varepsilon\le1$.
\end{enumerate}
\end{proposition}

\begin{proof}
(i) Fix the phase $y$, and let $(u,p)$ solve the system with right-hand side $(g,h)\in L^2(\R)^2$, where $a_\varepsilon=a(\cdot/\varepsilon+y)$. Write $X=\nrm u$, $P=\nrm p$ and $N=\nrm g+\nrm h$. Pairing the first equation with $u$ and using $u'=h-a_\varepsilon^{-1}p$ gives
\[
\ip{a_\varepsilon^{-1}p}{p}-\ip ph-z\nrm u^2=\ip gu .
\]
\begin{itemize}
\item \emph{Imaginary part.} Since $a$ is real, it gives $|\Im z|X^2\le N(X+P)$.
\item \emph{Real part.} Write $a_\varepsilon^{-1}-m=\varepsilon\partial_x[\tilde\chi(\cdot/\varepsilon+y)]$ and integrate by parts:
\[
|\ip{(a_\varepsilon^{-1}-m)p}{p}|\le2\varepsilon\nrm{\tilde\chi}_\infty P\nrm{p'},\qquad\nrm{p'}\le\nrm g+|z|X .
\]
Hence $|m|P^2\le PN+|z|X^2+NX+2\varepsilon\nrm{\tilde\chi}_\infty P(N+|z|X)$.
\end{itemize}
By the imaginary-part bound, $X\le N/|\Im z|+(NP/|\Im z|)^{1/2}$. Inserting this, every term on the right is of the form $CN^aP^b$ with $b<2$ and $a+b=2$ (for $\varepsilon\le1$). Young's inequality then gives $P\le CN$, hence $X\le CN$, with $C$ independent of $\varepsilon\in(0,1]$ and of $y$. The adjoint system, with $\bar z$, satisfies the same estimate, so the fibre operators are bijective, and $\nrm{\Ae^{-1}}\le C_A$.

(ii) $\Hflu=\{0\}$ in one dimension, so $\Kmic=0$ and $\Mz$ is the $y$-average of $\sM$. The symbol of $\Az$ has determinant $\xi^2-zm$, and $|\xi^2-zm|\ge c(1+\xi^2)$ because $\Im(zm)=m\Im z\neq0$. Solving gives $\hat u_0=m\hat f/(\xi^2-zm)$, so $u_0\in H^2$.

(iii) The residual is $\Ns U=(\Mz-\sM)U=(0,(m-a^{-1})p_0)=(0,g(y)u_0')$, with $g=a_{\hom}/a-1$, which is bounded and has mean zero, as in Appendix~\ref{app:1d}. With $\Phi=(\chi u_0',0)$, $\chi'=g$, the defect is zero, and $\Jmac\Phi=(0,\chi u_0'')$. Lemma~\ref{lem:id} needs only $\nrm{\Ae^{-1}}\le C_A$ (Remark~\ref{rem:shiftgen}), so $\nrm{\Ee(f,0)}\le C\varepsilon\nrm{u_0}_{H^2}$.

(iv) Let $u^y_\varepsilon$ denote the $u$-component of the fibre at $y$. The function $u_\varepsilon(\cdot+\varepsilon y)$ solves the phase-$y$ problem with source $f(\cdot+\varepsilon y)=f+G'$, where $G(x):=\int_x^{x+\varepsilon y}f$ and $\nrm G\le\varepsilon|y|\nrm f$. The pair $(u_\varepsilon(\cdot+\varepsilon y),\ p(\cdot+\varepsilon y)-G)$ minus the solution with source $f$ solves the system with data $(0,-a_\varepsilon^{-1}G)$. By (i), $\nrm{u_\varepsilon(\cdot+\varepsilon y)-u^y_\varepsilon}\le C\varepsilon|y|\nrm f$. Since translations are unitary and $\nrm{u_0(\cdot+\varepsilon y)-u_0}\le\varepsilon|y|\nrm{u_0'}$, integrating over $y\in[0,1)$ and using (iii) gives (iv).
\end{proof}

\paragraph{Numerical test.} Let $a=1$ on half of the cell and $a=-\kappa$ on the other half, with $z=i$ and $f=e^{-x^2}$. We use the finite-difference scheme of Appendix~\ref{app:num} on $[-30,30]$ with $40$ points per period, and $\varepsilon=2^{-2},\dots,2^{-7}$.

\begin{center}
\begin{tabular}{cccc}
\hline
$\kappa$ & $a_{\hom}$ & $\nrm{u_\varepsilon-u_0}_{L^2}$ at $\varepsilon=2^{-2}$, $2^{-7}$ & fitted slope\\\hline
$3$ & $3$ & $2.38\cdot10^{-2}$, $7.43\cdot10^{-4}$ & $1.00$\\
$1/3$ & $-1$ & $4.20\cdot10^{-2}$, $1.31\cdot10^{-3}$ & $1.00$\\\hline
\end{tabular}
\end{center}

The flux norm $\nrm{p_\varepsilon}$ varies by less than $0.1\%$ over this range of $\varepsilon$, consistent with (i). With $20$ and $80$ points per period the errors change by less than $1\%$. For $\kappa=1/3$ the effective operator is $+d^2/dx^2-z$: the homogenized medium is negative although half of the microstructure is positive.

\begin{remark}[Higher dimensions: a conditional statement]\label{rem:indefhigher}
For $d\ge2$, let $a$ be real, periodic and possibly sign-changing, and suppose:
\begin{enumerate}[label=(\alph*)]
\item a uniform bound $\sup_{0<\varepsilon\le1}\nrm{\Ae^{-1}}\le C_A$ for the lifted system;
\item well-posedness of the cell problems, i.e.\ hypothesis \ref{C};
\item $a_{\hom}$ is definite (positive or negative), so that $\Az$ is invertible and $\nrm{u_0}_{H^2}\le C\nrm f$ for $\Im z\neq0$.
\end{enumerate}
Then the construction of Proposition~\ref{prop:skew}, which is purely algebraic, and Lemma~\ref{lem:id} give $O(\varepsilon)$ for $L^2$ data.
\end{remark}

\section{Evolutionary Equations}\label{sec:evol}

\paragraph{Setting.} For $\rho>0$ and a Hilbert space $X$, let $L^2_\rho(\R;X)$ be the space of $U$ with $e^{-\rho t}U\in L^2(\R;X)$. The Fourier--Laplace transform
\[
(\mathcal L_\rho U)(\omega):=(2\pi)^{-1/2}\int_\R e^{-(i\omega+\rho)t}U(t)\,dt
\]
is unitary from $L^2_\rho(\R;X)$ onto $L^2(\R;X)$. It turns the time derivative into multiplication: $\partial_{t,\rho}=\mathcal L_\rho^*(i\omega+\rho)\mathcal L_\rho$ \cite{Picard2009}. In Picard's framework \cite{Picard2009}, an evolutionary equation $(\partial_{t,\rho}\mathsf P+\sJ)U=F$ with $\sJ$ skew-adjoint and $\mathsf P$ bounded, self-adjoint and $\mathsf P\ge c_P>0$ is well posed in $L^2_\rho(\R;X)$, and the solution is causal. Its solution is $U=\mathcal L_\rho^*\big[(z\mathsf P+\sJ)^{-1}\mathcal L_\rho F\big]$ with $z=i\omega+\rho$. We apply the point-free theory on each line $\Re z=\rho$, with the material law $\sM(z):=z\mathsf P$, where $\mathsf P=\mathsf P(y)$ is periodic.

\begin{lemma}[Uniformity on vertical lines]\label{lem:evol}
Assume \ref{H1}--\ref{H3}, and let $\mathsf P$ be bounded and self-adjoint with $\mathsf P\ge c_P>0$. For $z=\rho+i\omega$ with $\rho>0$, the operator $\sM(z)=z\mathsf P$ satisfies \ref{S}, \ref{S0} and \ref{H4} with $\Ps=0$, $\sT=\sT'=I$, $c=\rho c_P$, $C=0$ and $\Mb=|z|\nrm{\mathsf P}$. Hence $C_A=1/(\rho c_P)$ for every $\omega\in\R$. Moreover, $\Kmic$ does not depend on $z$, and $\Mz(z)=z\mathsf P_0$, where $\mathsf P_0$ is the Schur complement \eqref{eq:5.1} of $\mathsf P$.
\end{lemma}

\begin{proof}
$\Re\ip{z\mathsf Pv}{v}=\rho\ip{\mathsf Pv}{v}\ge\rho c_P\nrm v^2$, and the same holds for $\sM(z)^*=\bar z\mathsf P$. Theorem~\ref{thm:corr} gives $C_A=1/(\rho c_P)$. In $\Kmic=-(\Pflu z\mathsf P|_{\Hflu})^{-1}\Pflu z\mathsf P\Pmac$ the factor $z$ cancels.
\end{proof}

\begin{theorem}[Time-domain estimate from frequency-explicit bounds]\label{thm:evol}
Assume the hypotheses of Lemma~\ref{lem:evol} and let $\Hsrc\subset\Hmac$ be closed. Suppose that for every $z$ with $\Re z=\rho$ and every $F\in\Hsrc$ there is $\Phi_z\in G^0_F$ (with zero defect, for $u=\Az(z)^{-1}F$) such that
\begin{equation}\label{eq:evolass}
\Big(1+\frac{|z|\nrm{\mathsf P}}{\rho c_P}\Big)\nrm{\Phi_z}+\frac{\nrm{\Jmac\Phi_z}}{\rho c_P}\le K\Big(1+\frac{|z|}\rho\Big)^2\nrm F .
\end{equation}
Let $F\in L^2_\rho(\R;\Hsrc)$ with $\partial_t^2F\in L^2_\rho(\R;\Hsrc)$. Let $U_\varepsilon$ and $U_0$ solve $(\partial_{t,\rho}\mathsf P_\varepsilon+\sJ)U_\varepsilon=F$ and $(\partial_{t,\rho}\mathsf P_0+\Jz)U_0=F$, where $\mathsf P_\varepsilon=\Se\mathsf P\Se^{-1}$. Then
\[
\nrm{U_\varepsilon-(I+\Ke)U_0}_{L^2_\rho(\R;\Hs)}\le2K\varepsilon\Big(\nrm F_{L^2_\rho(\R;\Hs)}+\rho^{-2}\nrm{\partial_t^2F}_{L^2_\rho(\R;\Hs)}\Big).
\]
\end{theorem}

\begin{proof}
By Lemma~\ref{lem:id} with zero defect and Lemma~\ref{lem:evol}, for $\hat F(\omega):=(\mathcal L_\rho F)(\omega)$ and $z=\rho+i\omega$,
\[
\nrm{\big(\Ae(z)^{-1}-(I+\Ke)\Az(z)^{-1}\big)\hat F(\omega)}\le\varepsilon K\Big(1+\frac{|z|}\rho\Big)^2\nrm{\hat F(\omega)} .
\]
Since $\Ke$ does not depend on $z$, it commutes with $\mathcal L_\rho$, and $\mathcal L_\rho(U_\varepsilon-(I+\Ke)U_0)(\omega)$ is the vector on the left. By Plancherel, together with $(1+|z|/\rho)^2\le2(1+|z|^2/\rho^2)$ and $\mathcal L_\rho\partial_t^2F=z^2\hat F$,
\[
\nrm{U_\varepsilon-(I+\Ke)U_0}_{L^2_\rho}\le2\varepsilon K\big(\nrm{\hat F}_{L^2}+\rho^{-2}\nrm{z^2\hat F}_{L^2}\big),
\]
which is the claim.
\end{proof}

\begin{corollary}[Wave equation]\label{cor:wave}
Let $a$ be real, symmetric, $\Z^d$-periodic, with $\lambda\le a\le\Lambda$. In the lifted setting of Section~\ref{sec:3.1}, let $\mathsf P=\operatorname{diag}(1,a(y)^{-1})$ and $F=(f,0)$. Then $U_\varepsilon=(v_\varepsilon,p_\varepsilon)$ solves the lifted form of
\[
\partial_tv+\nabla\cdot p=f,\qquad a_\varepsilon^{-1}\partial_tp+\nabla v=0,
\]
i.e.\ $v=\partial_tu$ and $p=-a_\varepsilon\nabla u$, where $\partial_t^2u-\nabla\cdot a_\varepsilon\nabla u=f$ with $u$ vanishing in the far past. There is $C=C(d,\lambda,\Lambda)$, independent of $\rho$, such that, for $f\in L^2_\rho(\R;L^2(\R^d))$ with $\partial_t^2f\in L^2_\rho$,
\[
\nrm{U_\varepsilon-(I+\Ke)U_0}_{L^2_\rho(\R;\Hs)}\le C\varepsilon\Big(\nrm f_{L^2_\rho}+\rho^{-2}\nrm{\partial_t^2f}_{L^2_\rho}\Big).
\]
\end{corollary}

\begin{proof}
Here $c_P=\min\{1,\Lambda^{-1}\}$ and $\nrm{\mathsf P}=\max\{1,\lambda^{-1}\}$. The corrector problem is that of Section~\ref{sec:3.1}, so $\Kmic U_0=(0,Bp_0)$ and $\mathsf P_0=\operatorname{diag}(1,a_{\hom}^{-1})$.

\emph{Residual and gauge.} The second row of $\Az(z)U_0=F$ reads $\nabla v_0+z\,a_{\hom}^{-1}p_0=0$. Hence, as in Proposition~\ref{prop:skew}, $\Ns U_0=(-B_{jk}\partial_jp_{0,k},\ (\partial_{y_i}\chi_j)\partial_jv_0)$, and the skew gauge $\Phi_z=(\chi_j\partial_jv_0,\,-\sigma^k_{ij}\partial_jp_{0,k})$ satisfies $\Jmic\Phi_z=\Ns U_0$ and $\Jmac\Phi_z=(0,\chi_j\nabla\partial_jv_0)$. So $\nrm{\Phi_z}\le C(\nrm{\nabla v_0}+\nrm{\nabla p_0})$ and $\nrm{\Jmac\Phi_z}\le C\nrm{\nabla^2v_0}$.

\emph{Frequency bounds.} Eliminating $p_0=-a_{\hom}\nabla v_0/z$ gives $(z^2+\varkappa(\xi))\hat v_0=z\hat f$, with $\varkappa(\xi)=\xi\cdot a_{\hom}\xi$. Since $\Re\big((z^2+\varkappa)/z\big)=\rho+\varkappa\rho/|z|^2\ge\rho$, we have $|z^2+\varkappa|\ge\rho|z|$. We bound three multipliers.
\begin{itemize}
\item $|\xi|\,|\hat v_0|/|\hat f|=|\xi||z|/|z^2+\varkappa|$. If $\varkappa\ge2|z|^2$, then $|z^2+\varkappa|\ge\varkappa/2$ and $|\xi|\ge c|z|$, so the multiplier is at most $C$. Otherwise $|\xi|\le C|z|$, and it is at most $C|z|/\rho$.
\item $|\xi|\,|\hat p_0|/|\hat f|\le C|\xi|^2/|z^2+\varkappa|\le C\varkappa/|z^2+\varkappa|\le C(1+|z|^2/|z^2+\varkappa|)\le C(1+|z|/\rho)$.
\item $|\xi|^2|\hat v_0|/|\hat f|\le C\varkappa|z|/|z^2+\varkappa|\le C|z|(1+|z|/\rho)$.
\end{itemize}
Hence the left side of \eqref{eq:evolass} is at most
\[
C\Big(1+\frac{|z|}{\rho}\Big)^2\nrm{\hat f}+\frac{C|z|}{\rho}\Big(1+\frac{|z|}\rho\Big)\nrm{\hat f}\le C\Big(1+\frac{|z|}\rho\Big)^2\nrm{\hat f}.
\]
Theorem~\ref{thm:evol} applies, with $\Hsrc=\{(f,0)\}$.
\end{proof}

\begin{corollary}[Time-dependent Maxwell system]\label{cor:maxevol}
In the setting of Section~\ref{sec:maxwell}, let $\mathsf P=\operatorname{diag}(\epsilon(y),\mu(y))$, so that $(\partial_{t,\rho}\mathsf P_\varepsilon+\sJ)U=F$ is $\partial_t(\epsilon_\varepsilon E)-\operatorname{curl}H=F_1$, $\partial_t(\mu_\varepsilon H)+\operatorname{curl}E=F_2$. For $F\in L^2_\rho(\R;\Hsrc)$ with $\partial_t^2F\in L^2_\rho$, where $\Hsrc$ consists of divergence-free sources as in Proposition~\ref{prop:maxper}, the estimate of Theorem~\ref{thm:evol} holds with $K=C(\lambda,\Lambda)$ independent of $\rho$.
\end{corollary}

\begin{proof}
Lemma~\ref{lem:maxres} holds with $s$ replaced by $z$. From $z\operatorname{div}(\epsilon_{\hom}E_0)=\operatorname{div}F_1=0$, $\operatorname{curl}E_0=F_2-z\mu_{\hom}H_0$ and $\nrm{U_0}\le\nrm F/(\rho\lambda)$, the argument of Proposition~\ref{prop:maxper} gives $\nrm{\nabla U_0}\le C(1+|z|/\rho)\nrm F$. Hence $\nrm{\Phi_z}\le C(|z|\nrm{U_0}+\nrm{\nabla U_0})\le C(1+|z|/\rho)\nrm F$ and $\nrm{\Jmac\Phi_z}\le C|z|\nrm{\nabla U_0}\le C|z|(1+|z|/\rho)\nrm F$, and \eqref{eq:evolass} follows as in Corollary~\ref{cor:wave}.
\end{proof}

\begin{remark}[Scope]\label{rem:evol}
\begin{itemize}
\item \emph{Price of time.} The flux-corrector bounds grow like $(1+|z|/\rho)^2$, and this is paid for with two time derivatives of the source. In exchange, the source needs only $L^2$ regularity in space (for the wave equation), and the constant does not depend on $\rho$. The estimate holds in the exponentially weighted norm of $L^2_\rho(\R;\Hs)$, which controls $\int_{-\infty}^T\nrm{\cdot}^2\,dt$ up to the factor $e^{2\rho T}$.
\item \emph{Comparison.} Operator error estimates for periodic hyperbolic equations in $\R^d$, of the form $\nrm{\cos(\tau\mathcal A_\varepsilon^{1/2})-\cos(\tau\mathcal A_0^{1/2})}_{H^q\to L^2}\le C(1+|\tau|)\varepsilon$ with spatially regular data, together with a discussion of their sharpness, are due to Dorodnyi, Suslina and Meshkova \cite{DorodnyiSuslina2018,DorodnyiSuslina2021,Meshkova2017}. Corollary~\ref{cor:wave} is of a different type (weighted in time, time-regular data, lifted with corrector), and we have not compared the two in detail. The reference \cite{NeukammVargaWaurick2021} is qualitative.
\item \emph{Physical estimate and randomness.} We have stated lifted estimates; unlifting proceeds as in Corollary~\ref{cor:unlift}, frequency by frequency. In the random setting, the regularized gauge of Proposition~\ref{prop:skewrand} gives, by the same proof, $\varepsilon\fa(\theta)+\fd(\theta)$ in place of $\varepsilon$. This yields annealed rates under \ref{E2}, in particular unconditionally under the assumptions of Theorem~\ref{thm:frd}.
\end{itemize}
\end{remark}

\part{Structure and Limits}

\chapter{The Homogenization Flow}\label{sec:flow}

This section collects observations on the group structure of Definition~\ref{def:homo-syst}. They are not used in the proofs. Where we use physical analogies, they are interpretations, not statements.

\paragraph{The generator.} In the periodic case, $(\mathsf T_tU)(x,y)=U(x,y+tx)$, so formally $\mathsf T_t=e^{t\,x\cdot\nabla_y}$: a unitary group whose generator is the skew-adjoint closure of $x\cdot\nabla_y$ (``position times microscopic momentum''). In the random case the generator is $x\cdot D$. By the chain rule,
\[
\mathsf T_t^{-1}\nabla_x\mathsf T_t=\nabla_x+t\nabla_y,\qquad t=1/\varepsilon .
\]
This resembles the action of a Galilean boost on momentum, or a minimal coupling $\nabla_x\mapsto\nabla_x+t\mathcal A$ with ``vector potential'' $\mathcal A=\nabla_y$ and coupling constant $t$.

\begin{proposition}[Fibre decomposition of the lifted operator]\label{prop:fibre}
Suppose that $\sJ(\nabla)=\sum_k\sJ_k\partial_k$ has constant matrix coefficients, and that $\sM=\sM(y)$ does not depend on $x$. Let $\mathcal F_x$ denote the Fourier transform in $x$. Then
\[
\mathcal F_x\big(\Jmac+\varepsilon^{-1}\Jmic+\sM\big)\mathcal F_x^{-1}=\int^\oplus_{\R^d}\varepsilon^{-1}\Big[\sJ(\nabla_y+i\varepsilon\xi)+\varepsilon\sM(y)\Big]\,d\xi .
\]
\end{proposition}

\begin{proof}
$\Jmac=\sum_k\sJ_k\partial_{x_k}$ becomes $\sum_k\sJ_k\,i\xi_k$, and $\Jmic=\sum_k\sJ_k\partial_{y_k}$ is unchanged. Hence the sum is $\varepsilon^{-1}\sum_k\sJ_k(\partial_{y_k}+i\varepsilon\xi_k)+\sM$.
\end{proof}

The fibre at macroscopic frequency $\xi$ is $\varepsilon^{-1}$ times the Floquet--Bloch fibre, at quasimomentum $\varepsilon\xi$, of the unit-cell operator $\sJ(\nabla)+\varepsilon\sM$. Along the flow, each fixed $\xi$ is thus moved towards zero quasimomentum, which is the threshold regime of Birman and Suslina \cite{BirmanSuslina2009}. The lift followed by a Fourier transform in $x$ is a Gelfand-type transform, so the Bloch decomposition is \emph{recovered} here, not avoided. What the point-free framework adds is that the lifted operator is defined without it, and that, in a random medium, where there is no Brillouin zone, the same fibre operator $\sJ(D+i\varepsilon\xi)+\varepsilon\sM$ on $L^2(\Omega)$ is still defined. Proposition~\ref{prop:fibre} requires $x$-independent $\sM$; it does not apply to the non-stationary Dirac model.

\paragraph{The Schur complement of the Bloch fibre.} We work in the periodic setting of Proposition~\ref{prop:fibre}, with $\sM=\sM(y)$ bounded and $x$-independent (any shift $\alpha\Ps$ is included in $\sM$). On $L^2(\T^d;\C^N)$ we use the fibre versions of the projections:
\begin{itemize}
\item $\mathsf P_0$ onto constants (the fibre of $\Hmac$);
\item $\mathsf F$ onto the mean-zero part of $\ker\sJ(\nabla_y)$ (the fibre of $\Hflu$);
\item $\mathsf G$ onto $\overline{\ran}\,\sJ(\nabla_y)$ (the fibre of $\Hgra$).
\end{itemize}
For $\xi\in\R^d$ and $\varepsilon>0$, the fibre operator of Proposition~\ref{prop:fibre} is $\varepsilon^{-1}\mathsf N_\varepsilon(\xi)$, with $\mathsf N_\varepsilon(\xi):=\sJ(\nabla_y+i\varepsilon\xi)+\varepsilon\sM$.

\begin{proposition}[The effective symbol as a Schur complement of the Bloch fibre]\label{prop:germ}
Assume \ref{G1} with gap $\lambda_0$, assume $\nrm{\Mff^{-1}}\le C_F$ for $\Mff=\mathsf F\sM\mathsf F|_{\ran\mathsf F}$, and assume that $\sJ(i\xi)$ maps $\ran\mathsf F$ into $\ran\mathsf G$ for every $\xi$ (the fibre form of \ref{H2}(iii)). There are $c_0,C>0$, depending only on $\lambda_0$, $C_F$, $\nrm\sM$ and $\max_k|\sJ_k|$, such that, if $\varepsilon(1+|\xi|)^2\le c_0$, the Schur complement $\mathsf S_\varepsilon(\xi)$ of $\mathsf N_\varepsilon(\xi)$ onto $\ran\mathsf P_0$ exists and satisfies
\[
\Big|\varepsilon^{-1}\mathsf S_\varepsilon(\xi)-\big(\sJ(i\xi)+\Mz\big)\Big|\le C\,\varepsilon(1+|\xi|)^2 .
\]
Here $\Mz$ is the homogenized coefficient \eqref{eq:5.1}, a constant $N\times N$ matrix.
\end{proposition}

\begin{proof}
The operator $\sJ(i\varepsilon\xi)$ is a constant matrix, so it preserves each Fourier mode. Hence it has no blocks between $\mathsf P_0$ (mode $0$) and $\mathsf F\oplus\mathsf G$ (modes $\neq0$). On $\ran\mathsf P_0$ it equals $\varepsilon\sJ(i\xi)$, its $\mathsf F\mathsf F$ block vanishes by assumption, and $\sJ(\nabla_y)$ vanishes on $\ran(\mathsf P_0+\mathsf F)$. Write $\delta:=\varepsilon(1+|\xi|)$.

\emph{Eliminating $\mathsf G$.} The $\mathsf G\mathsf G$ block is $\sJ(\nabla_y)_{\mathsf G}+O(\delta)$. Since $\nrm{\sJ(\nabla_y)_{\mathsf G}^{-1}}\le\lambda_0^{-1}$, it is invertible with inverse $O(1)$ if $\delta$ is small. The couplings of $\mathsf G$ to $\mathsf P_0$ are $\varepsilon\sM_{\mathsf P_0\mathsf G}=O(\varepsilon)$, and those to $\mathsf F$ are $O(\delta)$. After this elimination:
\begin{itemize}
\item the $\mathsf P_0\mathsf P_0$ block is $\varepsilon(\sJ(i\xi)+\sM_{\mathsf P_0\mathsf P_0})+O(\varepsilon^2)$;
\item the $\mathsf P_0\mathsf F$ blocks are $\varepsilon\sM_{\mathsf P_0\mathsf F}+O(\varepsilon\delta)$;
\item the $\mathsf F\mathsf F$ block is $\varepsilon\big(\Mff+O(\delta^2/\varepsilon)\big)$, where $\delta^2/\varepsilon=\varepsilon(1+|\xi|)^2$.
\end{itemize}

\emph{Eliminating $\mathsf F$.} If $\varepsilon(1+|\xi|)^2\le c_0$, the $\mathsf F\mathsf F$ block is invertible with inverse $\varepsilon^{-1}(\Mff^{-1}+O(\varepsilon(1+|\xi|)^2))$. The Schur complement onto $\mathsf P_0$ is then
\[
\varepsilon\big(\sJ(i\xi)+\sM_{\mathsf P_0\mathsf P_0}-\sM_{\mathsf P_0\mathsf F}\Mff^{-1}\sM_{\mathsf F\mathsf P_0}\big)+O\big(\varepsilon^2(1+|\xi|)^2\big),
\]
and the bracket is $\sJ(i\xi)+\Mz$ by \eqref{eq:5.1}.
\end{proof}

\begin{corollary}[The Birman--Suslina germ]\label{cor:germ}
Let $\sJ$ be the elliptic operator of Section~\ref{sec:3.1} and $\sM=\operatorname{diag}(0,a(y)^{-1})$, with $a$ symmetric. This is the first-order form of $-\nabla\cdot a\nabla$, without shift. Eliminating also the flux component of $\ran\mathsf P_0$ from $\sJ(i\xi)+\Mz=\left(\begin{smallmatrix}0&i\xi^T\\ i\xi&a_{\hom}^{-1}\end{smallmatrix}\right)$ gives the scalar
\[
0-(i\xi)^Ta_{\hom}(i\xi)=\xi\cdot a_{\hom}\xi=|\xi|^2\,S(\hat\xi),
\]
where $S(\theta)=\theta\cdot g^0\theta$, $g^0=a_{\hom}$, is the spectral germ of Birman and Suslina for $-\nabla\cdot a\nabla$ \cite{BirmanSuslina2009}. Consequently, by Proposition~\ref{prop:germ}, the Schur complement of $\mathsf N_\varepsilon(\xi)$ onto the constant scalars equals $\varepsilon\,|\xi|^2S(\hat\xi)+O(\varepsilon^2(1+|\xi|)^4)$. In terms of the quasimomentum $k=\varepsilon\xi$, the leading term is $\varepsilon^{-1}|k|^2S(\hat k)$. The statement is an asymptotic one for fixed $\xi$ as $\varepsilon\to0$, which is the threshold regime $k\to0$.
\end{corollary}

\begin{proof}
The flux block $a_{\hom}^{-1}+O(\varepsilon(1+|\xi|)^2)$ is invertible, and the off-diagonal blocks are $i\xi+O(\varepsilon(1+|\xi|)^2)$. The Schur complement in the scalar variable is therefore $\xi\cdot a_{\hom}\xi$ up to an error $O(\varepsilon(1+|\xi|)^4)$, after the factor $\varepsilon$ is removed. That $g^0$ equals the classical effective matrix $a_{\hom}$ is part of \cite{BirmanSuslina2009}.
\end{proof}

\begin{remark}\label{rem:germ}
\begin{itemize}
\item \emph{What the proposition adds.} That the Birman--Suslina effective matrix coincides with the classical one is part of their theory. Proposition~\ref{prop:germ} shows that the germ is produced by the same elimination that defines $\Mz$, now applied to the Bloch fibre (a Feshbach--Schur reduction), with an explicit remainder. The fluctuations $\mathsf F$ are eliminated at order $\varepsilon$, and the stiff gradients $\mathsf G$ at order $\varepsilon^2$, through the gap.
\item \emph{What it does not add.} We do not rederive the threshold eigenvalue asymptotics $\lambda_1(k)=|k|^2S(\hat k)+O(|k|^3)$ of \cite{BirmanSuslina2009}, which rest on analytic perturbation theory. Proposition~\ref{prop:germ} is a statement about Schur complements, not about eigenvalues.
\item \emph{Random media.} In a random medium the fibre operator $\sJ(D+i\varepsilon\xi)+\varepsilon\sM$ on $L^2(\Omega)$ still exists, but \ref{G1} fails. The $\mathsf G$-elimination is then no longer uniform, and this is exactly where the regularization and the density of states of Section~\ref{sec:stoch} enter.
\end{itemize}
\end{remark}

\begin{lemma}[The material law along the flow]\label{lem:heis}
Put $\sM(t):=\mathsf T_t\sM\mathsf T_t^{-1}$, i.e.\ $\sM(1/\varepsilon)=\Me$. For $u,v\in\Hmac$ and all $t>0$, $\ip{\sM(t)u}{v}=\ip{\Pmac\sM\Pmac u}{v}$.
\end{lemma}

\begin{proof}
$\mathsf T_t$ fixes $\Hmac$, so $\ip{\mathsf T_t\sM\mathsf T_t^{-1}u}{v}=\ip{\sM\mathsf T_t^{-1}u}{\mathsf T_t^{-1}v}=\ip{\sM u}{v}$.
\end{proof}

So the macroscopic compression of the material law is constant along the flow and equals the arithmetic mean $\Pmac\sM\Pmac$. On the other hand, by Theorem~\ref{thm:qual}, the macroscopic compression of the \emph{resolvent} converges: $\Pmac(\sJ+\sM(t)+\alpha)^{-1}f\to(\Jz+\Mz+\alpha)^{-1}f$ for $f\in\Hmac$, with the Schur complement $\Mz$ in place of the arithmetic mean. The difference between the two is
\[
\Pmac\sM\Pmac-\Mz=-\Pmac\sM\Kmic\Pmac,
\]
the corrector term. One may read this as a resolvent-sense ergodic statement for the conjugated material law; we use this only as an interpretation.

\begin{lemma}[Spectral bookkeeping for one field]\label{lem:bookkeep}
Let $g$ be a stationary mean-zero field and put $F(\theta):=\int\frac{|\xi|^2}{\theta+|\xi|^2}\,d\mu_g$, $\fa(\theta)^2:=\int\frac{|\xi|^2\,d\mu_g}{(\theta+|\xi|^2)^2}$ and $\fd(\theta)^2:=\int\frac{\theta^2\,d\mu_g}{(\theta+|\xi|^2)^2}$. Then
\[
F'(\theta)=-\fa(\theta)^2,\qquad\fd(\theta)^2=\nrm g^2-F(\theta)+\theta F'(\theta).
\]
\end{lemma}

\begin{proof}
Differentiate under the integral (dominated convergence) to get the first identity. For the second, $1-\frac{|\xi|^2}{\theta+|\xi|^2}-\frac{\theta|\xi|^2}{(\theta+|\xi|^2)^2}=\frac{\theta(\theta+|\xi|^2)-\theta|\xi|^2}{(\theta+|\xi|^2)^2}=\frac{\theta^2}{(\theta+|\xi|^2)^2}$, and $\mu_g(\R^d)=\nrm g^2$ because $\mu_g(\{0\})=0$.
\end{proof}

Heuristically, $F(\theta)$ is the part of the variance of $g$ carried by frequencies $|\xi|\gtrsim\sqrt\theta$, $\fa^2=-F'$ is the rate at which it changes with $\theta$, and $\fd^2$ measures what remains near zero frequency. The choices of $\theta$ in Lemma~\ref{lem:specint} ($\theta=\varepsilon^2$ for $\gamma\le2$, $\theta=\varepsilon^{4/\gamma}$ for $\gamma>2$) balance these quantities.

\paragraph{Comparison with coarse-graining.} Table~\ref{tab:AK} compares, schematically, the present framework with the coarse-graining theory of Armstrong and Kuusi \cite{ArmstrongKuusi2025a,ArmstrongKuusi2025b}. The analogies in the table are interpretations; we do not claim that they are precise correspondences.

\begin{table}[ht]\centering\small
\renewcommand{\arraystretch}{1.25}
\begin{tabular}{>{\raggedright\arraybackslash}p{2.4cm}>{\raggedright\arraybackslash}p{4.2cm}>{\raggedright\arraybackslash}p{4.2cm}}
\hline
& Coarse-graining \cite{ArmstrongKuusi2025a,ArmstrongKuusi2025b} & Point-free framework\\\hline
scale parameter & blocks of increasing size & $t=1/\varepsilon$; regularization $\theta$\\
coarse-grained object & coarse-grained matrices $\mathbf a(\square),\mathbf a_*(\square)$, defined variationally for each realization & Schur complement $\Mz$; fields $(\theta-\Delta)^{-1}g$ (annealed)\\
type of procedure & iteration over real-space blocks & unitary flow; elimination of fluctuations\\
rate input & proved, by iteration across scales & assumed: (DOS)$_\gamma$, or variance bounds (Proposition~\ref{prop:varDOS})\\
contrast & high contrast treated & bounded $\sM$ only\\\hline
\end{tabular}
\caption{Schematic comparison with coarse-graining.}\label{tab:AK}
\end{table}

\chapter{What the Framework Cannot Do}\label{sec:limits}

This chapter collects the limitations of the method. Section~\ref{sec:quenched} shows that random estimates on the whole space must be annealed or localized. Section~\ref{sec:bdd} shows what the method gives on bounded domains. Section~\ref{sec:scope} records the structural restrictions of the axioms.

\section{Quenched Statements and Their Limits}\label{sec:quenched}

The estimates of Theorem~\ref{thm:ellrand}, Corollary~\ref{cor:ellrand} and Corollary~\ref{cor:randphys} are annealed: they bound $\sup_{\nrm h\le1}\big(\E\nrm{\mathsf X^\omega_\varepsilon h}^2\big)^{1/2}$ for the random resolvent difference $\mathsf X^\omega_\varepsilon$. We show that this yields almost-sure bounds for \emph{localized} resolvent differences. We then show that almost-sure bounds in the full operator norm on $L^2(\R^d)$ fail in general, so that some localization or averaging is unavoidable.

\begin{proposition}[Localized quenched bounds]\label{prop:quenched}
Let $\{\mathsf X^\omega_\varepsilon\}_{\varepsilon>0}$ be bounded operators on $L^2(\R^d;\C^N)$, strongly measurable in $\omega$, with
\[
\sup_{\nrm h\le1}\big(\E\nrm{\mathsf X^\omega_\varepsilon h}^2\big)^{1/2}\le\kappa_\varepsilon .
\]
Let $\mathsf K$ be a Hilbert--Schmidt operator. Then $\big(\E\nrm{\mathsf X^\omega_\varepsilon\mathsf K}_{\mathrm{HS}}^2\big)^{1/2}\le\kappa_\varepsilon\nrm{\mathsf K}_{\mathrm{HS}}$. Consequently, for $\varepsilon_n:=2^{-n}$, almost surely
\[
\nrm{\mathsf X^\omega_{\varepsilon_n}\mathsf K}_{L^2\to L^2}\le n\,\kappa_{\varepsilon_n}\nrm{\mathsf K}_{\mathrm{HS}}\qquad\text{for all sufficiently large }n .
\]
For $\mathsf K=1_B(1-\Delta)^{-s}$ with $B\subset\R^d$ bounded and $s>d/4$, one has $\nrm{\mathsf K}_{\mathrm{HS}}=(2\pi)^{-d/2}|B|^{1/2}\nrm{(1+|\xi|^2)^{-s}}_{L^2(\R^d)}\sqrt N<\infty$.
\end{proposition}

\begin{proof}
Let $(e_i)$ be an orthonormal basis. By monotone convergence, $\E\nrm{\mathsf X\mathsf K}_{\mathrm{HS}}^2=\sum_i\E\nrm{\mathsf X\mathsf Ke_i}^2\le\kappa_\varepsilon^2\sum_i\nrm{\mathsf Ke_i}^2=\kappa_\varepsilon^2\nrm{\mathsf K}_{\mathrm{HS}}^2$. By Chebyshev's inequality, $\mathbb P\big(\nrm{\mathsf X_{\varepsilon_n}\mathsf K}_{\mathrm{HS}}>n\kappa_{\varepsilon_n}\nrm{\mathsf K}_{\mathrm{HS}}\big)\le n^{-2}$, which is summable; the Borel--Cantelli lemma and $\nrm\cdot\le\nrm\cdot_{\mathrm{HS}}$ give the claim. The formula for $\nrm{\mathsf K}_{\mathrm{HS}}$ is the standard identity $\nrm{f(x)g(-i\nabla)}_{\mathrm{HS}}=(2\pi)^{-d/2}\nrm f_{L^2}\nrm g_{L^2}$.
\end{proof}

Applied to $\mathsf X^\omega_\varepsilon=(\mathsf H^\omega_\varepsilon)^{-1}-\mathsf H_0^{-1}$ (Corollary~\ref{cor:ellrand}), or to the Dirac resolvent differences of Corollary~\ref{cor:randphys}, Proposition~\ref{prop:quenched} gives almost-sure convergence, with rate $O(\kappa_\varepsilon\log(1/\varepsilon))$ along dyadic $\varepsilon$, of the resolvents acting on data that are localized in a fixed bounded set and slightly smoothed. The rates of Theorems~\ref{thm:ellrand} and \ref{thm:verified} transfer accordingly.

\begin{proposition}[No quenched norm-resolvent convergence on $\R^d$ in general]\label{prop:noquenched}
Let $a$ be a random checkerboard: $a(x,\omega)=a_j(\omega)I$ for $x\in j+U+[0,1)^d$, $j\in\Z^d$, where the $a_j$ are i.i.d., take the values $\lambda_1<\lambda_2$ with positive probabilities, and $U$ is uniform on $[0,1)^d$ and independent of them. Then there is $c_0>0$ such that, almost surely, $\nrm{(\mathsf H^\omega_\varepsilon)^{-1}-\mathsf H_0^{-1}}_{L^2\to L^2}\ge c_0$ for every $\varepsilon>0$.
\end{proposition}

\begin{proof}
1. \emph{Rare regions.} For each $n\in\mathbb N$, the events ``$a\equiv\lambda_1$ on the cube $Q$'', for infinitely many disjoint cubes $Q$ of side $n$ made of cells, are independent and each has probability $p^{n^d}>0$. Hence almost surely such a cube exists for every $n$.

2. \emph{A local test.} By the Voigt--Reuss bounds, $a_{\hom}$ is bounded below by the harmonic mean of the $a_j$, which is strictly larger than $\lambda_1$. So the symbols $(1+\lambda_1|\xi|^2)^{-1}$ and $(1+\xi\cdot a_{\hom}\xi)^{-1}$ differ for $\xi\neq0$. Fix $f_0\in C_c^\infty(B_1)$ with $\nrm{f_0}=1$ and $c_1:=\nrm{((-\lambda_1\Delta+1)^{-1}-\mathsf H_0^{-1})f_0}>0$; by translation invariance, the same $c_1$ works for every translate $f_0(\cdot-x_0)$.

3. \emph{Localization.} Fix $\varepsilon>0$ and $R\ge2$. By step~1, with $n\ge R/\varepsilon$, there is $x_0$ such that $a(\cdot/\varepsilon)\equiv\lambda_1$ on $B(x_0,R)$. Let $f:=f_0(\cdot-x_0)$, $v:=(-\lambda_1\Delta+1)^{-1}f$, and $w:=(\mathsf H^\omega_\varepsilon)^{-1}f-v$. Then $-\nabla\cdot a(\cdot/\varepsilon)\nabla w+w=\nabla\cdot\big((a(\cdot/\varepsilon)-\lambda_1)\nabla v\big)$. The right-hand side involves only $\nabla v$ outside $B(x_0,R)$, where $|\nabla v|$ decays like $e^{-|x-x_0|/\sqrt{\lambda_1}}$. By the energy estimate, $\nrm w\le C(\lambda_2-\lambda_1)\nrm{\nabla v}_{L^2(\R^d\setminus B(x_0,R))}\le Ce^{-cR}$, with $C,c$ independent of $\varepsilon$, $\omega$ and $x_0$. Hence $\nrm{((\mathsf H^\omega_\varepsilon)^{-1}-\mathsf H_0^{-1})f}\ge c_1-Ce^{-cR}\ge c_1/2=:c_0$ for $R$ large.
\end{proof}

\begin{remark}\label{rem:noquenched}
\begin{itemize}
\item On $\R^d$, a fixed realization contains arbitrarily large regions where the medium looks homogeneous but different from $a_{\hom}$, and data concentrated there are not homogenized at any scale $\varepsilon$. The annealed formulation, which averages the error for each fixed datum, and the localized quenched formulation of Proposition~\ref{prop:quenched} are therefore the natural ones. For perturbations by first-order terms with arbitrary (not necessarily periodic or random) coefficients, Borisov \cite{Borisov2023} characterizes norm-resolvent convergence by the convergence of local mean values of the coefficients; Proposition~\ref{prop:noquenched} concerns the principal part, which that characterization does not cover.
\item For the stability arguments of \cite{BalDang2024}, which use a fixed realization, Proposition~\ref{prop:quenched} provides the localized form; Corollary~\ref{cor:spec}(iv) then holds almost surely for $f(\mathsf H)$ composed with localizing Hilbert--Schmidt operators.
\end{itemize}
\end{remark}
\begin{remark}[Random alloys]\label{rem:balalloys}
For random alloys, Bal \cite{Bal2026b} takes a different route. He proves deterministic norm-resolvent bounds, uniform over configurations, towards operators with configuration-dependent local effective coefficients \cite[Theorem~2.8]{Bal2026b}. He shows that when the law charges a periodic pattern with a gapless homogenized operator, rare regions put spectrum in the homogenized gap, so that no deterministic gapped operator is an almost-sure norm-resolvent limit \cite[Theorem~2.10]{Bal2026b}; this is a Dirac counterpart of Proposition~\ref{prop:noquenched}. He then proves that the interface index along cuts avoiding the rare regions equals the bulk-difference invariant of the operator homogenized in the Born approximation \cite[Theorems~2.13 and~2.14]{Bal2026b}.
\end{remark}

\section{Bounded Domains}\label{sec:bdd}

Let $\mathcal O\subset\R^d$ be a bounded domain with $C^3$ boundary, let $a$ be $\Z^d$-periodic as in Section~\ref{sec:3.2}, and consider the Dirichlet problem $u-\nabla\cdot(a(x/\varepsilon)\nabla u)=f$ in $\mathcal O$, $u=0$ on $\partial\mathcal O$. We lift to $\Hs:=L^2(\mathcal O\times\T^d;\R^{d+1})$ with $\Se$ as in \eqref{eq:79}, $x\in\mathcal O$. We take $\sJ=\Jmac$ as in \eqref{eq:38}, acting in $x$, with domain
\[
\dom\sJ=\big\{(u,p):\ u\in L^2(\T^d;H^1_0(\mathcal O)),\ p\in L^2(\T^d;H(\mathrm{div},\mathcal O))\big\},
\]
which is skew-adjoint, since $\nabla:H^1_0\to L^2$ and $-\nabla\cdot:H(\mathrm{div})\to L^2$ are mutually adjoint. The operator $\Jmic$ is as in Section~\ref{sec:3.1}. The boundary condition is carried by $u$ only, it is preserved by $\Se$, and \ref{H1}--\ref{H4} are verified as in Section~\ref{sec:3.1}. The effective problem is the Dirichlet problem for $\mathsf H_0=-\nabla\cdot a_{\hom}\nabla+1$ in $\mathcal O$.

\begin{proposition}[Boundary layers cost $\varepsilon^{1/2}$]\label{prop:bdd}
There is $C=C(d,\lambda,\Lambda,\mathcal O)$ such that, for $0<\varepsilon\le1$ and $f\in H^1(\mathcal O)$,
\[
\nrm{\big(\Ae^{-1}-(I+\Ke)\Az^{-1}\big)(f,0)}_{L^2(\mathcal O\times\T^d)}\le C\varepsilon^{1/2}\nrm f_{H^1(\mathcal O)} .
\]
\end{proposition}

\begin{proof}
1. \emph{Effective solution and corrector.} By elliptic regularity, $u_0\in H^3(\mathcal O)\cap H^1_0(\mathcal O)$ with $\nrm{u_0}_{H^3}\le C\nrm f_{H^1}$, and $p_0=-a_{\hom}\nabla u_0\in H^2$. The corrector $\Kmic U=(0,Bp_0)$ needs no boundary condition, since the $p$-component lies in $H(\mathrm{div})$. Hence $\Kmic U\in\dom\Jmac$, and $\Ns U$ is given by the formula of Proposition~\ref{prop:verify}.

2. \emph{Cut-off gauge.} Let $\mathcal O_\delta:=\{x\in\mathcal O:\operatorname{dist}(x,\partial\mathcal O)<\delta\}$. Let $\zeta_\varepsilon\in C^\infty(\overline{\mathcal O})$ satisfy $\zeta_\varepsilon=0$ on $\mathcal O_\varepsilon$, $\zeta_\varepsilon=1$ outside $\mathcal O_{2\varepsilon}$, and $|\nabla\zeta_\varepsilon|\le C/\varepsilon$. Put $\tilde\Phi:=\zeta_\varepsilon\Phi^g_F$, with $\Phi^g_F=(\chi_j\partial_ju_0,\psi)$ from Proposition~\ref{prop:skew}. Its $u$-component vanishes near $\partial\mathcal O$, so $\tilde\Phi\in\mathcal D$. Since $\zeta_\varepsilon$ depends only on $x$, $\Jmic\tilde\Phi=\zeta_\varepsilon\Ns U$, so the defect is $\varrho=(\zeta_\varepsilon-1)\Ns U$. Moreover $\Jmac\tilde\Phi=\big(\nabla\zeta_\varepsilon\cdot\psi,\ (\nabla\zeta_\varepsilon)\chi_j\partial_ju_0+\zeta_\varepsilon\chi_j\nabla\partial_ju_0\big)$, because $\nabla_x\cdot\psi=0$.

3. \emph{Layer estimates.} For Lipschitz $\mathcal O$, $\nrm h_{L^2(\mathcal O_\delta)}\le C\delta^{1/2}\nrm h_{H^1(\mathcal O)}$. By the tensor identity:
\begin{itemize}
\item $\nrm{\tilde\Phi}\le\nrm{\Phi^g_F}\le C\nrm f$;
\item $\varepsilon\nrm{\Jmac\tilde\Phi}\le C\big(\nrm{\nabla p_0}_{L^2(\mathcal O_{2\varepsilon})}+\nrm{\nabla u_0}_{L^2(\mathcal O_{2\varepsilon})}\big)+C\varepsilon\nrm{u_0}_{H^2}\le C\varepsilon^{1/2}\nrm{u_0}_{H^3}$;
\item $\nrm\varrho\le\nrm{\Ns U}_{L^2(\mathcal O_{2\varepsilon}\times\T^d)}\le C\varepsilon^{1/2}\nrm{u_0}_{H^3}$.
\end{itemize}
Lemma~\ref{lem:id} (with $C_A=1/\mb$) concludes.
\end{proof}

\begin{remark}\label{rem:bdd}
\begin{itemize}
\item The lifted error contains the flux error, so it is comparable to an $H^1$-type error with a first-order corrector. For such errors, $\varepsilon^{1/2}$ is the classical boundary-layer order in the absence of boundary-layer correctors \cite{Shen2018}. Improving it requires boundary-layer analysis \cite{ArmstrongEtAl2017,GerardVaretMasmoudi2012,KenigPrange2015}, which the point-free framework does not provide.
\item For $L^2$ data, the splitting of Theorem~\ref{thm:quant}(ii), with $\kappa=C\varepsilon^{1/2}$, gives $O(\varepsilon^{1/4})$.
\item The sharp $O(\varepsilon)$ estimate for $\nrm{\mathsf H_\varepsilon^{-1}-\mathsf H_0^{-1}}_{L^2\to L^2}$ on bounded domains is obtained in the literature by duality arguments (see \cite{Shen2018} and references therein); it is not reproduced here.
\end{itemize}
\end{remark}

\section{Bounded contrast and the linear setting}\label{sec:scope}

Two restrictions are built into the axioms.

\paragraph{Bounded material operators.} Hypothesis \ref{H4}(i) requires $\nrm\sM\le\Mb$, and the constants $C_A$ and $C_*$ of Theorem~\ref{thm:corr} and Lemma~\ref{lem:id} depend on $\Mb$. High-contrast media, in which the material law degenerates or becomes unbounded as $\varepsilon\to0$, are therefore outside the scope of the present framework. This includes perforated domains with Dirichlet conditions on the holes, and hence both Darcy's law for Stokes flow in porous media and the critical-size perforations of Cioranescu and Murat, whose effective operators contain a zeroth-order term generated by the capacity of the holes. For high contrast by operator methods in the periodic case, see \cite{CherednichenkoCooper2016,CooperKamotskiSmyshlyaev2023}.

\paragraph{Linear Hilbert-space structure.} The framework is linear. Nonlinear problems are not covered.

\paragraph{Constrained systems.} The multiplier condition \ref{H4}(iii) asks for coercivity of the material on $\ran(I-\Ps)$. In saddle-point systems some components carry no material coefficient. In the first-order form of the Stokes system with oscillating viscosity, for instance, the pressure enters only through the differential structure. If $\Ps$ projects onto the velocity, then $\sM v=0$ for every $v$ with only a pressure component, so \ref{H4}(iii) fails for every bounded $\sT$. Such systems are not covered by the present framework.

\chapter{Summary}\label{sec:conclusion}\label{sec:summary}

This memoir builds a point-free framework for quantitative homogenization in which the abstract estimates use no spatial regularity of the heterogeneous problem, no Fourier/Bloch analysis, and no probabilistic structure; these enter only when the hypotheses are verified in examples, together with the regularity of the homogenized problem. The algebraic core consists of three pieces.
\begin{itemize}
\item The effective coefficient $\Mz$ is the Schur complement of $\sM$ with respect to the orthogonal decomposition $\Pz=\Hmac\oplus\Hflu$, and the abstract corrector $\Kmic$ is the associated block-elimination operator (Theorem~\ref{thm:corr}). Under a fluctuation-density condition this yields qualitative homogenization in the strong corrector form $\Ae^{-1}f-(I+\Ke)\Az^{-1}f\to0$ (Theorem~\ref{thm:qual}).
\item The error identity of Lemma~\ref{lem:id} solves the generalized Sylvester (commutator) equation of Theorem~\ref{thm:freq} on a single vector, through an antiderivative $\Phi$ of the residual with an arbitrary defect $\varrho=\Jmic\Phi-\Ns u$. The unbounded operator $\sJ$ is eliminated exactly as in Theorem~\ref{thm:freq}, and the error is controlled by $\varepsilon\nrm\Phi$, $\varepsilon\nrm{\Jmac\Phi}$ and $\nrm\varrho$. Theorem~\ref{thm:freq} itself is not used in the applications (Remark~\ref{rem:role23}).
\item Choosing $\Phi$ is where the media differ. With a spectral gap of $\Jmic$, $\Phi$ is an exact antiderivative (Corollaries~\ref{thm:main} and \ref{cor:gaugegap}); the gauge freedom $\Phi\mapsto\Phi+\Pz$ decides whether $\Jmac\Phi$ costs a derivative of the data, and hence whether the rate on $L^2$ data is $\varepsilon$ or $\varepsilon^{1/2}$ (Remark~\ref{rem:optimal}). Without a gap, $\Phi$ is a regularized antiderivative, and the rate is governed by the low-frequency behaviour of explicit spectral measures: those of $\Jmic$ on the residual (Theorem~\ref{thm:quant}) or, more effectively in stationary random media, those of finitely many stationary coefficient fields under the generators of the group action (Lemma~\ref{lem:specint}).
\end{itemize}

The two main applications are the following.
\begin{itemize}
\item For the second-order elliptic divergence operator (Chapter~\ref{sec:ell}) with periodic coefficients, a skew-symmetric gauge gives the sharp $O(\varepsilon)$ norm-resolvent rate for the lifted problem (Theorem~\ref{thm:31}) and for the physical operators (Corollary~\ref{cor:unlift}). For stationary ergodic coefficients, a regularized skew-symmetric flux corrector gives annealed norm-resolvent convergence for $L^2$ data with no further assumption, and, under the density-of-states hypothesis \ref{E2} on the corrector gradient and flux fields, the rates \eqref{eq:rates} (Theorem~\ref{thm:ellrand}, Corollary~\ref{cor:ellrand}). With exponent $\gamma=d$, which follows from central-limit-type variance bounds for spatial averages known under strong mixing assumptions (Remark~\ref{rem:E2}), these are $\varepsilon^{1/2}$, $\varepsilon(\log1/\varepsilon)^{1/2}$ and $\varepsilon$ in $d=1$, $d=2$ and $d\ge3$. We do not derive \ref{E2} here. Quasi-periodic coefficients are covered at the qualitative level; quantitative rates require \ref{E2}, which then encodes a small-divisor condition. We verify \ref{E2} directly in dimension one (Theorem~\ref{thm:verified}): for random coefficients with integrable covariance, which gives the rate $\varepsilon^{1/2}$, and for quasi-periodic coefficients with Diophantine frequencies, which gives the rate $\varepsilon$. In general dimension \ref{E2} follows from central-limit-type variance bounds for spatial averages (Proposition~\ref{prop:varDOS}).
\item For a first-order Dirac-type system modeling topological Anderson insulators (Chapter~\ref{sec:top}), the same machinery reproduces and extends the norm resolvent convergence result of Bal and Dang \cite{BalDang2024,BalDangCorrigendum} (see also \cite{Bal2026b}), with rate $O(\varepsilon)$ for every $z\in\C\setminus\R$, without Bloch/Floquet theory or asymptotic expansion (Corollary~\ref{cor:physD}). The only model-specific analytic input is a G\aa rding-type resolvent bound, uniform in $\varepsilon$ despite the singular potential (Lemma~\ref{lem:unif}). For stationary random potentials with a stationary antiderivative, we obtain annealed norm-resolvent convergence and, under the density-of-states hypothesis (R3), explicit rates (Theorem~\ref{thm:randlift}, Corollary~\ref{cor:randphys}). Hypothesis (R3) holds with $\gamma=2$ whenever the antiderivative has a finite range of dependence (Proposition~\ref{prop:R3}), for instance for smooth marked hard-core shot-noise potentials (Remark~\ref{rem:matern}). This gives the rate $\varepsilon(\log1/\varepsilon)^{1/2}$. For periodic potentials the norm-resolvent estimate yields $O(\varepsilon)$ convergence of spectra on compact sets, persistence of spectral gaps, and $O(\varepsilon)$ convergence of $f(\mathsf H_\varepsilon)$ for $f\in C_c^\infty$ and of spectral projections onto bounded intervals with endpoints in the resolvent set (Corollary~\ref{cor:spec}). These are the forms in which such estimates enter stability arguments for topological phases.
\end{itemize}

In this sense, the point-free framework separates which part of quantitative homogenization is algebraic (true for any $\sM$ satisfying \ref{H4} and any skew-adjoint $\sJ$ with the scaling structure) from which part is geometric or probabilistic: the existence of a good gauge, and the low-frequency behaviour of the relevant spectral measures. On the positive side, the structural condition \ref{SK} isolates what a good gauge requires: the corrector must be a microscopic co-boundary of a $2$-form, together with two identities of the de Rham complex. Moreover, under a spectral gap the $O(\varepsilon)$ rate is sharp (Theorem~\ref{thm:lower}). On the negative side, random estimates on the whole space cannot be upgraded to almost-sure operator-norm bounds without localization (Propositions~\ref{prop:quenched} and \ref{prop:noquenched}), and on bounded domains the method gives only the boundary-layer order $\varepsilon^{1/2}$ for the lifted estimate (Proposition~\ref{prop:bdd}). The numerical experiments of Appendix~\ref{app:num} reproduce the predicted rates $\varepsilon$ (periodic and quasi-periodic) and $\varepsilon^{1/2}$ (random) in dimension one.

\chapter*{Acknowledgements}
The ideas underlying this work originated many years ago when the first author was a graduate student under the guidance of his co-authors. However, they would never be realized into this work without the following people, whom the first author wants to send his heartfelt acknowledgement. First, to Guillaume Bal for numerous insightful discussions, one of which inspired the qualitative homogenization theorem. Second, to Valery P.\ Smyshlyaev for introducing him to the operator approach to homogenization in the tradition of Birman, Suslina, Zhikov, Pastukhova, and others. Finally, to Tho Nguyen Duc and Dinh Duong Nguyen, for the many hours spent studying functional calculus and spectral theory together. We also thank Oscar Lopez-Pamies for a stimulating discussion.

Google Gemini, Claude, and Grammarly were used for improving the exposition, grammar, and auditing the manuscript. All the output was checked by the authors, who are responsible for the correctness of statements and proofs. We thank Antoine Gloria for sending us 2 Claude reports on the manuscripts and helpful comments.

\appendix
\chapter{A Worked Example in Dimension One}\label{app:1d}

In dimension one every object of the framework is explicit. Let $a$ be $1$-periodic and measurable, with $\lambda\le a\le\Lambda$, and consider $u-(a(x/\varepsilon)u')'=f$ on $\R$.

\paragraph{Lifted structure.} $\Hs=L^2(\R\times\T;\C^2)$, $U=(u,p)$, $\sJ=\Jmac=\begin{pmatrix}0&\partial_x\\ \partial_x&0\end{pmatrix}$, $\Jmic=\begin{pmatrix}0&\partial_y\\ \partial_y&0\end{pmatrix}$, and $\sM=\operatorname{diag}(0,a(y)^{-1})$, with $\Ps=\Pi_u$ and $\alpha=1$. Hypothesis \ref{H4} holds with $\sT=\sT'=I$, $c=\Lambda^{-1}$, $C=0$ and $\Mb=\lambda^{-1}$, so $C_A=1/\mb$ with $\mb:=\min\{1,\Lambda^{-1}\}$. On the $y$-mode $k\neq0$, $\Jmic$ acts as $2\pi i\left(\begin{smallmatrix}0&k\\k&0\end{smallmatrix}\right)$, which is invertible with eigenvalues $\pm2\pi ik$. Therefore
\[
\begin{gathered}\Pz=\ker\Jmic=\Hmac=\{y\text{-independent fields}\},\qquad\Hflu=\{0\},\\ \Hgra=\{\text{fields with zero $y$-mean}\},\end{gathered}
\]
and \ref{G1} holds with $\lambda_0=2\pi$. A mean-zero divergence-free field on $\T$ vanishes, which is why $\Hflu=\{0\}$.

\paragraph{Corrector and effective coefficient.} Since $\Hflu=\{0\}$, $\Kmic=0$, and the Schur complement reduces to the average: $\Mz=\Pmac\sM\Pmac=\operatorname{diag}(0,\langle a^{-1}\rangle)$. Thus $a_{\hom}=\langle a^{-1}\rangle^{-1}$, the harmonic mean. In the first-order formulation the conserved quantity is the flux $p$, on which $\sM$ acts by $a^{-1}$, and averaging $a^{-1}$ is all that homogenization does in one dimension. The effective solution is $u_0-a_{\hom}u_0''=f$, $p_0=-a_{\hom}u_0'$.

\paragraph{Residual and flux corrector.} With $g:=a_{\hom}/a-1$ (mean zero),
\[
\Ns U=(\Mz-\sM)U=\big(0,(\langle a^{-1}\rangle-a^{-1})p_0\big)=\big(0,g(y)u_0'(x)\big).
\]
Let $\chi$ be the mean-zero solution of $\chi'=g$. Then the default flux corrector is $\Phi_F=(\chi(y)u_0'(x),0)\in\Hgra$, with $\Jmic\Phi_F=\Ns U$ and $\Jmac\Phi_F=(0,\chi(y)u_0''(x))$. Since $\Kmic=0$, the default gauge is already good. Fourier analysis gives $\nrm{u_0'}\le(2\sqrt{a_{\hom}})^{-1}\nrm f$ and $\nrm{u_0''}\le a_{\hom}^{-1}\nrm f$. Lemma~\ref{lem:id} then gives, with $C_*=1+(\Mb+1)/\mb$,
\[
\nrm{\Ee F}\le\varepsilon\nrm\chi_{L^2(\T)}\Big(\frac{C_*}{2\sqrt{a_{\hom}}}+\frac{1}{\mb\,a_{\hom}}\Big)\nrm f .
\]
\paragraph{Lower bound.} Here $\Xi_F=-\Jmic^{-1}\Phi_F=(0,-X(y)u_0'(x))$, where $X'=\chi$ and $\langle X\rangle=0$. Theorem~\ref{thm:lower} gives $\nrm{\Ee F}\ge\varepsilon\nrm\chi_{L^2(\T)}\nrm{u_0'}-C\varepsilon^2$, so the rate is exactly $\varepsilon$ unless $a$ is constant.

\paragraph{Physical operators.} Corollary~\ref{cor:unlift} reads
\[
\nrm{u_\varepsilon-u_0}\le C\varepsilon\nrm f,\qquad\Big\|u_\varepsilon'-\frac{a_{\hom}}{a(\cdot/\varepsilon)}\widetilde{\mathsf S}_\varepsilon u_0'\Big\|\le C\varepsilon\nrm f .
\]
\paragraph{Random and quasi-periodic coefficients.} Without periodicity, the corrector gradient $g=a_{\hom}/\tilde a-1$ is still explicit and stationary, while $\chi$ in general is not stationary. The regularized antiderivative of Proposition~\ref{prop:skewrand} is $\Phi_\theta=(\phi_\theta u_0',0)$ with $\phi_\theta=-(\theta-\Delta)^{-1}Dg$, and Theorem~\ref{thm:verified}(a),(b) gives the rates $\varepsilon^{1/2}$ and $\varepsilon$, respectively.

\paragraph{Bounded interval.} On $(0,1)$ with Dirichlet conditions, the corrector term $\chi(x/\varepsilon)u_0'(x)$ does not vanish at the endpoints. Proposition~\ref{prop:bdd} gives $\varepsilon^{1/2}$ for the lifted error.

\chapter{Numerical Illustration}\label{app:num}

We solve $u-(a(x/\varepsilon)u')'=f$ with $f(x)=e^{-x^2}$ on $[-12,12]$, with homogeneous Dirichlet conditions; the solutions decay like $e^{-|x|/\sqrt\Lambda}$, so the truncation error is negligible. We use a conservative second-order finite-difference scheme with $40$ grid points per period, with $a$ sampled at cell midpoints, and $\varepsilon=2^{-k}$, $k=2,\dots,7$. We compare $u_\varepsilon$ with the solution $u_0$ of the homogenized equation, where $a_{\hom}$ is the harmonic mean. We report $\nrm{u_\varepsilon-u_0}_{L^2}$ and the corrected-gradient error $\nrm{u_\varepsilon'-(a_{\hom}/a(\cdot/\varepsilon))u_0'}_{L^2}$. Three media are considered:
\begin{itemize}
\item \emph{periodic:} $a(y)=2+\sin(2\pi y)$, $a_{\hom}=\sqrt3$;
\item \emph{quasi-periodic:} $a(x)=3+\cos(2\pi(\omega_1+x))+\cos(2\pi(\omega_2+\varphi x))$, with $\varphi=(1+\sqrt5)/2$ (a Diophantine frequency, $\tau_0=1$) and $a_{\hom}\approx2.6040$ computed by quadrature on $\T^2$. The phase $\omega$ is uniform on $\T^2$, and the errors are root-mean-square averages over $8$ phases, as in Theorem~\ref{thm:verified}(b);
\item \emph{random checkerboard:} i.i.d.\ values in $\{1,4\}$ with equal probability on the cells $\varepsilon(j+U+[0,1))$, $a_{\hom}=1.6$, with root-mean-square averages over $40$ samples, as in Theorem~\ref{thm:verified}(a).
\end{itemize}

\begin{figure}[ht]
\centering
\includegraphics[width=\textwidth]{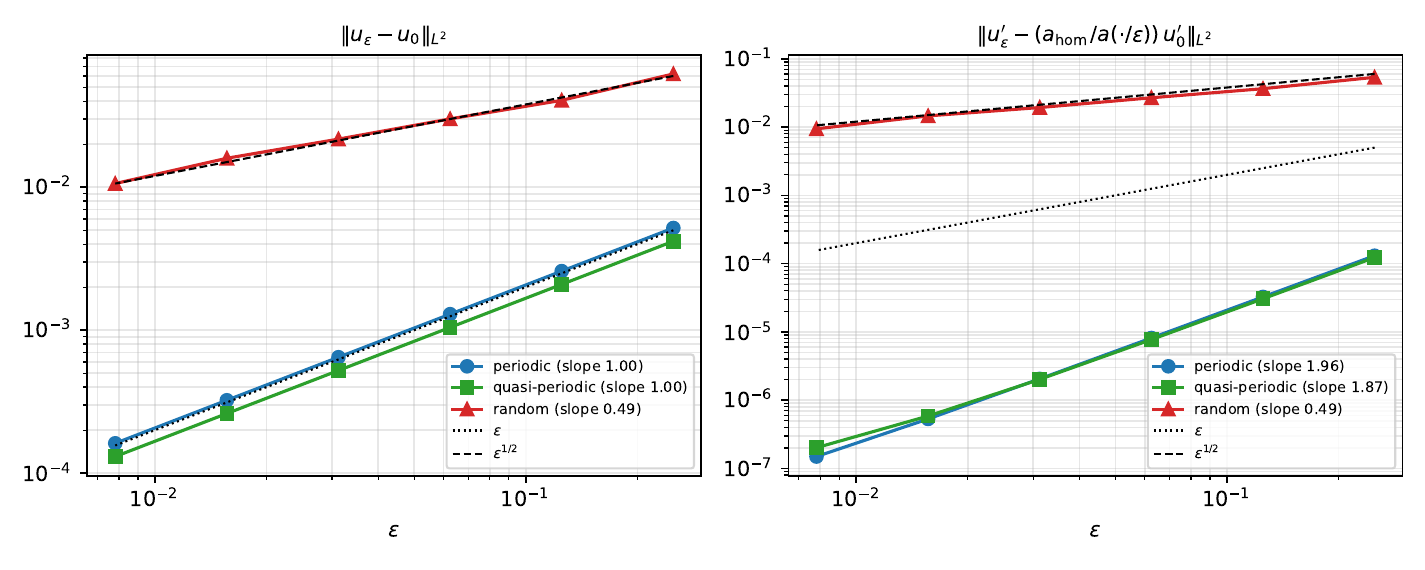}
\caption{Errors versus $\varepsilon$ in log-log scale; fitted slopes in the legends. Left: $\nrm{u_\varepsilon-u_0}_{L^2}$. Right: corrected-gradient error.}\label{fig:rates}
\end{figure}

\begin{center}
\begin{tabular}{lccc}
\hline
& periodic & quasi-periodic & random\\\hline
fitted slope, $\nrm{u_\varepsilon-u_0}_{L^2}$ & $1.00$ & $1.00$ & $0.49$\\
fitted slope, corrected gradient & $1.96$ & $1.87$ & $0.49$\\
predicted (upper bound) & $1$ & $1$ & $1/2$\\\hline
\end{tabular}
\end{center}

The $L^2$ errors follow the predicted rates $\varepsilon$, $\varepsilon$ and $\varepsilon^{1/2}$ (Corollary~\ref{cor:unlift}, Theorem~\ref{thm:verified}). In the periodic and quasi-periodic cases, the corrected-gradient error decays faster than the proved bound, at rate close to $\varepsilon^2$. This is a one-dimensional phenomenon: the corrected flux $a(\cdot/\varepsilon)u_\varepsilon'-a_{\hom}u_0'$ has derivative $u_\varepsilon-u_0$, whose leading term $\varepsilon\chi(x/\varepsilon)u_0'(x)$ oscillates with zero mean, so its antiderivative is of order $\varepsilon^2$. We do not use this observation elsewhere. In the random case, both errors decay at the central-limit rate $\varepsilon^{1/2}$, consistent with the exponent $\gamma=1$ in (DOS) for $g=a_{\hom}/a-1$.

\end{document}